\documentclass[a4paper,11pt,oneside]{scrreprt}
\usepackage{latexsym}
\usepackage{amssymb}
\usepackage[fleqn]{amsmath}
\usepackage{amscd}
\usepackage{amsthm}
\usepackage{float}
\usepackage{wrapfig}
\usepackage{graphicx}
\usepackage{graphicx,caption}
\usepackage{comment}
\usepackage{subfig}
\usepackage{chngcntr}
\usepackage{everyshi}% http://ctan.org/pkg/everyshi
\usepackage{changepage}
\usepackage{setspace}
\usepackage{bm}
\usepackage[toc,page]{appendix}
\graphicspath{ {figures/} }
\usepackage{array}
\RequirePackage{xparse, graphicx, caption, picins}
\usepackage[shortlabels]{enumitem}
\usepackage[a4paper,bindingoffset=0in,left=1.875in,right=0.5in,top=1in,bottom=1in,footskip=0in,
  headheight=17pt, % as per the warning by fancyhdr
includehead,includefoot,
heightrounded]{geometry}
\usepackage{fancyhdr}
\fancypagestyle{plain}{%
	\fancyhf{} % clear all header and footer fields
	\fancyhead{} % clear all header fields
	\fancyfoot{} % clear all footer fields
	\fancyhead[R]{\thepage}
	\renewcommand{\headrulewidth}{0pt}
	\renewcommand{\footrulewidth}{0pt}
}
\newif\ifcaptionlinebreak

\DeclareCaptionOption{linebreak}{\csname captionlinebreak#1\endcsname}

\newif\ifcaptionlinebreak

\DeclareCaptionOption{linebreak}{\csname captionlinebreak#1\endcsname}

\usepackage[export]{adjustbox}

\DeclareCaptionStyle{mystyle}[linebreak=false]{linebreak=true}
\usepackage{parskip}
\usepackage{atbegshi}

\makeatletter
\renewcommand\paragraph{\@startsection{paragraph}{4}{\z@}{-3.25ex \@plus -.05ex \@minus -0.2ex}{0.01pt}{\raggedsection\normalfont\sectfont\nobreak\size@paragraph}}
\makeatother

\newtheorem{theorem}{Theorem}[section]

\theoremstyle{definition}
\newtheorem{definition}[theorem]{Definition}

\theoremstyle{remark}

\numberwithin{equation}{section}

\def\thetitle{Relative Equilibria of Dumbbells Orbiting\\in a Planar Newtonian Gravitational System}
\def\theauthor{Jodin Morey}

\def\themonthyear{April, 2022}

\newcommand{\mathitem}{\hspace*{1.2em}&\bullet\hspace*{\labelsep}}
\newcounter{preamblecounter}
\EveryShipout{\stepcounter{preamblecounter}}% Step pagecntr every page
\protected\def\isgreaterthan#1{\ifnum#1>1000 #1\else0\fi}
\protected\def\isgreaterthanb#1{\ifnum#1<1000 #1\else0\fi}

\AtBeginShipout{\ifodd\isgreaterthan{\value{preamblecounter}}\edef\mytemp{\ht\AtBeginShipoutBox=\the\ht\AtBeginShipoutBox\relax
		\dp\AtBeginShipoutBox=\the\dp\AtBeginShipoutBox\relax
	}\sbox\AtBeginShipoutBox{\raisebox{0.7in}{\usebox\AtBeginShipoutBox}}\mytemp\fi
\ifodd\isgreaterthanb{\value{preamblecounter}}\edef\mytemp{\ht\AtBeginShipoutBox=\the\ht\AtBeginShipoutBox\relax
	\dp\AtBeginShipoutBox=\the\dp\AtBeginShipoutBox\relax
}\sbox\AtBeginShipoutBox{\raisebox{-0.8in}{\usebox\AtBeginShipoutBox}}\mytemp\fi
}
\newtheorem*{theorem*}{Theorem}
\begin{document}
%\newgeometry{twoside=false,left=1in, right=1in, top=1.125in, bottom=1.25in,headheight=0pt,includeall,textheight=9in,nomarginpar}
\raggedbottom 
\newlength{\oldparskip}\setlength\oldparskip{\parskip}\parskip=0.3in
\thispagestyle{empty}
\begin{center}
\textbf{\thetitle\\ \vspace{.4in}
A THESIS\\
SUBMITTED TO THE FACULTY OF\\ THE GRADUATE SCHOOL OF\\ THE UNIVERSITY OF MINNESOTA
\\BY \vspace{.1in}
\\ \theauthor}\\ \vspace{.1in}
\vfill
\noindent\normalsize
\textbf{IN PARTIAL FULFILLMENT OF THE REQUIREMENTS\\FOR THE DEGREE OF\\DOCTOR OF PHILOSOPHY\\ \vspace{.3in}
Richard Moeckel}\vspace{.3in}

\begin{center}
	\textbf{\themonthyear}
\end{center}
\vspace*{\fill}\end{center}

\normalsize\parskip=\oldparskip
\parskip=8pt
\pagenumbering{gobble}
\topskip0pt
%\vspace*{\fill}
\newpage
\begin{center}  
\textcopyright\hspace{.2in} \textbf{\theauthor \hspace{.2in}\themonthyear\\
ALL RIGHTS RESERVED}
\end{center}
\pagenumbering{gobble}
%\vspace*{\fill}
\chapter*{Acknowledgments}
\addcontentsline{toc}{chapter}{\protect\numberline{}Acknowledgments}
\pagenumbering{roman}
%\doublespacing
I am grateful to my wife, family, and friends for their patience as I dedicated my time to the writing of this thesis. I also thank Richard Moeckel, my advisor, who provided support, expertise, and encouragement when needed.

\newpage
%\newgeometry{left=1in, right=1in, top=1in, bottom=1in,headheight=0pt}
\chapter*{Abstract}
 \addcontentsline{toc}{chapter}{\protect\numberline{}Abstract}
%\vspace*{\fill}
%\begin{center} 
%  {\textbf{Abstract}}\\
%\vspace{.5in}
%\end{center}
%\doublespaced
\noindent
In the cosmos, any two bodies share a gravitational attraction.  When in proximity to one another in empty space, their motions can be modeled by Newtonian gravity. Newton found their orbits when the two bodies are infinitely small, the so-called two-body problem. The general situation in which the bodies have varying shapes and sizes, called the full two-body problem, remains open. We find relative equilibria (RE) and their stability for an approximation of the full two-body problem, where each body is restricted to a plane and consists of two point masses connected by a massless rod, a dumbbell. In particular, we find symmetric RE in which the bodies are arranged colinearly, perpendicularly, or trapezoidally.  When the masses of the dumbbells are pairwise equal, we find asymmetric RE bifurcating from the symmetric RE.  And while we find that only the colinear RE have nonlinear/energetic stability (for sufficiently large radii), we also find that the perpendicular and trapezoid configurations have radial intervals of linear stability. We also provide a geometric restriction on the location of RE for a dumbbell body and any number of planar rigid bodies in planar orbit (an extension of the Conley Perpendicular Bisector Theorem).

\vspace*{\fill}
%\begin{center}
%	\theauthor\\
%	\theadvisor
%\end{center}

\newpage
%\singlespaced
\doublespacing
\setstretch{1.1}

\tableofcontents
\addcontentsline{toc}{chapter}{\protect\numberline{}List of Figures}\newpage
\listoffigures\newpage
\addcontentsline{toc}{chapter}{\protect\numberline{}List of Tables}
\listoftables

%\restoregeometry\pagenumbering{gobble}
\newpage
\setcounter{preamblecounter}{1}
\setlength{\voffset}{-0.125in}
\setlength{\topmargin}{0.125in}
\setlength{\headheight}{0pt}
\setlength{\headsep}{0.35in}
\setlength{\hoffset}{0.4in} %left horizontal offset
\setlength{\oddsidemargin}{0in}
\setlength{\evensidemargin}{0in}
\setlength{\textwidth}{5.925in}
\setlength{\textheight}{8.5in}
\setlength{\parindent}{0ex}
\setlength\parskip{1em plus 0.1em minus 0.2em}

\setstretch{1.28}
\pagestyle{fancy}
%\fancyhead[L]{\includegraphics[height=30pt]{images/header/header_image_\arabic{page}.png}}
\fancyhead[R]{\arabic{page}}
\renewcommand{\headrulewidth}{1pt}
\renewcommand{\footrulewidth}{0pt}
\pagenumbering{arabic}
\chapter{Introduction}
%Include Roadmap
We will explore the dynamics of a planar Newtonian two-body system. That is, a system in which two bodies interact by way of Newtonian gravity, and whose orbits are restricted to a plane. The Two-Body Problem (2BP), where each body is modeled as a point mass, was solved by Newton in the Principia in 1687 \cite{newton1687}.  He also proved Shell Theorem, which states that spherically symmetric rigid bodies (which also lack tidal forces) affect ``external objects gravitationally as though all of its mass were concentrated at a point at its center.'' However, when one considers the Full Two-Body Problem (F2BP) in which the bodies have nonspherical shapes, nonuniform densities, or tidal forces (like planets or asteroids); general solutions to the equations of motion (differential equations describing the motions of the system) are intractable.  Despite the difficulty in finding solutions, information about the dynamics of some of these types of systems has been found. In this paper, we will use an approximation to the F2BP, where each body is represented as a dumbbell (a dumbbell consists of two point masses connected by a massless rod).
\section{Relative Equilibria (RE)}
As a step in understanding these more complicated systems, we would like to locate some kind of equilibria. While we don't expect literal equilibria from orbiting bodies, we can identify so-called relative equilibria. A relative equilibrium (RE) of a dynamical system is a solution which becomes an equilibrium in some uniformly rotating coordinate system. Put another way, a RE is an equilibrium point for a dynamical system which has been reduced through the quotienting out of some variable (a rotation angle in our case), as was shown in \cite{marsden1992}.

Put yet another way, for a two-body (or n-body) problem in a uniformly rotating reference frame, RE configurations of the parameters (initial positions, velocities, mass distributions) are those for which the bodies are static. That is, the distance between the bodies is constant, and neither body rotates relative to the other (they are tidally locked). An observer in an \textit{inertial reference frame} however, would see the two bodies rotating in circles about the system's axis of angular momentum (which includes the center of mass) in a rigid fashion (constant radius and such that the face that each body reveals to the other remains constant). For point mass bodies, note that only the radius requirement is relevant, as a point mass has no defined state of rotation.  Lastly, RE can be characterized as the critical points of an ``amended potential,'' which we will discuss in Section \ref{AmendedPotential}.

Motivating our interest in RE, we note that Pluto and Charon are in a near RE \cite{Michaely2017}. Many known bodies also exist in quasi-RE (the less massive of the two bodies is tidally locked, but the more massive is not). Examples include recently discovered orbiting binary asteroids \cite{Naidua2020}, Jupiter and its Galilean moons (the four largest)\cite{Laplace1787}, and closer to home, the Moon and Earth. Additionally, as we send satellites out to orbit comets and asteroids, we will wish to have knowledge of configurations allowing these types of static orbits (particularly those which are stable).

\paragraph{Outline for the paper}
\begin{itemize}[parsep=3pt]
	\item In Chapter \ref{AmendedPotential} of this paper, we will introduce the amended potential and discuss its role in finding and characterizing the stability of RE.
	\item In Chapter \ref{RENewtonian}, we will re-solve Newton's 2BP using the amended potential technique as a way of introducing the methods we will use in our dumbbell models.
	\item In Chapter \ref{RE1DB}, we examine the dumbbell/point mass problem. We expand on the work of Beletskii and Ponomareva who explored this model in 1990 \cite{beletskii1990}, finding colinear and perpendicular (isosceles) RE.  However, our work will differ as we will use the amended potential method.  We use this technique and others to identify RE and examine stability (both nonlinear/energetic and linear).  Unlike Beletskii and Ponomareva (who found Lyapunov stability), we find stable RE as minimal energy states of the Hamiltonian.
	\item In Chapter \ref{RE2DB}, we will use similar techniques to look for RE and their stability in the two-dumbbell problem. Here, there are many more families of RE, and classifying all of them is not as simple as in the dumbbell/point mass problem. In addition to symmetric RE (colinear, perpendicular, trapezoid), we apply bifurcation analyses to identify families of asymmetric RE bifurcating from the symmetric ones.\end{itemize}

\paragraph{Results}
\begin{itemize}[parsep=3pt]
	\item The most general result is an extension of Conley's Perpendicular Bisector Theorem \cite{Moeckel1990}, applied to a system containing a dumbbell and other planar rigid bodies in planar orbit. Our theorem gives a geometric restriction on the location of the bodies while in RE.
	\item For the dumbbell/point mass problem, we expand the results of Beletskii and Ponomareva \cite{beletskii1990} to include energetically stable colinear RE in the radially overlapped region. We also identify regions of the parameter space that vary in the number of RE by using angular momentum as a bifurcation parameter.
	\item For the two-dumbbell problem, we locate symmetric RE (colinear, perpendicular, trapezoid), and families of asymmetric RE bifurcating from the symmetric ones. We identify regions of the parameter space that vary in the number of RE by using angular momentum as a bifurcation parameter. And while the asymmetric RE are found to be unstable, we find the colinear RE to be energetically stable for sufficiently large radii. We also identify radial intervals of linear stability for perpendicular and trapezoid RE. \end{itemize}

\section{Historical Examples}
%\begin{figure}[H]
%	\centering\includegraphics[width=0.3\linewidth]{"images/Euler_3_body_colinear.jpg"}
%	\captionsetup{labelformat=empty}
%	\label{fig:euler-3-body-colinear}
%\end{figure}
The first orbital two-body RE discovered were in Newton's 2BP solutions. If one assumes the radius to be constant in his solutions, they produce RE consisting of circular planar orbits (which we will show below). Then, while studying the Earth-Moon-Sun system, Euler (in 1766) discovered RE solutions to the Circular Restricted Three-Body Problem (CR3PB), where two bodies have masses significantly larger than the third and the motions of the larger two are restricted to circular orbits around their center of mass \cite{Euler1766}. His RE consisted of colinear configurations with the Moon taking positions L1, L2, or L3 in Figure \ref{fig:lagrangianpoint}. In the following year, he found these colinear RE for the general three-body problem where the three bodies were free to take on any mass \cite{Euler1767}. Then, in 1772, Lagrange found the two remaining ``Lagrange Points'' (L4, L5) for the general three-body problem. Bodies orbiting at these points form equilateral triangles with the other two bodies.
\begin{figure}[H]
	\centering
	\includegraphics[width=0.4\linewidth]{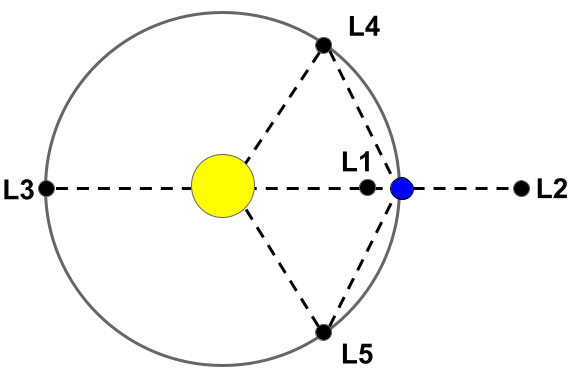}
	\caption{\\\bfseries Lagrange Points}{\vspace{0.8em}\footnotesize RE for the three-body problem.\\[-0.801em]}
	\label{fig:lagrangianpoint}
\end{figure}
Sometime later (1859) while studying the rings of Saturn, Sir James Clerk Maxwell showed the existence of RE for an (n+1)-body problem \cite{Maxwell1859}. The configuration has n orbiting bodies of equal mass positioned regularly on a circular orbit (forming a regular n-gon) orbiting the ``+1'' central mass. Maxwell found the rings were stable if the central mass was sufficiently large relative to the ring's mass. Richard Moeckel later identified a minor error in the original calculations that put a lower bound of 7 on the number of bodies necessary to achieve stability (1994) \cite{Moeckel1994}.
\begin{figure}[H]
	\centering
	\includegraphics[width=0.3\linewidth]{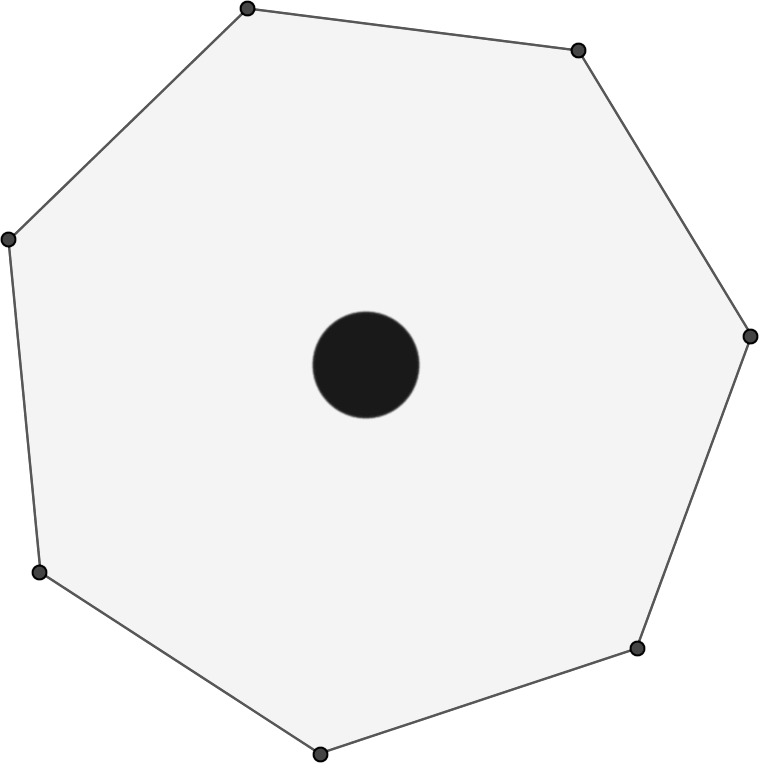}
	\caption{\\\bfseries 7-gon}
	\label{fig:regular7gonwhite}{\vspace{0.8em}\footnotesize A RE to the (n+1)-body problem.\\[-0.801em]}
\end{figure}
However, when one moves from point masses to extended rigid bodies (ERBs), the complexity of the problem for two bodies may preclude comprehensive classification of all RE configurations.  Nonetheless, a number of papers have explored the possible configurations for and stability of RE when dealing with non-spherical ERBs.  Wang and Maddocks (1992) proved the existence of non-Lagrangian RE (where the orbits of the two bodies exist in distinct parallel planes) \cite{wang1992}. And Maciejewski (1995) proved the existence of at least 36 non-Lagrangian RE in the limit as the distances between the bodies go (and angular momentum goes) to infinity \cite{maciejewski1995}.

Scheeres (2006) subsequently discovered necessary and sufficient conditions for RE of a system consisting of a point mass and an ERB. He also analyzed the linear (eigenvalues of the linearized equations of motion) and nonlinear energetic stability (whether the RE are at strict minima of an energy function) of these RE \cite{scheeres2006}. In another paper (2009), he located RE and determined the stability properties for a system (with an approximate potential) where both bodies were non-spherical but limited to planar motion \cite{scheeres2009}. A few years later (2012), he generalized by allowing (in each body) internal interactions (like tidal forces) that can cause dissipation of energy. Then, using an energy function for the system, he identified minimum energy configurations at fixed values of angular momenta \cite{scheeres2012}. Recently (2018), my advisor Richard Moeckel gave lower bounds on the number of RE for the F2BP where the angular momentum (and radius) of the system is large, but finite \cite{moeckel2018}.

For this paper, instead of looking at continuous ERBs, we will build an approximating structure out of a finite number of point masses, connected by massless rods. We build on work conducted by Beletskii and Ponomareva (1990) \cite{beletskii1990}. This point mass approximation has the benefit of massively simplifying the potential function in our equations.
\begin{figure}[H]
	\centering{\begin{flushleft}\textbf{Calculating the potential involves working with:}\end{flushleft}}
	\includegraphics[width=0.72\linewidth]{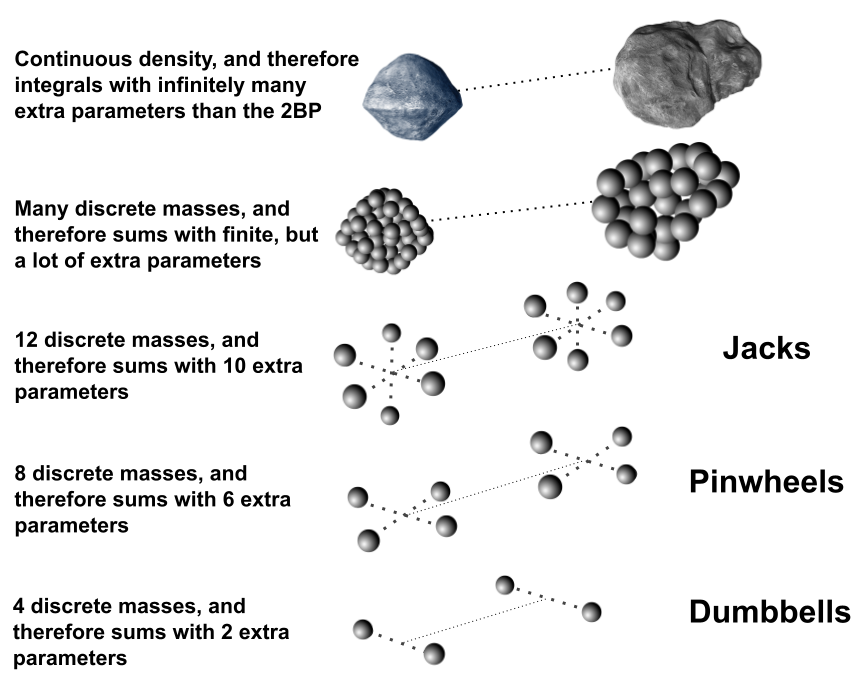}
	\caption{\\\bfseries Approximating F2BP with Point Masses}
	{\vspace{0.8em}\footnotesize Credit: NASA (dart.jhuapl.edu) 8/16/2021\\[-0.801em]}
	%\captionsetup{labelformat=empty}
	\label{fig:approximatingasteroid2kd}
\end{figure}

\chapter{Amended Potential}
\label{AmendedPotential}The process to discover equilibria often begins with writing down the Lagrangian of a system $\mathcal{L} = K - U$, where $K$ is the kinetic energy and $U$ is the potential energy (with conserved energy given by the Hamiltonian $\mathcal{H} = K + U$). One then uses the Euler Lagrange equations $\frac{d}{dt}\frac{\partial\mathcal{L}}{\partial \dot{q}_{i}}-\frac{\partial\mathcal{L}}{\partial q_{i}} = 0$, (where the $q_{i}$'s are the system's variables) to generate equations of motion. The equilibria are the solutions to these equations when the time derivatives of the variables are set to zero.

However, a technique used by several of the above mentioned papers, and what we will use in our analyses, is the amended (effective) potential for finding RE. The technique uses symmetries (conserved quantities) in the system to reduce the equations of motion by removing symmetry related velocities. The resulting amended potential $V$ is essentially the regular potential $U$ plus contributions we get from being in a rotating reference frame.

This technique is useful for at least three reasons. First, you form an amended potential by using conserved quantities/symmetries to eliminate velocity variables and therefore unnecessary complexity from your system. Second, Smale \cite{smale1970} showed that the RE of the original system are identical to the critical points of the amended potential. So once you form your amended potential, you can use standard tools for identifying critical points and therefore the RE. Third, if you can identify a critical point as a \textit{strict} minimum of this amended potential, Smale showed that this critical point is also a strict minimum of the original energy function, and therefore the system is energetically stable for the associated RE by the Dirichlet-Lagrange stability theorem \cite{Dirichlet1846}.
\begin{adjustwidth}{2.5em}{0pt}
\textbf{The Dirichlet-Lagrange Stability Theorem:} An equilibrium point (velocity $\vec{v}=0$) at position $\vec{r}=\vec{r}_0$ of Newton’s equations for a particle of mass $m$, moving under the influence of a potential $V$, which has a local strict minimum at $\vec{r}_0$, is stable.
\end{adjustwidth}
%Original Lagrange: We have shown that the [potential energy] function Φ is in a minimum or maximum, when the configuration of the system is one of equilibrium; we are now going to demonstrate that if this function is in a minimum then the equilibrium will be stable, such that the system, being assumedinequilibrium and displaced by a small amount, will tend to return to it by itself while making infinitely small oscillations. \textit{Méchanique Analytique}, La Veuve Desaint, Paris, Part 1, Sect. 3, No. 16, (1788) 38.
\vspace{1em}The general difficulty in finding energetic stability in dynamical systems makes this technique very attractive. Energetic stability means that if your trajectory is sufficiently close to the minimum of the energy function, for small enough perturbations, the resulting trajectory remains close to that minimum.

So how does one calculate the amended potential? First, we reduce our system by using an inertial reference frame moving with the same constant velocity as the system's center of mass, which we place at the origin. We then form a Lagrangian for our dynamical system $\mathcal{L} = K - U$ and use the Euler Lagrange equations to form equations of motion. Next, we note that the dynamics are invariant under a change of the rotation $\phi$ of the reference frame. As a result, $\phi$ does not appear in our Lagrangian (it is a cyclic variable). So, our Lagrangian possesses a symmetry, and we can perform a reduction of the system. By Noether \cite{Noether1918}, we then have the equation:
\begin{align}
	{\displaystyle \frac{\partial\mathcal{L}}{\partial \dot{\phi}} =L:=|\vec{L}|.\label{eq:Noether}}
\end{align}
The angular momentum $\vec{L}$ is a conserved quantity (first integral). Equation \eqref{eq:Noether} allows us to eliminate our velocity variable $\dot{\phi}$ by solving for it explicitly in terms of the other variables, and then substituting this back into our equations of motion. Once eliminated, the resulting reduced Lagrangian $\mathcal{L}_{red}$ is still a Lagrangian ($K_{red}-V$), and includes the amended potential $V$. Essentially, you are quotienting out a symmetry group (a circle in our case), and the resulting system exists on a quotient manifold with one fewer degrees of freedom. And the critical points of $V$ are RE of our original system. Now we will apply these techniques to our first model, Newton's 2BP.
\chapter{RE of Newtonian Point Mass Two-Body Problem (2BP)}\label{RENewtonian}
Newton used geometric arguments in his solution to the 2BP for point masses. In order to familiarize ourselves with the methods described in the previous section, let us use them to calculate RE for the 2BP. The 2BP (through a simple change of variables shown below), can be modeled as a restricted 2BP (Kepler Problem), or one in which the central body is assumed to be motionless (approximating the case when the mass of the central body far exceeds the mass of the orbiting body). This also meets the definition of a body in a \textbf{central force}. A central force is one which acts on masses $M_{i}$ such that:
\begin{itemize}
\item The force $\vec{F}_i$ on $M_{i}$ is always directed toward, or away from a fixed point $O$; and
\item The magnitude of the force $|\vec{F}_i|$ only depends on the distance $r$ between $M_{i}$ and $O$.
\end{itemize}
\begin{flushleft}
Characterizing our system in this way will simplify our calculations of the equations of motion. Also, we show that even though the bodies exist in three dimensional space, their motion is restricted to a plane, and therefore we can simplify our analysis by examining the motion in just two dimensions.
\end{flushleft}
\section{Central Force Motion is Planar}
To solve the Kepler problem, we assume the fixed body is located at the origin of our system, and without loss of generality we set $M_2$ as that fixed body. Observe that the initial position $\vec{r}_0$ and velocity vector $\vec{v}_0$ for the moving mass $M_1$ define a plane. We will show that $\vec{r}$ and $\vec{v}$ remain in this plane. 

Observe that the dot product of the angular momentum $\vec{L}=\vec{r}\times M_1\vec{v}$ with the mass' position vector $\vec{r}$ is zero: $\vec{r}\cdot\vec{L}=\vec{r}\cdot (\vec{r}\times M_1\vec{v})=M_1\vec{v}\cdot(\vec{r}\times\vec{r})=0$. Therefore the position $\vec{r}$ (and similarly the velocity $\vec{v}=\frac{dr}{dt}$) at each moment lies in a plane perpendicular to $\vec{L}$.  And so if $\vec{L}$ is constant, then $\vec{r}$ and $\vec{v}$ remain in the plane perpendicular to $\vec{L}$.  Recall that the time derivative of angular momentum is: 
\begin{center}
	$\frac{d\vec{L}}{dt} =\frac{d}{dt}(\vec{r}\times M_1\vec{v})=(\vec{v}\times M_1\vec{v})+(\vec{r}\times M_1\frac{d}{dt}\vec{v})=\vec{r}\times\vec{F} =$ net torque.
\end{center}
Since $\vec{F}$ points in the opposite direction as  $\vec{r}$ ($\vec{F}$ is central), we have $\vec{r}\times\vec{F} =0$. Therefore, $\vec{L}$ is constant, and our central force motion is planar. With this information, let us calculate the equations of motion for the Kepler problem.

\section{Kepler Problem: Equations of Motion}
We calculate the equations of motion using the Lagrangian: $\mathcal{L}=T-U $. Recall that force can be defined as the negative of the vector gradient of a potential field:  $\vec{F}(\vec{r}):=-\nabla U=-\frac{d}{d\vec{r}} U(r)$ (where $r:=|\vec{r}|$). And furthermore, from Newton's universal law of gravitation, the force between two masses $M_1$ and $M_2$ is: $\vec{F}(\vec{r})=-\frac{GM_1 M_2}{\vert\vec{r}\vert^3}\vec{r}$. So, integrating we find: $U(r)=-\int_{\vec{r}}^{\infty}\vec{F}(\vec{s})\cdot d\vec{s}=-\int_{\vec{r}}^{\infty}-\frac{GM_1 M_2}{\vert\vec{s}\vert^3}\vec{s}\cdot d\vec{s}=-\frac{GM_1 M_2}{r}$. And since our motion is restricted to two dimensions, we have kinetic energy $T=\frac{1}{2}M_1(\dot{r}^2+r^2\dot{\phi}
^2)$. So the Lagrangian is: $\mathcal{L}=\frac{1}{2}M_1(\dot{r}^2+r^2\dot{\phi}
^2)+\frac{GM_1 M_2}{r}$. Applying the Euler-Lagrange equations, we have $\frac{d}{dt}\frac{\partial \mathcal{L}}{\partial \dot{q}_{i}}-\frac{\partial \mathcal{L}}{\partial q_{i}}=0$, where the $q_{i}$'s in our case are $r$ and $\phi$. Taking these derivatives we find two coupled equations of motion:
\begin{subequations}\label{eq:keplerEOM}
	\makeatletter\@fleqntrue\makeatother
		\begin{align}
			\begin{split}
				\mathitem \text{$M_1\ddot{r}+M_1 r\dot{\phi}^2-\frac{GM_1 M_2}{r^2}=0\label{eq:keplerEOMa}$, and}
			\end{split}\\
			\begin{split}
				\mathitem \text{$\frac{d}{dt}(M_1r^2\dot{\phi})=0\label{eq:keplerEOMb}$.}
			\end{split}
		\end{align}
\end{subequations}

Now we can either note that \eqref{eq:keplerEOMb} upon integrating gives us a conserved quantity $L:=M_1 r^2\dot{\phi}$, or we can appeal to Noether and the existence of the cyclic variable $\phi$, and calculate $\frac{\partial \mathcal{L}}{\partial \dot{\phi}}=L$. Either way, we then perform the reduction from Section \ref{AmendedPotential} using our conserved quantity $L$. Solving for the rotation speed from the equation above, we have: 
\begin{align}{\displaystyle\dot{\phi} =\frac{L}{M_1 r^2}\label{eq:keplerphidot}}.
\end{align}
Substituting $\dot{\phi}$ into \eqref{eq:keplerEOMa}, we have: $M_1\ddot{r}+\frac{L^2}{M_1r^3}-\frac{GM_1 M_2}{r^2}=0$.  To find our reduced Lagrangian $\mathcal{L}_{red}$ and the amended potential, note that the kinetic and potential energy necessary for $\frac{d}{dt}\frac{\partial \mathcal{L}_{red}}{\partial \dot{r}}-\frac{\partial \mathcal{L}_{red}}{\partial r}=0$ to produce the previous equation would be: $\mathcal{L}_{red}=T_{red}-V$, where $T_{red}=\frac{1}{2}M_1\dot{r}^2$ and $V=\frac{L^2}{2M_1 r^2}-\frac{GM_1 M_2}{r}$. So, we have eliminated the velocity variable $\dot{\phi}$, decoupled the system, and reduced to a single equation of motion: 
\begin{align}{\displaystyle
\ddot{r}=\frac{\partial_rV}{M_1}=\frac{GM_2}{r^2}-\frac{L^2}{M_1^2 r^3}\label{eq:keplerreducedeom}}.
\end{align}
\begin{figure}[H]
	\captionsetup[subfigure]{justification=centering}
	\centering
	\subfloat[Kepler Amended Potential showing strict\\minimum, and therefore a stable RE.]
	{\includegraphics[width=0.51\textwidth]{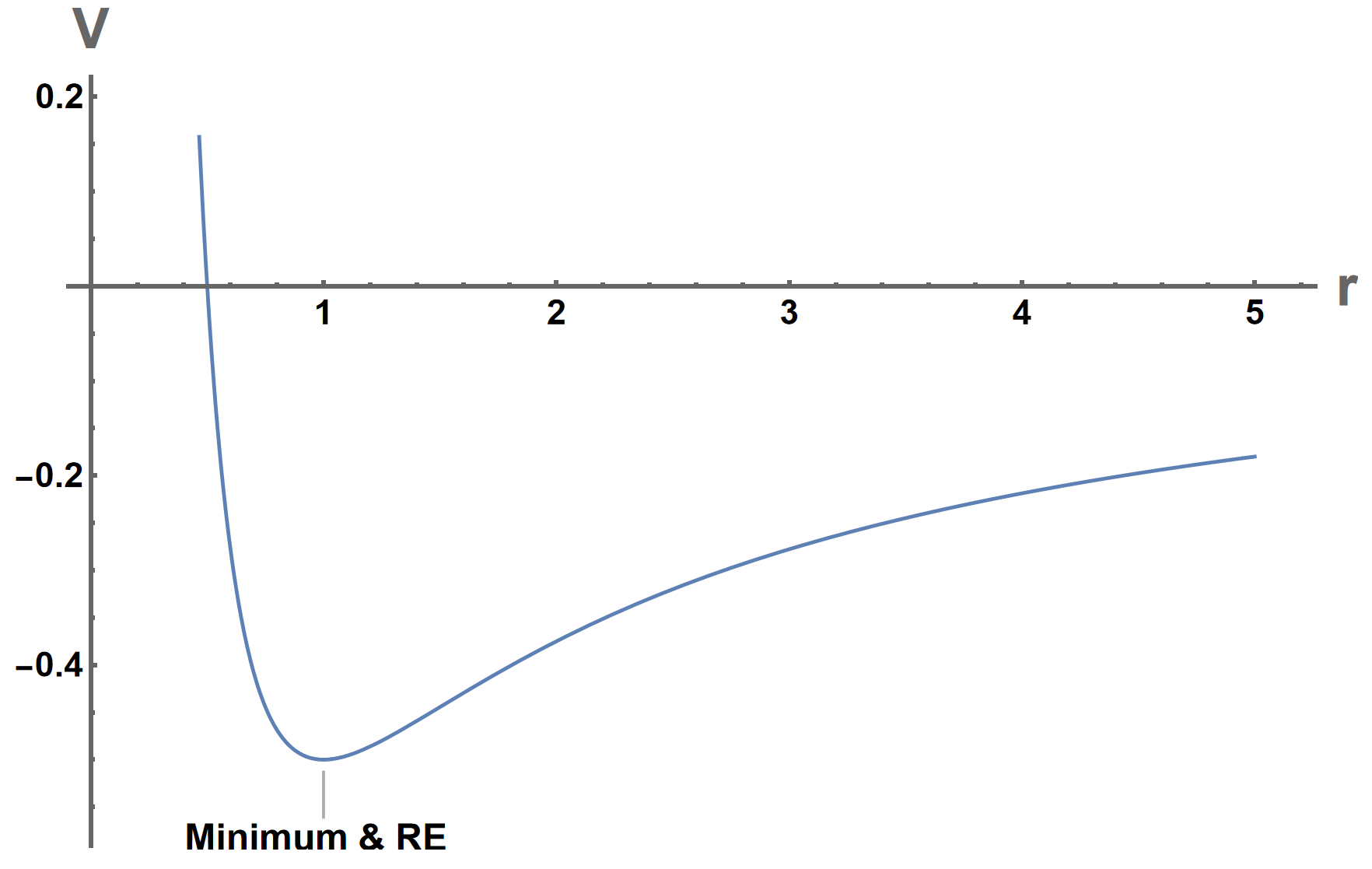}}
	\hspace{.5cm}
	\subfloat[Kepler Phase Portrait of V showing nearly\\circular nearby elliptic periodic orbits.]
	{\includegraphics[width=0.44\textwidth]{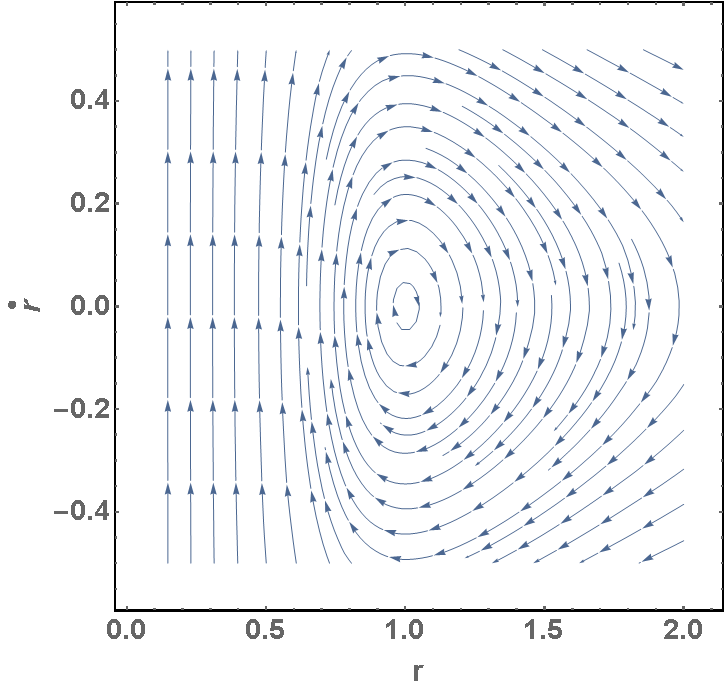}}
	\caption{\\\bfseries Kepler Dynamics}{\vspace{0.8em}\footnotesize with $M_1=M_2=L=G=1$\\[-0.801em]}
	\label{fig:kepler_amendpot_phase}
\end{figure}
%Documented Kepler_Graph.mx
To find RE of our system, one option is to look for fixed points of this reduced system. Note that a constant radius implies a circular orbit. Additionally, since point masses have no meaningful sense of rotation, all two-body circular orbits of point masses meet the definition of being RE.  So, from \eqref{eq:keplerreducedeom}, we use the fact that RE occur when the time derivatives are zero (except rotational $\dot{\phi}$), and set $\ddot{r} =0$. Solving for the radius in what remains, we conclude: $r=\frac{L^2}{GM_1 M_2^2}$. Observe the one-to-one relationship between radii and positive $L$ in this equation. From \eqref{eq:keplerEOMa}, we can determine the rotation speed: $\dot{\phi} =\sqrt{\frac{GM_1}{r}}=\frac{GM_1 M_2}{L}$. Again, a one-to-one relationship between $\dot{\phi}$ and $L$. So, angular momentum uniquely determines both the $\dot{\phi}$ and $r$ for RE.

Additionally, for each fixed $L$, note that the circular orbits are critical points (minimums) of the amended potential: $\partial_{r}V=-\frac{L^2}{M_1 r^3}+\frac{GM_1 M_2}{r^2}=0$, when $r=\frac{L^2}{GM_1 M_2^2}$, as expected (see Figure \ref{fig:kepler_amendpot_phase}a).
%Also note in the phase portrait (Figure \ref{fig:kepler_amendpot_phase}b) that the nearby elliptic periodic orbits are nearly circular. 
As noted earlier, the RE of the original system \eqref{eq:keplerEOM} are the same as the equilibria of the reduced system \eqref{eq:keplerreducedeom}. Now that we have found the RE, let us see if they are stable.  That is, if something were to perturb the RE orbit (some solar wind, or small asteroid), will the resulting orbit remain close to the RE?
\subsubsection*{Energetic Stability of Kepler}
As we saw in Section \ref{AmendedPotential}, to determine energetic stability we verify that the RE are strict minima of the amended potential $V$.
We find positive curvature ($\partial^2_{r}V=\frac{3L^2}{M_1 r^4}-\frac{2GM_1 M_2}{r^3}=\frac{M_1^4M_2^7G^4}{L^6}>0$), when $r=\frac{L^2}{GM_1 M_2^2}$. So, these RE are strict minima. So, by Smale, these RE are energetically stable. In the Figure \ref{fig:kepler_amendpot_phase}b, you see the phase portrait for the amended potential. Note that surrounding the equilibrium are nearby periodic orbits.{\text{ }\\[-1.801em]}
\section{Unrestricted 2BP}
Generalizing, what are the RE when both bodies can move freely? So this is not a Kepler problem. First we may wish to reconsider the location of the origin of our inertial reference frame. Due to conservation of linear momentum, we can assume the system's center of mass moves at a constant rate. This allows us to choose an inertial reference frame such that our choice of origin coincides with the system's center of mass. We then denote the positions of the two bodies as $\vec{r}_{1},\vec{r}_2$, and the vector between them as $\vec{r}:=\vec{r}_2-\vec{r}_1$.

Constructing the dynamics, recall from Newton's second ($\vec{F}=M\vec{a}$) and third ($\vec{F}_{12}=-\vec{F}_{21}$) laws that $M_2 \ddot{r}_2=\vec{F}_{21}=-\vec{F}_{12}=-M_1\ddot{r}_1$, or equivalently $\ddot{r}=\frac{\vec{F}_{21}}{M_2}-\frac{\vec{F}_{12}}{M_1}=(\frac{1}{M_2}+\frac{1}{M_1})\vec{F}_{21}$, where $\vec{F}_{ij}$ is the force on $\vec{r}_i$ by $\vec{r}_j$. And denoting the ``reduced mass'' $M:=\frac{M_1M_2}{M_1+M_2}$, we have: $M\ddot{r}=\vec{F}_{21}$. So, by the form of this equation, we see that we have characterized the motion of the two bodies as equivalent to the motion of a single body at radius $\vec{r}$ with mass $M$, moving under the influence of $\vec{F}_{21}$. And if the motion described here is due to a central force, we can use the results from the Kepler Problem to find our equation of motion. Note that in this system, the center of mass serves as our fixed point, and always lies between $\vec{r}_1$ and $\vec{r}_2$, so the force is directed toward it. Additionally, since the force $\frac{GM_1M_2}{|\vec{r}|^3}\vec{r}$ only depends on $\vec{r}$, it fits the definition of a central force. And so the Kepler problem above gives us our solutions.

Shell Theory lets us also apply these solutions to spherically symmetric ERBs with constant density. However, to find the RE when one of the bodies is not spherically symmetric, or with irregular density is complicated. As a step in that direction, our next model alters one of the point mass bodies by splitting it into two point masses connected by a massless rod, i.e. a ``dumbbell.''
\chapter{Planar Dumbbell/Point Mass Problem}\label{RE1DB}
\begin{figure}[H]
	\centering
	\includegraphics[width=0.5\linewidth]{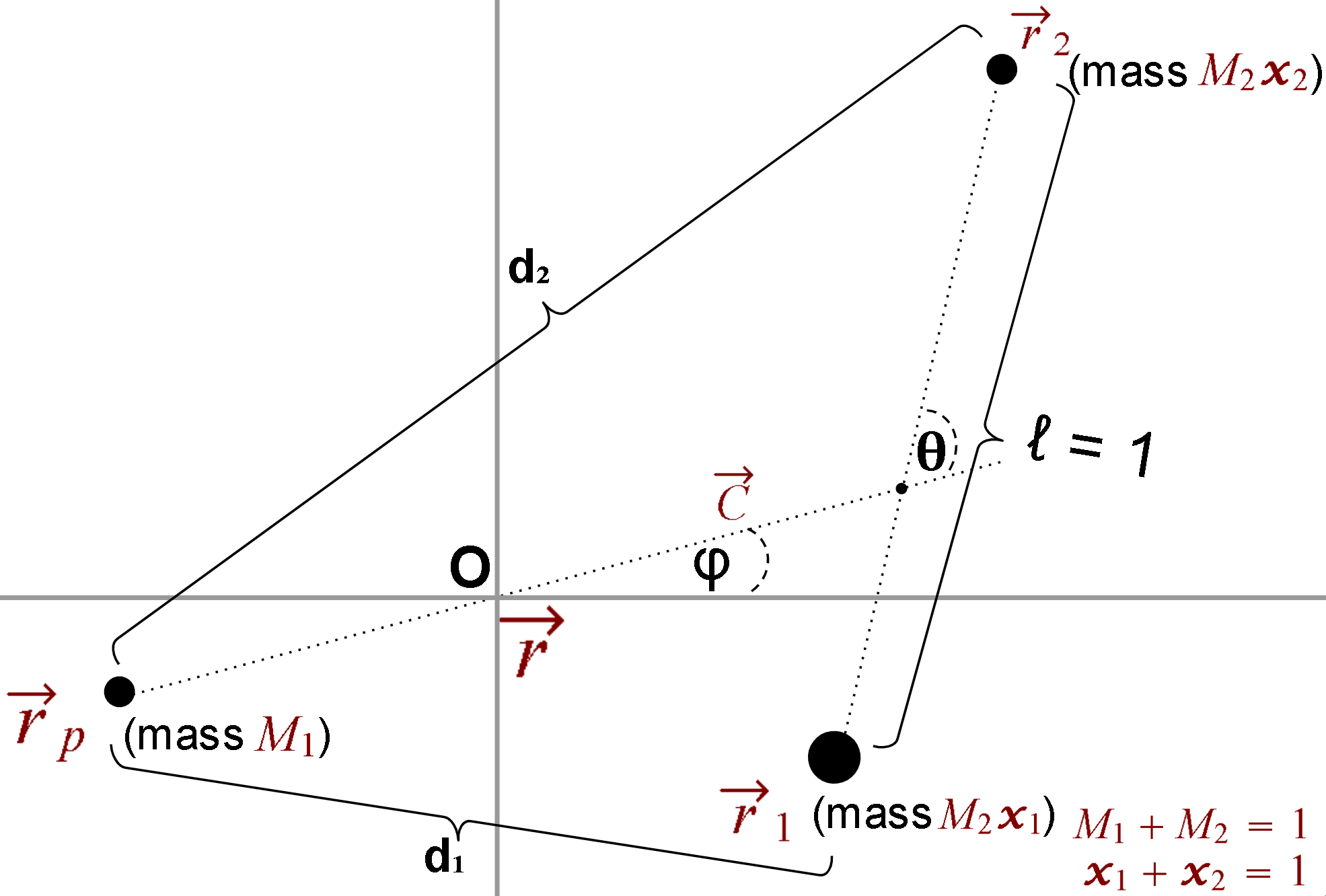}
	\caption{\\\bfseries Dumbbell/Point Mass Problem}{\text{ }\\[-1.801em]}
	\label{fig:1kdb_configurationnotation}
\end{figure}
It was shown in 2013 \cite{Kinoshita2013} that the dumbbell/point mass problem (seen above) is non-integrable. Nonetheless, various researchers have identified its RE and their stability (\cite{beletskii1990},\cite{ElipeETAL2006},\cite{MaciejewskiETAL2018},\cite{DilaoMurteira2020}). Unlike Newton's model with two point masses, the dumbbell/point mass problem need not have planar orbits. For example, researchers in 2020 \cite{DilaoMurteira2020} examined the dumbbell/point mass problem, allowing for three dimensional motion, and identified Lyapunov stability of RE. However, a likely location for RE of our subsequent two-dumbbell problem \textit{is} in planar motion, so it is in this context we will study the dumbbell/point mass problem.

We follow some of the calculations of Beletskii and Ponomareva (1990) \cite{beletskii1990}, who also looked at RE for this planar model. In particular, the authors were interested in how stability (Lyapunov and linear) of RE changes as you alter parameter values (the mass and length of the dumbbell).

Below, we extend the results of these previous papers, formulating the problem through the use of an amended potential to identify and characterize the energetic stability of RE via Smale. We also identify linear stability, and use angular momentum as a bifurcation parameter in the equations of motion. This will identify regions of the parameter space that differ in the number of RE and help us identify extrema in the angular momentum, assisting with the energetic stability analysis. 

As previous authors identified, the bodies of the planar configuration RE are either colinear, or form an isosceles triangle. In Section \ref{PerpBisThm}, we prove an extension of the Conley Perpendicular Bisector Theorem, which limits the possible RE configurations to only the colinear or isosceles.
\begin{theorem*}[Perpendicular Bisector Theorem for RE of a Dumbbell and Rigid Planar Bodies]
	Let a dumbbell $\overline{r_1r_2}$ and one or more planar rigid bodies $\mathcal{B}_2,...,\mathcal{B}_n$ be in a planar RE. Then, if one of the two open cones determined by the lines through $\overline{r_1r_2}$ and its perpendicular bisector contains one or more rigid bodies, the other open cone cannot be empty.
\end{theorem*}
In particular, applying our theorem shows that any RE must have the point mass body lie on the line defined by the dumbbell's massless rod (colinear configuration), or on the perpendicular bisector of that rod (isosceles configuration).
\section{Set-up and Notation}
For the following description, refer to Figure \ref{fig:1kdb_configurationnotation}.  Without loss of generality, we denote the mass of the point mass body as $M_1$, and the total dumbbell's mass as $M_2$ such that the system mass is $M_1+M_2=1$. We also denote the mass ratios of the dumbbell masses as $x_1$ and $x_2$ such that $x_1+x_2=1$.

We label the origin and center of mass of our system reference frame as $O$. The masses of the dumbbell are connected by a massless rod of length $\ell$ which we set to 1 (by scaling the system distances). The dumbbell has a center of mass located at $\vec{C}$. We label the ``radius'' vector connecting the point mass body to $\vec{C}$ as $\vec{r}$, and $r:=|\vec{r}|$. We also label the acute angle between $\vec{r}$ and the horizontal axis of our system reference frame as $\phi$. We let $\theta$ represent the angle between $\vec{r}$ and the dumbbell's massless rod, where if $\theta =0$, then $x_1$ is the mass located on $\overline{OC}$. Without loss of generality, we limit $\theta\in[0,\pi)$. Note we can replicate dynamics on the rest of the range by switching the dumbbell masses. The system has three degrees of freedom $r,\phi$, and $\theta$.

The position vectors for the dumbbell masses are:\\ $\vec{r}_{i}:=M_1r(\cos\phi,\sin\phi)+(-1)^{i}x_{k}(\cos(\phi+\theta),\sin(\phi +\theta))$, where $k\neq i$, and the point mass body position vector is $\vec{r}_{p}:=-M_2r(\cos\phi,\sin\phi)$. Note that the mass associated with $\vec{r}_{i}$ is $M_2 x_{i}$. Also, the distance between $\vec{r}_{p}$ and $\vec{r}_{i}$ is:
\begin{align}\label{eq:1DB_Distances}
{\textstyle d_{i}(r,\theta):=\sqrt{r^2+(-1)^{i}2x_{k}r\cos\theta+x_{k}^2}}.
\end{align}
For a table of the notation used in this paper, see Appendix \ref{appendix:Notation Used in This Paper}.
\section{Equations of Motion}
In order to find the equations of motion, we let $\vec{v}_{i},\vec{v}_{p}$ denote the velocities of $\vec{r}_{i},\vec{r}_{p}$. We will calculate the Lagrangian $\mathcal{L}=T-U$, where $T=\frac{1}{2}M_1 \vec{v}_{p}^{\hspace{.2em}2}+\frac{1}{2}\Sigma_{i\in\{1,2\}}M_2 x_{i}\vec{v}_{i}^{\hspace{.2em}2}$ is the kinetic energy, and (from Newton) the potential energy of the system is (with $G=1$, and scaled by $\frac{1}{M_1 M_2}$):
\begin{align}
	{\textstyle U(r,\theta):=-\frac{x_1}{d_1}-\frac{x_2}{d_2}.}
\end{align}
To calculate kinetic energy, we write:\\
$\text{\enspace\enspace\enspace\enspace\enspace}\textstyle  \vec{v}_{i}=\frac{d}{dt}\vec{r}_{i}=M_1 \dot{r}(cos\phi,sin\phi)$\\
$\text{\enspace\enspace\enspace\enspace\enspace\enspace\enspace\enspace\enspace\enspace\enspace\enspace\enspace\enspace\enspace}\textstyle  +M_1 r\phi(-sin\phi,cos\phi)+(-1)^{i}x_{k}(\dot{\phi}+\dot{\theta})(-sin(\phi+\theta),cos(\phi+\theta))$, and\\
$\text{\enspace\enspace\enspace\enspace\enspace}\textstyle \vec{v}_{p}=-M_2\dot{r}(cos\phi,sin\phi)-M_2r\dot{\phi}(-sin\phi,cos\phi).$\\
So, using some trigonometric identities, we find our Lagrangian (scaled by $\frac{1}{M_1 M_2}$) is:
\begin{align}\label{eq:1DB_Lagrangian}
	{\textstyle \mathcal{L}=T-U=\frac{1}{2}(\dot{r}^2+r^2\dot{\phi}^2)+\frac{1}{2}B(\dot{\theta}+\dot{\phi})^2 +(\frac{x_1}{d_1}+\frac{x_2}{d_2}),}
\end{align}
\begin{align}\label{eq:1DB_MOI}
	{\textstyle \text{where }B:=\frac{x_1 x_2}{M_1}}
\end{align}
is the dumbbell's scaled moment of inertia relative to its center of mass. The equations of motion are given by the Euler-Lagrange equation: $\frac{d}{dt}\frac{\partial L}{\partial \dot{q}_{i}} -\frac{\partial L}{\partial q_{i}}=0$, for each degree of freedom $q_i\in\{r,\phi,\theta\}$. So, we calculate:
\begin{subequations}\label{eq:1DB_EOMs}
	\makeatletter\@fleqntrue\makeatother
\begin{align}
	\mathitem \text{$\ddot{r}-r\dot{\phi}^2+\frac{\partial U}{\partial r}=0,\label{eq:1DBEOMsA}$}\\
	\mathitem \text{$r^2\ddot{\phi}+2\dot{r}\dot{\phi} r+B(\ddot{\theta}+\ddot{\phi})=-\frac{\partial U}{\partial\phi}=0$,  and$\label{eq:1DBEOMsB}$}\\
	\mathitem \text{$B(\ddot{\theta}+\ddot{\phi})+\frac{\partial U}{\partial\theta}=0.\label{eq:1DBEOMsC}$}
\end{align}
\end{subequations}
\section{Finding RE}
To find RE, we first wish to reduce our system using the amended potential method in Section \ref{AmendedPotential} by exploiting the conserved angular momentum. As we did in the $2BP,$ we calculate scalar angular momentum (scaled by $\frac{1}{M_{1}M_{2}}$):\ $\frac{\partial 
	\mathcal{L}}{\partial\overset{\cdot}{\varphi}}=r^{2}\overset{\cdot}{\varphi}+B\left(\overset{\cdot}{\theta}+\overset{\cdot}{\varphi}\right) =:L.$ Solving for the rotational speed:
\begin{align}\label{eq:1DB_PhiDot}
	{\textstyle \overset{\cdot}{\varphi}\,=\frac{L-B\overset{\cdot}{\theta}}{r^{2}+B}\text{\enspace\enspace}\text{ or }\text{\enspace\enspace}\overset{\cdot}{\varphi}\,=\frac{L}{r^{2}+B}\text{ at RE}.}
\end{align}
We can now eliminate the rotation speed $\overset{\cdot}{\varphi}$ from our system using this relation. Upon substituting $\overset{\cdot}{\varphi}$ into $\eqref{eq:1DBEOMsB},$ and solving for $\overset{\cdot\cdot}{\varphi}$, we see that $\eqref{eq:1DBEOMsB}$ is equivalent to $\overset{\cdot\cdot}{\varphi}\,=-\frac{2r\overset{\cdot}{r}\left(L-B\overset{\cdot}{\theta}\right)+B\overset{\cdot \cdot}{\theta}\left(r^{2}+B\right)}{\left(r^{2}+B\right)^{2}}.$
 Substituting $\overset{\cdot}{\varphi}$ and $\overset{\cdot\cdot}{\varphi}$ into $\eqref{eq:1DBEOMsA}$ and $\eqref{eq:1DBEOMsC}$, we determine the reduced Lagrangian $\mathcal{L}_{red}$ for which $\frac{d}{dt}\frac{\partial\mathcal{L}_{red}}{\partial \overset{\cdot}{r}_{i}}-\frac{\partial \mathcal{L}_{red}}{\partial r_{i}}$ will
generate these reduced equations. We find: $\mathcal{L}_{red}=T_{red}-V=\frac{1}{2}\left(\overset{\cdot}{r}^{2}\allowbreak+\,\frac{Br^{2}}{r^{2}+B}\overset{\cdot}{\theta}^{2}\right)+\frac{BL}{r^{2}+B}\overset{\cdot}{\theta}-\left(\frac{L^{2}}{2\left(r^{2}+B\right)}+U\right)$.

So, we have our amended potential (scaled by $\frac{1}{M_{1}M_{2}}$):
\begin{align}\label{eq:1DB_AmendedPot}
	{\textstyle
		 V=\frac{L^{2}}{2\left(r^{2}+B\right)}-\frac{x_{1}}{d_{1}}-\frac{x_{2}}{\,\,d_{2}}.}
\end{align}
Recall we can characterize the RE of our system as the critical points of the amended potential $V.$ Taking our derivatives, we find for $\theta$, the \textbf{angular requirement}:
\begin{align}\label{eq:1DB_Angular_REQ}
	{\textstyle \partial_{\theta}V=x_1 x_2 r\sin\theta\left(\frac{1}{d_{1}^{3}}-\frac{1}{d_{2}^{3}}\right)=0.}
\end{align}
And for $r,$ the \textbf{radial requirement}: 
\begin{align}\label{eq:1DB_Radial_REQ}
	{\textstyle \partial_{r}V=\frac{x_{1}\left(r-x_{2}\cos\theta\right)}{d_{1}^{3}}+\allowbreak \frac{x_{2}\left(r+x_{1}\cos \theta\right)}{d_{2}^{3}}\allowbreak -\frac{L^{2}}{\left(r^{2}+B\right)^{2}}r=0}.
\end{align}
Observe that if we find a $\theta$ which satisfies \eqref{eq:1DB_Angular_REQ}, then substitute it into \eqref{eq:1DB_Radial_REQ}, we will have a relationship between $r$ and $L$ for RE. This allows us to parameterize the RE by $r$. Also note that for RE, if you fix angular momentum $L$, then for each $r$, equation \eqref{eq:1DB_PhiDot} provides a unique rotation speed $\overset{\cdot}{\varphi}.$
\begin{definition}[Nontrivial Dumbbell]
	A nontrivial dumbbell is one that has positive length, both masses are positive, and whose distance from the system's center of mass is positive i.e., $x_{i}\notin\left\{0,1\right\},$ and ${\ell},r>0$.
\end{definition}

The authors of \cite{beletskii1990} observed that for $\eqref{eq:1DB_Angular_REQ}$ to hold with a nontrivial dumbbell, there are only two configurations: $\theta \in\left\{ 0,\pi\right\}$ or $d_{1}=d_{2}$. Remarkably, the only configurations allowed correspond to RE with masses being in a colinear alignment or in the shape of an isosceles triangle. Particularly, the isosceles RE exist with any choice of dumbbell masses (they need not be equal). These results seem less mysterious when considered as the only allowed configurations of our Perpendicular Bisector Theorem for RE of a Dumbbell and Rigid Planar Bodies.
\begin{figure}[H]
	\centering
	\begin{minipage}{0.45\textwidth}
		\centering
		\includegraphics[width=0.4\textwidth]{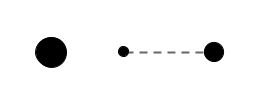}\\{$\theta =0$}
%		\captionsetup{labelformat=empty}
	\end{minipage}\hfill or \hfill
	\begin{minipage}{0.45\textwidth}
		\centering
		\includegraphics[width=0.4\textwidth]{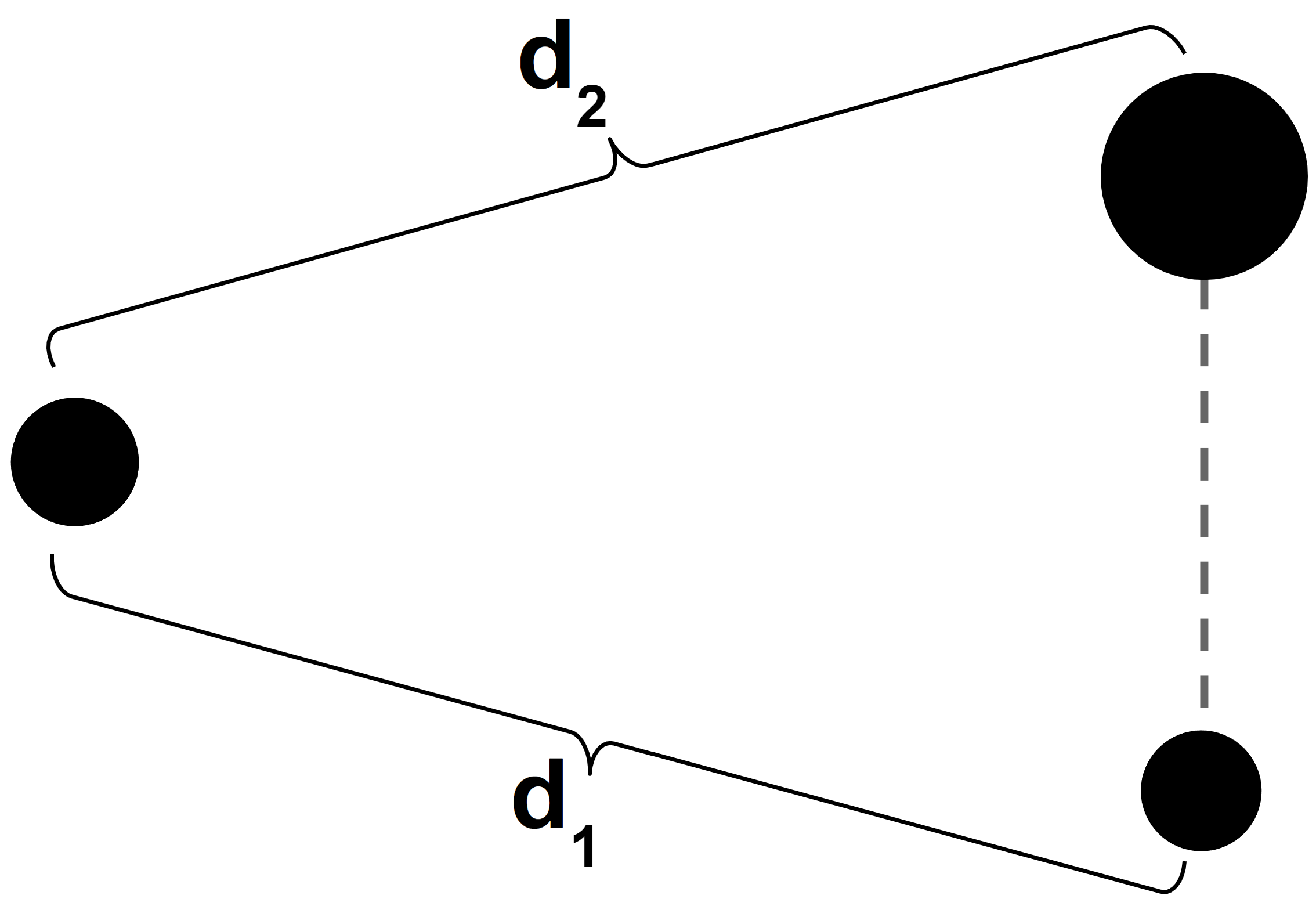}\\{$d_{1}=d_{2}$}
%		\captionsetup{labelformat=empty}
	\end{minipage}
\caption{\\\bfseries Colinear and Isosceles Configurations}{\text{ }\\[-1.801em]}
\end{figure}
Before we look at each of these configurations individually, we first develop the tools we will need to analyze the energetic and linear stability.

\paragraph{Energetic Stability}
As we saw in Section \ref{AmendedPotential}, to determine energetic stability we check if RE are strict minima of the amended potential $V$. So we calculate the Hessian of $V$: $H:=\begin{bmatrix}\partial_{r}^2 V & \partial_{r,\theta}^2 V \\ 
\partial_{r,\theta}^2 V & \partial_{\theta}^2 V\end{bmatrix}$. Any strict minima will produce real, positive $H$ eigenvalues, making $H$ positive definite. To simplify our calculations, we first recharacterize $\partial_{r}V$ (equation \ref{eq:1DB_Radial_REQ}) as:
$\left(g\left(r,\theta\right) \allowbreak -L^{2}\right) \frac{r}{\left(r^{2}+B\right)^{2}},$ where $g\left(r,\theta\right) :=\frac{\left(r^{2}+B\right)^{2}}{r}\left(\frac{x_{1}\left(r-x_{2}\cos\theta\right)}{d_{1}^{3}}+\allowbreak\frac{x_{2}\left(r+x_{1}\cos\theta\right)}{d_{2}^{3}}\right)$. So at a RE (critical point of $\partial_{r}V$), we have: $g\left(r,\theta\right) =L^{2}.$ Also, $\textstyle \partial_{r}^{2}V=\frac{r}{\left(r^{2}+B\right)^{2}}\partial_{r}g\left(r,\theta\right)+\left(g\left(r,\theta\right)-L^{2}\right)\left(1-\frac{4r^{2}}{B+r^{2}}\right)\frac{1}{\left(r^{2}+B\right)^{2}}$, which at RE is: 
\begin{align}
	{\displaystyle \partial_{r}^{2}V|_{RE}=\frac{r}{\left(r^{2}+B\right)^{2}}\partial_{r}g\left(r,\theta\right).\label{eq:1DB_PartialV_AtRE}}
\end{align}
We can now determine the sign of $\partial_{r}^{2}V$ at a RE by looking at $\partial_{r}g,$
or $\partial_{r}L^{2}$. In other words, with the slopes of graphs like Figure \ref{fig:1DB_Colinear_LS} (positive slopes indicate $\partial_{r}^{2}V>0$). Next, for $\theta $, we calculate:
\begin{align}\label{eq:1DB_SecondThetaPartialofV}
	{\displaystyle \partial_{\theta}^{2}V=x_1 x_2 r\cos \theta\left(\frac{1}{d_{1}^{3}}-\frac{1}{d_{2}^{3}}\right) -3x_{1}x_2 r^{2}\sin ^{2}\theta
\left(\frac{x_{2}}{d_{1}^{5}}+\frac{x_{1}}{d_{2}^{5}}\right).}
\end{align}
 And lastly: $\text{\enspace\enspace}\partial_{\theta ,r}V=\partial_{r,\theta}V$\\$\text{\hspace{20pt}}=x_1 x_{2}\sin \theta\left(\frac{1}{d_{1}^{3}}-\frac{1}{d_{2}^{3}}\right)+3x_{1}x_2 r\sin \theta\left(\frac{r+x_{1}\cos \theta}{\left(r^{2}+2x_{1}r\cos \theta +x_{1}^{2}\right)^{\frac{5}{2}}}-\allowbreak \frac{r-x_{2}\cos \theta}{\left(r^{2}-2x_{2}r\cos \theta +x_{2}^{2}\right)^{\frac{5}{2}}}\right)$.

For a nontrivial (${ r\neq 0}$) critical point this becomes:
\begin{align}
	{\displaystyle \partial_{\theta,r}V|_{RE}\,{ =\,3}x_1 x_2 r\sin{\theta}\left(\frac{r+x_{1}\cos\theta}{\left(r^{2}+2x_{1}r\cos \theta +x_{1}^{2}\right)^{\frac{5}{2}}}\allowbreak -\allowbreak \frac{r-x_{2}\cos \theta}{\left(r^{2}-2x_{2}r\cos \theta +x_{2}^{2}\right)^{\frac{5}{2}}}\right).\label{eq:1DB_Mixed_PartialV_AtRE}}
\end{align}
In the configurations below, we will use these in the second partial derivative test to determine whether $H$ is positive definite, and therefore energetic stability of our RE.
\paragraph{Linear Stability}
%But what if we find RE that lack energetic stability?  Are we still interested in them? Is linear stability useful to us? Yes, linear stability tells us that nearby orbits will (in the short term) have stable dynamics.
We determine linear stability by linearizing our equations of motion \eqref{eq:1DB_EOMs} as $\overset{\cdot}{v}\,=A\vec{v}$, and requiring purely imaginary $A$ eigenvalues at RE. We start by rewriting the reduced Lagrangian in terms of the amended potential. So:
%$\mathcal{L}_{red}=\frac{1}{2}\left(\overset{\cdot}{r}^{2}+\frac{Br^{2}}{r^{2}+B}\overset{\cdot}{\theta}^{2}\right)+\frac{BL}{r^{2}+B}\overset{\cdot}{\theta}-\left(\frac{L^{2}}{2\left(r^{2}+B\right)}+U\left(r,\theta\right)\right) $ becomes
$\textstyle\mathcal{L}_{red}=\frac{1}{2}\left(\overset{\cdot}{r}^{2}+\frac{2BL}{r^{2}+B}\overset{\cdot}{\theta}+\frac{Br^{2}}{r^{2}+B}\overset{\cdot}{\theta}^{2}\right) -V\left(r,\theta\right)$, where $V\left(r,\theta\right) :=\frac{L^{2}}{2\left(r^{2}+B\right)}+U\left(r,\theta\right) $ is the amended potential. Applying the Euler Lagrange equation gives us equations of motion:
\begin{subequations}\label{eq:1DBLinearizedEOMs}
	\makeatletter\@fleqntrue\makeatother
	\begin{align}
		\begin{split}
			&\text{$ \overset{\cdot\cdot}{r}=\frac{Br\overset{\cdot}{\theta}\left(B\overset{\cdot}{\theta}-2L\right)}{\left(r^{2}+B\right)^{2}}-\partial_{r}V,\label{eq:1DBLinearizedEOMsA}$}
		\end{split}\\
		\begin{split}
			&\text{$ \overset{\cdot\cdot}{\theta}=2\frac{L-B\overset{\cdot}{\theta}}{r\left(r+B^{2}\right)}\overset{\cdot}{r}-\frac{r^{2}+B}{Br^{2}}\partial_{\theta}V.\label{eq:1DBLinearizedEOMsB} $}
		\end{split}
	\end{align}
\end{subequations}
For RE, we have $\overset{\cdot}{r}_{RE}=\overset{\cdot}{\theta}_{RE}= \overset{\cdot\cdot}{r}_{RE}=\overset{\cdot\cdot}{\theta}_{RE}=0$, giving us:
\begin{subequations}
	\makeatletter\@fleqntrue\makeatother
	\begin{align}
		\begin{split}
		&\text{$\overset{\cdot\cdot}{r}_{RE}=-\partial_{r}V=0\nonumber$,}
	\end{split}\\
\begin{split}
		&\text{$\overset{\cdot\cdot}{\theta}_{RE}=-\frac{r^{2}+B}{Br^{2}}\partial_{\theta}V=0\nonumber$,}
			\end{split}
	\end{align}
\end{subequations}
%\eqref{eq:1DB_EOMs_for_LinearStab}
from which we can see our previous requirements for RE that $\partial_{r}V=\partial_{\theta}V=0$. Letting $\vec{v}:=\left[r\text{\enspace\enspace}\theta\text{\enspace\enspace}\overset{\cdot}{r}\text{\enspace\enspace}\overset{\cdot}{\theta}\right] $, we linearize our system $\eqref{eq:1DBLinearizedEOMs}$ as $\overset{\cdot}{v}\,=A\vec{v}$, where $A:=\begin{bmatrix}0 & I \\ 
	A_3 & A_4\end{bmatrix}$,\\ $A_3=\begin{bmatrix}\frac{B\overset{\cdot}{\theta}\left(B\overset{\cdot}{\theta}-2L\right)}{\left(r^{2}+B\right)^{2}}-4r\frac{Br\overset{\cdot}{\theta}\left(B\overset{\cdot}{\theta}-2L\right)}{\left(r^{2}+B\right)^{3}}-\partial_{r}^{2}V & -\partial_{r,\theta}^{2}V \\2\frac{\left(B\overset{\cdot}{\theta}-L\right)\left(2r+B^{2}\right)}{r^{2}\left(r+B^{2}\right)^{2}}\overset{\cdot}{r}+\left(2\frac{r^{2}+B}{Br^{3}}-\frac{2}{Br}\right)\partial_{\theta}V-\frac{r^{2}+B}{Br^{2}}\partial_{r,\theta}^{2}V& -\frac{r^{2}+B}{Br^{2}}\partial_{\theta}^{2}V\end{bmatrix}$, \\$A_4=\begin{bmatrix}0 & \frac{Br\left(B\overset{\cdot}{\theta}-2L\right)+B^{2}r\overset{\cdot}{\theta}}{\left(r^{2}+B\right)^{2}}\\2\frac{L-B\overset{\cdot}{\theta}}{r\left(r+B^{2}\right)}&-\frac{2B}{r\left(r+B^{2}\right)}\overset{\cdot}{r}\end{bmatrix}$, and $I,0$ are the $3\times 3$ identity and zero matrices, respectively. At RE ($\overset\cdot r=\overset{\cdot}{\theta}=\partial_{r}V=\partial_{\theta}V=0$) we get $A_{RE}=\begin{bmatrix}0 & I \\ 
	A_{3_{RE}} & A_{4_{RE}}\end{bmatrix}$,\\where
$A_{3_{RE}}=\begin{bmatrix}-\partial_{r}^{2}V & -\partial_{r,\theta}^{2}V \\-\frac{r^{2}+B}{Br^{2}}\partial_{r,\theta}^{2}V & -\frac{r^{2}+B}{Br^{2}}\partial_{\theta}^{2}V\end{bmatrix}$ and $A_{4_{RE}}=\begin{bmatrix}0 & -\frac{2BLr}{\left(r^{2}+B\right)^{2}}\\2\frac{L-B\overset{\cdot}{\theta}}{r\left(r+B^{2}\right)}&0\end{bmatrix}$.\\
Calculating the characteristic polynomial of $A_{RE}$, we find: $z^{4}+c_{1}z^{2}+c_{0}$,\\where $c_{1}=\partial_{r}^{2}V+\frac{4BL^{2}}{\left(r^{2}+B\right)^{3}}+\frac{r^{2}+B}{Br^{2}}\partial_{\theta}^{2}V$ and $c_{0}=\frac{r^{2}+B}{Br^{2}}\left(\partial_{\theta}^{2}V\partial_{r}^{2}V-\left(\partial_{r,\theta}^{2}V\right)^{2}\right)$.

Note for linear stability we need imaginary eigenvalues.
So our linear stability conditions are:
\begin{align}\label{eq:1DB_Linear_Stability_Criteria}
	{\displaystyle  c_{1}^{2}\geq 4c_{0},\text{\enspace\enspace}c_{1}\geq 0\text{,\enspace\enspace and \enspace\enspace}c_{0}\geq 0.}
\end{align}
We will use these conditions in the configurations below to establish linear stability.
\vspace*{-5mm}\subsection{Colinear Configuration: $r>0$ and $\theta =0$}
\begin{figure}[H]
	\centering
	\begin{minipage}{0.45\textwidth}
		\centering
		\includegraphics[width=0.4\textwidth]{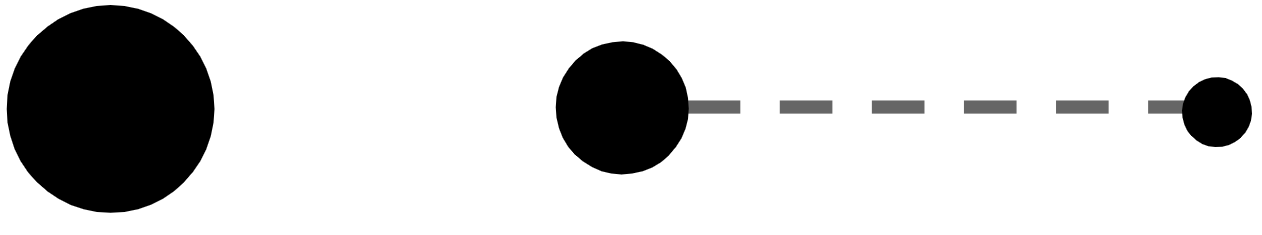}\\{$r>x_2$}
		\label{fig:1DB_Overlap_Config_A}
	\end{minipage}\hfill or \hfill
	\begin{minipage}{0.45\textwidth}
		\centering
		\includegraphics[width=0.4\textwidth]{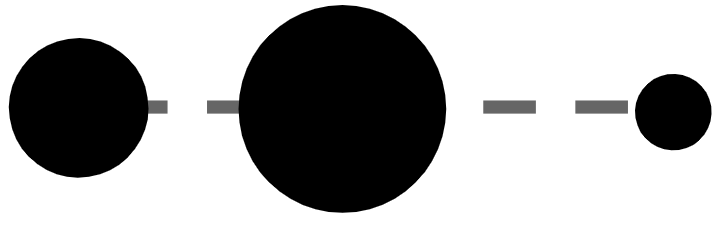}\\{$r<x_2$}
		\label{fig:1DB_Colinear_Config}
	\end{minipage}
\caption{\\\bfseries Colinear Non-Overlap and Overlap}{\text{ }\\[-1.801em]}
\end{figure}
Let us first look at the simpler of the two configurations, when the masses are colinear.  Note, there is a singularity when $r=x_{2}$, so this configuration naturally divides up into $r>x_{2}$ and $r<x_{2}.$ We will not address the situation where $r=x_{2}$, as a model like ours, which uses point masses and Newton's interaction potential, does not accurately model any real situation when bodies are this close. For work that attempts to model these scenarios, see Scheeres \cite{scheeres2007} where he studies the case when the orbiting bodies come in contact with each other. Note also that the case $r<x_{2}$ is physically impossible due to the ``massless rod'' colliding with the $\vec{r}_1$ mass. So this model is explored more as a mathematical curiosity.

From the radial requirement \eqref{eq:1DB_Radial_REQ} we find: $\frac{x_{1}}{\left(r-x_{2}\right)\left\vert r-x_{2}\right\vert}+\allowbreak \frac{x_{2}}{\left(r+x_{1}\right)^2}\allowbreak =\frac{L^{2}}{\left(r^{2}+B\right)^{2}}r.$ So since $r>0$, our bifurcation equation is: 
\begin{align}\label{eq:1DB_Colinear_L_Squared}
	{\displaystyle L^{2}(r)=\frac{\left(r^{2}+B\right)^{2}}{r}\left(
		\frac{x_{1}}{\left(r-x_{2}\right)\left\vert r-x_{2}\right\vert}+\allowbreak \frac{x_{2}}{\left(r+x_{1}\right)^{2}}\allowbreak\right).}
\end{align}

Graphing\ $L^2$ for $x_{2}<x_{1}$ with $\left(x_{1},M_{1}\right) =\left(\frac{6}{10},\frac{1}{2}\right)$, and $x_{1}<x_{2}$ with $\left(x_{1},M_{1}\right) =\left(\frac{1}{10},\frac{1}{2}\right)$ we have Figure \ref{fig:1DB_Colinear_LS}. Given $x_{1},M_{1},$ one can vary a horizontal line (representing the square of the angular momentum) in such a graph, using it as a bifurcation parameter. For $x_{2}<x_{1}$ (Figure \ref{fig:1DB_Colinear_LSA}), there appear to be RE only in the non-overlapped $\left(r>x_{2}\right) $ part of the domain. So no RE for sufficiently low angular momenta $L^{2}$. However, for sufficiently large $L^{2}$, the line intersects the graph at two $r$ values of different colinear RE. For $x_{1}<x_{2}$ (in Figure \ref{fig:1DB_Colinear_LSB}), we see RE in the overlapped part of the domain as well. This particular graph suggests there is only one RE for every $L^{2}$ in the overlap. However, we will see below that the overlapped region is a bit more complicated.
\begin{figure}[H]
	\captionsetup[subfigure]{justification=centering}
	\centering
	\subfloat[$L^2$ for $x_{2}<x_{1}$\\$\left(x_{1},M_{1}\right) =\left(\frac{6}{10},\frac{1}{2}\right)$]
	{\includegraphics[width=0.35\textwidth]{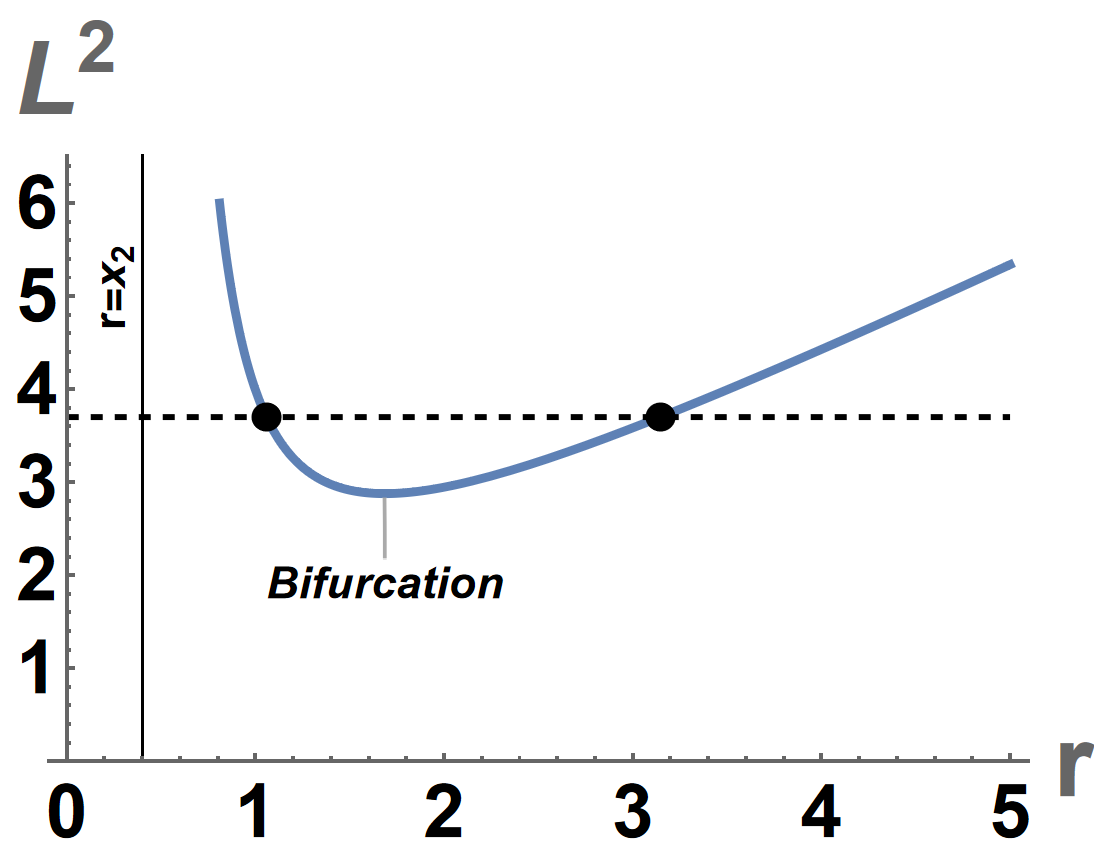}\label{fig:1DB_Colinear_LSA}}
	\hspace{1cm}
	\subfloat[$L^2$ for $x_{1}<x_{2}$\\$\left(x_{1},M_{1}\right) =\left(\frac{1}{10},\frac{1}{2}\right)$]
	{\includegraphics[width=0.35\textwidth]{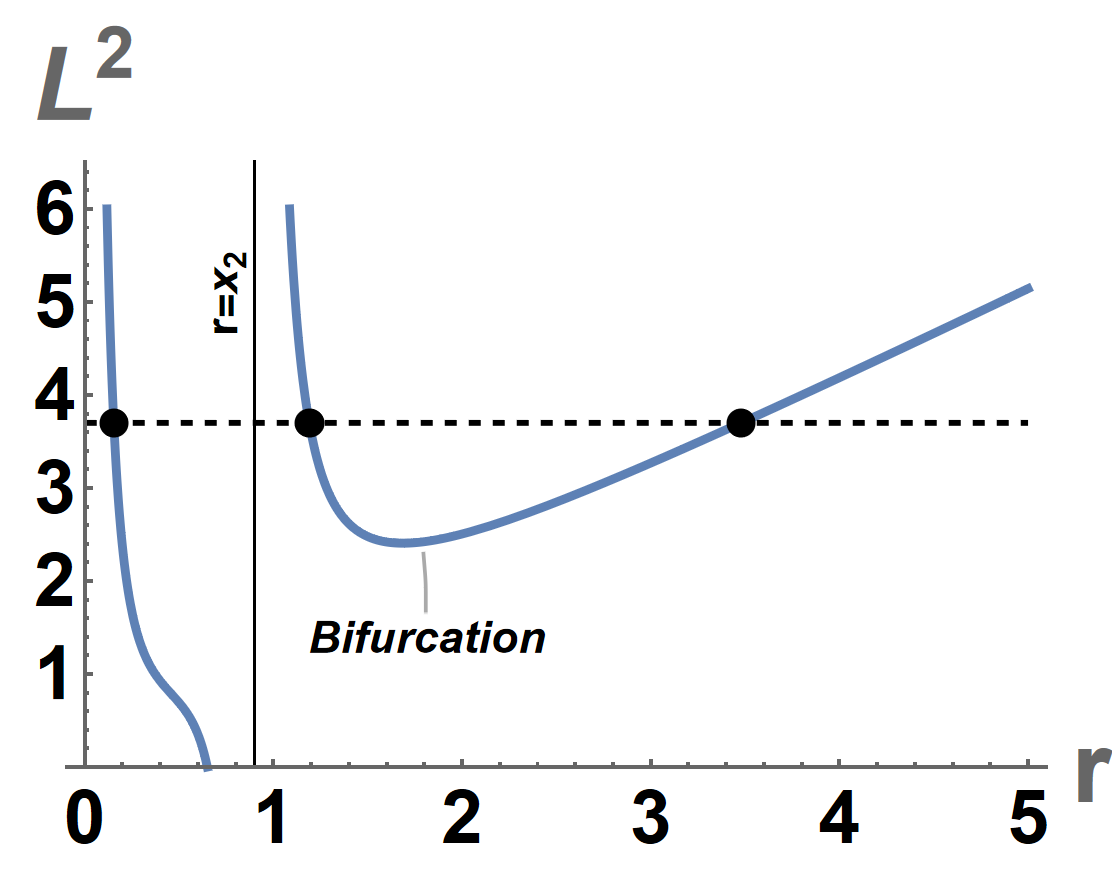}\label{fig:1DB_Colinear_LSB}}
	{\includegraphics[width=0.2\textwidth]{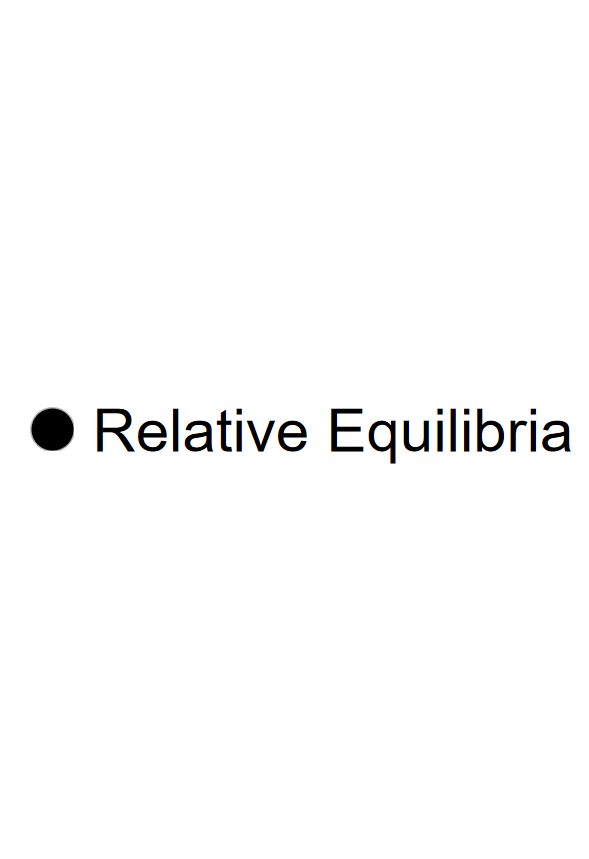}}
	\caption{\\\bfseries Colinear $L^2$ Curves}{\vspace{0.8em}\footnotesize The dashed line is a hypothetical constant angular\\momentum, the solid thin line is the overlap radius.\\[-0.801em]}
	\label{fig:1DB_Colinear_LS}
\end{figure}%

\vspace*{-5mm}\subsubsection{Subcase 1: Non-Overlapped: $r>x_{2}$}
\begin{figure}[H]
	\centering
	\begin{minipage}{0.35\textwidth}
		\centering
		\includegraphics[width=0.4\textwidth]{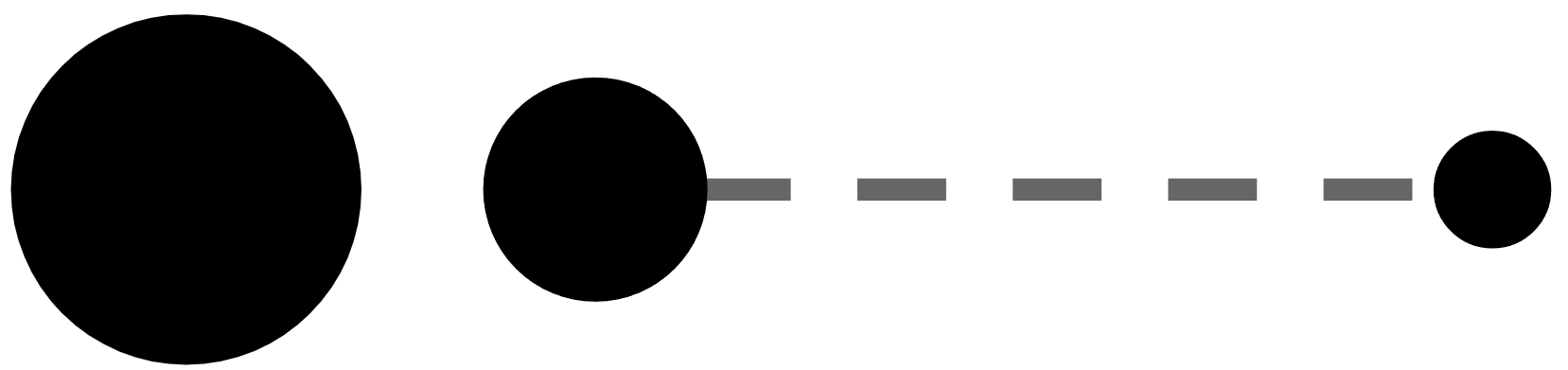}\\ \footnotesize Unstable RE
		%	\captionsetup{labelformat=empty}
		\label{fig:1DB_Overlap_Config_A}
	\end{minipage}\hfill
	\begin{minipage}{0.47\textwidth}
		\centering
		\includegraphics[width=0.6\textwidth]{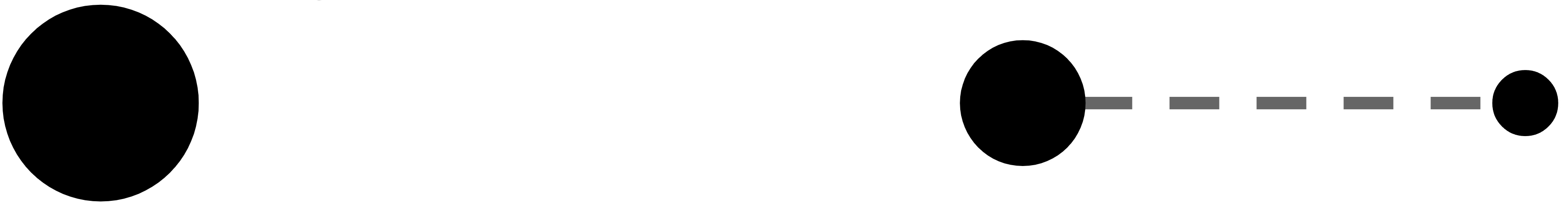}\\ \footnotesize Stable RE
		%	\captionsetup{labelformat=empty}
		\label{fig:1DB_Overlap_Config_B}
	\end{minipage}
	\caption{\\\bfseries Colinear Non-Overlap Visualizations}{\vspace{0.8em}\footnotesize Shows the radii of unstable/stable RE, respectively.\\[-0.801em]}
		\label{fig:1DB_Overlap_Config}
\end{figure}
To examine bifurcations when the dumbbell is not overlapping, we look at $r>x_{2}$. So let $R:=r-x_{2}$ or $r=R+x_{2}.$ With this substitution, the dumbbells do not overlap for $R\in\left(0,\infty\right).$ 
\begin{theorem}[Single Bifurcation of RE for Non-Overlap Colinear Angular Momenta]
	Let a dumbbell and point mass be in a colinear non-overlapped RE. There are no RE for sufficiently low angular momenta, and two RE for angular momenta greater than some value $L_b>0$.
\end{theorem}
\begin{proof}
We will check that in the non-overlap region, the graph of $L^2$ is qualitatively similar to Figure \ref{fig:1DB_Colinear_LS}, that is, concave up ( $(L^2)^{\;\prime \prime }(R)>0$) with one positive minimum.
For our analysis, it is helpful to do the following change of variable: let $x_{1}=\frac{u}{1+u}$ and $x_{2}=\frac{1}{1+u}.$ Note that we still have $x_{1}+x_{2}=1,$ but now we have characterized $x_{1},x_{2}$ with one variable $0<u<\infty $. Making these changes in $(L^2)^{\;\prime \prime}(R)$, and multiplying by the positive expression\\ $h(R):=\frac{1}{2}M_{1}^{2}R^{4}(R+1)^{4}\left(u+1\right)^{4}(1+R+uR)^{3}$, we find:\\
$\textstyle h(R)L^2=M_{1}(3M_{1}+2)u(u+1)^{5}R^{8}+8M_{1}u(u+1)^{4}(3M_{1}+u+1)R^{7}\\
\text{\enspace} +6u(u+1)^{3}\left(M_{1}^{2}(3u+14)+2M_{1}\left(u^{2}+5u+1\right)+u\right) R^{6}\\
\text{\enspace}+4u(u+1)^{2}\left(3M_{1}^{2}\left(u^{2}+8u+14\right)+2M_{1}\left(u^{3}+11u^{2}+22u+1\right)+3u(2u+1)\right)R^{5}\\
\text{\enspace}\textstyle+u(u+1)[3M_{1}^{2}\left({u}^{3}{+19u}^{2}{+71u+70}\right)+2M_{1}\left({u}^{4}{+27u}^{3}{+123u}^{2}{+137u+1}\right)\\
\text{\enspace}\text{\enspace}\text{\enspace}\textstyle+3u\left({12u}^{2}{+23u+2}\right)] R^{4}\\
\textstyle\text{\enspace}+12u\left(M_{1}^{2}\left(u^{3}+9u^{2}+21u+14\right)+M_{1}u\left(u^{3}+11u^{2}+29u+21\right)+u^{2}\left(2u^{2}+8u+7\right)\right) R^{3}\\
\text{\enspace}+2u\left(M_{1}^{2}\left(9u^{2}+42u+42\right)+2M_{1}u\left(6u^{2}+32u+35\right)+u^{2}\left(3u^{2}+22u+28\right)\right)R^{2}\\
\text{\enspace}+4u\left(3M_{1}^{2}(u+2)+M_{1}u(5u+11)+u^{2}(2u+5)\right) R+3u(M_{1}+u)^{2}$.\\
Observe (by inspection) the expression is always positive, so $L^2$ is concave up, giving at most one bifurcation. We also check that there is a critical point. We observe $(L^2)^{\;\prime}(R)=0$ when:\\
$\text{\enspace\enspace}0=M_{1}R^{6}+3M_{1}R^{5}+\left(3M_{1}\left(x_{1}^{2}+x_1 x_{2}+x_{2}^{2}\right) -3x_{1}x_{2}\right) R^{4}$\newline
$\text{\enspace\enspace\enspace\enspace}+\left(M_{1}\left(x_{1}^{3}-2x_{1}x_{2}\left(5x_{2}-1\right)+x_{2}^{3}\right)-3x_{1}x_{2}(2x_{1}+1)\right)R^{3}$\newline
$\text{\enspace\enspace\enspace\enspace}-3x_{1}x_{2}(3+2x_{2})(M_{1}x_{2}+x_{1})R^{2}$\\
$\text{\enspace\enspace\enspace\enspace}-3x_{1}x_{2}\left(2x_{2}+1\right)(M_{1}x_{2}+x_{1})R-2x_{1}x_{2}^{2}(M_{1}x_{2}+x_{1})$.

Note that this expression is continuous and takes negative (when $R=0$ for instance), and positive (as $R\rightarrow \infty $) values. Therefore, by the intermediate value theorem $(L^2)^{\;\prime }(R)$ has a zero. So $L^{2}$ must have a minimum, giving a bifurcation of the number of RE as $L^{2}$ is varied. Lastly, we check that this minimum is positive, giving these RE physical relevance (real angular momenta). Observe by inspection that for $r>x_{2}$ we have (using \eqref{eq:1DB_Colinear_L_Squared}): $L^{2}=\frac{1}{r}\left(r^{2}+B\right)^{2}\left(\frac{x_{1}}{\left(r-x_{2}\right)^{2}}+\allowbreak\frac{x_{2}}{\left(r+x_{1}\right)^{2}}\allowbreak\right)$, which is always positive. Therefore, for $r>x_{2}$ we see no RE for small angular momenta, and two RE for angular momenta larger than some positive bifurcation value $L_b$.
\end{proof}
Now let's look at the overlap case.
\vspace*{-5mm}\subsubsection{Subcase 2: Overlap $r<x_2$}
\begin{figure}[H]
	\centering
	\begin{minipage}{0.45\textwidth}
		\centering
		\includegraphics[width=0.4\textwidth]{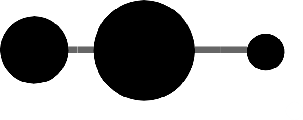} 
%	\captionsetup{labelformat=empty}
		\label{fig:1DB_Overlap_Config_A}
	\end{minipage}\hfill
	\begin{minipage}{0.45\textwidth}
		\centering
		\includegraphics[width=0.4\textwidth]{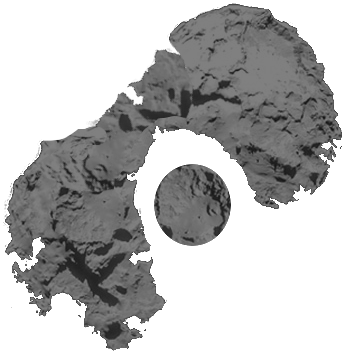}\\ \footnotesize Could be approximated by overlap model
%	\captionsetup{labelformat=empty}
		\label{fig:1DB_Overlap_Config_B}
	\end{minipage}
\caption{\\\bfseries Colinear Overlap Visualizations}{\text{ }\\[-1.801em]}
\end{figure}
Note that for $r<x_{2}$, the bodies are overlapping. When this occurs, $\vec{r}_{p}$ is located within the massless rod of the dumbbell. So, obviously this has limited practical application. This model could be used to loosely represent oddly shaped orbiting asteroids, the larger of which is shaped roughly like a bent dumbbell, and the smaller asteroid could be ``orbiting'' the larger in a RE within the concavity of the dumbbell (see the image above). But mostly, we study this case for mathematical curiosity. So let $R:=\frac{r}{x_{2}-r}$ or $r=x_{2}\frac{R}{R+1}.$ With this substitution, the bodies overlap for $R\in\left(0,\infty\right).$ From \eqref{eq:1DB_Colinear_L_Squared} we have:
\begin{align}\label{eq:1DB_Colinear_L_Squared_Overlap}
	{\displaystyle  L^{2}(R)=\frac{R+1}{x_2 R}\left(B+x_{2}^{2}\frac{R^{2}}{\left(R+1\right)^{2}}\right)^{2}\left(\frac{x_{2}}{\left(x_{2}\frac{R}{R+1}+x_{1}\right)^{2}}-\frac{x_{1}}{\left(x_{2}\frac{R}{R+1}-x_{2}\right)^{2}}\allowbreak \allowbreak\right).}
\end{align}

Then, multiplying $(L^{2})^{\;\prime }(R)$ by the positive expression $\frac{M_{1}^{2}x_{2}(R+1)^{2}(R+x_{1})^{3}}{2\left(M_{1}x_2 R^{2}+x_{1}(R+1)^{2}\right)}$, and collecting powers of $R$ we define:\\
$\textstyle\text{\enspace}k(R):=-2x_{1}(M_{1}x_{2}+x_{1})R^{6}-3x_{1}(2x_{1}+1)(M_{1}x_{2}+x_{1})R^{5}-3x_{1}^{2}\left(2x_{1}+3\right)(M_{1}x_{2}+x_{1})R^{4}\\
\text{\enspace\enspace}+\left(-10x_{1}^{5}-(11M_{1}+6)x_{1}^{4}x_{2}+3(1-3M_{1})x_{1}^{3}x_{2}^{2}+3(M_{1}-1)x_1 x_{2}^{4}+M_{1}x_{2}^{5}\right)R^{3}\\
\text{\enspace\enspace}-3x_{1}x_{2}\left(x_{2}-x_{1}\right)\left((2-M_{1})\left(x_{1}^{2}+x_1 x_{2}+x_{2}^{2}\right)+x_1 x_{2}\right)R^{2}\\
\text{\enspace\enspace}-3x_{1}\left(x_{2}-x_{1}\right)\left(x_{1}^{2}+x_1 x_{2}+x_{2}^{2}\right)R-x_{1}^{2}\left(x_{2}-x_{1}\right)\left(x_{1}^{2}+x_1 x_{2}+x_{2}^{2}\right).$

We will learn from analyzing $k$ that we have two qualitatively different configurations; one when $x_{1}<\frac{1}{2}<x_{2},$ and another when $x_{2}<\frac{1}{2}<x_{1}$.
\paragraph{Subcase 2a: $x_1<\frac{1}{2}<x_2$}
\begin{figure}[H]
	\centering
	\includegraphics[width=0.2\textwidth]{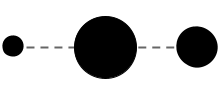} 
%	\captionsetup{labelformat=empty}
	\label{fig:1DB_Colinear_Config}
	\caption{\\\bfseries Colinear Overlap with Small $x_1$}{\text{ }\\[-1.801em]}
\end{figure}
\begin{theorem}[One or Three RE per Angular Momentum for Overlap Colinear with Small $x_1$]
	Let a dumbbell and point mass be in a colinear overlapped RE. For $x_{1}<x_{2}$, with $x_1$ sufficiently large, we have $(L^{2})^{\;\prime}{ (R)}<0$ for all $R,$ and therefore only one RE for each $L^{2}$ (Figure \ref{fig:1DB_Overlap_x1_small}a). However, for $x_{1}\ll x_{2},$ we have $(L^{2})^{\;\prime}>0$ for one interval, and therefore three RE for some interval of $L^{2}$ values, and one RE otherwise (Figure \ref{fig:1DB_Overlap_x1_smallb}).
\end{theorem}
\begin{figure}[H]
	\captionsetup[subfigure]{justification=centering}
	\centering
	\subfloat[$L^2$ for $x_{1}<x_{2}$\\$(x_1,x_2,M_1)=(0.1,0.9,0.5)$\label{fig:1DB_Overlap_x1_smalla}]
{\includegraphics[width=0.35\textwidth]{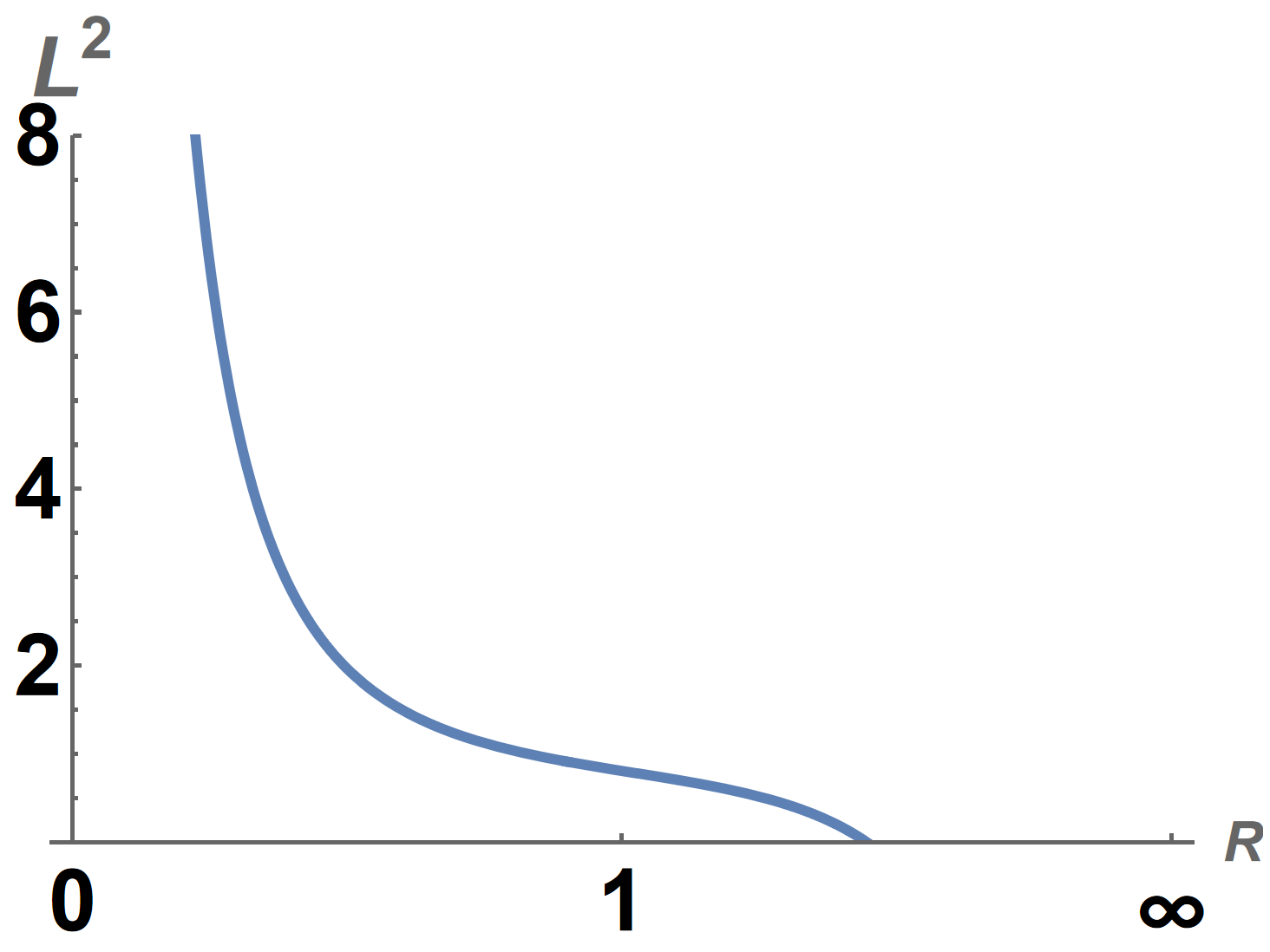}}
	\hspace{1cm}
	\subfloat[$L^2$ for $x_{1}\ll x_{2}$\\$(x_1,x_2,M_1)=(0.008,0.992,0.5)$\label{fig:1DB_Overlap_x1_smallb}]
	{\includegraphics[width=0.38\textwidth]{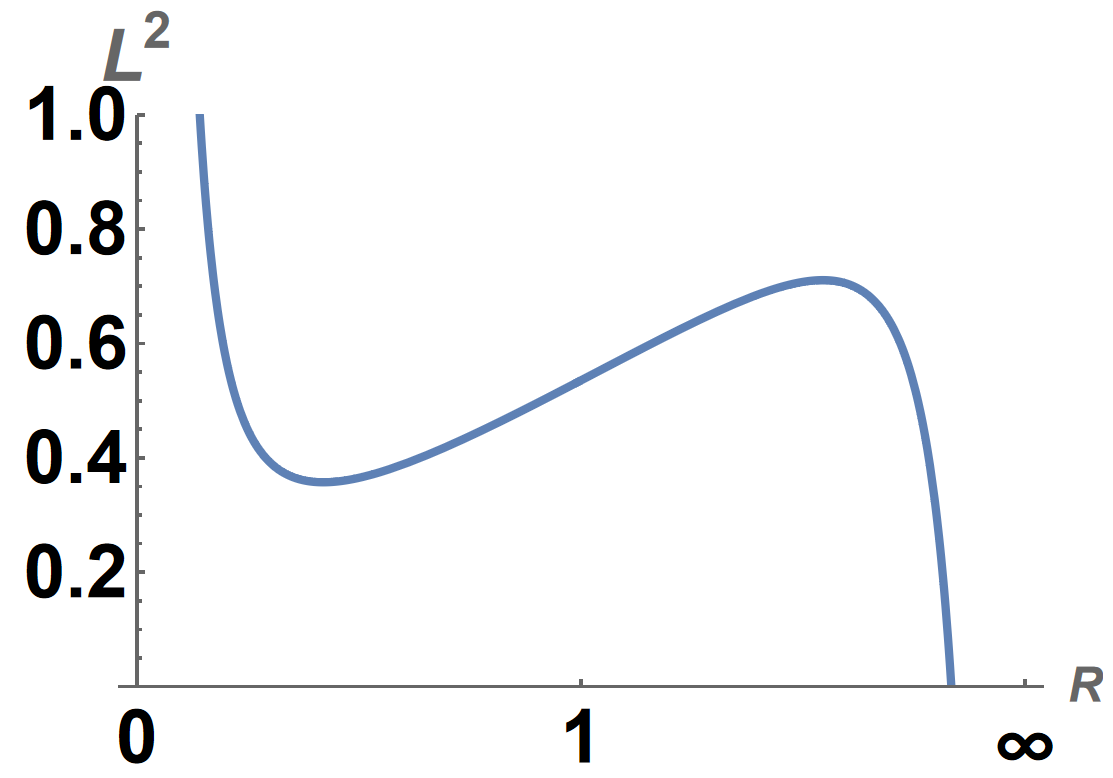}}
	\caption{\\\bfseries Colinear Overlap $L^2$ Curves with Small $x_1$}{\vspace{0.8em}\footnotesize One RE per $L^2$ for most $x_{1}<x_{2}$, but one or three RE for $x_{1}\ll x_{2}$.\\[-1.801em]}
	\label{fig:1DB_Overlap_x1_small}
\end{figure}
In order to visually cover the entire range $R\in\left(0,\infty\right)$, here and in other graphs in this paper, we graph the horizontal axis as $R=\frac{z}{2-z}$ for $z\in(0,2)$. 
\begin{proof}
We look for zeros of $(L^2)^{\;\prime }{ (R).}$ Observe that in $k(R),$ every term is negative except perhaps the $R^{3}$ term. Also, every term has an $x_{1}$ factor except the $R^{3}$ term. So, as $x_{1}\rightarrow 0,$ every term except $R^{3}$ gets arbitrarily small. The $R^{3}$ term can remain large, since there is a $x_{2}^{5}$ in the coefficient which remains (and increases) as $x_{1}\rightarrow 0.$ Note in particular, that the sign of $k(R)$ will change since the positive $R^{3}$ term will survive, and the others (while being negative) will shrink to zero. So, for fixed $R$ and $x_{1}\ll x_{2}$ we find positive slope $(L^2)^{\;\prime }(R)$, and as you then increase $R$, larger powers of $R$ overtake the $R^{3}$ term, and the slope $(L^2)^{\;\prime }(R)$ changes from positive back to negative (Figure \ref{fig:1DB_Overlap_x1_smallb}).

To ensure that there is only one interval in which $(L^2)^{\;\prime }>0$, we bound the number of zeros of $(L^2)^{\;\prime }(R)$ using Descartes' rule of signs on $-k(R).$ Observe that when $x_{1}<x_{2}$, all of the coefficients of $-k(R)$ except the coefficient of $R^{3}$ are trivially positive. And as we observed above, the coefficient of $R^{3}$ can change sign depending on the size of $x_{1}$. So, using the rule of signs, we observe that the sign changes twice or not at all, and we conclude that the number of (positive) roots is either two or zero. We also observe that as $R\rightarrow 0$ the constant term gives us $x_{1}^{2}\left(x_{1}-x_{2}\right)\left(x_{1}^{2}+x_1 x_{2}+x_{2}^{2}\right) <0$ (since $x_{1}<x_{2}$). And as $R\rightarrow \infty $, we see from the $R^{6}$ term that $-2x_{1}(M_{1}x_{2}+x_{1})<0$, and the boundary behavior is what we suspect from the graphs above. So, with only two potential roots for $(L^2)^{\;\prime}$, we see $L^2$ can have at most one increasing interval.\\[-1.801em]
\end{proof}
\vspace{-5pt}So if the overlapped dumbbell mass $x_1$ is very small, for the right choice of angular momentum, the two bodies find themselves in a RE at one of three (overlapped) radii. Below, we graph $k=\frac{d}{dR}k=0$ (representing points in parameter space at which the extrema and inflection points of graphs like \ref{fig:1DB_Overlap_x1_smallb} collide) in Figure \ref{fig:1DB_Colinear_Overlap_Bifurcation_Regions}. It shows regions of the $x_{1}M_{1}$-plane in which we have qualitatively different graphs of $(L^2)\left(R\right)$. Region $A$ has only one RE ($(L^2)^{\;\prime }<0$ for all $R,$ Figure \ref{fig:1DB_Overlap_x1_smalla}), and $B$ has one or three RE depending upon $L$ ($(L^2)^{\;\prime }>0$ on one interval, Figure \ref{fig:1DB_Overlap_x1_smallb}).
\begin{figure}[H]
		\centering
		\includegraphics[width=0.3\textwidth]{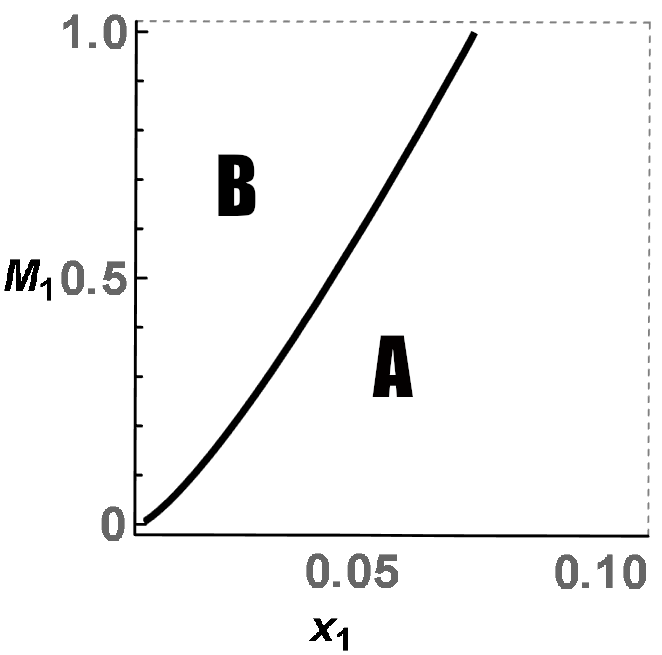} 
		\caption{\\\bfseries Colinear Overlap Bifurcation Regions}{\vspace{0.8em}\footnotesize Parameter space showing Region A which has 1 RE, and B which has 1 or 3 RE (depending upon $L^2$).\\[-1.801em]}
		\label{fig:1DB_Colinear_Overlap_Bifurcation_Regions}
\end{figure}
Below, we depict the orbital configurations of RE for a particular choice of $\left(x_{1},M_{1}\right) $ in Region $B$. Specifically, we depict the locations of the two bifurcation points which occur for $\left(x_{1},M_{1}\right) =\left(0.01, 0.75\right) $ at $L\approx 0.6536$ and$\;L\approx 0.05795,$ respectively.
\begin{figure}[H]
	\centering
	\begin{minipage}{0.45\textwidth}
		\centering
		\includegraphics[width=0.8\textwidth]{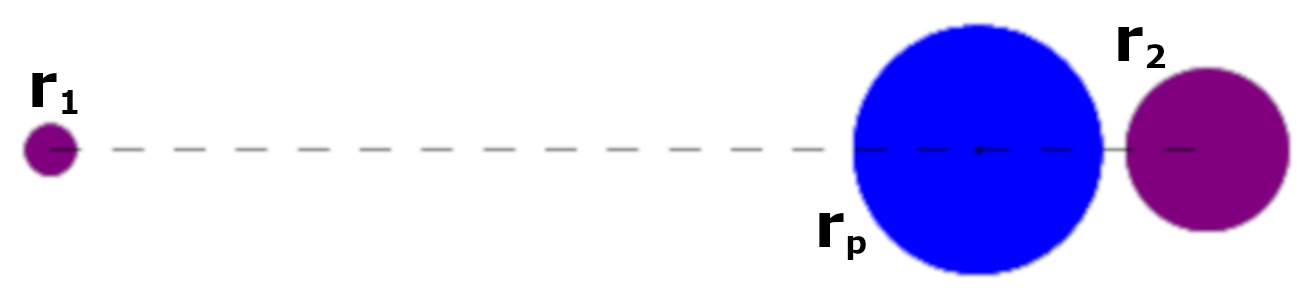}
		%\newline{$r\approx 0.188372$, $L\approx 0.6536$}
		\newline{$r\approx 0.1884$, $L\approx 0.6536$}
		\label{fig:1DB_Bifurcation_Config_A}
	\end{minipage}\hfill
	\begin{minipage}{0.45\textwidth}
		\centering
		\includegraphics[width=0.8\textwidth]{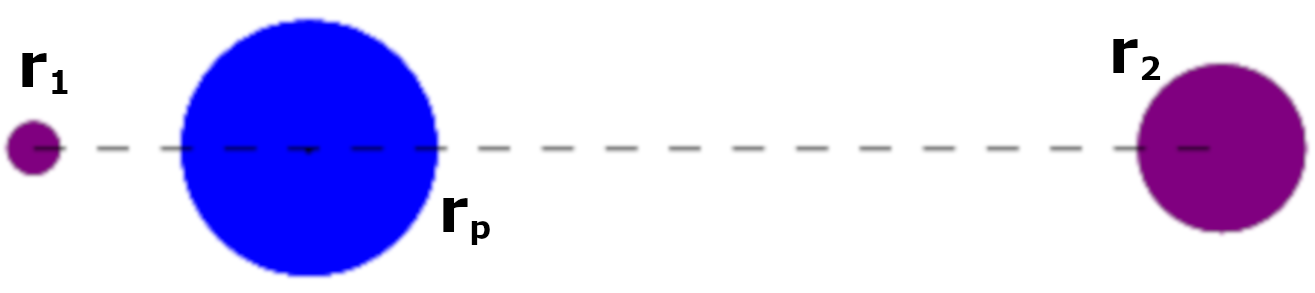} 
		%\newline{$r\approx 0.758019$, $L\approx 0.05795$}
		\newline{$r\approx 0.7580$, $L\approx 0.05795$}
		\label{fig:1DB_Bifurcation_Config_B}
	\end{minipage}
\caption{\\\bfseries Colinear Overlap Bifurcation Radii Visualization}{\vspace{0.8em}\footnotesize Depicts the RE found at the two bifurcation angular momenta when $\left(x_{1},M_{1}\right) =\left(0.01, 0.75\right)$.\\[-0.801em]}
\end{figure}
\paragraph{Subcase 2b: $x_2<\frac{1}{2}<x_1$}
\begin{figure}[H]
	\centering
	\includegraphics[width=0.2\textwidth]{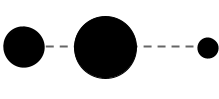} 
%	\captionsetup{labelformat=empty}
\caption{\\\bfseries Colinear Overlap with Large $x_1$}{\text{ }\\[-1.801em]}
	\label{fig:1DB_Colinear_Config}
\end{figure}
Figure \ref{fig:1DB_Colinear_LSA} implies that in the overlap ($r<x_{2}$) there are no RE when $x_2<x_1$. To verify this, we show $L^{2}<0$, which would correspond to $L$ with nonzero imaginary part, which are not physically realizable angular momenta. From \eqref{eq:1DB_Colinear_L_Squared_Overlap} we have:\\
$L^{2}=-\frac{\left(\left(x_{2}^{2}+B\right) R^{2}+\left(2y+1\right) B\right)^{2}\left(\left({R}+2x_{1}\right) x_{1}\allowbreak {R}+\left(x_{1}^{3}-x_{2}^{3}\right)\right) \allowbreak }{x_{2}^{3}R\left({R}+1\right)\left({R}+x_{1}\right)^{2}}.$
Observe by inspection that when \\$x_{2}<x_{1},$ we have $L^{2}$ negative. Therefore, there are no physically real angular momenta for $x_{2}<x_{1}$\ in the radius $r<x_{2}$, and therefore $x_{2}<x_{1}$ has no RE. Having characterized the location of the colinear RE, let us look at their stability.
\vspace*{-5mm}\subsubsection{Energetic Stability of Colinear}
\label{1DB_Colinear_Energetic_Stability}As we saw in Section \ref{AmendedPotential}, to determine energetic stability we check if the RE are strict minima of the amended potential $V$. We are looking for a positive definite $H$, the Hessian of $V$. Note from \eqref{eq:1DB_Mixed_PartialV_AtRE} $\partial_{\theta,r}V|_{RE}=0$ when $\theta =0$, so our Hessian is diagonal. For stability we need $\partial^2_rV, \partial^2_{\theta}V|_{RE}>0$. Recall that we can determine the sign of $\partial_{r}^{2}V|_{RE}$ by the slope of our $L^2$ curve. Next, calculating $\partial_{\theta}^{2}V|_{RE}$, we have: $\frac{x_1 x_2 r}{\left(r^{2}-2x_{2}r+x_{2}^{2}\right)^{3/2}}-\frac{x_1 x_2 r}{\left(r^{2}+2x_{1}r+x_{1}^{2}\right)^{3/2}}.$ Observe this is greater than zero when $r>\frac{x_2-x_1}{2}.$ Note that this is always true when the bodies are not overlapping (i.e., $r>x_{2}$). And therefore, when $(L^2)^{\;\prime}>0$, we conclude that our critical points are strict minima of the amended potential and energetically stable.

Looking back at Figure \ref{fig:1DB_Colinear_LS}, we recall that for sufficiently large angular momenta, in the non-overlap region ($r>x_2$) we have two RE, and now we see that the one with a greater radius is a strict minimum, and the other a saddle of $V$.

When the bodies are overlapping ($r<x_2$) and $x_1>\frac{1}{2}$ we find physically unreal angular momenta ($L^{2}<0$). For $x_1<\frac{1}{2}$ (see Figure \ref{fig:colinear_overlap_energetic_stability}) we find the region $r>\frac{x_2-x_1}{2}$ (to the right of the solid line) intersects with the region $(L^2)^{\;\prime}>0$ (below the dashed line) for low $x_{1}$, and therefore we have stability. We can therefore conclude that our critical points are energetically stable in this region.

In the overlap case when $r<\frac{x_2-x_1}{2}$, the critical points are maxima of the amended potential when $(L^2)^{\;\prime}<0,$\ and saddle points otherwise. However, since the kinetic energy in $\eqref{eq:1DB_Lagrangian}$ is positive definite, all the points in the overlap with $r<\frac{x_2-x_1}{2}$ represent saddles in the energy manifold.

Note that in the overlap region when $(x_1,M_1)=(0.008,\frac{1}{2})$ (as in Figure \ref{fig:1DB_Overlap_x1_small}) and with appropriate $L^2$, the second of the three RE occurs at a positive slope. In particular, when $L^2 = \frac{7}{10}$ we have three RE at $r\in(0.0865293, 0.721838, 0.80048)$ with the second and third one being larger than $\frac{x_2-x_1}{2}=0.492$. Therefore, the configuration in the Figure \ref{fig:1DB_Colinear_Stable} is energetically stable (we have curved the massless rod to make it a bit more physically feasible).
\begin{figure}[H]
	\centering
	\includegraphics[width=0.40\textwidth]{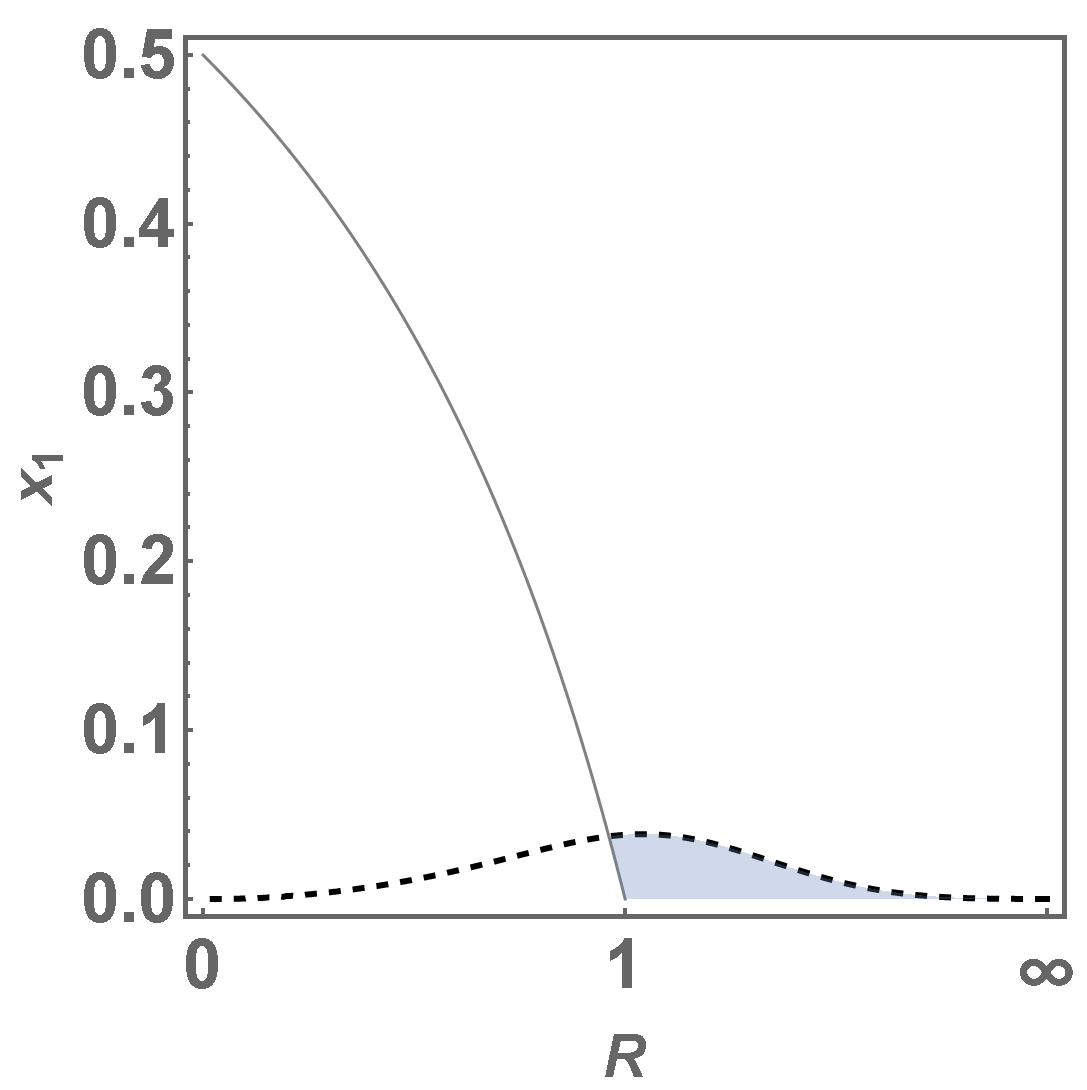}
	\caption{\\\bfseries Colinear Overlap Energetic Stability Region with $M_1=\frac{1}{2}$}{\vspace{0.8em}\footnotesize $r=\frac{x_2-x_1}{2}$ (solid line), $\partial_r L^2=0$ (dashed line), Stable Region (shaded)\\The shaded region is where $r>\frac{x_2-x_1}{2}$ and $\partial_r L^2>0$.\\We find stable RE when the point mass body is close to $\vec{r}_1$, and $x_1$ is small (see Figure \ref{fig:1DB_Colinear_Stable}).\\[-0.801em]}
	\label{fig:colinear_overlap_energetic_stability}
\end{figure}
\begin{figure}[H]
	\centering
	\includegraphics[width=0.4\linewidth]{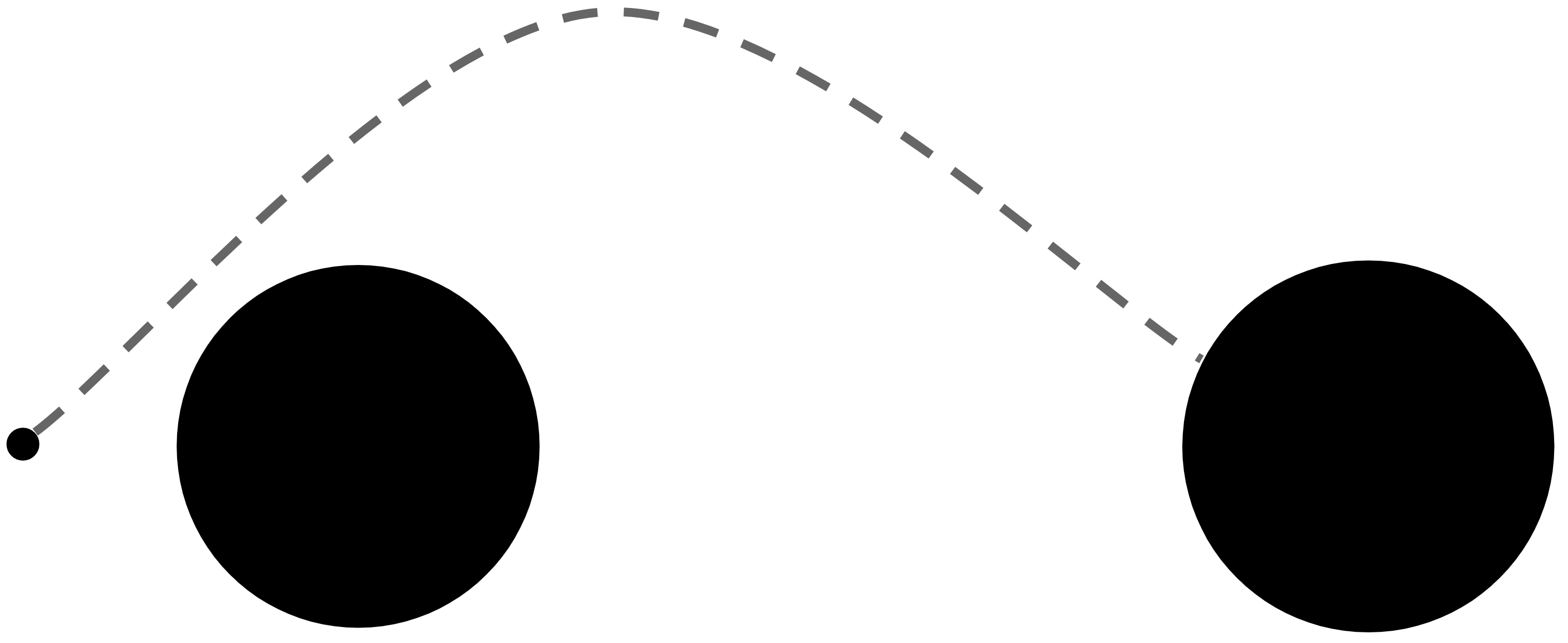}
%	\captionsetup{labelformat=empty}
	\caption{\\\bfseries Colinear Overlap Visualization, Energetically Stable}{\vspace{0.8em}\footnotesize Radius/masses chosen for energetic stability, and bent rod for physical feasibility.\\[-0.801em]}
	\label{fig:1DB_Colinear_Stable}
\end{figure}
%[Documented in Re_Scaled.nb - 8_2_21]
\vspace*{-5mm}\subsubsection{Linear Stability of Colinear}
\label{1DB_Colinear_Linear_Stability}Mapping the linear stability conditions \eqref{eq:1DB_Linear_Stability_Criteria} in the $Rx_{1}$-plane, for the non-overlapped colinear configuration, we have Figure \ref{fig:Col_M1_0p5}. The dashed curve going through the plane is $\partial_r L^{2}=0$. So, numerically it appears the linear and energetic stability coincide.

In Figure \ref{fig:Col_M1_0p5_Overlap}, we map the linear stability conditions \eqref{eq:1DB_Linear_Stability_Criteria} for the overlapped colinear configuration. We limit our graph to $x_{1}\in\left(0,\frac{1}{2}\right) $ where stability (and physically realizable angular momenta) could be found. Recall that we found
energetic stability when $r>\frac{x_2-x_1}{2}$ and $\partial_{r}L^{2}>0$. That is, the filled in region below the $\partial_r L^{2}=0$ dashed curve. So the energetic stability coincides with the linear stability in this region. In addition, we also see some linear stability when $r<\frac{x_2-x_1}{2}$, but this region is where $L^2<0$, and therefore represents physically unrealizable RE.
\begin{figure}[H]
	\centering
	\includegraphics[width=0.4\linewidth]{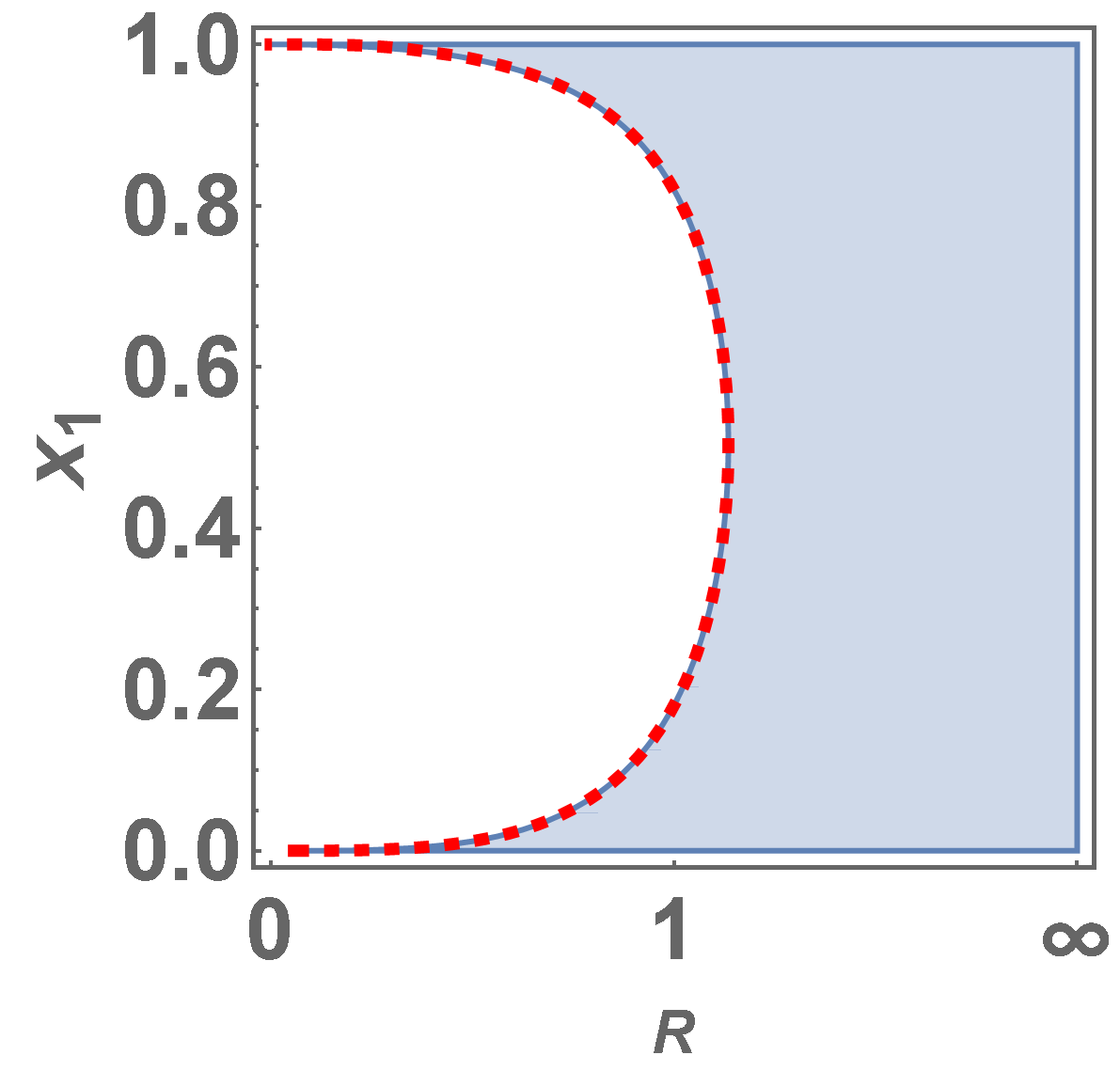}
	\caption{\\\bfseries Colinear Non-Overlap Energetic Stability Region with $M_1=\frac{1}{2}$}{\vspace{0.8em}\footnotesize $\partial_r L^2=0$ (dashed line), Energetically Stable Region (shaded)\\Energetic stability region coincides with $\partial_rL^2>0$. \\[-0.801em]}
	\label{fig:Col_M1_0p5}
\end{figure}
%[Documented BatonReduction2]
\begin{figure}[H]
	\centering
	\includegraphics[width=0.4\linewidth]{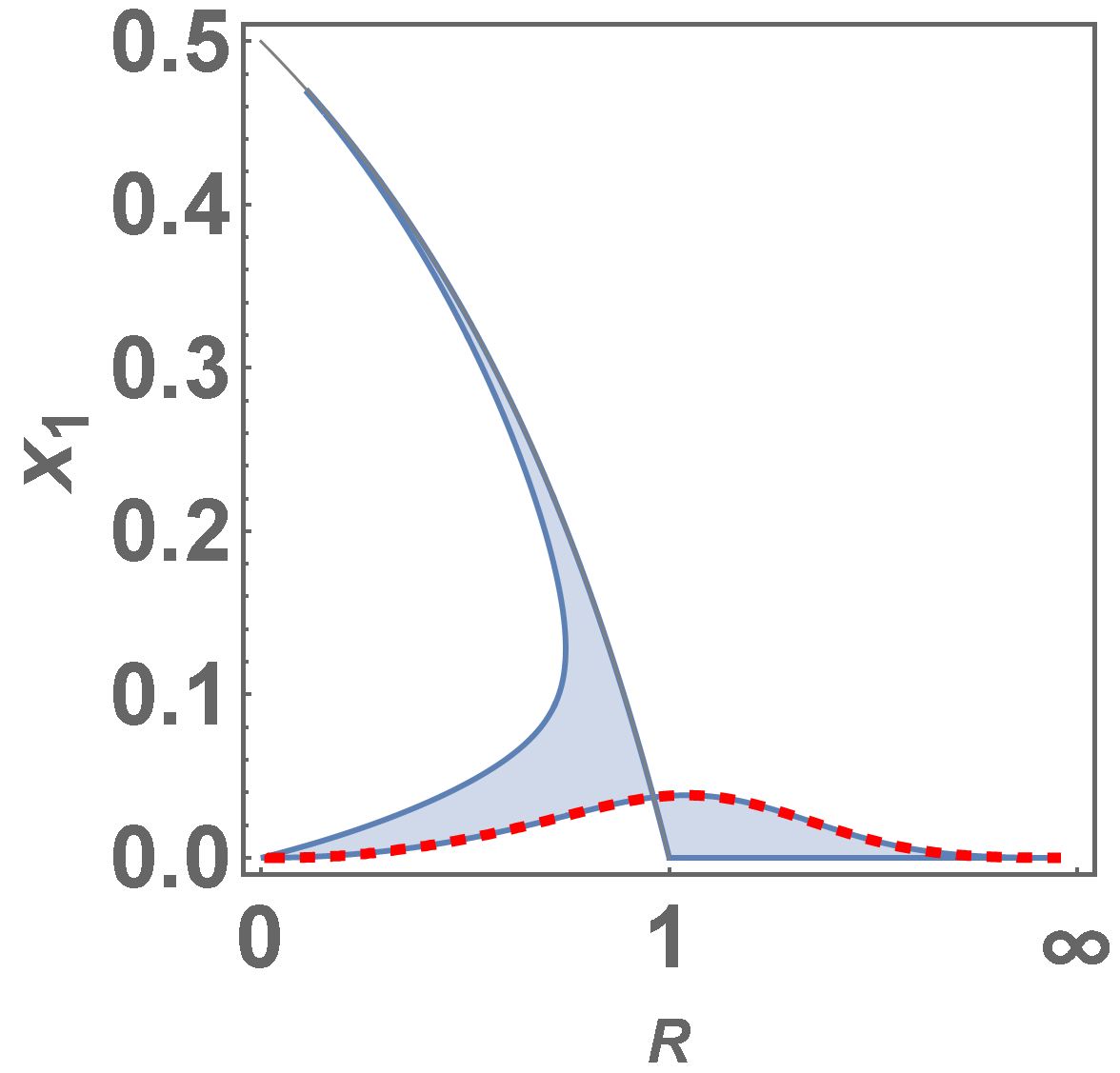}
	\caption{\\\bfseries Colinear Overlap Linear Stability Region with $M_1=\frac{1}{2}$}
	\label{fig:Col_M1_0p5_Overlap}{\vspace{0.8em}\footnotesize $r=\frac{x_2-x_1}{2}$ (solid line), $\partial_r L^2=0$ (dashed line), Stable Region (shaded)\\Shows linear stability when point mass body is close to $\vec{r}_1$, coinciding with energetic stability seen in Figure \ref{fig:colinear_overlap_energetic_stability}. The additional shaded region in this graph represents RE which are physically unrealizable due to having complex angular momenta.\\[-0.801em]}
\end{figure}
Now let us look at the other configuration having RE, when the bodies take the shape of an isosceles triangle.
\vspace*{-5mm}\subsection{Isosceles Triangle Configuration: $r>0$ and $d_{1}=d_{2}$}
\label{1DB_Isos}\begin{figure}[H]
	\captionsetup[subfigure]{justification=centering}
	\centering
	\subfloat[Equal $x_i$ masses\\$x_1=x_2$, $\theta=\frac{\pi}{2}$]
	{\includegraphics[width=0.33\textwidth]{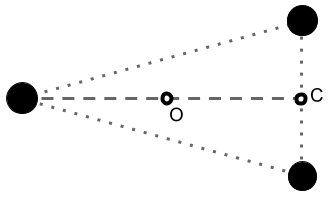}}
	\hspace{1cm}
	\subfloat[Unequal $x_i$ masses\\$x_1\ne x_2$, $\theta\ne\frac{\pi}{2}$]
	{\includegraphics[width=0.33\textwidth]{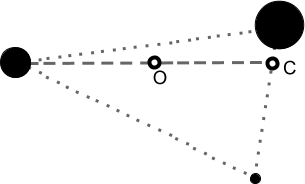}}
%	\captionsetup{labelformat=empty}
	\caption{\\\bfseries Isosceles Configuration}{\text{ }\\[-1.801em]}
\end{figure}
Observe from \eqref{eq:1DB_Distances} that $d_{1}=d_{2}$ implies a relationship (found by Beletskii and Ponomareva) between the radius and the angle $\theta\in\left(-\frac{\pi}{2},\frac{\pi}{2}\right]$:
\begin{align}\label{eq:1DB_Isosceles_Relation}
	{\displaystyle\cos\theta =\frac{x_{2}-x_{1}}{2r}.}
\end{align}
Note that $\theta =\frac{\pi}{2}$ when $x_{1}=x_{2}$, for any $r>0.$ And when $x_{1}<x_{2}$, we have a one-to-one relationship for $r\in\left[ \frac{x_{2}-x_{1}}{2},\infty\right)$, where $\theta \in\left[ 0,\frac{\pi}{2}\right)$. Similarly, for $x_{2}<x_{1}$, we have the one-to-one relationship for $r\in\left[ \frac{x_{1}-x_{2}}{2},\infty\right)$, and $\theta \in\left(-\frac{\pi}{2},0\right]$. So, we have a minimum radius $\frac{\left\vert x_{2}-x_{1}\right\vert}{2}.$

We see from \eqref{eq:1DB_PhiDot} and the radial requirement \eqref{eq:1DB_Radial_REQ} that when $d_{1}=d_{2}$ there is a simple relationship between rotational speed and the distances between the masses: $\left(\overset{\cdot}{\varphi}\right)^{2}=\frac{L^{2}}{\left(r^{2}+B\right)^{2}}=\frac{1}{d_{1}^{3}}.$ Rearranging this to find bifurcations of our RE with parameter $L,$ we have:
\begin{align}
	{\displaystyle L^{2}=\frac{\left(r^{2}+B\right)^{2}}{\left(r^{2}+M_{1}B\right)^{\frac{3}{2}}},}
\end{align}
where $B:=\frac{x_1 x_2}{M_1}$
is the dumbbell's scaled moment of inertia relative to its center of mass. Graphing for a couple of particular parameters $\left(x_1,M_1\right)$, we have Figure \ref{fig:1DB_Isos_L_Plot}.
\begin{theorem}[Zero, One, or Two RE per Angular Momentum for Isosceles Configuration]
	Let a dumbbell and point mass body be in an isosceles RE. The $x_{1}M_{1}$-plane divides into regions which differ by the possible number of RE. Particularly for\\$M_{1}<\frac{12 x_1 x_2}{(x_1 - x_2)^2+16x_1 x_2}$, and as $L^2$ increases from zero, we initially have no RE until a bifurcation point at $L_b=\sqrt[4]{\frac{256}{27}BM_{2}}$, where two RE appear at $r_b=\sqrt{B(3-4M_1)}$. Subsequently, a RE merges with the origin.  As a result, we have have zero, then two, then one RE. For larger $M_{1}$, the bifurcation curve has no local minima, but does have an absolute minimum at $r=\frac{\left\vert x_{2}-x_{1}\right\vert}{2}$. Therefore, we have zero, then one RE as $L^2$ increases from zero.
\end{theorem} 
\begin{figure}[H]
	\captionsetup[subfigure]{justification=centering}
	\centering
	\setcounter{subfigure}{0}
	\subfloat[$L^2$ with $(x_1,M_1)=(\frac{3}{4},\frac{9}{20})$]
	{\includegraphics[width=0.33\textwidth]{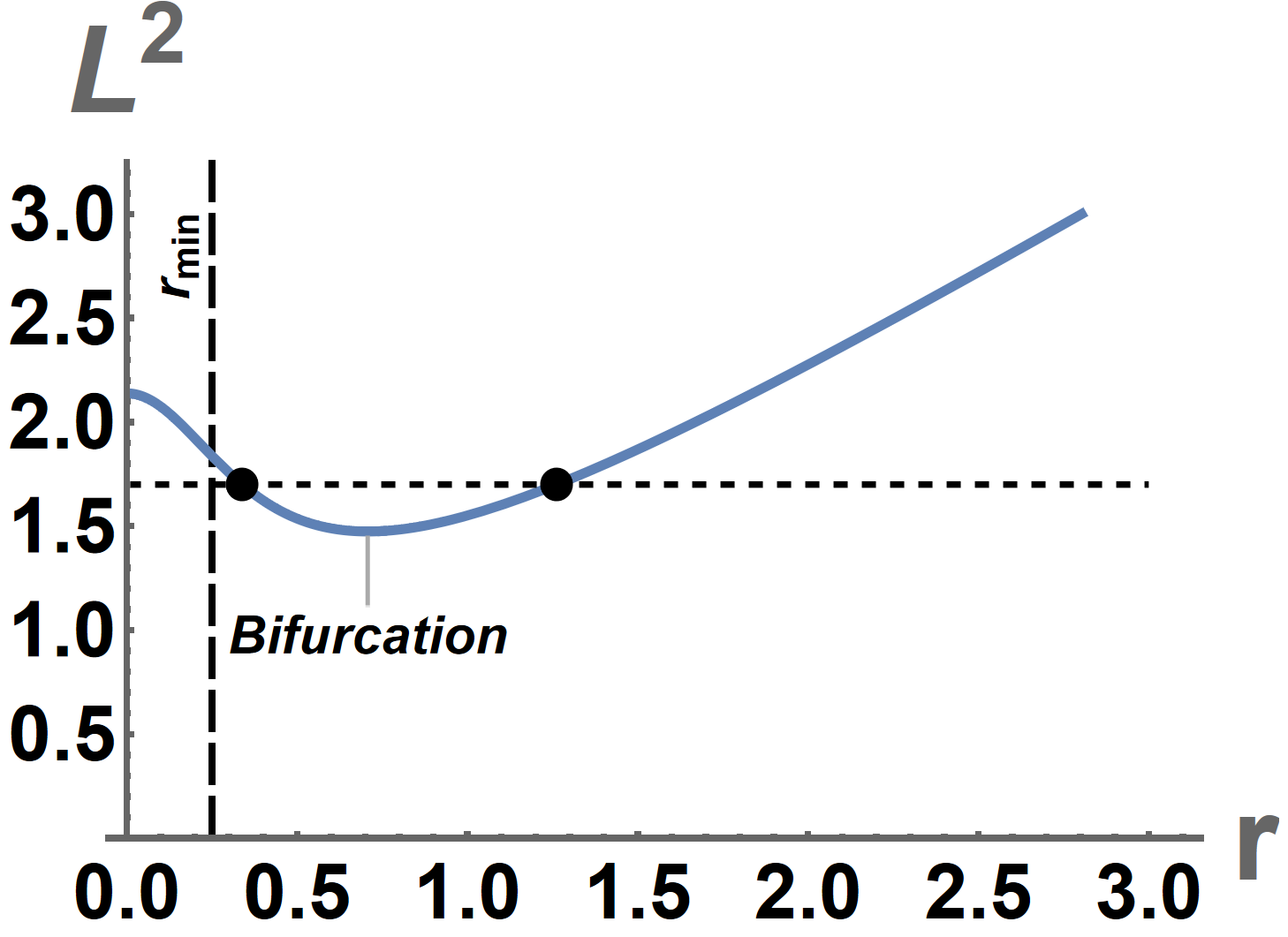}\label{fig:1DB_Isos_L_PlotA}}
	\hspace{1cm}
	\subfloat[$L^2$ with $(x_1,M_1)=(\frac{2}{3},\frac{4}{5})$]
	{\includegraphics[width=0.47\textwidth]{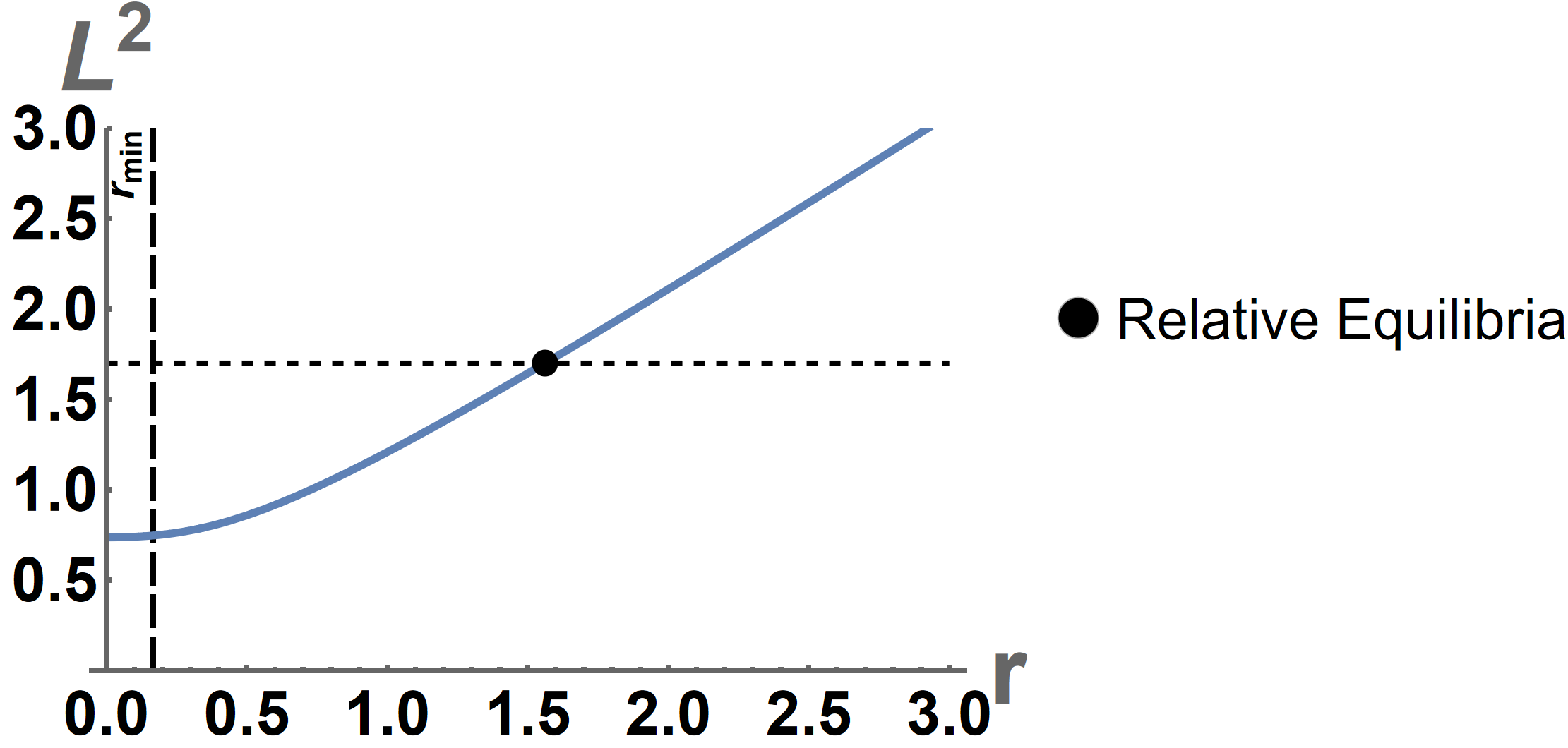}\label{fig:1DB_Isos_L_PlotB}}
	\caption{\\\bfseries Isosceles $L^2$ Curves}{\vspace{0.8em}\footnotesize Figure (a) shows $L^2$ curve with point mass body sufficiently small such that, as $L^2$ increases from zero, there are zero,  two, or one RE possible; (b) has a more massive point mass body such that there are zero or one RE possible.\\[-0.801em]}
	\label{fig:1DB_Isos_L_Plot}
\end{figure}

\begin{proof}
Note that a requirement for bifurcation is:  $(L^{2})^{\;\prime}(r)=\frac{4r\left(r^{2}+B\right)\left(r^{2}+BM_{1}\right) -3r\left(r^{2}+B\right)^{2}}{\left(r^{2}+BM_{1}\right)^{\frac{5}{2}}}$\\$=0$, from which we find only one positive critical point $r_{b}:=\sqrt{B\left(3-4M_{1}\right)},$ when $M_{1}<\frac{3}{4}$. We must also ensure that $r_{b}$ is greater than or equal to our previously calculated minimum radius $r_{min}=\frac{|x_2-x_1|}{2}$. Comparing these expressions, we find the additional restriction $M_1\le\frac{12x_1x_2}{(x_1-x_2)^2+16x_1x_2}$. A simple calculation shows the maximum of the right-hand side of this inequality occurs when $x_1=\frac{1}{2}$ and $M_1=\frac{3}{4}$. Therefore, making this inequality strict satisfies $M_{1}<\frac{3}{4}$. So, if the mass of the point mass body $M_1$ is too large, the related graph has a RE curve with no bifurcation (as in Figure \ref{fig:1DB_Isos_L_PlotB}). For instance, this would be the case for a planet and an orbiting small (dumbbell shaped) satellite. For smaller $M_1$ (see Figure \ref{fig:1DB_Isos_L_PlotB}), we compute the (square of the) angular momentum at the point of bifurcation $L^{2}(r_{b})$:
$L_{r_{b}}:=\sqrt[4]{\frac{256}{27}BM_{2}}.$
To ensure our critical point is a minimum, we calculate $(L^{2})^{\;\prime\prime}(r_b)=\frac{8\sqrt{3}(4M_1-3)}{27(M_1-1)\sqrt{-B(M_1-1)}}>0$. So there is only one bifurcation and it is located at $\left(r_{b},L_{r_{b}}\right)$. As $r\rightarrow 0,$ we find the angular momentum for RE becomes: $L_{0}:=\sqrt[4]{\frac{B}{M_{1}^{3}}}>0$. Or, since our minimum radius is $\frac{\left\vert x_{2}-x_{1}\right\vert}{2}$, a more relevant initial angular momentum is: $L_{r_{\min}}:=\underset{r\rightarrow\frac{\left\vert x_{2}-x_{1}\right\vert}{2}}{\lim}(L)=\frac{\left(x_{2}-x_{1}\right)^{2}+4B}{\sqrt{2}\left(\left(x_{2}-x_{1}\right)^{2}+4M_{1}B\right)^{\frac{3}{4}}}>0.$ 
\end{proof}
Below, Figure \ref{fig:1DB_Isos_bif_regions_scaled} represents the regions of the $x_{1}M_{1}$-plane in which we have qualitatively different graphs of $L^{2}(r).$ In the region below the curve, the graph of $L^{2}(r)$ has a minimum as calculated above (with a graph similar to Figure \ref{fig:1DB_Isos_L_PlotA}), and therefore a bifurcation above which, there are two RE until $L^{2}>L_{0}$, where there exists a unique RE. The region above the curve has no local minimums, and therefore has a unique RE for each $L^{2}>L_{r_{min}}$ (Figure \ref{fig:1DB_Isos_L_PlotB}).
\begin{figure}[H]
	\centering
	\includegraphics[width=0.4\textwidth]{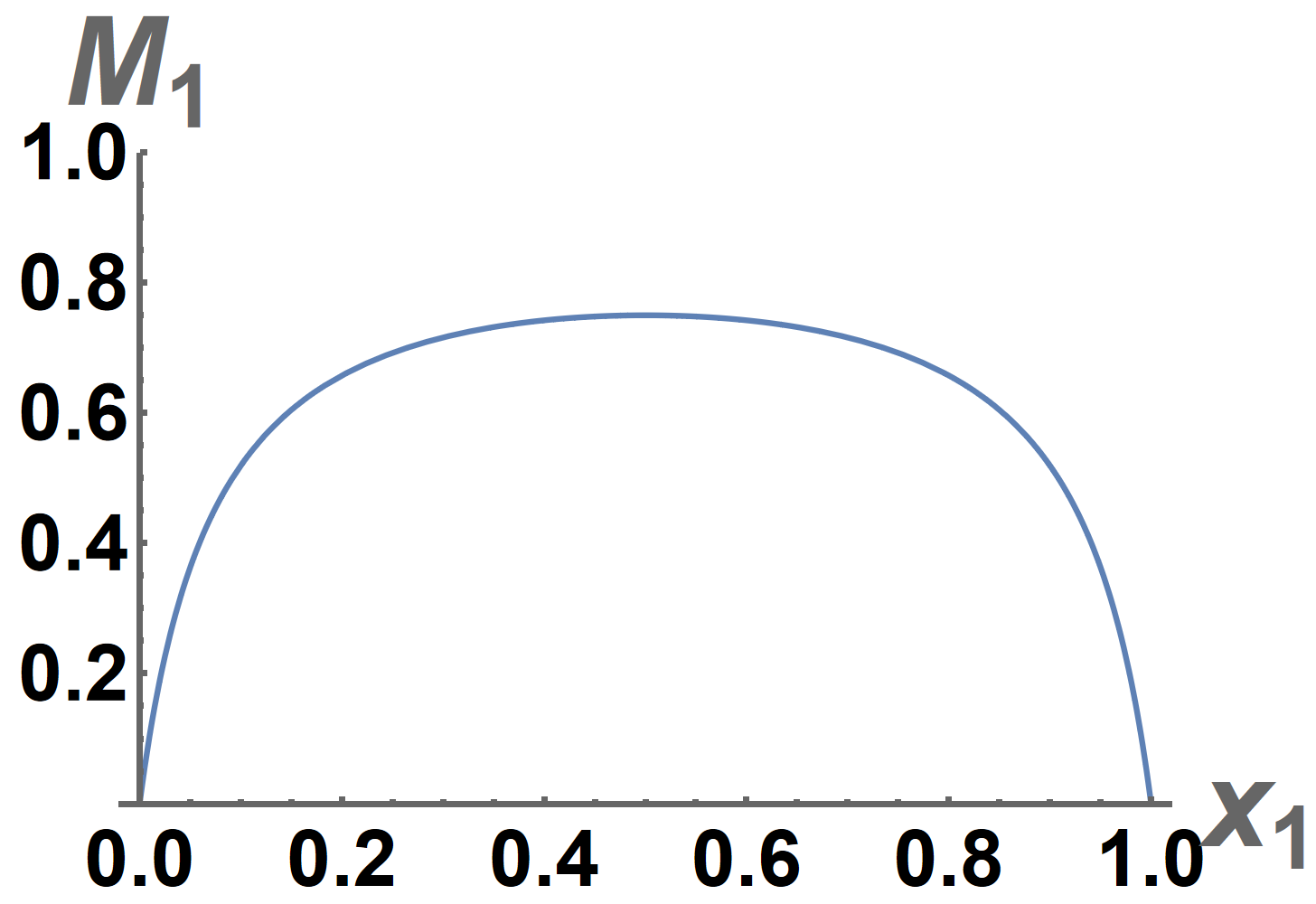} 
	\caption{\\\bfseries Isosceles Bifurcation Regions}{\vspace{0.8em}\footnotesize The curve separates a lower region with an $L^{2}$ RE bifurcation,\\and an upper region with no such bifurcation.\\[-1.801em]}
	\label{fig:1DB_Isos_bif_regions_scaled}
\end{figure}
What do the two isosceles RE look like in the lower region, and how do they differ? Looking at the example in Figure \ref{fig:1DB_Isos_L_PlotA} (where $\left(x_{1},M_{1}\right) =\left(\frac{3}{4},\frac{9}{20}\right)$ with $L^{2}=1.7$), we calculate the radii to be 
%$r\approx 0.33838$ or $1.2622.$
$r\approx 0.3384$ or $1.262.$ And the related $\theta $ are approximately 
%$2.4021$ and $1.7702$
$0.7646\pi$ and $0.5634\pi$, respectively. See the orbital configurations below.
\begin{figure}[H]
	\captionsetup[subfigure]{justification=centering}
	\centering
	\subfloat[$r\approx 0.3384,\theta\approx 0.7646\pi,$ $\overset{\cdot}{\varphi}{\approx 2.455}$]
	%\subfloat[$r\approx 0.3384,\theta\approx 2.402,$\\ $\overset{\cdot}{\varphi}{\approx 2.4547}$]
	{\hspace{1.3cm}\includegraphics[width=0.14\textwidth]{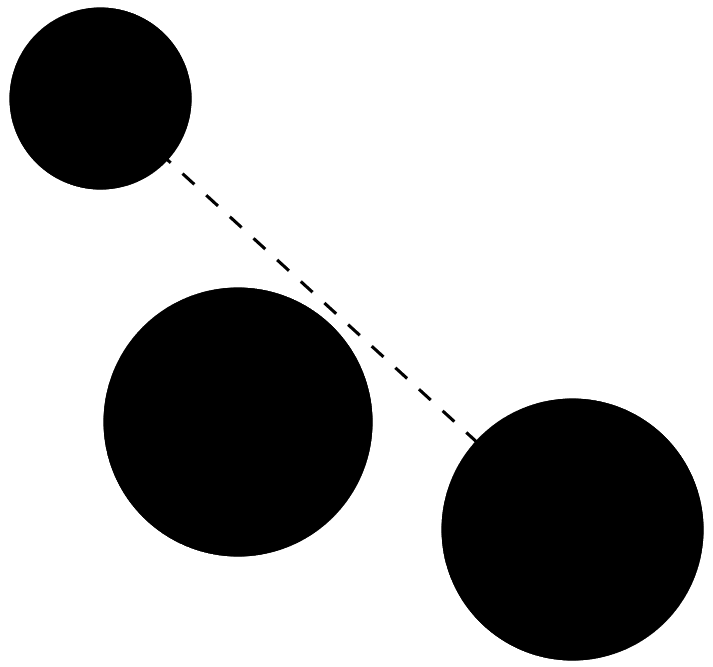}\hspace{1.3cm}}
	\hspace{1.5cm}
	\subfloat[$r\approx 1.262,\theta\approx 0.5634\pi,$ $\overset{\cdot}{\varphi}{\approx 0.6487}$]
	%\subfloat[$r\approx 1.262,\theta\approx 1.770,$\\$\overset{\cdot}{\varphi}{\approx 0.64874}$]
	{\hspace{0.75cm}\includegraphics[width=0.23\textwidth]{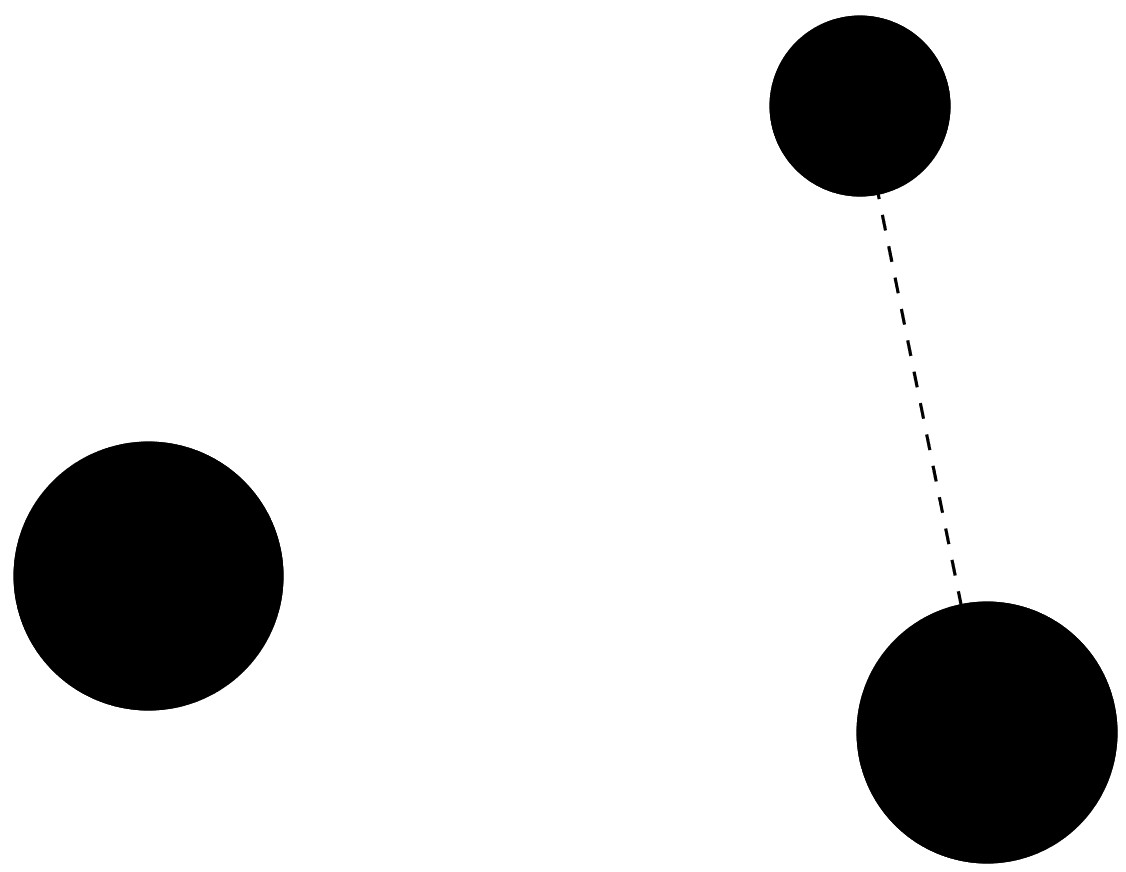}\hspace{0.75cm}}
	\caption{\\\bfseries Isosceles RE Visualization with $\left(x_{1},M_{1}\right) =\left(\frac{3}{4},\frac{9}{20}\right)$ and $L^{2}\approx1.7$}{\vspace{0.8em}\footnotesize What the RE configurations from Figure \ref{fig:1DB_Isos_L_PlotA} would look like.\\[-0.801em]}
	\label{fig:1DB_Isos_Config}
\end{figure}

For the smaller radius RE, the rotational speed is much higher $\overset{\cdot}{\varphi}\,\approx 2.45$ compared to the larger radius which has $\overset{\cdot}{\varphi}\,\approx 0.649.$ Now that we have located our isosceles RE, let us see if they are stable.

\paragraph{Energetic Stability of Isosceles}
\label{1DB_Energetic_Stability_Isosceles}As we saw in Section \ref{AmendedPotential}, to determine energetic stability we check if the RE are strict minima of the amended potential $V$. So, we are looking for a positive definite $H$, the Hessian of $V$. Using the second partial derivative test, with $\cos\theta = \frac{x_2-x_1}{2r}$ observe from \eqref{eq:1DB_SecondThetaPartialofV} that:\\$\text{\enspace\enspace}\partial_{\theta}^{2}V=x_1 x_2 r\cos\theta\left(\frac{1}{d_{1}^{3}}-\frac{1}{d_{2}^{3}}\right)-3x_{1}x_2 r^{2}\sin^{2}\theta\left(\frac{x_{2}}{d_{1}^{5}}+\frac{x_{1}}{d_{2}^{5}}\right)=-\frac{3x_{1}x_2 r^{2}\sin^{2}\theta}{d_{1}^{5}}<0.$

So, the only possibilities are maxima or saddles. Evaluating the eigenvalues, one finds that these are maxima of $V.$ So we will have no stable RE for the isosceles configuration. If we wish to classify what type of critical points these are on the energy manifold $\mathcal{H}$, note that the kinetic energy in \eqref{eq:1DB_Lagrangian} is positive definite, so that the RE found in the isosceles configuration are all saddle points in the energy manifold.
\paragraph{Linear Stability of Isosceles}
Although we lack energetic stability, when we map the linear stability conditions \eqref{eq:1DB_Linear_Stability_Criteria} for the isosceles configuration (See Figure \ref{fig:Isos_M1_0p5}), for $M_1<\frac{12 x_1 x_2}{(x_1 - x_2)^2+16x_1 x_2}\le\frac{3}{4}$ we see that for each $x_{1}$ we have linear stability for some interval(s) of (small) $r$. We graph $r_{min}$ as a light gray curve and $\partial_r L^{2}=0$ as a dashed curve. With respect to $x_1$, numerically we find linear stability to be bounded above and below by $M_1<\frac{12 x_1 x_2}{(x_1 - x_2)^2+16x_1 x_2}\le\frac{3}{4}$ (causing the horizontal region of linear stability to narrow as $M_1$ increases). Radially, we find linear stability bounded below by $r_{min}$ (where the radial eigenvalue of the amended potential's Hessian turns negative) and above by $\partial_r L^{2}=0$ (where the radial eigenvalue turns positive). Since $\partial_r L^{2}>0$ for all radii when $M_1>\frac{3}{4}$, we see the reason for the $M_1$ restriction.
\begin{figure}[H]
	\captionsetup[subfigure]{justification=centering}
	\centering
	\subfloat[$M_1=0.1875$]
	{\includegraphics[width=0.3\textwidth]{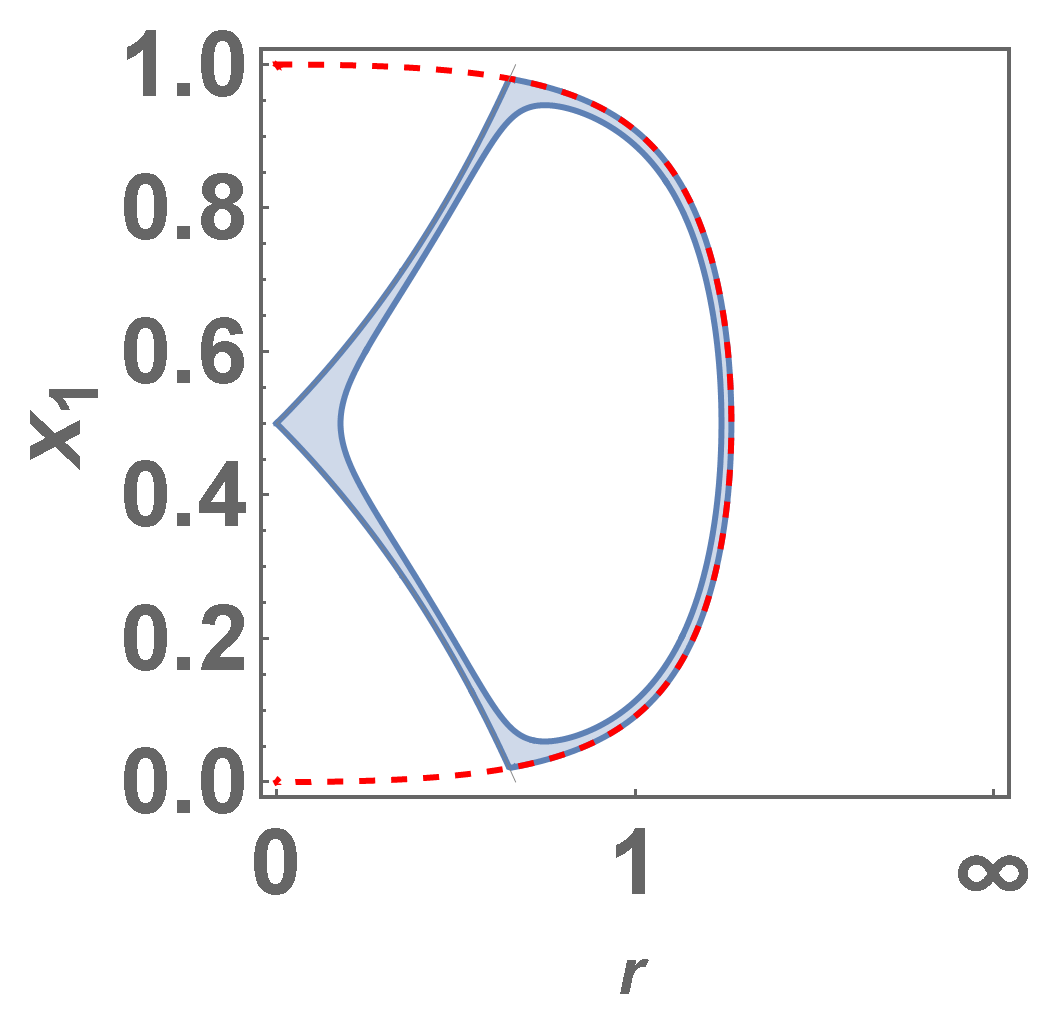}}
	\hspace{.5cm}
	\subfloat[$M_1=0.5$]
	{\includegraphics[width=0.3\textwidth]{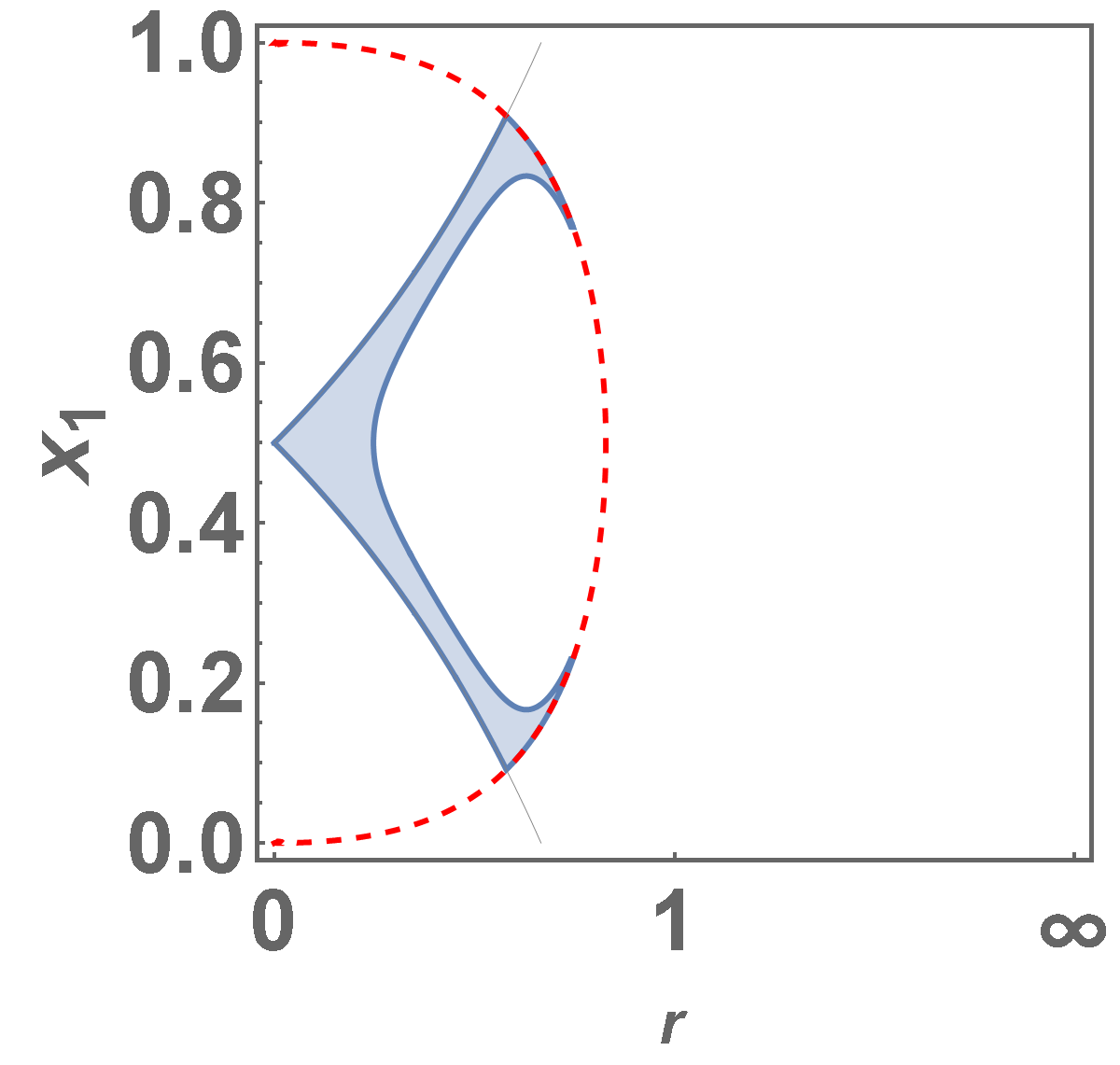}}
	\hspace{.5cm}
	\subfloat[$M_1=0.6875$]
	{\includegraphics[width=0.3\textwidth]{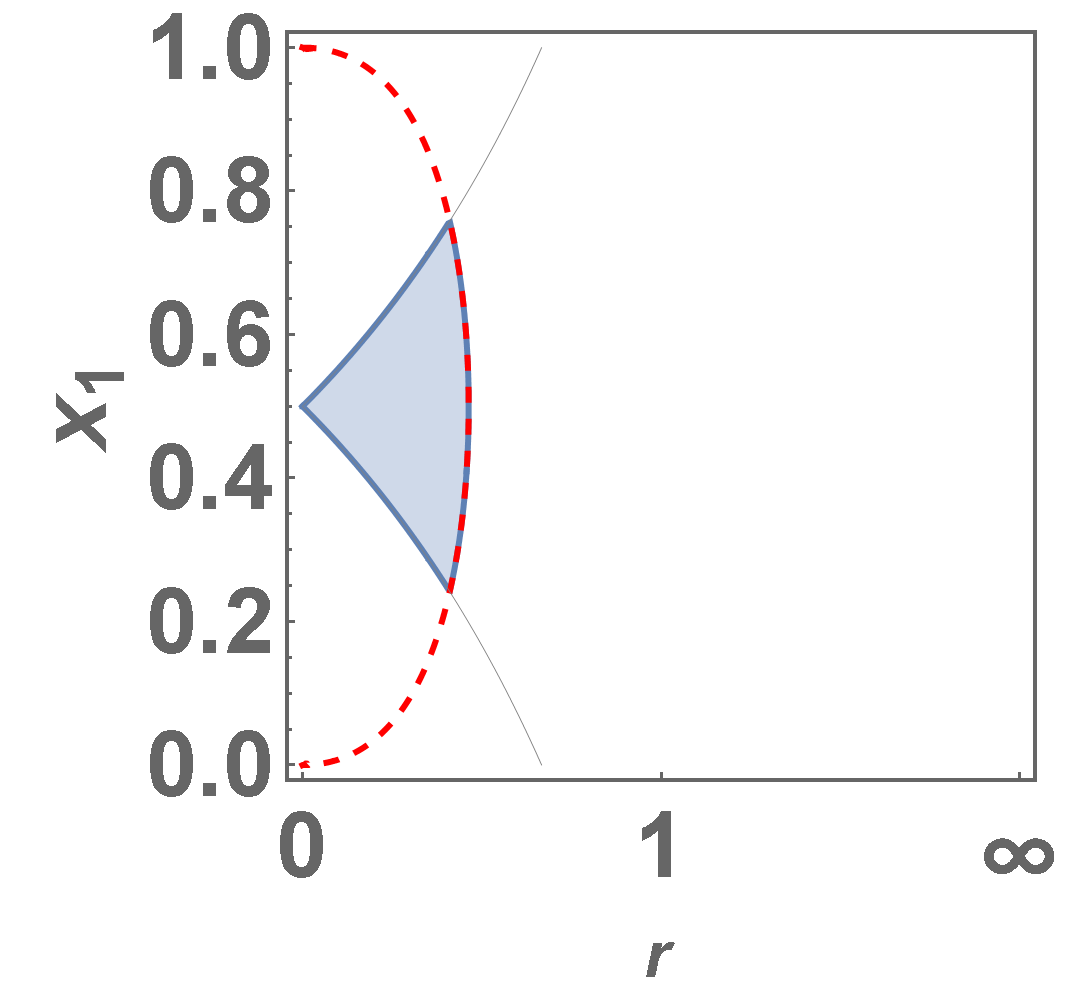}}
	\caption{\\\bfseries Isosceles Linear Stability}{\vspace{0.8em}\footnotesize $r_{min}$ (light gray curve), $\partial_r L^2=0$ (dashed curve), Linearly Stable Region (shaded)\\The linear stability region is bounded above by $\partial_r L^2=0$ and below by $r_{min}$.\\[-0.801em]}
	\label{fig:Isos_M1_0p5}
\end{figure}
Since there is no linear stability for $M_1\ge\frac{12 x_1 x_2}{(x_1 - x_2)^2+16x_1 x_2}$, that means there is no linear stability in the upper region of Figure \ref{fig:1DB_Isos_bif_regions_scaled}, or for bifurcation curves like Figure \ref{fig:1DB_Isos_L_PlotB}. Physically, this means that for linear stability we need the dumbbell body to be a significant portion of the overall mass of the system.  This is an unreasonable scenario for a large astronomical object and a dumbbell shaped artificial satellite (since the mass differential would be too great), but this could certainly be accomplished by two natural objects, or a small asteroid and an artificial satellite. If instead, we consider the point mass body as modeling a somewhat spherical shaped artificial satellite with $M_1$ very small, and the dumbbell modeling a massive oblong asteroid or moon, we can find linear stability.

As a next step up in complexity of our model, we examine the RE of a planar system with two dumbbells.
%%%%%%%%%%%%%%%%%%%%%%%%%%%%%%%%%%%%%%%%%%%%%%%%%%%%%%%%%%%%%%%%
\chapter{Planar Two-Dumbbell Problem}\label{RE2DB}
\begin{figure}[H]
	\centering
	\includegraphics[width=0.4\linewidth]{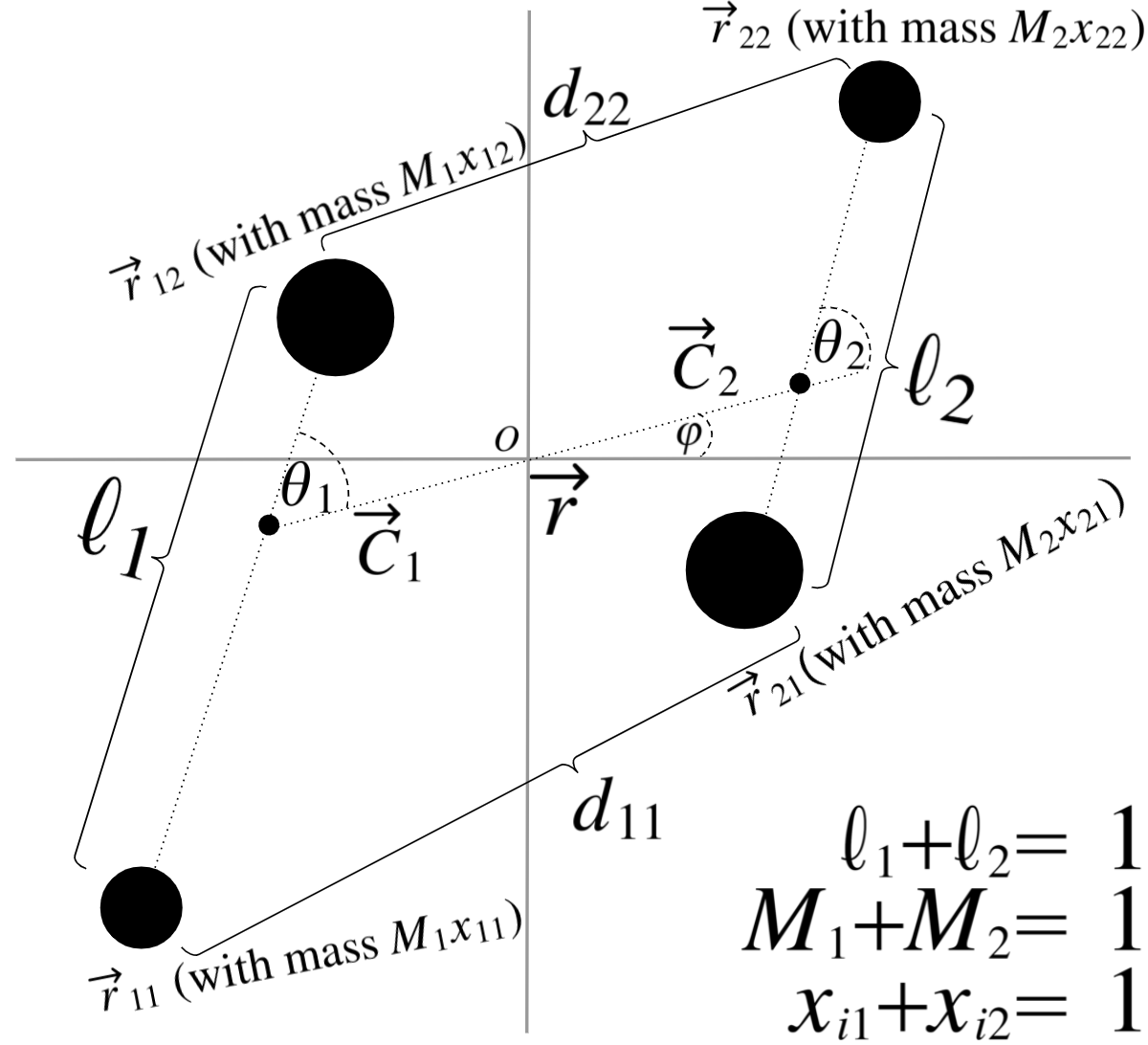}
	%		\captionsetup{labelformat=empty}
	\caption{\\\bfseries Two-Dumbbell Problem}
	\label{fig:6_configuration_Notation}
\end{figure}
Recall how the possible configurations for the planar dumbbell/point mass problem (colinear and isosceles) were easily determined from the angular requirements. For the planar two-dumbbell problem, however, finding the complete set of solutions to the angular requirements is not possible. Instead, we consider simplifications which are likely to lead to RE.  Intuitively, and from our experience with the dumbbell/point mass problem, we suspect having dumbbells in symmetric configurations (colinear, perpendicular, parallel) may lead to RE.  Another option involves symmetry of masses on the dumbbells (equal masses).  In what follows, you will see this process play out as we find RE with these symmetric qualities.

Of course, this process begs the question: ``Are there any asymmetric RE?'' Yes. In addition to the symmetric RE, below we locate families of RE with asymmetric angular rotations bifurcating from the symmetric RE, but having some symmetry with respect to the dumbbell mass values.  We have also numerically located RE where all of the parameters and rotation angles are asymmetric.

In Section \ref{PerpBisThm} of this paper, we prove an extension of the Conley Perpendicular Bisector Theorem which puts geometric restrictions on the location of the dumbbell bodies for these RE.  
\begin{theorem*}[Perpendicular Bisector Theorem for RE of a Dumbbell and Rigid Planar Bodies]
	Let a dumbbell $\overline{r_1r_2}$ and one or more planar rigid bodies $\mathcal{B}_2,...,\mathcal{B}_n$ be in a planar RE. Then, if one of the two open cones determined by the lines through $\overline{r_1r_2}$ and its perpendicular bisector contains one or more rigid bodies, the other open cone cannot be empty.
\end{theorem*}
In particular, from that theorem we can conclude in advance that the dumbbells cannot be wholly contained in an open quadrant defined by the other dumbbell's rod and that rod's perpendicular bisector. In particular, we might suspect a colinear or a perpendicular configuration could be RE. Additionally, our theorem allows for RE when one of the dumbbell's masses is in an open quadrant, and the other mass is in a neighboring quadrant. The RE we locate in this chapter do indeed obey the restrictions identified by our theorem.
\section{Set-up and Notation}
We will follow the notation conventions used in the dumbbell/point mass problem. For the following description, refer to Figure \ref{fig:6_configuration_Notation}. Let the origin and center of mass of our system reference frame be $O$. We denote the mass ratios of the dumbbell bodies as $M_{i}$ such that the total system mass is scaled to $M_{1}+M_{2}=1.$ Each dumbbell body has a center of mass denoted $\vec{C}_{i}$. Each dumbbell body consists of two masses, we denote their mass ratios as $x_{i1}$ and $x_{i2}$ such that $x_{i1}+x_{i2}=1$. The masses are located at $\vec{r}_{i1}$ and $\vec{r}_{i2}$ (with respect to the system frame), which are connected by a massless rod of length $\ell_{i}$. We scale the system distances to let $\ell_{1}+\ell_{2}=1.$

We label the ``radius'' vector $\overrightarrow{C_{1}C_{2}}$ as $\vec{r},$ and $r:=\left\vert \vec{r}\right\vert$. We let $\varphi $ be the acute angle between the system's horizontal axis and $\vec{r}, $ so $\vec{C}_{i}=\left(-1\right)^{i}M_{n}r\left(\cos\varphi,\sin\varphi\right) $, where $n\neq i.$ And we let $\theta_{i}$ represent the angle between $\vec{r}$ and each dumbbell's rod, where if $\theta_{i}=0$, then $x_{in}$ is the mass lying on $\overline{C_1,C_2}$, where $n\neq i$. We see that we have four degrees of freedom: $r,$ $\varphi$, $\theta_{1},$ and $\theta_{2}$. And the system position vectors for the masses are:\\ $\vec{r}_{ij}=\left(-1\right)^{i}M_{n}r\left(\cos \varphi ,\sin\varphi\right)+\left(-1\right)^{j}x_{ik}\ell_{i}\left(\cos\left(\varphi +\theta_{i}\right) ,\,\sin\left(\varphi+\theta_{i}\right)\right)$, where $n\neq i$ and $k\neq j.$ Note, the mass associated with each $\vec{r}_{ij}$ is $M_{i}x_{ij}$. 
Also, we denote the distance between $\vec{r}_{1u}$ and $\vec{r}_{2v}$ as:
$d_{uv}\left(r,\theta_{1},\theta_{2}\right) :=$
\begin{align}\label{eq:2DB_Distances}
	\text{$\scriptstyle\sqrt{{r}^{2}-\left(-1\right)^{u}{2x_{1\overline{u\,}}\ell}_{1}{r}\cos{\theta}_{1}+\left(-1\right)^{v}{2x_{2\overline{v\,}}\ell_{2}r}\cos{\theta}_{2}-\left(-1\right)^{u+v}{2x_{1\overline{u\,}}x_{2\overline{v\,}}\ell}_{1}{\ell}_{2}\cos\left(\theta_{1}-\theta_{2}\right){+x_{1\overline{u\,}}^{2}\ell}_{1}^{2}{+x_{2\overline{v\,}}^{2}\ell}_{2}^{2}}$,}
\end{align}
where $u\neq \overline{u\,}$ and $v\neq \overline{v\,}.$

\section{Equations of Motion}
Unless otherwise stated, we assume both dumbbells are nontrivial (as we have already considered the dumbbell/point mass problem in the previous chapter). In order to find the equations of motion, we let $\vec{v}_{ij}$ denote the velocities of $\vec{r}_{ij}$. We calculate the Lagrangian $\mathcal{L}\,=T\,-U\,,$ where $T\,=\frac{1}{2}\sum_{i,j\in\{1,2\}}M_{i}x_{ij}\vec{v}_{ij}^{\text{ }2}$ is the kinetic energy, and the potential energy of the system (with $G=1,$ and scaled by $\frac{1}{M_{1}M_{2}}$) is $U\,=-\sum_{u,v\,\in\left\{1,2\right\} }^{{}}\frac{x_{1u}x_{2v}}{d_{uv}}$. To calculate kinetic energy, we first calculate the velocity vectors:\\ $\text{\enspace\enspace}\vec{v}_{ij}=\left({-1}\right)^{i}{M}_{{n}}\overset{\cdot}{r}\left(\cos{\varphi},\sin{\varphi}\right){\,+\,}\left({-1}\right)^{i}{M}_{{n}}{r}\overset{\cdot}{{\varphi}}\left(-\sin{\varphi},\cos{\varphi}\right)\\\text{\enspace\enspace\enspace\enspace\enspace\enspace\enspace}+\left({-1}\right)^{j}{x}_{ik}{\ell}_{i}\left(\overset{\cdot}{{\varphi}}+\overset{\cdot}{{\theta}}_{i}\right)\left(-\sin\left({\varphi+\theta}_{i}\right),\cos\left({\varphi+\theta}_{i}\right)\right).$

Recalling some trigonometric identities, we find:\\
$\text{\enspace\enspace}\textstyle\vec{v}_{ij}^{\hspace{.2em}2}=\,M_{n}^{2}\left(\overset{\cdot}{r}^{2}+r^{2}\overset{\cdot}{\varphi}^{2}\right)+x_{ik}^{2}\ell_{i}^{2}\left(\overset{\cdot}{\theta}_{i}+\overset{\cdot}{\varphi}\right)^{2}$\\$\text{\enspace\enspace\enspace\enspace}+\left(-1\right)^{i+j}2M_{n}x_{ik}\ell_{i}\left(\overset{\cdot}{\theta}_{i}+\overset{\cdot}{\varphi}\right)\left(r\overset{\cdot}{\varphi}\cos \theta_{i}-\overset{\cdot}{r}\sin \theta_{i}\right),$
where $n\neq i$ and $k\neq j.$\\And the kinetic energy is:\\
$\text{\enspace\enspace}T\,=\frac{1}{2}M_{1}M_{2}\left(\overset{\cdot}{r}^{2}+r^{2}\overset{\cdot}{\varphi}^{2}\right)+\frac{1}{2}x_{11}x_{12}M_{1}\ell_{1}^{2}\left(\overset{\cdot}{\theta}_{1}+\overset{\cdot}{\varphi}\right)^{2}+\frac{1}{2}x_{21}x_{22}M_{2}\ell_{2}^{2}\left(\overset{\cdot}{\theta}_{2}+\overset{\cdot}{\varphi}\right)^{2}.$

So, our Lagrangian (scaled by $\frac{1}{M_{1}M_{2}}$) is:
\begin{align}\label{eq:2DB_Lagrangian}
	{\textstyle \mathcal{L}=T\,-U\,=\frac{1}{2}\left(\overset{\cdot}{r}^{2}+r^{2}\overset{\cdot}{\varphi}^{2}\right)+\frac{1}{2}B_{1}\left(\overset{\cdot}{\theta}_{1}+\overset{\cdot}{\varphi}\right)^{2}+\frac{1}{2}B_{2}\left(\overset{\cdot}{\theta}_{2}+\overset{\cdot}{\varphi}\right)^{2}-U,}
\end{align}
\begin{align}\label{eq:2DB_MomOfInertExpressions}
	{\textstyle \text{ where } B_{i}:=\frac{x_{i1}x_{i2}}{M_{n}}\ell_{i}^{2}  \text{ (with $n\neq i$)}}
\end{align}
are the scaled moments of inertia relative to the centers of mass for each of the dumbbells.
The equations of motion are given by the Euler-Lagrange equation: $\frac{d}{dt}\frac{\partial\mathcal{L}}{\partial\overset{\cdot}{q_{i}}}-\frac{\partial \mathcal{L}}{\partial q_{i}}=0$, for each degree of freedom $q_i\in\left\{ r,\varphi ,\theta_{1},\theta_{2}\right\}.$ We calculate:
\begin{subequations}
	\makeatletter\@fleqntrue\makeatother
	\begin{align}
		\mathitem \text{$\overset{\cdot\cdot}{r}-r\overset{\cdot}{\varphi}^{2}=-\frac{\partial U}{\partial r}$,\label{eq:2DB_EOMsA}}\\
		\mathitem \text{$r^{2}\overset{\cdot\cdot}{\varphi}+2r\overset{\cdot}{r}\overset{\cdot}{\varphi}+B_{1}\left(\overset{\cdot\cdot}{\theta}_{1}+\overset{\cdot\cdot}{\varphi}\right)+B_{2}\left(\overset{\cdot\cdot}{\theta}_{2}+\overset{\cdot\cdot}{\varphi}\right) \,=-\frac{\partial U}{\partial\varphi}=0,$\label{eq:2DB_EOMsB}}\\
		\mathitem \text{$B_{i}\left(\overset{\cdot\cdot}{\theta}_{i}+\overset{\cdot\cdot}{\varphi}\right) \,\,=-\frac{\partial U}{\partial \theta_{i}},$ for $i\in\left\{ 1,2\right\}.$\label{eq:2DB_EOMsC}}
	\end{align}
	\label{eq:2DB_EOMs}
\end{subequations}

\section{Finding RE}
We now wish to reduce our system using the amended potential method in Section \ref{AmendedPotential} by exploiting the conserved angular momentum. So, as in the previous models, we calculate angular momentum (scaled by $\frac{1}{M_{1}M_{2}}$) as: $\frac{\partial\mathcal{L}}{\partial \overset{\cdot}{\varphi}}={r}^{2}\overset{\cdot}{\varphi}+B_{1}\left(\overset{\cdot}{\theta}_{1}+\overset{\cdot}{\varphi}\right)+B_{2}\left(\overset{\cdot}{\theta}_{2}+\overset{\cdot}{\varphi}\right) =:L.$ Or, solving for the rotational speed:
\begin{align}\label{eq:2DB_Rotation_Speed}
	{\textstyle
	 \overset{\cdot}{\varphi}\,=\frac{L-B_{1}\overset{\cdot}{\theta}_{1}-B_{2}\overset{\cdot}{\theta}_{2}}{{r}^{2}+B_{1}+B_{2}}\text{\enspace\enspace}\text{ or }\text{\enspace\enspace}\overset{\cdot}{\varphi}\,=\frac{L}{{r}^{2}+B_{1}+B_{2}}\text{ at RE}.}
\end{align}
We can now eliminate velocity $\overset{\cdot}{\varphi}$ from our system using \eqref{eq:2DB_Rotation_Speed}. Upon substituting $\overset{\cdot}{\varphi}$ into $\eqref{eq:2DB_EOMsB},$ and solving for $\overset{\cdot\cdot}{\varphi}$, we see it is equivalent to $\overset{\cdot\cdot}{\varphi}\,=-2r\overset{\cdot}{r}\frac{L-B_{1}\overset{\cdot}{\theta}_{1}-B_{2}\overset{\cdot}{\theta}_{2}}{\left({r}^{2}+B_{1}+B_{2}\right)^{2}}-\frac{B_{1}\overset{\cdot\cdot}{\theta}_{1}+B_{2}\overset{\cdot\cdot}{\theta}_{2}}{r^{2}+B_{1}+B_{2}}.$ And substituting $\overset{\cdot}{\varphi}$ and $\overset{\cdot\cdot}{\varphi}$ into $\eqref{eq:2DB_EOMsA}$ and $\eqref{eq:2DB_EOMsC}$, we determine the reduced Lagrangian $\mathcal{L}_{red}$ for which $\frac{d}{dt}\frac{\partial \mathcal{L}_{red}}{\partial\overset{\cdot}{r}_{i}}-\frac{\partial \mathcal{L}_{red}}{\partial r_{i}}$ will generate these reduced equations. We find:\\ $\mathcal{L}_{red}=T_{red}-V$\\$\text{\enspace\enspace}=\frac{1}{2}\left(\overset{\cdot}{r}^{2}+B_{1}\frac{{r}^{2}+B_{2}}{{r}^{2}+B_{1}+B_{2}}\overset{\cdot}{\theta}_{1}^{2}+B_{2}\frac{{r}^{2}+B_{1}}{{r}^{2}+B_{1}+B_{2}}\overset{\cdot}{\theta}_{2}^{2}\right)+\frac{L\left(B_{1}\overset{\cdot}{\theta}_{1}+B_{2}\overset{\cdot}{\theta}_{2}\right)-B_{1}B_{2}\overset{\cdot}{\theta}_{1}\overset{\cdot}{\theta}_{2}}{{r}^{2}+B_{1}+B_{2}}$\\$\text{\enspace\enspace\enspace\enspace\enspace}-\left(\frac{L^{2}}{2\left({r}^{2}+B_{1}+B_{2}\right)}+U\right).$

So, we have our amended potential (scaled by $\frac{1}{M_{1}M_{2}}$): $V:=\frac{L^{2}}{2\left(r^{2}+B_{1}+B_{2}\right)}+U.$

Recall we can characterize the RE of our system as the critical points of the amended potential $V.$ Taking the $r$ derivative of $V$, we find the \textbf{radial requirement}:\\[1.1em]
$\text{\enspace\enspace}\partial_{r}V=x_{11}\left(\frac{x_{21}\left({r+}x_{12}{\ell}_{1}\cos{\theta}_{1}-x_{22}{\ell}_{2}\cos{\theta}_{2}\right)}{d_{11}^{3}}+\frac{x_{22}\left({r+}x_{12}{\ell}_{1}\cos{\theta}_{1}+x_{21}{\ell}_{2}\cos{\theta}_{2}\right)}{d_{12}^{3}}\right)$\\[-1.3em]
$\text{\enspace\enspace\enspace\enspace\enspace\enspace}$
\begin{align}\label{eq:2DB_RadialREQ}
	{\text{\enspace}\textstyle +x_{12}\left(\frac{x_{21}\left({r-}x_{11}{\ell}_{1}\cos{\theta}_{1}-x_{22}{\ell}_{2}\cos{\theta}_{2}\right)}{d_{21}^{3}}+\frac{x_{22}\left({r-}x_{11}{\ell}_{1}\cos{\theta}_{1}+x_{21}{\ell}_{2}\cos{\theta}_{2}\right)}{d_{22}^{3}}\right)-\frac{rL^{2}}{\left(r^{2}+B_{1}+B_{2}\right)^{2}}=0.}
\end{align}
In the cases pursued below, we will wish to examine bifurcations of the related $L^{2}$ graph in order to identify qualitatively different parts of our parameter space as it relates to the quantity and stability of the RE. To this end, we rearrange \eqref{eq:2DB_RadialREQ} as:\\[1.1em]
$\text{\enspace\enspace}L^{2}=\frac{\left({r}^{2}+B_{1}+B_{2}\right)^{2}}{r}x_{11}\left(\frac{x_{21}\left({r+}x_{12}{\ell}_{1}\cos{\theta}_{1}-x_{22}{\ell}_{2}\cos{\theta}_{2}\right)}{d_{11}^{3}}+\frac{x_{22}\left({r+}x_{12}{\ell}_{1}\cos{\theta}_{1}+x_{21}{\ell}_{2}\cos{\theta}_{2}\right)}{d_{12}^{3}}\right)$\\[-1.3em]
\begin{align}\label{eq:2DB_Radial_REQ}
	{\textstyle +\frac{\left({r}^{2}+B_{1}+B_{2}\right)^{2}}{r}x_{12}\left(\frac{x_{21}\left({r-}x_{11}{\ell}_{1}\cos{\theta}_{1}-x_{22}{\ell}_{2}\cos{\theta}_{2}\right)}{d_{21}^{3}}+\frac{x_{22}\left({r-}x_{11}{\ell}_{1}\cos{\theta}_{1}+x_{21}{\ell}_{2}\cos{\theta}_{2}\right)}{d_{22}^{3}}\right).}
\end{align}

Taking ${\theta}_{i}$ derivatives of $V$, we find \textbf{angular requirements}:
\begin{flalign}\label{eq:2DB_Angular_REQ}
	\begin{aligned}
		\textstyle \partial_{\theta_{1}}V=&x_{11}x_{12}{\ell}_{1}\left(x_{21}\left(x_{22}{\ell}_{2}\sin\left(\theta_{1}-\theta_{2}\right)-{r}\sin{\theta}_{1}\right)\left(\frac{1}{d_{11}^{3}}-\frac{1}{d_{21}^{3}}\right)\right)\\
		&\textstyle +x_{11}x_{12}{\ell}_{1}\left(x_{22}\left(x_{21}{\ell}_{2}\sin\left(\theta_{1}-\theta_{2}\right)+{r}\sin{\theta}_{1}\right)\left(\frac{1}{d_{22}^{3}}-\frac{1}{d_{12}^{3}}\right)\right) =0,\\
		\textstyle \partial_{\theta_{2}}V=&x_{21}x_{22}{\ell}_{2}\left(x_{11}\left(
		x_{12}{\ell}_{1}\sin\left(\theta_{1}-\theta_{2}\right) -{r}\sin{\theta}_{2}\right)\left(\frac{1}{d_{12}^{3}}-\frac{1}{d_{11}^{3}}\right)\right)\\
		&\textstyle +x_{21}x_{22}{\ell}_{2}\left(x_{12}\left(x_{11}{\ell}_{1}\sin\left(\theta_{1}-\theta_{2}\right)+{r}\sin{\theta}_{2}\right)\left(\frac{1}{d_{21}^{3}}-\frac{1}{d_{22}^{3}}\right)\right) =0.
	\end{aligned}
\end{flalign}
Simplifying we have:
\begin{flalign}\label{eq:2DB_Angular_Requirements}
%	\makeatletter\@fleqntrue\makeatother
	\begin{aligned}
	 \text{$0=x_{21}\left(x_{22}{\ell}_{2}\sin\left({\theta}_{1}{-\theta}_{2}\right)-{r}\sin{\theta}_{1}\right)\left(\frac{1}{d_{11}^{3}}-\frac{1}{d_{21}^{3}}\right)$}\\
	  \text{$\text{\enspace\enspace\enspace}+x_{22}\left(x_{21}{\ell}_{2}\sin\left({\theta}_{1}{-\theta}_{2}\right)+{r}\sin{\theta}_{1}\right)\left(\frac{1}{d_{22}^{3}}-\frac{1}{d_{12}^{3}}\right),$}\\
 	\text{$0=x_{11}\left(x_{12}{\ell}_{1}\sin\left({\theta}_{1}{-\theta}_{2}\right)-{r}\sin{\theta}_{2}\right)\left(\frac{1}{d_{11}^{3}}-\frac{1}{d_{12}^{3}}\right)$}\\
 	\text{$\text{\enspace\enspace\enspace}-x_{12}\left(x_{11}{\ell}_{1}\sin\left({\theta}_{1}{-\theta}_{2}\right)+{r}\sin{\theta}_{2}\right)\left(\frac{1}{d_{21}^{3}}-\frac{1}{d_{22}^{3}}\right).$}
	\end{aligned}
\end{flalign}
Note that similar to Beletskii and Ponomareva's work on the dumbbell/point mass problem, if we can solve the angular requirements above for $\theta_{1},\theta_{2}$, we can substitute these into the radial requirement and find $r,L,$ pairs, which are then associated with a unique rotational speed $\overset{\cdot}{\varphi}$ given by \eqref{eq:2DB_Rotation_Speed}. However, finding the complete set of RE solutions to the angular requirements is nontrivial. To proceed beyond this impasse, we will consider symmetric configurations allowed by the Perpendicular Bisector Theorem for RE of a Dumbbell and Rigid Planar Bodies Theorem, for which one might suspect RE.

\paragraph{Energetic Stability}
\label{2DB_Energetic_Stability}Once we find RE, we will want to analyze their stability. Recall in the dumbbell/point mass problem that we had energetic stability for the colinear case only. So we hope to find some energetic stability for a colinear configuration here as well. As we saw in Section \ref{AmendedPotential}, to determine energetic stability we check if the RE are strict minima of the amended potential $V$. To this end, we calculate the Hessian $H$ of $V$: 
\begin{equation}\label{eq:2DB_Hession}
	H=\begin{bmatrix}\partial^2_{r}V&\partial^2_{r,\theta_1}V&\partial^2_{r,\theta_2}V\\ 
		\partial^2_{\theta_1,r}V&\partial^2_{\theta_1}V&\partial^2_{\theta_1\theta_2}V\\
		\partial^2_{\theta_2,r}V&\partial^2_{\theta_2,\theta_1}V&\partial^2_{\theta_2}V
	\end{bmatrix}.
\end{equation}
For several of the RE configurations examined below, we find $H$ becomes block diagonal\\($\partial^2_{r,\theta_i}V=0$). So let us recharacterize $\partial_r V$ in a way that helps us determine the sign of $\partial^2_r V$, and therefore the sign of $H$'s radial eigenvalue. Calculating $\partial_r V$, we see we can recharacterize it as:\\ $\partial_{r}V=\left(g\left(r,\theta_{1},\theta_{2}\right) -L^{2}\right)\frac{r}{\left(r^{2}+B_{1}+B_{2}\right)^{2}},$ where:\\ $\textstyle\text{\enspace\enspace} g\left(r,\theta_{1},\theta_{2}\right) :=\frac{\left(r^{2}+B_{1}+B_{2}\right)^{2}}{r}x_{11}\left(\frac{x_{21}\left({r+}x_{12}{\ell}_{1}\cos{\theta}_{1}-x_{22}{\ell}_{2}\cos{\theta}_{2}\right)}{d_{11}^{3}}+\frac{x_{22}\left({r+}x_{12}{\ell}_{1}\cos{\theta}_{1}+x_{21}{\ell}_{2}\cos{\theta}_{2}\right)}{d_{12}^{3}}\right)$\\
$\text{\hskip65pt}\textstyle+\frac{\left(r^{2}+B_{1}+B_{2}\right)^{2}}{r}x_{12}\left(\frac{x_{21}\left({r-}x_{11}{\ell}_{1}\cos{\theta}_{1}-x_{22}{\ell}_{2}\cos{\theta}_{2}\right)}{d_{21}^{3}}+\frac{x_{22}\left({r-}x_{11}{\ell}_{1}\cos{\theta}_{1}+x_{21}{\ell}_{2}\cos{\theta}_{2}\right)}{d_{22}^{3}}\right).$

This is done so that at a critical point of $\partial_r V$, we have: $g=L^{2}.$
$\partial_{r}^{2}V$ is then:\\ $\frac{r}{\left(r^{2}+B_{1}+B_{2}\right)^{2}}\partial_{r}g+\left(g-L^{2}\right)\left(1-\frac{4r^{2}}{r^{2}+B_{1}+B_{2}}\right) \frac{1}{\left(r^{2}+B_{1}+B_{2}\right)^{2}},$ which at a critical point of $\partial_r V$ becomes: $\partial_{r}^{2}V|_{RE}=\frac{r}{\left(r^{2}+B_{1}+B_{2}\right)^{2}}\partial_{r}g$. So, we can determine the sign of $\partial_{r}^{2}V$ with $\partial_{r}g,$ or equivalently $\partial_{r}L^{2}$. 
\begin{equation}\label{eq:2DB_Vr_Slope_Equal_LSquared_Slope}
	\partial_{r}L^{2}>0\implies\partial_{r}^{2}V>0.
\end{equation}
In other words, we can determine the sign of $\partial^2_r V$ with the slopes of graphs like Figure \ref{fig:6_3a_no_overlap} below.\vspace{-12pt}
\paragraph{Linear Stability}
We will determine linear stability for the two-dumbbell problem in the same way as we did with the dumbbell/point mass problem, by first rewriting the reduced Lagrangian \eqref{eq:2DB_Lagrangian} in terms of the amended potential. So:\\
$\text{\enspace\enspace}\scriptstyle\mathcal{L}_{red}=\frac{1}{2}\overset{\cdot}{r}^2+\frac{1}{2(r^2+B_1+B_2)}\left(B_1\left(r^2+B_2\right)\overset{\cdot}{\theta}_1^2+B_2\left(r^2+B_1\right)\overset{\cdot}{\theta}_{2}^{2}+2L\left(B_{1}\overset{\cdot}{\theta}_{1}+B_{2}\overset{\cdot}{\theta}_{2}\right)-B_{1}B_{2}\overset{\cdot}{\theta}_{1}\overset{\cdot}{\theta}_{2}\right) -V\left(r,\theta_{i}\right)$,\\where $V\left(r,\theta_{i}\right):=\frac{L^{2}}{2\left({r}^{2}+B_{1}+B_{2}\right)}+U\left(r,\theta_{i}\right) $ is the amended potential. Applying the Euler Lagrange equation 

and solving for acceleration gives us equations of motion:
\begin{subequations}\label{eq:2DB_EOMs_for_Linearization}
	\makeatletter\@fleqntrue\makeatother
	\begin{align}
		\begin{split}
			&\text{$\overset{\cdot\cdot}{r}\,=\frac{r\left(B_{1}\overset{\cdot}{\theta}_{1}+B_{2}\overset{\cdot}{\theta}_{2}\right)\left(B_{1}\overset{\cdot}{\theta}_{1}+B_{2}\overset{\cdot}{\theta}_{2}-2L\right)}{\left(r^{2}+B_{1}+B_{2}\right)^{2}}-\partial_{r}V,$}
		\end{split}\\
		\begin{split}
			&\text{$\overset{\cdot\cdot}{\theta}_{i}=-\frac{2\overset{\cdot}{r}\left(B_{1}\overset{\cdot}{\theta}_{1}+B_{2}\overset{\cdot}{\theta}_{2}-L\right)}{r\left(r^{2}+B_{1}+B_{2}\right)}-\left(\frac{1}{B_{i}}+\frac{1}{r^{2}}\right)\partial_{\theta_{i}}V-\frac{1}{r^{2}}\partial_{\theta_{n}}V,$}
		\end{split}\\
	\begin{split}
		\hspace{20pt}\text{$\text{ for }i\in\{1,2\}\text{ and }n\neq i.\nonumber$}
	\end{split}
	\end{align}
\end{subequations}
For RE, we have $\overset{\cdot}{r}_{RE}\,=\overset{\cdot}{\theta}_{iRE}=\overset{\cdot\cdot}{r}_{RE}=\overset{\cdot\cdot}{\theta}_{iRE}=0$, giving us:
\begin{subequations}
	\makeatletter\@fleqntrue\makeatother
	\begin{align}
		\mathitem \text{$\overset{\cdot\cdot}{r}_{RE}=-\partial_{r}V=0\nonumber$,}\\
		\mathitem \text{$\overset{\cdot\cdot}{\theta}_{iRE}=-\frac{1}{r^{2}}\left(\frac{r^{2}+B_{i}}{B_{i}}\partial_{\theta_{i}}V+\partial_{\theta_{n}}V\right)=0\text{, for }i\in\{1,2\}\text{ and }n\neq i,\nonumber$}
	\end{align}
\end{subequations}

from which we can see our previous requirements for RE that $\partial_{r}V=\partial_{\theta_{i}}V=0$.

Letting $v:=\left[r\text{\enspace}\theta_{1}\text{\enspace}\theta_{2}\text{\enspace}\overset{\cdot}{r}\text{\enspace}\overset{\cdot}{\theta}_{1}\text{\enspace}\overset{\cdot}{\theta}_{2}\right] $, we linearize our system $\eqref{eq:2DB_EOMs_for_Linearization}$ as $\overset{\cdot}{v}\,=Av$, which at RE ($\overset{\cdot}{r}=\overset{\cdot}{\theta}_{i}=\partial_{r}V=\partial_{\theta_{i}}V=0$) has: ${A}_{{RE}}{ :=}\begin{bmatrix}
	0 & I\\ 
	A_3 & A_4\\ \end{bmatrix}_{{ RE}}$, where\\
$A_3={\begin{bmatrix}
	{-\partial}_{r}^{2}{V} & {-\partial}_{r,\theta_{1}}^{2}{V} & {-\partial}_{r,\theta_{2}}^{2}{V}\\ 
	\frac{-{1}}{{r}^{2}}\left({\partial}_{{r\theta}_{2}}^{2}{V+}\frac{{r}^{2}{+B}_{1}}{B_{1}}{\partial}_{{r\theta}_{1}}^{2}{V}\right) &\frac{-{1}}{{r}^{2}}\left({\partial}_{{\theta}_{1}{\theta}_{2}}^{2}{V+}\frac{{r}^{2}{+B}_{1}}{B_{1}}{\partial}_{{\theta}_{1}}^{2}{V}\right) & \frac{-{1}}{{r}^{2}}\left({\partial}_{{\theta}_{2}}^{2}{V+}\frac{{r}^{2}{+B}_{1}}{B_{1}}{\partial }_{{\theta}_{1}{\theta}_{2}}^{2}{V}\right)\\ 
	\frac{-{1}}{{r}^{2}}\left({\partial}_{{r\theta}_{1}}^{2}{V+}\frac{{r}^{2}{+B}_{2}}{B_{2}}{\partial }_{{r\theta}_{2}}^{2}{V}\right) & \frac{-{1}}{{r}^{2}}\left({\partial}_{{\theta}_{1}}^{2}{V+}\frac{{r}^{2}{+B}_{2}}{B_{2}}{\partial}_{{\theta}_{1}{\theta}_{2}}^{2}{V}\right) & \frac{-{1}}{{r}^{2}}\left({\partial }_{{\theta}_{1}{\theta}_{2}}^{2}{ V+}\frac{{r}^{2}{+B}_{2}}{B_{2}}{\partial }_{{\theta}_{2}}^{2}{V}\right)\end{bmatrix}}$,\\
$A_4={\begin{bmatrix}
	0	& -\frac{2B_{1}Lr}{\left({r}^{2}{+B}_{1}{+B}_{2}\right)^{2}} &-\frac{2B_{2}Lr}{\left({r}^{2}{+B}_{1}{+B}_{2}\right)^{2}} \\ 
	\frac{2L}{{r}\left({r}^{2}{+B}_{1}{+B}_{2}\right)} & 0 & 0 \\ 
	\frac{2L}{{r}\left({r}^{2}{+B}_{1}{+B}_{2}\right)} & 0 & 0\end{bmatrix}}$, and $I,0$ are respectively the $3\times3$ identity and zero matrices.

For linear stability, we need the eigenvalues of $A_{RE} $ to be purely imaginary. Calculating the characteristic polynomial
of $A_{RE}$, we find: $p(z):=z^{6}+c_{2}z^{4}+c_{1}z^{2}+c_{0}$, where\\ $c_2=\frac{4 L^2 (B_1+B_2)}{\left(r^2+B_1+B_2\right)^3}+\frac{ r^2+B_1}{B_1 r^2}\partial^2_{\theta_1}V+\frac{r^2+B_2}{B_2 r^2}\partial^2_{\theta_2}V+\frac{2}{r^2}\partial^2_{\theta_1\theta_2}V+\partial^2_rV$,\\
$c_1=\frac{4 L^2}{B_1 B_2 \left(r^2+B_1+B_2\right)^3}\left(B_1^2 \partial^2_{\theta_2}V+B_2^2 \partial^2_{\theta_1}V-2 B_1 B_2 \partial^2_{\theta_1\theta_2}V\right)\\\text{\enspace\enspace\enspace}+\frac{r^2+B_1+B_2}{B_1 B_2 r^2}\left(\partial^2_{\theta_1}V \partial^2_{\theta_2}V-\partial^2_{\theta_1\theta_2}V^2\right)+\frac{r^2+B_1}{B_1 r^2}\left(\partial^2_{r}V \partial^2_{\theta_1}V-(\partial^2_{\theta_1}V)^2\right)\\\text{\enspace\enspace\enspace}+\frac{r^2+B_2}{B_2 r^2}\left(\partial^2_{r}V \partial^2_{\theta_2}V-(\partial^2_{\theta_2}V)^2\right)+\frac{2}{r^2}(\partial^2_{r}V \partial^2_{\theta_1\theta_2}V-\partial^2_{\theta_1}V \partial^2_{\theta_2}V)$,\\
$c_0=-\frac{r^2+B_1+B_2}{B_1 B_2 r^2}\left(\partial^2_{\theta_1}V \left((\partial^2_{\theta_2}V)^2-\partial^2_{r}V \partial^2_{\theta_2}V\right)+\partial^2_{r}V \partial^2_{\theta_1\theta_2}V^2\right)\\\text{\enspace\enspace\enspace}-\frac{r^2+B_1+B_2}{B_1 B_2 r^2}\left((\partial^2_{\theta_1}V)^2 \partial^2_{\theta_2}V-2 \partial^2_{\theta_1}V \partial^2_{\theta_2}V \partial^2_{\theta_1\theta_2}V\right)$.\\
Considering the polynomial $p\left(z^{2}\right) =\left(z^{2}\right)^{3}+c_{2}\left(z^{2}\right)^{2}+c_{1}z^{2}+c_{0}$ as a cubic in $z^{2},$ our requirement is that $z^{2}$ roots are real and negative. For a cubic to have all real roots we need the discriminant $\Delta\geq0$, where $\Delta=-27c_{0}^{2}-4c_{1}^{2}+18c_{0}c_{1}c_{2}+c_{1}^{2}c_{2}^{2}-4c_{0}c_{2}^{3}$. This becomes our first criteria for linear stability.

To ensure that the roots are negative, we can employ the Routh-Hurwitz stability criteria \cite{Rahman2002}. These criteria require the coefficients to be real, which we have. For a cubic, the criteria are: $c_{i}>0$. So our criteria for $z^{2}$ negative roots become:
\begin{align}\label{eq:2DB_Linear_Stability_Criteria}
	{\displaystyle \Delta\geq \,0\text{ and }c_{i}>0.}
\end{align}
Roots matching these criteria give us $z$ purely imaginary, and also linear stability. We will use these criteria for the configurations examined below. For our first configuration, as with the dumbbell/point mass problem, we find RE when the dumbbells are colinear.
\vspace*{-5mm}\subsection{Colinear Configuration: $\theta_1,\theta_2 =0$}
\label{2DB_Col}\begin{figure}[H]
	\centering
	\includegraphics[width=0.4\linewidth]{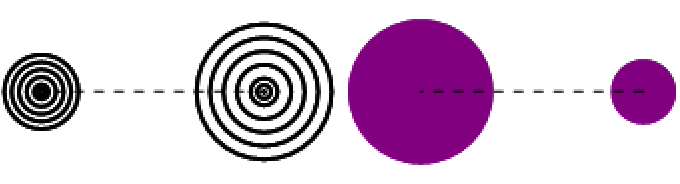}
%	\captionsetup{labelformat=empty}
	\caption{\\\bfseries Colinear Configuration}
\end{figure}
Note that this configuration ($\theta_1,\theta_2 =0$) immediately satisfies the angular requirements \eqref{eq:2DB_Angular_Requirements}. In terms of $L^{2}$, from the radial requirement \eqref{eq:2DB_RadialREQ} we find:\\
$L^{2}(r)=\frac{\left({r}^{2}+B_{1}+B_{2}\right)^{2}}{r}x_{11}\left(\frac{x_{21}}{\left({r+}x_{12}{\ell}_{1}-x_{22}{\ell}_{2}\right)\left\vert{r+}x_{12}{\ell}_{1}-x_{22}{\ell}_{2}\right\vert}+\frac{x_{22}}{\left({r+}x_{12}{\ell}_{1}+x_{21}{\ell}_{2}\right)^{2}}\right)$
\begin{align}\label{eq:2DB_Col_LS_EQ}
	{\textstyle +\frac{\left({r}^{2}+B_{1}+B_{2}\right)^{2}}{r}x_{12}\left(\frac{x_{21}}{\left({r-}x_{11}{\ell}_{1}-x_{22}{\ell}_{2}\right)\left\vert{r-}x_{11}{\ell}_{1}-x_{22}{\ell}_{2}\right\vert}+\frac{x_{22}}{\left({r-}x_{11}{\ell}_{1}+x_{21}{\ell}_{2}\right)\left\vert{r-}x_{11}{\ell}_{1}+x_{21}{\ell}_{2}\right\vert}\right).}
\end{align}

Observe that we have singularities (collisions of the masses) when $d_{uv}=0$, or equivalently when $r\in\left\{r_{4},r_{3},r_{2},r_{1}\right\}:=\left\{-{x_{21}\ell}_{2}{-x_{12}\ell}_{1},\;-{x_{21}\ell}_{2}+{x_{11}\ell}_{1},\;{x_{22}\ell}_{2}-{x_{12}\ell}_{1},\;{x_{22}\ell}_{2}+{x_{11}\ell}_{1}\right\} $.

For this configuration, not only do we have singularities, but also ``overlap.'' That is, for sufficiently small radii, the inner mass of each body is located within the massless rod of the other body. 
\begin{figure}[H]
	\centering
	\includegraphics[width=0.4\linewidth]{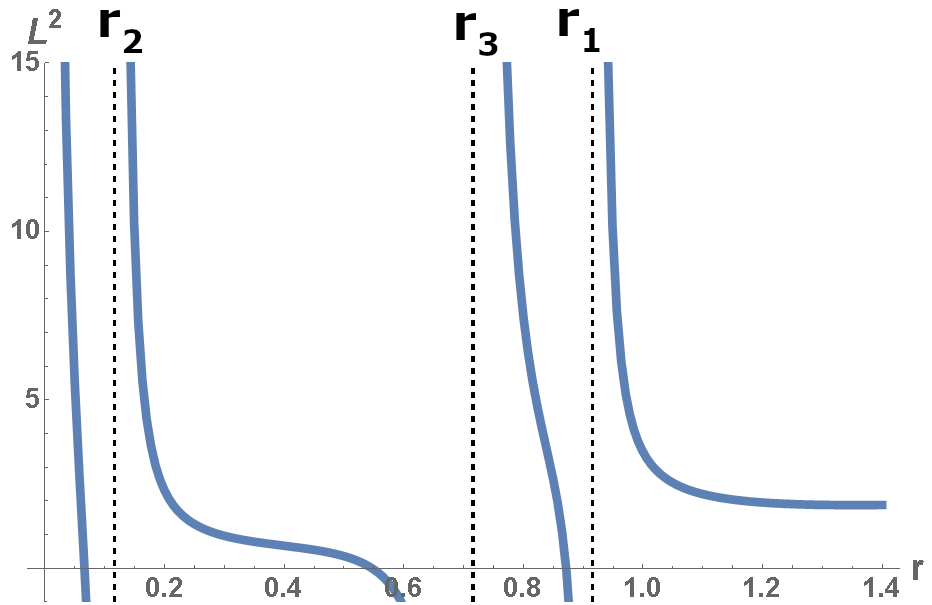}
	\caption{\\\bfseries Colinear $L^2$ Curves}{\vspace{0.8em}\footnotesize $x_{11}=0.9,x_{21}=0.08,\ell_1=0.2,M_1=\frac{1}{2}$\\The dotted lines ($r_i$) are collision radii.\\[-0.801em]}
	\label{fig:2KD_Colinear_LS}
\end{figure}
We see five intervals around these singularities: $-\infty<r_{4}< r_{3},r_{2}$ $<r_{1}<\infty$. We can ignore the interval $r<r_{4}$ (non-overlap, but negative radius), as this case is the same as $r>r_{1}$ with the $x_{ij}$ and ${\ell}_{i}$ permuted. Indeed, when $r<r_{2},r_{3}$, we may have negative radius, but we can disregard these radially negative cases for the same reason. However, the behavior of our system has the most physical relevance when $r>r_{1}$, the non-overlapped case. We will restrict our attention to this case.
\vspace*{-5mm}\subsubsection{Case: $r>r_1$, Non-Overlap}
\begin{figure}[H]
	\centering
	\includegraphics[width=0.4\linewidth]{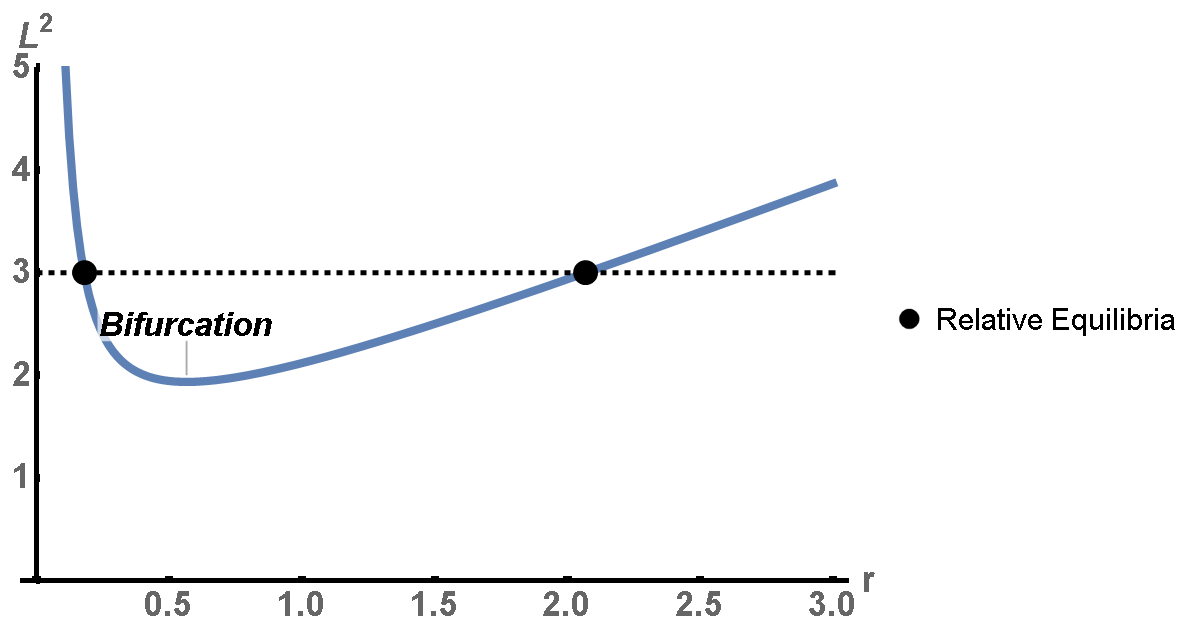}
	\caption{\\\bfseries Colinear $L^2$ Curve for Non-Overlap}{\vspace{0.8em}\footnotesize $x_{11}=\frac{1}{2},x_{21}=\frac{1}{10},\ell_1=\frac{1}{2},M_1=\frac{1}{2}$\\We have a single bifurcation, so zero or two RE depending upon angular momentum.\\[-1.901em]}
	\label{fig:6_3a_no_overlap}
\end{figure}
To examine this interval of most interest, let: $R:=r-{x_{22}\ell}_{2}-{x_{11}\ell}_{1}$ or $r=R+{x_{11}\ell}_{1}+{x_{22}\ell}_{2}$. So for $R\in\left(0,\infty\right)$, we have $r>r_{1}$. Substituting this into $\eqref{eq:2DB_Col_LS_EQ}$:\\ $L^2(R)=$
$\frac{1}{R+{x}_{11}\ell_{1}+{x}_{22}\ell_{2}}\left(\frac{{x}_{11}{x}_{21}}{\left(R+\ell_{1}\right)^{2}}+\frac{{x}_{12}{x}_{22}}{\left(R+\ell_{2}\right){}^{2}}\right)\left(\frac{{x}_{11}{x}_{12}}{M_{2}}{\ell}_{1}^{2}+\frac{{x}_{21}{x}_{22}}{M_{1}}{\ell}_{2}^{2}+\left({R+x}_{11}{\ell}_{1}{+x}_{22}{\ell}_{2}\right)^{2}\right)^{2}$\\[-1.001em]
\begin{equation}\label{eq:2DB_Colinear_NoOver}
	{\textstyle   +\frac{1}{R+{x}_{11}\ell_{1}+{x}_{22}\ell_{2}}\left(\frac{{x}_{11}{x}_{22}}{\left(R+1\right)^{2}}+\frac{{x}_{12}{x}_{21}}{R^{2}}\right)\left(\frac{{x}_{11}{x}_{12}}{M_{2}}{\ell}_{1}^{2}+\frac{{x}_{21}{x}_{22}}{M_{1}}{\ell}_{2}^{2}+\left({R+x}_{11}{\ell}_{1}{+x}_{22}{\ell}_{2}\right)^{2}\right)^{2}.}
\end{equation}\\[-0.901em]
Now we will prove the shape of the non-overlap $L^2$ graphs. In Figure \ref{fig:6_3a_no_overlap}, we see the RE bifurcate for a particular value of $L^2$.  Below this value, there are no RE, and above it there are two RE. 
\begin{theorem}[Number of Non-Overlapped Colinear RE as Angular Momentum Varies]
	\label{2DB_Colinear_NonOverlapped_Bifurcation_Thm}
	In the planar colinear non-overlapped two-dumbbell problem, and for sufficiently low angular momenta $L$, there are no RE. However, for some angular momentum $L_b>0$, and at some radius $r_b$, two RE bifurcate. For all angular momenta greater than $L_b$, there are two RE.
\end{theorem}\vspace{-1em}
\begin{proof}By inspection, observe that \eqref{eq:2DB_Colinear_NoOver} is always positive, and $L^2(R)\rightarrow \infty $ as $R\rightarrow\left\{0,\infty\right\}.$ To show that we have only one bifurcation as we vary $L^{2},$ we first show that $L^2$ has positive curvature. We do this by showing that $(L^2)^{\text{ }\prime\prime }(R)$ is always positive. Multiplying $(L^2)^{\text{ }\prime\prime }(R)$ by the positive
expression $h(R):=\left(R+{x_{11}\ell}_{1}+{x_{22}\ell}_{2}\right)^{3}$, we get: $h(R)(L^2)^{\prime\prime}(R)=$\\[-3.901em]
\begin{adjustwidth}{-0.7cm}{0cm}
\begin{flalign}\label{eq:2DB_Col_LS_Curvature_EQ}
	  \begin{aligned}
	\scriptstyle &\scriptstyle 6\left({R+x}_{11}{\ell}_{1}{+x}_{22}{\ell}_{2}\right)^{2}\left(\frac{{x}_{12}{x}_{22}}{\left(R+\ell_{2}\right){}^{4}}+\frac{{x}_{12}{x}_{21}}{R^{4}}+\frac{{x}_{11}{x}_{21}}{\left(R+\ell_{1}\right){}^{4}}+\frac{{x}_{11}{x}_{22}}{(R+1)^{4}}\right)\left(\left({R+x}_{11}{\ell}_{1}{+x}_{22}{\ell}_{2}\right){{}}^{2}+\frac{{x}_{11}{x}_{12}}{M_{2}}{\ell}_{1}^{2}{+}\frac{{x}_{21}{x}_{22}}{M_{1}}{\ell}_{2}^{2}\right){{}}^{2}\\
	&\scriptstyle-16\left({R+x}_{11}{\ell}_{1}{+x}_{22}{\ell}_{2}\right)^{3}\left(\frac{{x}_{12}{x}_{22}}{\left(R+\ell_{2}\right){}^{3}}+\frac{{x}_{12}{x}_{21}}{R^{3}}+\frac{{x}_{11}{x}_{21}}{\left(R+\ell_{1}\right){}^{3}}+\frac{{x}_{11}{x}_{22}}{(R+1)^{3}}\right)\left(\left({R+x}_{11}{\ell}_{1}{+x}_{22}{\ell}_{2}\right){}^{2}+\frac{{x}_{11}{x}_{12}}{M_{2}}{\ell}_{1}^{2}+\frac{{x}_{21}{x}_{22}}{M_{1}}{\ell}_{2}^{2}\right)\\
	&\scriptstyle+4\left({ R+x}_{11}{\ell}_{1}{+x}_{22}{\ell}_{2}\right)\left(\frac{{x}_{12}{x}_{22}}{\left(R+\ell_{2}\right){}^{3}}+\frac{{x}_{12}{x}_{21}}{R^{3}}+\frac{{x}_{11}{x}_{21}}{\left(R+\ell_{1}\right){}^{3}}+\frac{{x}_{11}{x}_{22}}{(R+1)^{3}}\right)\left(\left({R+x}_{11}{\ell}_{1}{+x}_{22}{\ell}_{2}\right){}^{2}+\frac{{x}_{11}{x}_{12}}{M_{2}}{\ell}_{1}^{2}+\frac{{x}_{21}{x}_{22}}{M_{1}}{\ell}_{2}^{2}\right){}^{2}\\
	&\scriptstyle+2\left(\frac{{x}_{12}{x}_{22}}{\left(R+\ell_{2}\right){}^{2}}+\frac{{x}_{12}{x}_{21}}{R^{2}}+\frac{{x}_{11}{x}_{21}}{\left(R+\ell_{1}\right){}^{2}}+\frac{{x}_{11}{x}_{22}}{(R+1)^{2}}\right)\left(\left({ R+x}_{11}{\ell}_{1}{+x}_{22}{\ell}_{2}\right){}^{2}+\frac{{x}_{11}{x}_{12}}{M_{2}}{\ell}_{1}^{2}+\frac{{x}_{21}{x}_{22}}{M_{1}}{\ell}_{2}^{2}\right){}^{2}\\
	&\scriptstyle-4\left({R+x}_{11}{\ell}_{1}{+x}_{22}{\ell}_{2}\right)^{2}\left(\frac{{x}_{12}{x}_{22}}{\left(R+\ell_{2}\right){}^{2}}+\frac{{x}_{12}{x}_{21}}{R^{2}}+\frac{{x}_{11}{x}_{21}}{\left(R+\ell_{1}\right){}^{2}}+\frac{{x}_{11}{x}_{22}}{(R+1)^{2}}\right)\left(\left({R+x}_{11}{\ell}_{1}{+x}_{22}{\ell}_{2}\right){}^{2}+\frac{{x}_{11}{x}_{12}}{M_{2}}{\ell}_{1}^{2}+\frac{{x}_{21}{x}_{22}}{M_{1}}{\ell}_{2}^{2}\right)\\
	&\scriptstyle+8\left(R+{x}_{11}\ell_{1}+{x}_{22}\ell_{2}\right)^{4}{}\left(\frac{{x}_{12}{x}_{22}}{\left(R+\ell_{2}\right){}^{2}}+\frac{{x}_{12}{x}_{21}}{R^{2}}+\frac{{x}_{11}{x}_{21}}{\left(R+\ell_{1}\right){}^{2}}+\frac{{x}_{11}{x}_{22}}{(R+1)^{2}}\right).
	  \end{aligned}
	\end{flalign}
\end{adjustwidth}
To simplify our analysis, we make the following changes of parameter:
\begin{flalign}\label{eq:2DB_Col_LS_Curvature_COVs}
	\begin{aligned}
	&\textstyle{x}_{11}\rightarrow\frac{u_{1}}{1+u_{1}},\text{\enspace\enspace}{x}_{12}\rightarrow \frac{1}{1+u_{1}},\text{\enspace\enspace}{x}_{21}\rightarrow \frac{u_{2}}{1+u_{2}},\text{\enspace\enspace}{x}_{22}\rightarrow \frac{1}{1+u_{2}},\\
	&\textstyle M_{1}\rightarrow\frac{m}{1+m},\text{\enspace\enspace}M_{2}\rightarrow\frac{1}{1+m},
	\text{\enspace\enspace}\ell_{1}\rightarrow\frac{\ell}{1+\ell},\text{\enspace\enspace}\text{ and }\ell_{2}\rightarrow\frac{1}{1+\ell}.
	\end{aligned}
\end{flalign}

Note that we still have ${x}_{i1}+{x}_{i2}=M_{1}+$ $M_{2}=\ell_{1}+\ell_{2}=1,$ but now we have characterized\ these 8 parameters as only 4 parameters $0<u_{i},m,\ell<\infty$. Upon substitution into \eqref{eq:2DB_Col_LS_Curvature_EQ}, the resulting expanded expression has 37,144 terms. However, since all of our parameters are defined to be positive, and the terms are all added together, the result is positive. Therefore the graph is concave up, giving at most one bifurcation. We calculate that $(L^2)^{\;\prime }(R)=0$ when:
\begin{equation}
	\begin{split}
	\scriptstyle 0& \scriptstyle=4\left({R+\ell}_{1}{x}_{11}{+\ell}_{2}{x}_{22}\right)^{2}\left(\left({R+\ell}_{1}{x}_{11}{+\ell}_{2}{x}_{22}\right)^{2}+\frac{{x}_{11}{x}_{12}}{M_{2}}{\ell}_{1}^{2}+\frac{{x}_{21}{x}_{22}}{M_{1}}{\ell}_{2}^{2}\right)\left(\frac{{x}_{12}{x}_{21}}{R^{2}}+\frac{{x}_{11}{x}_{22}}{\left(R+1\right)^{2}}+\frac{{x}_{11}{x}_{21}}{\left(R+\ell_{1}\right)^{2}}+\frac{{x}_{12}{x}_{22}}{\left(R+\ell_{2}\right)^{2}}\right)\nonumber \\
	&\scriptstyle-2\left({R+\ell}_{1}{x}_{11}{+\ell}_{2}{x}_{22}\right)\left(\left({R+\ell}_{1}{x}_{11}{+\ell}_{2}{x}_{22}\right)^{2}+\frac{{x}_{11}{x}_{12}}{M_{2}}{\ell}_{1}^{2}+\frac{{x}_{21}{x}_{22}}{M_{1}}{\ell}_{2}^{2}\right)^{2}\left(\frac{{x}_{12}{x}_{21}}{R^{3}}+\frac{{x}_{11}{x}_{22}}{\left(R+1\right)^{3}}+\frac{{x}_{11}{x}_{21}}{\left(R+\ell_{1}\right)^{3}}+\frac{{x}_{12}{x}_{22}}{\left(R+\ell_{2}\right)^{3}}\right)\nonumber \\
	&\scriptstyle-\left(\left({R+\ell}_{1}{x}_{11}{+\ell}_{2}{x}_{22}\right)^{2}+\frac{{x}_{11}{x}_{12}}{M_{2}}\ell_{1}^{2}+\frac{{x}_{21}{x}_{22}}{M_{1}}\ell_{2}^{2}\right)^{2}\left(\frac{{x}_{12}{x}_{21}}{R^{2}}+\frac{{x}_{11}{x}_{22}}{\left(R+1\right)^{2}}+\frac{{x}_{11}{x}_{21}}{\left(R+\ell_{1}\right)^{2}}+\frac{{x}_{12}{x}_{22}}{\left(R+\ell_{2}\right)^{2}}\right).\nonumber
	\end{split}
\end{equation}
We note that (on $R>0$) this expression is continuous and takes negative (as $R\rightarrow 0$ for instance), and positive (as $R\rightarrow \infty $) values. Therefore, by the intermediate value theorem $(L^2)^{\;\prime }(R)$ has a zero. So $L^2$ must have a minimum, giving us a bifurcation of the number of RE as $L^{2}$ is varied. As noted above, $L^2$ is always positive, so this minimum is positive, giving these RE physical relevance (real angular momenta).
\end{proof}
Now that we know the location of the colinear RE, let us determine their stability.

\paragraph{Energetic Stability of Colinear}
%[Documented Hessian_Collinear 8/8/21]
\begin{theorem}[Stability for the Colinear Non-Overlapped Two-Dumbbell Problem]
	\label{2DB_Colinear_Stability_Thm}
	RE of the colinear, non-overlapped two-dumbbell problem are stable when $\partial_{r}L^{2}>0$.
\end{theorem}
\begin{proof}
\label{2DB_Colinear_Energetic_Stability}As we saw in Section \ref{AmendedPotential}, to determine energetic stability we will need to check if the RE are strict minima of the amended potential $V$. 
We find that since $\theta_{1}=\theta_{2}=0$, we have $\partial_{R,{\theta}_{1}}V=\partial_{R,{\theta}_{2}}V=\partial_{{\theta}_{1},R}V=\partial_{{\theta}_{2},R}V=0.$ So our Hessian \eqref{eq:2DB_Hession} becomes block diagonal:\\
$\text{\enspace\enspace}H=\begin{bmatrix}\partial^2_{R}V&0&0\\ 
	0&\partial^2_{\theta_1}V&\partial^2_{\theta_1\theta_2}V\\
	0&\partial^2_{\theta_2,\theta_1}V&\partial^2_{\theta_2}V
\end{bmatrix}.$

Therefore, stability of our RE will be determined by the sign of $\partial^2_{R}V$ (which we learned from \eqref{eq:2DB_Vr_Slope_Equal_LSquared_Slope} was the sign of the slope of our $L^2$ graphs) and the sub-Hessian $H_s=\begin{bmatrix}\partial_{{\theta}_{1},{\theta}_{1}}V & \partial_{{\theta}_{1},{\theta}_{2}}V \\ 
\partial_{{\theta}_{2},{\theta}_{1}}V & \partial_{{\theta}_{2},{\theta}_{2}}V\end{bmatrix}$. Determining whether this matrix is positive definite is nontrivial due to the complexity of the component expressions. In an effort to simplify things, we make the same substitutions $\eqref{eq:2DB_Col_LS_Curvature_COVs}$ as we did with $\eqref{eq:2DB_Col_LS_Curvature_EQ}$. Then, the diagonal components as well as the determinant are long expressions of $u_1,u_2,m,$ and $\ell$. And since the parameters are defined to be positive, and all of the terms of these expressions are added together, we find the expressions to be positive, the sub-Hessian to be positive definite, and the RE to be stable when $\partial_rL^{2}>0$.
\end{proof}
%Observe that this is consistent with the stability results we saw previously for the dumbbell/point mass model in Section \ref{1DB_Colinear_Linear_Stability}. 
If you note that the colinear two-dumbbell problem becomes the colinear dumbbell/point mass problem as $\ell_1 \rightarrow 0$, we should expect the stability results for the dumbbell/point mass problem to be consistent with this limit of the two-dumbbell problem. Figure \ref{fig:2DB_Linearization}b suggests stability converging on the dumbbell/point mass stability results shown in Figure \ref{fig:Col_M1_0p5}.
\paragraph{Linear Stability of Colinear}
\label{2DB_Colinear_Linear_Stability}Mapping the linear stability criteria \eqref{eq:2DB_Linear_Stability_Criteria} in the $Rx_{21}$-plane for the colinear configuration with no overlap, we find graphs:
\begin{figure}[H]
	\captionsetup[subfigure]{justification=centering}
	\centering
	\subfloat[$M_1=\frac{1}{2},\ell_1=\frac{1}{2},x_{11}=\frac{1}{2}$]
	{\includegraphics[width=0.45\textwidth]{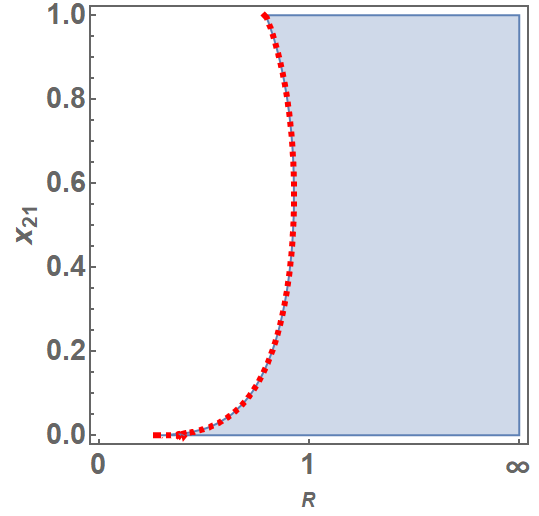}}
	\hspace{1cm}
	\subfloat[$M_1=\frac{1}{2},\ell_1=0.0001,x_{11}=0.9999$]
	{\includegraphics[width=0.45\textwidth]{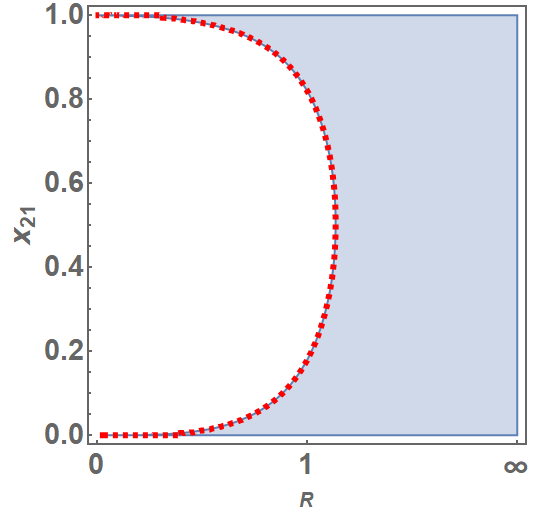}}
	\caption{\\\bfseries Colinear Energetic and Linear Stability}{\vspace{0.8em}\footnotesize $\partial_r L^2=0$ (dashed curve),  Energetically and Linearly Stable Region (shaded)\\Numerically, linear stability appears to coincide with energetic stability ($\partial_r L^{2}>0$).\\The parameters in (b) are close to the colinear dumbbell/point mass problem, approximating Figure \ref{fig:Col_M1_0p5}.\\[-0.801em]}
	\label{fig:2DB_Linearization}
\end{figure}
In order to visually cover the entire $R$ range, note that we have used $R=\frac{z}{2-z}$ as the horizontal axis in our graphs with $z\in\left(0,2\right)$. The dashed curve going through the plane is $\partial_r L^{2}=0$. So we see that the linear stability boundary coincides with the energetic stability boundary calculated in Section \ref{2DB_Energetic_Stability}. Figure \ref{fig:2DB_Linearization}a is the graph for our equal mass case, and we see that for each $x_{21}$, there is some radius $R\leq 1$ below which RE are unstable (as well as linearly unstable), and above which RE are stable (consistent with Theorem \ref{2DB_Colinear_Stability_Thm}). Figure \ref{fig:2DB_Linearization}b has parameters $\ell_{1}=0.0001$ and $x_{11}=0.9999,$ approaching the colinear dumbbell/point mass problem. Observe that the graph is indistinguishable from Figure \ref{fig:Col_M1_0p5} in Section \ref{1DB_Colinear_Linear_Stability}. We chose $M_{1}=\frac{1}{2}$ for these graphs, but qualitatively the shapes of these graphs do not change as you vary $M_{1}.$ The main difference is that the radius where stability begins is smaller for $M_{1}$ small, and larger for $M_{1}$ large. Now let us look at another symmetric configuration, where the bodies are perpendicular.

\vspace*{-5mm}\subsection{Perpendicular Configuration: $(\theta_1,\theta_2) =(\frac{\pi}{2},0)$}\label{2DB_Perp}
\begin{figure}[H]
	\centering
	\includegraphics[width=0.35\linewidth]{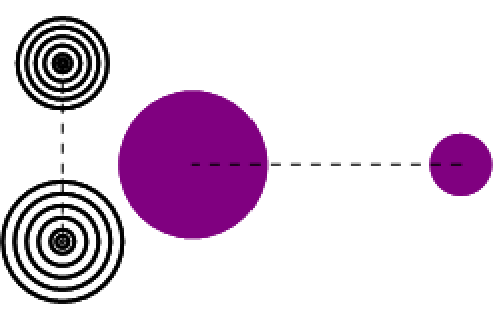}
	%	\captionsetup{labelformat=empty}
	\caption{\\\bfseries Perpendicular Configuration}
	\label{fig:QW491W03}
\end{figure}
We examine $(\theta_1,\theta_2) =(\frac{\pi}{2},0)$, but similar results are found for $(\theta_1,\theta_2) =(0,\frac{\pi}{2})$. For this perpendicular configuration, the distances between our masses \eqref{eq:2DB_Distances} become:\\
%$\text{\enspace\enspace}d_{uv}\left(r\right)|_{\theta_{i}:(\frac{\pi}{2},0)}=\sqrt{{r}^{2}+\left(-1\right)^{v}{2}x_{2\overline{v\,}}{\ell}_{2}{r+}x_{1\overline{u\,}}^{2}{\ell}_{1}^{2}+x_{2\overline{v\,}}^{2}{\ell}_{2}^{2}}.$
$\text{\enspace\enspace}d_{11}=\sqrt{{r}^{2}{-2}x_{22}{\ell}_{2}{r+}x_{12}^{2}{\ell}_{1}^{2}+x_{22}^{2}{\ell}_{2}^{2}},\text{\enspace\enspace}d_{12}=\sqrt{{r}^{2}{+2}x_{21}{\ell}_{2}{r+}x_{12}^{2}{\ell}_{1}^{2}+x_{21}^{2}{\ell}_{2}^{2}},$\\
$\text{\enspace\enspace}d_{21}=\sqrt{{r}^{2}-{2}x_{22}{\ell}_{2}{r+}x_{11}^{2}{\ell}_{1}^{2}+x_{22}^{2}{\ell}_{2}^{2}},\text{\enspace\enspace}d_{22}=\sqrt{{r}^{2}{+2}x_{21}{\ell}_{2}{r+}x_{11}^{2}{\ell}_{1}^{2}+x_{21}^{2}{\ell}_{2}^{2}}.$\\
And our angular requirements $\eqref{eq:2DB_Angular_REQ}$ become:
\begin{subequations}\label{eq:2DB_Perp_Angular_REQ}
	\makeatletter\@fleqntrue\makeatother
	\begin{align}
		\begin{split}
			&\text{$0=x_{21}\left(x_{22}{\ell}_{2}-{r}\right)\left(\frac{1}{d_{11}^{3}}-\frac{1}{d_{21}^{3}}\right)+x_{22}\left(x_{21}{\ell}_{2}+{r}\right)\left(\frac{1}{d_{22}^{3}}-\frac{1}{d_{12}^{3}}\right),$}\label{eq:2DB_Perp_Angular_REQa}
		\end{split}\\
		\begin{split}
			&\text{$0=x_{11}x_{12}{\ell}_{1}\left(\frac{1}{d_{12}^{3}}{-}\frac{1}{d_{11}^{3}}+\frac{1}{d_{21}^{3}}-\frac{1}{d_{22}^{3}}\right).$}\label{eq:2DB_Perp_Angular_REQb}
		\end{split}
	\end{align}
\end{subequations}
Unlike the colinear configuration, simply being perpendicular is insufficient to guarantee a RE.  Rather, the following theorem gives further restrictions on the shape and mass values.
\begin{theorem}[Perpendicular RE for the Two-Dumbbell Problem]\label{2DB_Number_of_Rhombus_RE_Thm}
For the perpendicular $(\theta_1,\theta_2) =(\frac{\pi}{2},0)$ configuration of the two-dumbbell problem, there is one family of isosceles RE ($d_{11}=d_{21}$ and $d_{12}=d_{22}$) where the masses on the vertical body are equal.
\end{theorem}
\begin{proof}
A simple calculation reveals that the requirement \eqref{eq:2DB_Perp_Angular_REQb}, for nontrivial dumbbells, reduces to:
\begin{align}\label{eq:2DB_Perpendicular_RE_Conditions}
	{\displaystyle  {\frac{1}{d_{12}^{3}}+\frac{1}{d_{21}^{3}}=}\frac{1}{d_{22}^{3}}+\frac{1}{d_{11}^{3}}.}
\end{align}
Let us examine when \eqref{eq:2DB_Perpendicular_RE_Conditions} is satisfied as we set various distances equal.

\textbf{Rhombus}\\When all of distances are equal, the dumbbells form a rhombus and $\eqref{eq:2DB_Perpendicular_RE_Conditions}$ and $\eqref{eq:2DB_Perp_Angular_REQa}$ are satisfied. Observe that $d_{11}=d_{21}$ implies $x_{11}=x_{12}=\frac{1}{2}.$ Also, $d_{11}=d_{12},$ implies:\\[-1em]
\begin{equation}\label{eq:2DB_Rhombus_REQ}
r=\frac{x_{22}-x_{21}}{2}\ell_2
\end{equation}
This also requires that $r<\frac{\ell_{2}}{2}$ (otherwise \eqref{eq:2DB_Rhombus_REQ} implies $x_{21}\le0$). $x_{21}$ must also be less than $\frac{1}{2}$, otherwise the equation implies negative radius. Therefore $r<\frac{\ell_2}{2}<\ell_2 x_{22}$ (since $x_{21}<\frac{1}{2}\implies x_{22}>\frac{1}{2}$), so the radius is in the overlap region. From \eqref{eq:2DB_Rhombus_REQ} you can choose various $x_{2j}$ masses, and the solution adjusts the radius and angular momentum necessary to maintain the rhombus shape.

\textbf{Isosceles}\\Another way we can set distances equal to each other is to satisfy the requirement \eqref{eq:2DB_Perpendicular_RE_Conditions} with $d_{11}=d_{21}$ and $d_{22}=d_{12}$ (which also implies $x_{11}=x_{12}=\frac{1}{2}$). Observe that these also satisfy \eqref{eq:2DB_Perp_Angular_REQa} giving isosceles triangles for any radius. Of course, the rhombus RE also satisfy this requirement, and are therefore a subset of the isosceles RE.

\textbf{Rhombus Again}\\The last way we could satisfy \eqref{eq:2DB_Perpendicular_RE_Conditions} with equal distances is in the case when $d_{11}=d_{12}$ and $d_{22}=d_{21}$. However, upon substituting these into $\eqref{eq:2DB_Perp_Angular_REQa}$, we find we must also have $d_{11}=d_{21}$, which gives us the rhombus configuration again.

\textbf{Unequal distances}\\The last situation to examine is when all of the distances are different, but still somehow manage to satisfy \eqref{eq:2DB_Perpendicular_RE_Conditions}. Rearranging \eqref{eq:2DB_Perp_Angular_REQa}, we have:
\begin{equation}\label{eq:2DB_Perp_Ang_REQ_b}
0={x}_{21}{\small\ell}_{2}\left(\frac{1}{d_{11}^{3}}{\small-}\frac{1}{d_{21}^{3}}+\frac{1}{d_{22}^{3}}{\small-}\frac{1}{d_{12}^{3}}\right) +{\small r}\left(\frac{1}{d_{22}^{3}}{\small-}\frac{1}{d_{12}^{3}}-\frac{1}{d_{11}^{3}}{\small +}\frac{1}{d_{21}^{3}}\right).
\end{equation} 
Notice that if we assume \eqref{eq:2DB_Perpendicular_RE_Conditions} is satisfied, this allows us to eliminate the first term in \eqref{eq:2DB_Perp_Ang_REQ_b}. And since our radius is greater than zero, \eqref{eq:2DB_Perp_Ang_REQ_b} becomes: 
\begin{equation}\label{eq:2DB_Perp_UnequalDist_Ang_REQ}
{\small \frac{1}{d_{22}^{3}}+\frac{1}{d_{21}^{3}}=}\frac{1}{d_{11}^{3}}{\small +}\frac{1}{d_{12}^{3}}.
\end{equation} 
 Solving for $\frac{1}{d_{22}^{3}}$ in \eqref{eq:2DB_Perpendicular_RE_Conditions} and substituting into \eqref{eq:2DB_Perp_UnequalDist_Ang_REQ} reduces to $\frac{1}{d_{11}^{3}}=\frac{1}{d_{21}^{3}}$, which implies $d_{11}=d_{21}.$ But we assumed the distances were all different, so we have a contradiction.  Therefore, the only possible RE for the perpendicular configuration is the isosceles family (with a subset of rhombus RE).
\end{proof}
We restrict our analysis to the area of most interest, the non-overlap radii.

\vspace*{-5mm}\subsubsection{Perpendicular - Isosceles}
\label{Isosceles}\begin{figure}[H]
	\centering
	\includegraphics[width=0.30\linewidth]{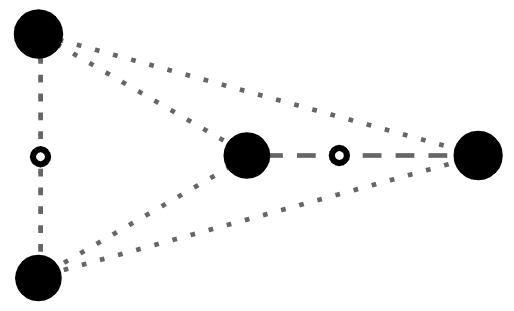}
%	\captionsetup{labelformat=empty}
	\caption{\\\bfseries Perpendicular - Isosceles}{\text{ }\\[-1.801em]}
	\label{fig:QW491W05}
\end{figure}
Substituting the isosceles mass and distance restrictions ($d_{11}=d_{21},$ $d_{22}=d_{12}$, and $x_{11}=\frac{1}{2}$) into our radial requirement \eqref{eq:2DB_Radial_REQ}, we find the angular momentum:\\[-1.801em]
\begin{align}
	{\textstyle L^{2}=\frac{\left({r}^{2}+B_{1}+B_{2}\right)^{2}}{r}\left(
		\frac{x_{21}\left({r-}x_{22}{\ell}_{2}\right)+x_{21}\left({r-}x_{22}{\ell}_{2}\right)}{d_{11}^{3}}+\frac{x_{22}\left({r+}x_{21}{\ell}_{2}\right)+x_{22}\left({r+}x_{21}{\ell}_{2}\right)}{d_{12}^{3}}\right) \nonumber}\\
	{\textstyle =\frac{2\left({r}^{2}+B_{1}+B_{2}\right)^{2}}{r}\left(x_{21}\frac{{r-}x_{22}{\ell}_{2}}{d_{11}^{3}}+x_{22}\frac{{r+}x_{21}{\ell}_{2}}{d_{12}^{3}}\right) . \nonumber}
\end{align}
Let us explore the shape of the curve to see how the RE bifurcate. Note that overlap of the dumbbell masses occurs when $r<\ell_{2}x_{22}.$ So we make the substitution $r\rightarrow R+\ell_{2}x_{22}$, and our distances become: $d_{11}=d_{21}=\sqrt{R^{2}{+\frac{1}{4}\ell}_{1}^{2}},$ and $
d_{12}=d_{22}=\sqrt{\left(R+\ell_{2}\right)^{2}{+\frac{1}{4}\ell}_{1}^{2}}.$ 
%Observe that $d_{i1}<d_{i2}$. 
Our angular momentum becomes: 
%$L^{2}=\frac{2\left({r}^{2}+B_{1}+B_{2}\right)^{2}}{r}\left(x_{21}\frac{{r-}x_{22}{\ell}_{2}}{d_{11}^{3}}+x_{22}\frac{{r+}x_{21}{\ell}_{2}}{d_{12}^{3}}\right)\rightarrow $
$L^{2}=\frac{2\left(\left(R+\ell_{2}x_{22}\right)^{2}+B_{1}+B_{2}\right)^{2}}{R+\ell_{2}x_{22}}\left(x_{21}\frac{R}{d_{11}^{3}}+x_{22}\frac{R+\ell_{2}}{d_{12}^{3}}\right)$, where we see that $L^2>0$, and therefore the angular momentum is real for the non-overlap region. %Angular momentum is  strictly increasing for $R>\frac{l_1}{\sqrt{6 M_2}}-l2 x_22$.
As $R\rightarrow 0$ we have: $L^{2}\rightarrow \frac{2\left(x_{22}^2\ell_{2}^{2}+\frac{x_{21}x_{22}}{M_{1}}\ell_{2}^{2} +\frac{1}{4M_{2}}\ell_{1}^{2}\right)^{2}}{\sqrt{\ell_{2}^{2}{+\frac{1}{4}\ell}_{1}^{2}}^{3}}$, and as $R\rightarrow\infty$ we have $L^{2}\rightarrow\infty$. However, the $L^2$ bifurcation curve is qualitatively different depending upon $(\ell_1,M_1,x_{21})$. Below, we graph three of the most common $L^2$ shapes.
\begin{figure}[H]
	\centering
	\includegraphics[width=0.48\linewidth]{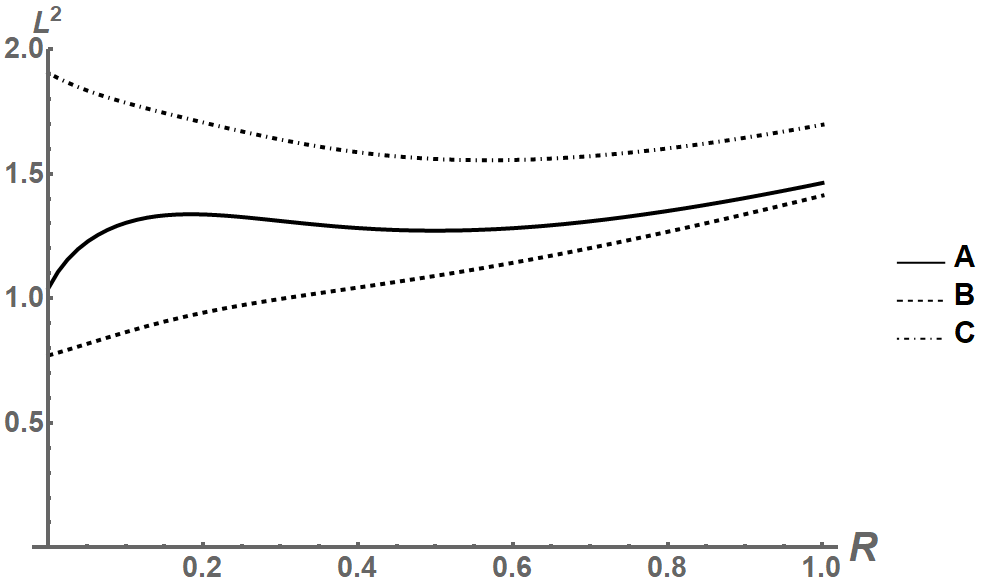}
	\caption{\\\bfseries Isosceles $L^2$ Curves}{\vspace{0.8em}\footnotesize $L^2$ graphs for parameter values with $\ell_1=\frac{3}{4}$ and\\(A) $M_1=\frac{1}{2},$ $x_{21}=\frac{3}{4}$.  (B) $M_1=\frac{1}{4},$ $x_{21}=\frac{1}{2}$.  (C) $M_1=\frac{1}{20},$ $x_{21}=\frac{1}{4}$.\\$L^2$ curves take a variety of shapes depending on parameters.  No simple relationship was found. \\[-0.801em]}
	\label{fig:2D_Perp_Isos}
\end{figure}
As bifurcation parameter $L^2$ increases, the number of RE found in the graphs above are A: ($0\rightarrow1\rightarrow3\rightarrow1$), B: ($0\rightarrow1$), and C: ($0\rightarrow2\rightarrow1$). And while these account for the vast majority of shapes, we have also found ($0\rightarrow2\rightarrow4\rightarrow2\rightarrow1$), ($0\rightarrow2\rightarrow1\rightarrow3\rightarrow1$), and ($0\rightarrow2\rightarrow4\rightarrow3\rightarrow1$). So no simple relationship between the shapes and parameter values was found. However, for all these curves, we start out with no RE for low angular momenta, and a single RE for sufficiently large angular momenta.

\paragraph{Energetic Stability of Isosceles}
\label{Energetic_Stability_of_Isosceles}As we saw in Section \ref{AmendedPotential}, to determine energetic stability we check if the RE are strict minima of the amended potential $V$. For $(\theta_{1},\theta_2)=(\frac{\pi}{2},0)$ and $x_{11}=x_{12}=\frac{1}{2}$, the resulting Hessian is block diagonal: 
\begin{equation}\label{eq:2DB_Perp_Hession}
	H=\begin{bmatrix}
		\partial_{r}^{2}V & 0 & 0 \\ 0 & \partial_{\theta_{1}}^{2}V & \partial_{\theta_{1},\theta_{2}}^{2}V\\0 & \partial_{\theta_{2},\theta_{1}}^{2}V & \partial_{\theta_{2}}^{2}V
	\end{bmatrix}.
\end{equation}
Recall from \eqref{eq:2DB_Vr_Slope_Equal_LSquared_Slope} that $\partial_{r}^{2}V$ is positive when the slope of our $L^2$ graph is positive. However, we also note by inspection that $\partial_{\theta_{1}}^{2}V$ is strictly negative:\\ $\text{\enspace\enspace}\partial_{\theta_{1}}^{2}V=-24\ell_{1}^{2}\left(\frac{x_{21}(r-x_{22}\ell_{2})^{2}}{\left(4(r-x_{22}\ell_{2})^{2}+\ell_{1}^{2}\right)^{5/2}}+\frac{x_{22}(r+x_{21}\ell_{2})^{2}}{\left(4(r+x_{21}\ell_{2})^{2}+\ell_{1}^{2}\right)^{5/2}}\right)$. 

Therefore, whenever the determinant of the first minor of $H$ is positive, the determinant of the second minor is negative. So, we have no strict minima or stability. Observe that we also lacked stability in \ref{1DB_Energetic_Stability_Isosceles} for the isosceles dumbbell/point mass problem. Note that the isosceles two-dumbbell problem above approaches the equal mass isosceles dumbbell/point mass problem as $\ell_2 \rightarrow 0$, so it's not surprising that in the limit we should also find no stability.

\paragraph{Linear Stability of Isosceles}
Despite the lack of energetic stability, when we map the linear stability criteria \eqref{eq:2DB_Linear_Stability_Criteria} in the $rM_{1}$-plane, we see linear stability for large $M_{1}$ on some non-overlapping radial intervals (with starting radius $r_s>\ell_2x_{22}$).  We observe that varying $x_{21}$ does not significantly affect the location of this region. As seen in the figures below, the region starts at higher $M_{1}$ when $\ell_{1}$ low and the radial intervals cease at the dashed line $\partial_r L^2=0$. Similar graphs for $\left(\theta_{1},\theta_{2}\right) =\left(0,\frac{\pi}{2}\right) $ exist with linear stability for small $M_{1}$. So for a sufficiently massive vertical body, a radially oriented satellite can find linear stability.
\begin{figure}[H]
	\captionsetup[subfigure]{justification=centering}
	\centering
	\subfloat[$\ell_1=0.2,x_{21}=\frac{1}{2}$]
	{\includegraphics[width=0.3\textwidth]{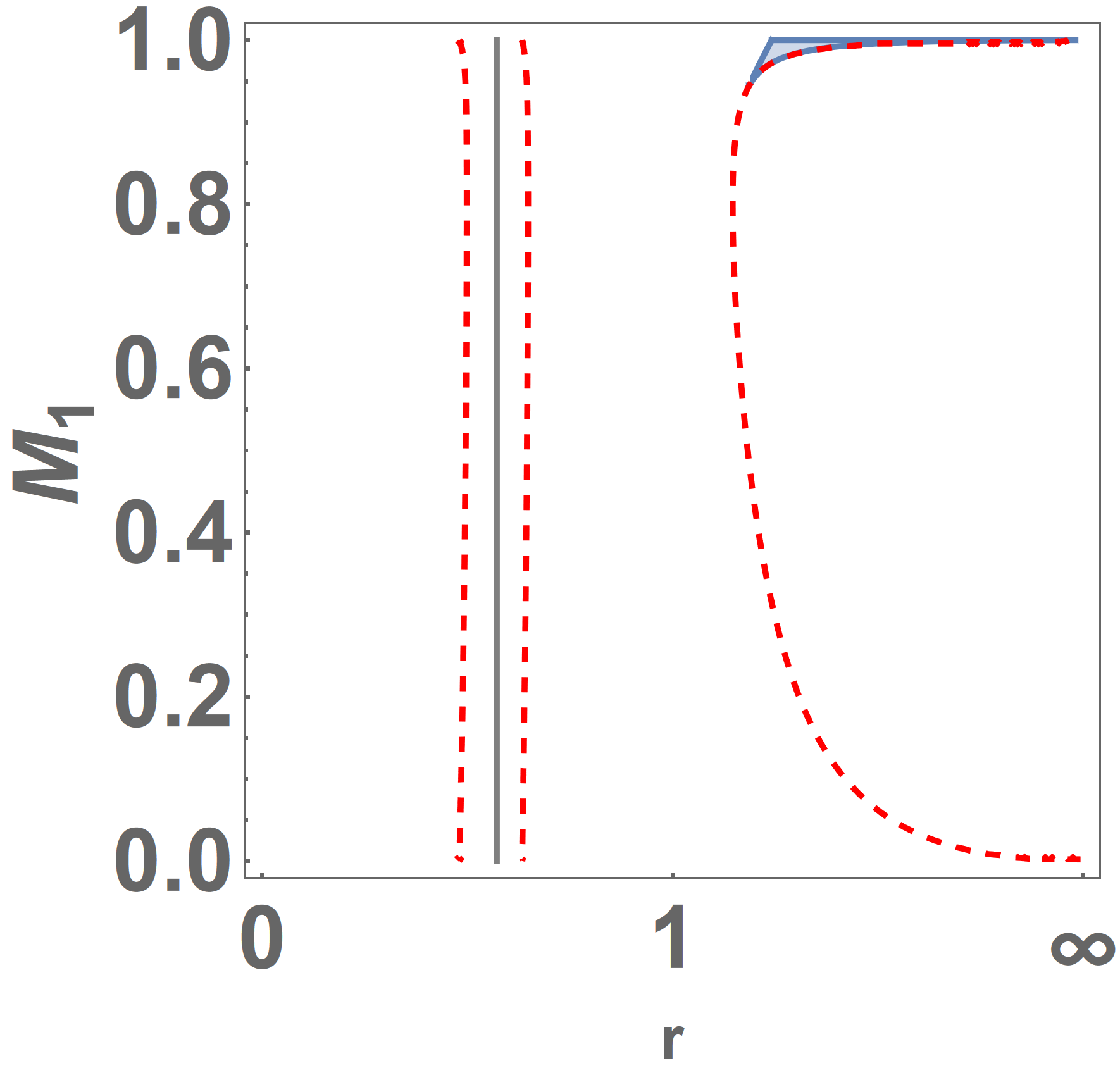}}
	\hspace{.1cm}
	\subfloat[$\ell_1=\frac{1}{2},x_{21}=\frac{1}{2}$]
	{\includegraphics[width=0.3\textwidth]{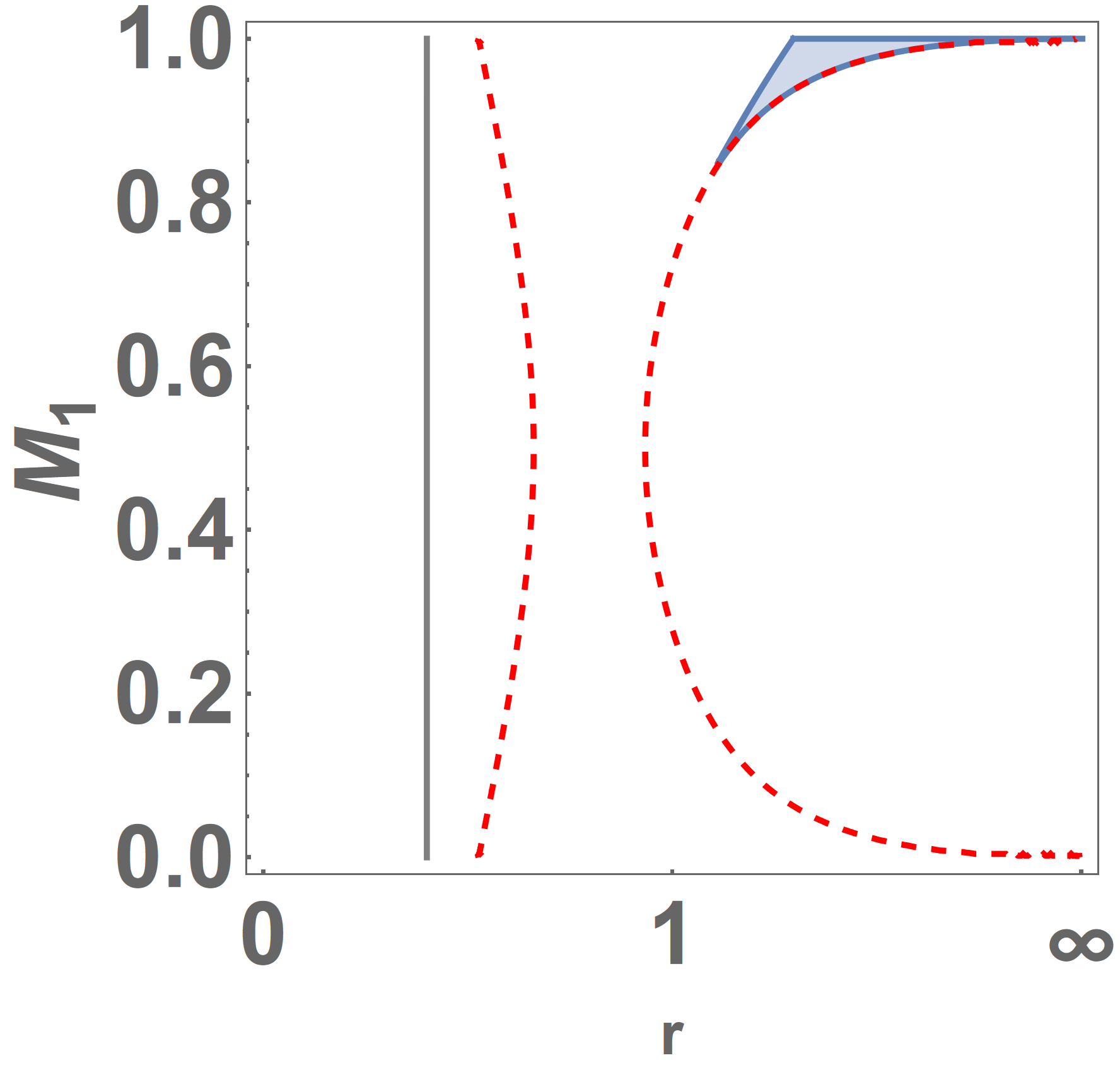}}
	\hspace{.1cm}
	\subfloat[$\ell_1=0.99,x_{21}=\frac{1}{2}$]
	{\includegraphics[width=0.3\textwidth]{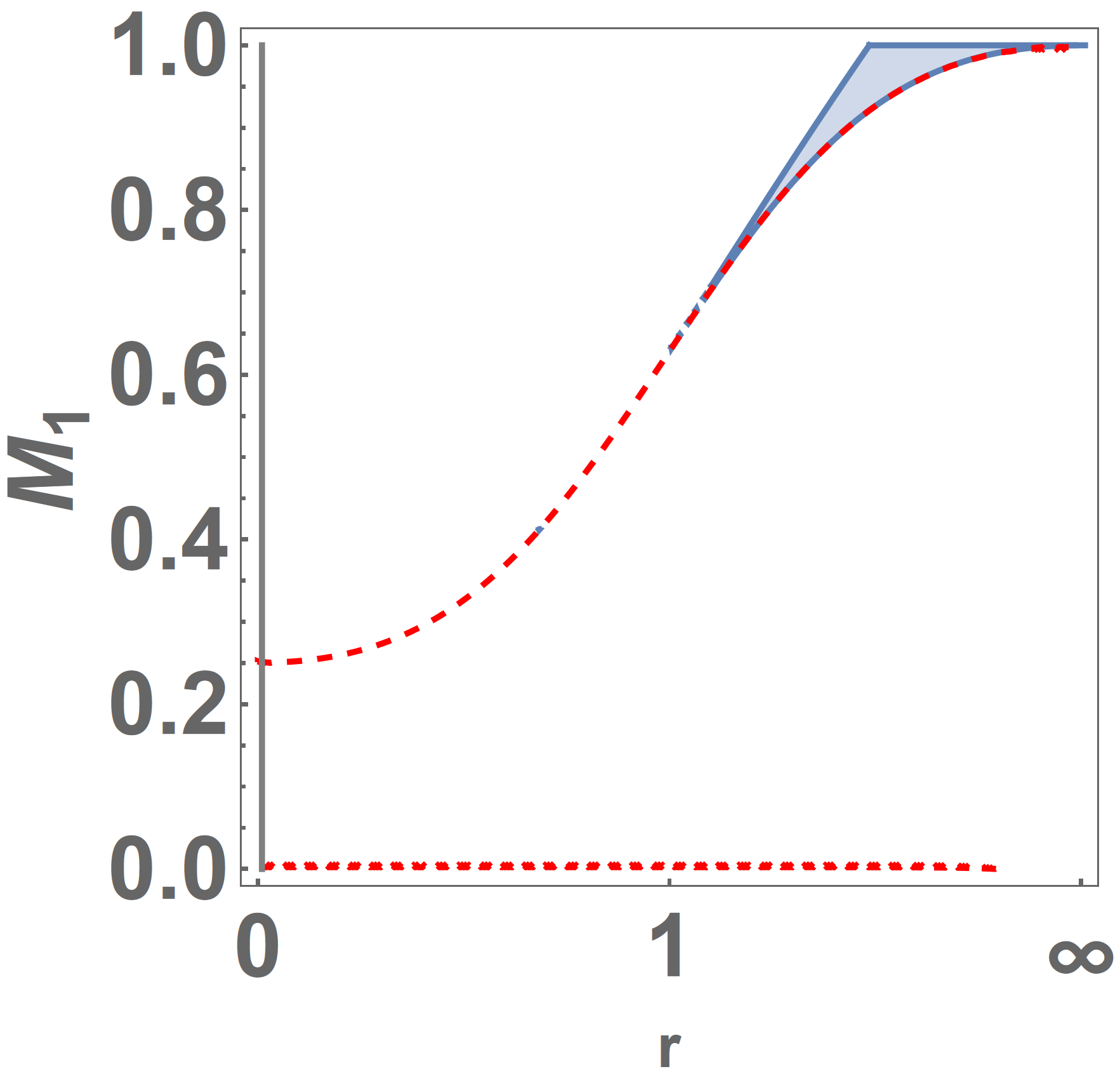}}
	\caption{\\\bfseries Isosceles Linear Stability}{\vspace{0.8em}\footnotesize $r=\ell_2x_{22}$ (solid line), $\partial_r L^2=0$ (dashed line), Linearly Stable Region (shaded)\\Linear stability with large vertical dumbbell mass $M_1$, bounded above by $\partial_r L^2=0$.\\[-0.801em]}
	\label{fig:Linearization_Isos_By_M1}
\end{figure}
\begin{wrapfigure}{R}{0.15\textwidth}
	\centering
	    \captionsetup{justification=centering}
	\raisebox{0pt}[\dimexpr\height-1.6\baselineskip\relax]{\includegraphics[width=0.15\textwidth]{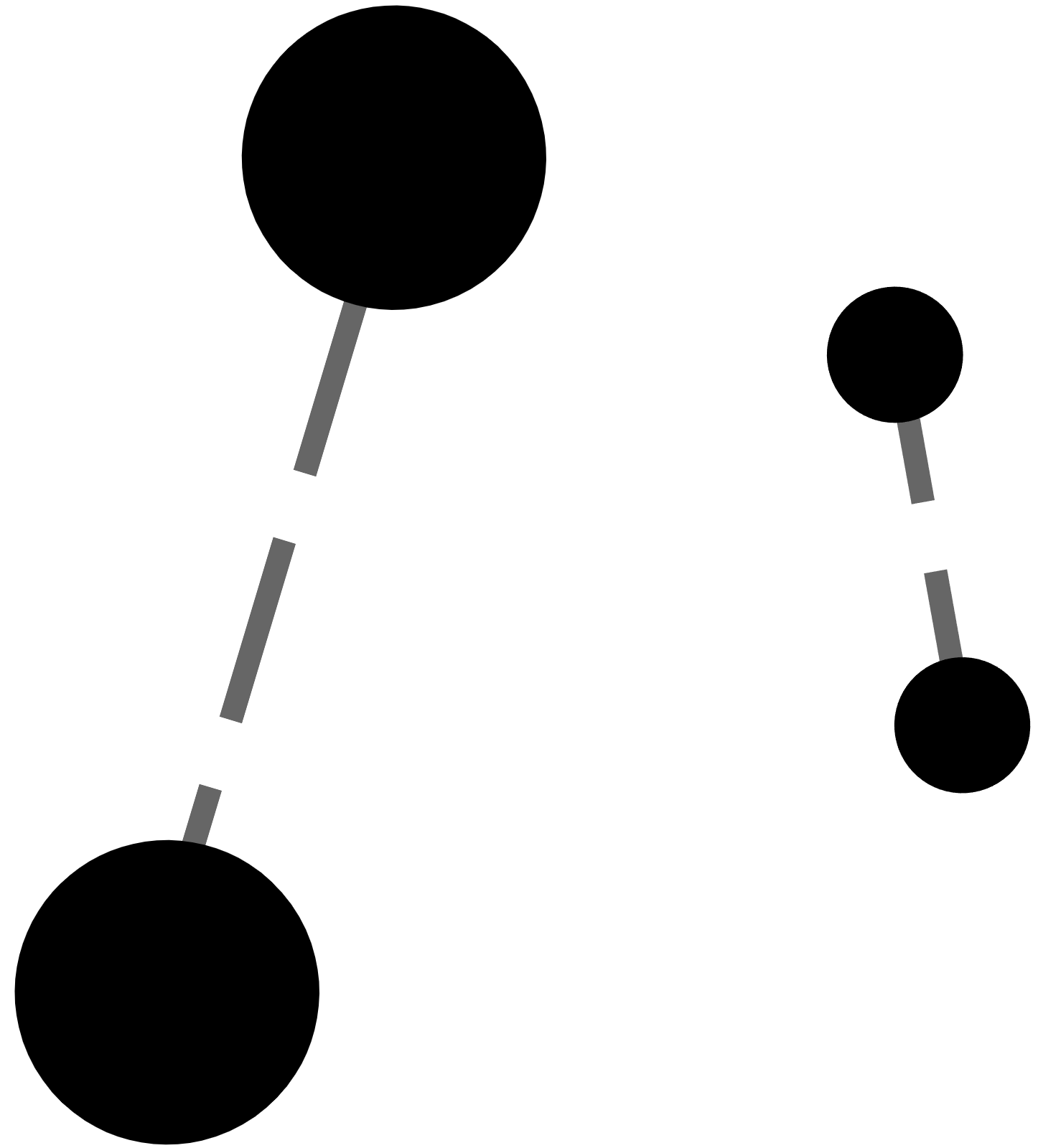}}
	%	\captionsetup{labelformat=empty}
	\caption{\\\bfseries  \protect{Equal Mass Configuration}}
\end{wrapfigure}\vspace*{-5mm}\subsection{(Pairwise) Equal Mass Configuration: $x_{11}=x_{12}$ and $x_{21}=x_{22}$}
Another obvious rotational symmetry is when the dumbbells are parallel, a trapezoid configuration. Applying $\theta_{1}=\theta_{2}=\frac{\pi}{2}$ to angular requirements \eqref{eq:2DB_Angular_Requirements}: 
\begin{equation}
	\begin{split}\label{eq:Trap_Angular_Requirements}
		&\textstyle {0 = }-x_{21}r\left(\frac{1}{d_{11}^{3}}-\frac{1}{d_{21}^{3}}\right)+x_{22}r\left(\frac{1}{d_{22}^{3}}-\frac{1}{d_{12}^{3}}\right)\text{, and}\\
		&\textstyle { 0 = \hspace{3pt}}x_{11}r\left(\frac{1}{d_{11}^{3}}-\frac{1}{d_{12}^{3}}\right)+x_{12}\left(\frac{1}{d_{21}^{3}}-\frac{1}{d_{22}^{3}}\right).
	\end{split}
\end{equation}
Finding all possible solutions to the system is nontrivial.  However, working with the system one finds $d_{11}=d_{22}$ and $d_{12}=d_{21}$ to be the most straightforward conditions solving it. These conditions further imply either the relationship $\frac{\ell_1} {\ell_2} =\frac{x_{12}-x_{11}}{x_{22}-x_{21}}$ between the parameters, or pairwise equal masses. Pairwise equal masses means the two dumbbells may have different masses ($M_1$ need not equal $M_2$), but the mass ratios on each dumbbell are equal: $x_{i1}=x_{i2}=\frac{1}{2}$. For this paper, we will focus on this (pairwise) equal mass configuration. In fact, this section will focus more generally on it (not assuming a trapezoid).  This will allow us to not only study the equal mass trapezoid configuration, but it provides a larger context in which to study asymmetric RE which bifurcate from the symmetric configurations we've studied so far. Substituting equal masses into our angular requirements \eqref{eq:2DB_Angular_Requirements}, we find:
\begin{equation}
	\begin{split}\label{eq:EqMass_Angular_Requirements}
		&\textstyle {0=}\left({\ell}_{2}\sin\left({\theta}_{1}{-\theta}_{2}\right) -2{r}\sin{\theta}_{1}\right)\left(\frac{1}{d_{11}^{3}}-\frac{1}{d_{21}^{3}}\right)+\left({\ell}_{2}\sin\left({\theta}_{1}{-\theta}_{2}\right)+2{r}\sin{\theta}_{1}\right)\left(\frac{1}{d_{22}^{3}}-\frac{1}{d_{12}^{3}}\right)\text{, and}\\
		&\textstyle { 0=}\left({\ell}_{1}\sin\left({\theta}_{1}{-\theta}_{2}\right) -2{r}\sin{\theta}_{2}\right)\left(\frac{1}{d_{11}^{3}}-\frac{1}{d_{12}^{3}}\right)-\left({\ell}_{1}\sin\left({\theta}_{1}{-\theta}_{2}\right)+2{r}\sin{\theta}_{2}\right)\left(\frac{1}{d_{21}^{3}}-\frac{1}{d_{22}^{3}}\right).
	\end{split}
\end{equation}
Or equivalently:\\
$\text{\enspace\enspace}\textstyle{0=2r\sin\theta_{1}\left(d_{11}^3d_{12}^3(d_{21}^3+d_{22}^3)-d_{21}^3d_{22}^3(d_{11}^3+d_{12}^3)\right)}\\ \text{\enspace\enspace\enspace\enspace}+\ell_{2}\sin\left({\theta}_{1}{-\theta}_{2}\right)\left(d_{11}^3d_{12}^3(d_{21}^3-d_{22}^3)-d_{21}^3d_{22}^3(d_{11}^3-d_{12}^3)\right)$, and\\
$\text{\enspace\enspace}\textstyle{0=2r\sin\theta_{2}\left(d_{11}^3d_{12}^3(d_{22}^3-d_{21}^3)-d_{21}^3d_{22}^3(d_{11}^3-d_{12}^3)\right)}\\ \text{\enspace\enspace\enspace\enspace}+\ell_{1}\sin\left({\theta}_{1}{-\theta}_{2}\right)\left(d_{11}^3d_{12}^3(d_{21}^3-d_{22}^3)-d_{21}^3d_{22}^3(d_{11}^3-d_{12}^3)\right)$, where 
\begin{align}\label{eq:EqMass_Distances}
	{\scriptstyle   d_{uv}=\sqrt{{r}^{2}-\left(-1\right)^{u}{\ell}_{1}{r}\cos{\theta}_{1}+\left(-1\right)^{v}{\ell_{2}r}\cos{\theta}_{2}-\left(-1\right)^{u+v}\frac{1}{2}{\ell}_{1}{\ell}_{2}\cos\left(\theta_{1}-\theta_{2}\right){+\frac{1}{4}}\left({\ell}_{1}^{2}{+\ell}_{2}^{2}\right)}}.
\end{align}

Our radial requirement becomes:\\
$\text{\enspace\enspace}L^{2}=\frac{\left({r}^{2}+B_{1}+B_{2}\right)^{2}}{8r}\left(\frac{2{r+\ell}_{1}\cos{\theta}_{1}{-\ell}_{2}\cos{\theta}_{2}}{d_{11}^{3}}+\frac{2{r+\ell}_{1}\cos{\theta}_{1}{+\ell}_{2}\cos{\theta}_{2}}{d_{12}^{3}}\right)$\\[-2em]
\begin{equation}\label{eq:EqMass_LS}
\textstyle\text{\enspace\enspace\enspace\enspace}+\frac{\left({r}^{2}+B_{1}+B_{2}\right)^{2}}{8r}\left(\frac{2{r-\ell}_{1}\cos{\theta}_{1}{-\ell}_{2}\cos{\theta}_{2}}{d_{21}^{3}}+\frac{2{r-\ell}_{1}\cos{\theta}_{1}{+\ell}_{2}\cos{\theta}_{2}}{d_{22}^{3}}\right)\text{, where }\scriptstyle B_{i}:=\frac{\ell_{i}^{2}}{4M_{j}}\text{ and }\scriptstyle n\ne i.
\end{equation}
So let us find some solutions for the angular requirements by looking at symmetric configurations and taking advantage of their accompanying simplifications.
We saw earlier that the colinear angular requirements are satisfied for any choice of masses, and the perpendicular angular requirements are met when the vertical masses are equal, so these configurations' requirements are certainly satisfied in the equal mass configuration. In an attempt to locate even more RE, we set the body rotation angles equal to each other, but not equal to zero.

\vspace*{-5mm}\subsubsection{Case: $\theta_1 =\theta_2\ne 0$}
\begin{adjustwidth}{0.5cm}{0cm}
The distances \eqref{eq:EqMass_Distances} become:\\
%$\scriptstyle d_{uv}=\sqrt{{r}^{2}+\left(\left(-1\right)^{v}{\ell_{2}-}\left(-1\right)^{u}{\ell}_{1}\right){r}\cos{\theta}_{1}-\left(-1\right)^{u+v}\frac{1}{2}{\ell}_{1}{\ell}_{2}{+\frac{1}{4}}\left({\ell}_{1}^{2}{+\ell}_{2}^{2}\right)},$
$\text{\enspace\enspace} d_{11}=\sqrt{{r}^{2}+\left({\ell}_{1}-{\ell_{2}}\right){r}\cos{\theta}_{1}+\frac{1}{4}\left(\ell_{1}-\ell_{2}\right)^{2}},\text{\enspace\enspace} d_{12}=\sqrt{{r}^{2}+{r}\cos{\theta}_{1}{+\frac{1}{4}}},$\\
$\text{\enspace\enspace} d_{21}=\sqrt{{r}^{2}-{r}\cos{\theta}_{1}{+\frac{1}{4}}}\text{,\enspace\enspace and }\text{\enspace\enspace}d_{22}=\sqrt{{r}^{2}-\left({\ell}_{1}-{\ell_{2}}\right){r}\cos{\theta}_{1}+\frac{1}{4}\left(\ell_{1}-\ell_{2}\right)^{2}}$.

Our angular requirements \eqref{eq:EqMass_Angular_Requirements} become:\\ $\text{\enspace\enspace}{0=d_{11}^3d_{12}^3(d_{22}^3+d_{21}^3)-d_{21}^3d_{22}^3(d_{11}^3+d_{12}^3)}$, and\\
$\text{\enspace\enspace} {0=d_{11}^3d_{12}^3(d_{22}^3-d_{21}^3)-d_{21}^3d_{22}^3(d_{11}^3-d_{12}^3)}$.

We find two obvious conditions satisfying these equations. The first is when $d_{11}=d_{22}$ and $d_{12}=d_{21}.$ Examining the distance formulas, we see this requires $\theta_{1}=\theta_{2}=\frac{\pi}{2}$, and we recapture the trapezoid configuration.
\begin{figure}[H]
	\centering
	\includegraphics[width=0.15\textwidth]{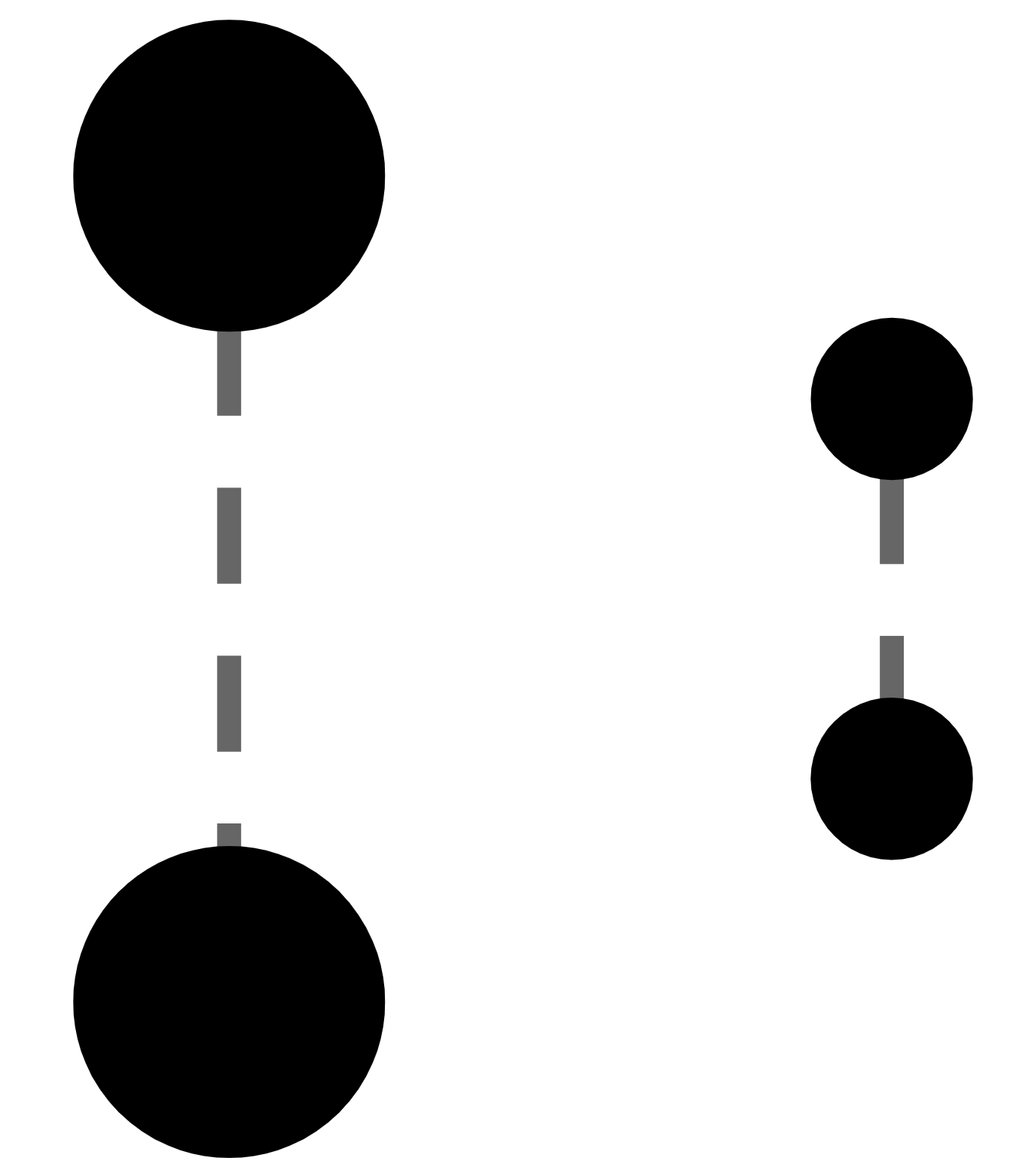}
	%	\captionsetup{labelformat=empty}
	\caption{\\\bfseries Trapezoid Equal Mass Configuration}{\text{ }\\[-1.801em]}
\end{figure}
The second condition is when $d_{22}=d_{12}$ and $d_{11}=d_{21},$ or $d_{22}=d_{21}$ and $d_{11}=d_{12}$. Setting the distances equal (from \eqref{eq:EqMass_Distances}) requires: $2{\ell}_{1}{r}\cos{\theta}_{1}={-\frac{1}{4}+}\frac{1}{4}\left(\ell_{1}-\ell_{2}\right)^{2}$ and $2{\ell}_{1}{r}\cos{\theta}_{1}={+\frac{1}{4}-}\frac{1}{4}\left(\ell_{1}-\ell_{2}\right)^{2}$. Setting the right hand sides of these equations equal gives us $1=(\ell_1-\ell_2)^2$, which requires a trivial dumbbell $\ell_{1}\in\left\{0,1\right\}$. Similarly, requiring $d_{11}=d_{12}$ and $d_{22}=d_{21}$ (left pointing isosceles triangles) also requires $\ell_{2}\in\left\{0,1\right\}$. But these imply the dumbbell/point mass problem which we examined in Chapter 4.
\end{adjustwidth}

So the RE found through these simplifications are symmetric configurations (colinear, perpendicular, trapezoid). But are there any asymmetric RE?
\vspace*{-5mm}\subsubsection{Pitchfork Bifurcations}
In an attempt to find asymmetric solutions not found above, we perform a bifurcation analysis of the symmetric solutions (colinear, perpendicular, trapezoid), using radius as our bifurcation parameter. We hope to locate asymmetrical solutions bifurcating from the symmetric ones. Encouragingly, when we plot the solution curves to the equal mass angular requirements using Mathematica, we see what appears to be pitchfork bifurcations coming from our symmetric RE. Below, we determine quadratic approximations of these bifurcation curves.

Consider a real system $\widetilde{f}\left(\vec{z};r\right) ,\widetilde{g}\left(\vec{z};r\right) =0$ of two equations in two variables $\vec{z}:=\left(z_{1},z_{2}\right) $ and depending upon a bifurcation parameter $r.$ Our goal is to find a local approximation of solutions to the system near a pitchfork bifurcation. To that end, let us first obtain some properties of this type of pitchfork. Assuming our bifurcation occurs at $\vec{z}=\vec{0}$, in order to be considered a pitchfork, $\widetilde{f},\widetilde{g}$ must be odd functions of $\vec{z}.$ In particular, we would then have $\vec{0}$ as a solution for all $r.$ Therefore, it is possible to write the equations as $\widetilde{f}=z_{1}\widetilde{f}_{1}+z_{2}\widetilde{f}_{2}$ and $\widetilde{g}=z_{1}\widetilde{g}_{1}+z_{2}\widetilde{g}_{2}$. From these, we can write the Jacobian $D\left(\widetilde{f},\widetilde{g}\right) \ $(with respect to $\vec{z}$) at the origin as: $D\left(\widetilde{f},\widetilde{g}\right) _{\left(\vec{0};r\right)}=\begin{bmatrix}
	\widetilde{f}_{1} & \widetilde{f}_{2} \\ 
	\widetilde{g}_{1} & \widetilde{g}_{2}\end{bmatrix}|_{\left(\vec{0};r\right)},$ with eigenvalues $\mu_1(r)$, $\mu_2(r)$.

Since we are presupposing a codimension-1 bifurcation, without loss of generality, assume at $r=0$ we have $\mu _{1}\left(0\right) =0,\;\mu_{2}\left(0\right) =:\mu \neq 0.$ Also, to be a pitchfork, we have $\frac{d\mu _{1}}{dr}=:k\neq 0,$ where the zero eigenvalue is crossing the imaginary axis (transversality).

By performing a Jordan normal form decomposition, we can find a linear change of coordinates $P\vec{z}=:\vec{u}.$ From this we define: $\text{\enspace\enspace}\left\{ f(\vec{u};r),g(\vec{u};r)\right\} :=P^{-1}\left\{\widetilde{f}\left(P^{-1}\vec{u};r\right) ,\widetilde{g}\left(P^{-1}\vec{u};r\right) \right\}$.

Written this way, $f,g$ are odd in $\left(u_{1},u_{2}\right) $, and can be expressed as:
\begin{subequations}\label{eq:Pitchfork_Taylor}
	\makeatletter\@fleqntrue\makeatother
	\begin{align}
		\begin{split}
			&\text{$f=u_{1}f_{1}+u_{2}f_{2},\text{\enspace\enspace} \text{ and}$\label{eq:Pitchfork_Taylora}}
		\end{split}\\
		\begin{split}
			&\text{$g=u_{1}g_{1}+u_{2}g_{2}.$\label{eq:Pitchfork_Taylorb}}
		\end{split}
	\end{align}
\end{subequations}

Making this change brings the benefit that at $\left(\vec{u};r\right) =\left(\vec{0};0\right) $ we have\\ $D\left(f,g\right)_{\left(\vec{0};0\right)}=\begin{bmatrix}
	f_{1} & f_{2} \\ 
	g_{1} & g_{2}\end{bmatrix}|_{\left(\vec{0};0\right)}=\begin{bmatrix}
	0 & 0 \\ 
	0 & \mu\end{bmatrix}.$ Also, $f_{1r}=\frac{d\mu _{1}}{dr}=k;$ where $f_{1r}$ denotes the partial derivative of $f_{1}$ with respect to $r$.\\
Note that the functions $f_{i},g_{i}$ are even functions of $\vec{u}$ (since $f,g$ are odd functions of $u_{1},u_{2}$). So at $\vec{u}=\left(0,0\right) $ and for all $r,$ their partial derivatives vanish: $f_{iu_{1}}=f_{iu_{2}}=g_{iu_{1}}=g_{iu_{2}}=0.$ We will use this in a Taylor expansion below.

We will now show that under the assumption that $f_{1u_{1}u_{1}}\left(\vec{0};0\right) =:l\neq 0$, we can find a curve of solutions of the form: $u_{2}=\alpha\left(u_{1}\right) $ and $r=\beta\left(u_{1}\right) $ near $u_{1}=0,$ with: $\alpha(0)=\beta(0)=\alpha^{\prime}(0)=\beta^{\prime}(0)=\alpha ^{\prime \prime }(0)=0,$ and $\beta ^{\prime \prime }(0)\neq 0.$

We wish to continue our bifurcation point of $\eqref{eq:Pitchfork_Taylor}$ into a curve of solutions. So the Implicit Function Theorem (IFT) will be helpful, but first we need to do a change of variable to obtain a nonzero Jacobian determinant. We will take advantage of the nonzero values $g_{2}=\mu $ and $f_{1r}=k$. Observe that we obtain $g_{2}$ from $\eqref{eq:Pitchfork_Taylorb}$ if we take a derivative with respect to $u_{2}$, and (if we can first get rid of the $u_{1}$ coefficient) we can obtain $f_{1r}$ from $\eqref{eq:Pitchfork_Taylora}$ upon taking a derivative with respect to $r$. Therefore, let us make the change of variable:
\begin{align}\label{eq:Pitchfork_COV}
	{\displaystyle u_{2}=u_{1}z.}
\end{align}
Upon substituting \eqref{eq:Pitchfork_COV} into $\eqref{eq:Pitchfork_Taylor}$, we then define (after canceling a factor of $u_{1}$):\\
$\text{\enspace\enspace}F\left(u_{1},z;r\right):=\frac{f}{u_{1}}=f_{1}\left(u_{1},u_{1}z;r\right)+zf_{2}\left(u_{1},u_{1}z;r\right) =0,$

$\text{\enspace\enspace}G\left(u_{1},z;r\right):=\frac{g}{u_{1}}=g_{1}\left(u_{1},u_{1}z;r\right)+zg_{2}\left(u_{1},u_{1}z;r\right) =0.$

The Jacobian $\frac{\partial (F,G)}{(z,r)}$ evaluated at our bifurcation point is: 
$\begin{bmatrix}
	f_{2} & f_{1r} \\ 
	g_{2} & g_{1r}\end{bmatrix}|_{\left(0;0\right)}=\begin{bmatrix}
	0 & k \\ 
	\mu & g_{1r}\end{bmatrix}.$
And since $\mu ,k$ are nonzero, the determinant is nonzero, and IFT
guarantees solutions of the form:
\begin{align}\label{eq:Pitchfork_IFT_Sols}
	{\displaystyle  z=\gamma (u_{1}),\;r=\beta (u_{1}),}
\end{align}
with $\gamma (0)=\beta (0)=0$.
Recapturing $u_{2}$ from $z,$ we find from \eqref{eq:Pitchfork_COV},\eqref{eq:Pitchfork_IFT_Sols} that
\begin{align}\label{eq:Pitchfork_Pre_Sols1}
	{\displaystyle  u_{2}=u_{1}\gamma(u_{1})=:\alpha (u_{1}),}
\end{align}
with $\alpha (0)=\alpha ^{\prime }(0)=0.$ We then discover our curve by calculating the derivatives of $\alpha(u_{1}),$ $\beta (u_{1})$ using implicit differentiation on the equations:

$\text{\enspace\enspace}F\left(u_{1},u_{1}\gamma (u_{1}),\beta (u_{1})\right) =f_{1}\left(u_{1},u_{1}\gamma (u_{1}),\beta (u_{1})\right)+\gamma (u_{1})f_{2}\left(u_{1},u_{1}\gamma (u_{1}),\beta (u_{1})\right) =0,$

$\text{\enspace\enspace}G\left(u_{1},u_{1}\gamma (u_{1}),\beta (u_{1})\right) =g_{1}\left(u_{1},u_{1}\gamma (u_{1}),\beta (u_{1})\right)+\gamma (u_{1})g_{2}\left(u_{1},u_{1}\gamma (u_{1}),\beta (u_{1})\right) =0.$

The first derivatives are:
\begin{flalign}\label{eq:2DB_EQMass_Bifurcation_First_Derivative}
	\begin{aligned}
		& \text{\enspace\enspace}f_{1u_{1}}+\left(u_{1}\gamma ^{\prime }+\gamma\right)f_{1u_{2}}+\beta^{\prime}f_{1r}+\gamma^{\prime}f_{2}+\gamma\left(f_{2u_{1}}+\left(u_{1}\gamma ^{\prime }+\gamma\right)f_{2u_{2}}+f_{2r}\right) =0,\\
		& \text{\enspace\enspace}g_{1u_{1}}+\left(u_{1}\gamma ^{\prime }+\gamma\right)g_{1u_{2}}+\beta ^{\prime }g_{1r}+\gamma ^{\prime }g_{2}+\gamma\left(g_{2u_{1}}+\left(u_{1}\gamma ^{\prime }+\gamma\right)g_{2u_{1}}+g_{2r}\right) =0,
	\end{aligned}
\end{flalign}
where the arguments have been suppressed. At $\left(\vec{u};r\right) =\left(\vec{0};0\right) $ we get: $k\beta ^{\prime }(0)=g_{1r}\beta ^{\prime }(0)+\mu \gamma^{\prime }(0)=0.$ So $\beta ^{\prime }(0)=\gamma ^{\prime }(0)=0.$ Since $\gamma ^{\prime }(0)=0,$ then from \eqref{eq:Pitchfork_Pre_Sols1} we find $\alpha ^{\prime \prime}(0)=\left(\gamma +u_{1}\gamma ^{\prime }\right)^{\prime }|_{0}=\left(\gamma ^{\prime }+\gamma ^{\prime }+u_{1}\gamma ^{\prime \prime }\right)|_{0}=0$. Differentiating the first equation of \eqref{eq:2DB_EQMass_Bifurcation_First_Derivative}, and ignoring terms which we determined above vanish at $\left(\vec{0};0\right) $ gives:
$f_{1u_{1}u_{1}}+f_{1r}\beta ^{\prime \prime }(0)=l+k\beta^{\prime \prime }(0)=0\;\;\Rightarrow \;\;\beta ^{\prime \prime }(0)=-\frac{l}{k}.$

So if $f_{1u_{1}u_{1}}\left(\vec{0};0\right) =l\neq 0$, we can find our curve of solutions up to second order as:
\begin{subequations}\label{eq:Pitchfork_Sols}
	\makeatletter\@fleqntrue\makeatother
	\begin{align}
		\begin{split}
			&\text{$u_{2}=0+\mathcal{O}\left(u_{1}^{3}\right)$\text{, and }\label{eq:Pitchfork_SolsA}}
		\end{split}\\
		\begin{split}
			&\text{$\;r=\frac{\;\beta^{\prime\prime}(0)}{2!}u_{1}^{2}+\mathcal{O}\left(u_{1}^{3}\right)=-\frac{l}{2k}u_{1}^{2}+\mathcal{O}\left(u_{1}^{3}\right).$\label{eq:Pitchfork_SolsB}}
		\end{split}
	\end{align}
\end{subequations}
\\
%%%%%%%%%%%%%%%%%%%%%%%%%%%%%%%%%%%%
Next, we show that if $P=P^{-1}$, then at the bifurcation point, $l=f_{1u_{1}u_{1}}$ in \eqref{eq:Pitchfork_SolsB}
is equal to $\frac{1}{3}f_{u_{1}u_{1}u_{1}}$. We point this out since, in practice,
it is easier to calculate $f_{u_{1}u_{1}u_{1}}$ than to calculate $f_{1u_{1}u_{1}}$ directly. Recall: $\left\{ f(\vec{u};r),g(\vec{u};r)\right\}
:=P^{-1}\left\{\widetilde{f}\left(P^{-1}\vec{u};r\right) ,\widetilde{g}\left(
P^{-1}\vec{u};r\right) \right\}$, where $\widetilde{f}=z_{1}\widetilde{f}_{1}+z_{2}\widetilde{f}_{2}$ and $\widetilde{g}=z_{1}\widetilde{g}_{1}+z_{2}\widetilde{g}_{2}.$ Then assuming: $P=P^{-1}=\,\begin{bmatrix}P_{11} & P_{12} \\ 
	P_{12} & -P_{11}\end{bmatrix}$, we have:\\
$\text{\enspace\enspace} f=u_{1}\left(P_{11}\left(P_{11}\widetilde{f}_{1}+P_{12}\widetilde{f}_{2}\right)+P_{12}\left(P_{11}\widetilde{g}_{1}+P_{12}\widetilde{g}_{2}\right)\right)$\\
$\text{\enspace\enspace\enspace\enspace\enspace\enspace}+u_{2}\left(P_{11}\left(P_{12}\widetilde{f}_{1}-P_{11}\widetilde{f}_{2}\right)+P_{12}\left(P_{12}\widetilde{g}_{1}-P_{11}\widetilde{g}_{2}\right)\right)$\\
$\text{\enspace\enspace\enspace\enspace}=:u_{1}f_{1}+u_{2}f_{2}.$

Note that $\partial_{u_{1}}^{3}f\,\,|_{u_{2}=0}=\partial
_{u_{1}}^{2}\left(f_{1}+u_{1}f_{1u_{1}}\right) =\partial_{u_{1}}\left(2f_{1u_{1}}+u_{1}f_{1u_{1}u_{1}}\right)
=3f_{1u_{1}u_{1}}+u_{1}f_{1u_{1}u_{1}u_{1}}.$ So, we do indeed find that $f_{u_{1}u_{1}u_{1}}\,\,|_{\vec{u}=0}=3f_{1u_{1}u_{1}}.$
%[Documented at Implicit\_Differentiation\_Bifurcation\_Analysis\_no\_Gamma,
%3\_28\_2021]
Now let us apply these bifurcation results to our equal mass configuration.

\paragraph{Bifurcation Analysis for Equal Mass Configuration}
\label{2DB_EqualMass_Bifurcation_Analysis}When looking for RE numerically, depending on $\ell_{1},$ there are several 1D families of RE curving through the equal mass configuration space $(r,\theta_{1},\theta_{2})$. In particular, there are symmetric families which consist of the colinear configuration $\mathcal{R}_{C}$ where $\left(\theta_{1},\theta_{2}\right) =\left(0,0\right) $, the perpendicular configurations $\mathcal{R}_{P_{1}},\mathcal{R}_{P_{2}}$ where $\left(\theta_{1},\theta_{2}\right) \in\left\{\left(\frac{\pi}{2},0\right) ,\left(0,\frac{\pi}{2}\right) \right\}$, and the trapezoid configuration $\mathcal{R}_{T}$ where $\left(\theta_{1},\theta_{2}\right) =\left(\frac{\pi}{2},\frac{\pi}{2}\right)$. Their existence is independent of $r,\ell_{1}.$ We also find asymmetric solutions which bifurcate from these symmetric ones.

After setting $r$, the configuration space consists of a $(\theta_{1},\theta_{2})$ torus. If we graph the angular RE requirements ($V_{\theta_{i}}=0$), the resulting RE appear at their intersections. In Figure \ref{fig:EqMassRSlice_l10p75_r1000p} you see an example of such a graph ($r=1000,\;\ell_{1}=\frac{3}{4}$). Observe that $\mathcal{R}_{T}$ is at the center of the graph, and $\mathcal{R}_{C},\mathcal{R}_{P_{1}},\mathcal{R}_{P_{2}}$ are on the boundary. Note that the edges are identified (this is a torus), so in total we have only four RE visualized here (not nine).

\textbf{Bifurcation curves for $\ell_1\neq\frac{1}{2}$}. We now describe bifurcations shown in Figures \ref{fig:EqMassRSlice_l10p75_r0p252_r0p362},...,\ref{fig:EqMassRSlice_l10p75_r0p388_r0p498}. We include hatched regions in the figures to be used later when examining stability. 
\begin{itemize}[parsep=3pt]
	\item As $r$ increases from $0$, we see $\mathcal{R}_{C}$ bifurcating (Figure \ref{fig:EqMassRSlice_l10p75_r0p252_r0p362}a) at radius $r_{1}$ (whose value depends upon $\ell_{1}$) into three RE, including two new RE branches which we label $\mathcal{B}_{CC^{\pm}}$ (a bifurcation from colinear, later merging back with colinear) where $\mathcal{B}_{CC^{+}}$ has $\theta_1$ increasing, and $\mathcal{B}_{CC^{-}}$ has $\theta_1$ decreasing. Then at some $r_{8}$ (Figure \ref{fig:EqMassRSlice_l1_NotHalf_r8}), $\mathcal{B}_{CC^{\pm }}$ merge back with $\mathcal{R}_{C}$.
	\item At some $r_{2}$ (Figure \ref{fig:EqMassRSlice_l1_NotHalf_r2}), $\mathcal{R}_{T}$ bifurcates into three RE, including two new RE branches which we label $\mathcal{B}_{TP^{\pm}}$ (bifurcation going from trapezoid to perpendicular). Then at some $r_{7}$ (Figure \ref{fig:EqMassRSlice_l1_NotHalf_r7}), $\mathcal{B}_{TP^{\pm }}$ merge with $\mathcal{R}_{P_{1}}$.
	\item At some $r_{3}$ (Figure \ref{fig:EqMassRSlice_l1_NotHalf_r3}), $\mathcal{R}_{P_{2}}$ bifurcates into three RE, including two new RE branches which we label $\mathcal{B}_{PC^{\pm}}$ (going from perpendicular to colinear). Then at some $r_{5}$  (Figure \ref{fig:EqMassRSlice_l1_NotHalf_r5}), $\mathcal{B}_{PC^{\pm}}$ merge with $\mathcal{R}_{C}$.
	\item At some $r_{4}$ (Figure \ref{fig:EqMassRSlice_l1_NotHalf_r4}), $\mathcal{R}_{C}$ bifurcates again into two RE which we label $\mathcal{B}_{CP^{\pm}}$ (going from colinear to perpendicular). Then at some $r_{6}$ (Figure \ref{fig:EqMassRSlice_l1_NotHalf_r6}), $\mathcal{B}_{CP^{\pm }}$ merge with $\mathcal{R}_{P_{2}}$.
\end{itemize}
\begin{figure}[H]
	\captionsetup[subfigure]{justification=centering}
	\centering
	\subfloat[$r=0.252$.\\Subsequent to $\mathcal{B}_{CC}(r_{1})$.\\Shows RE bifurcating from $\theta_{1,2}=0$.]
	{\includegraphics[width=0.35\textwidth]{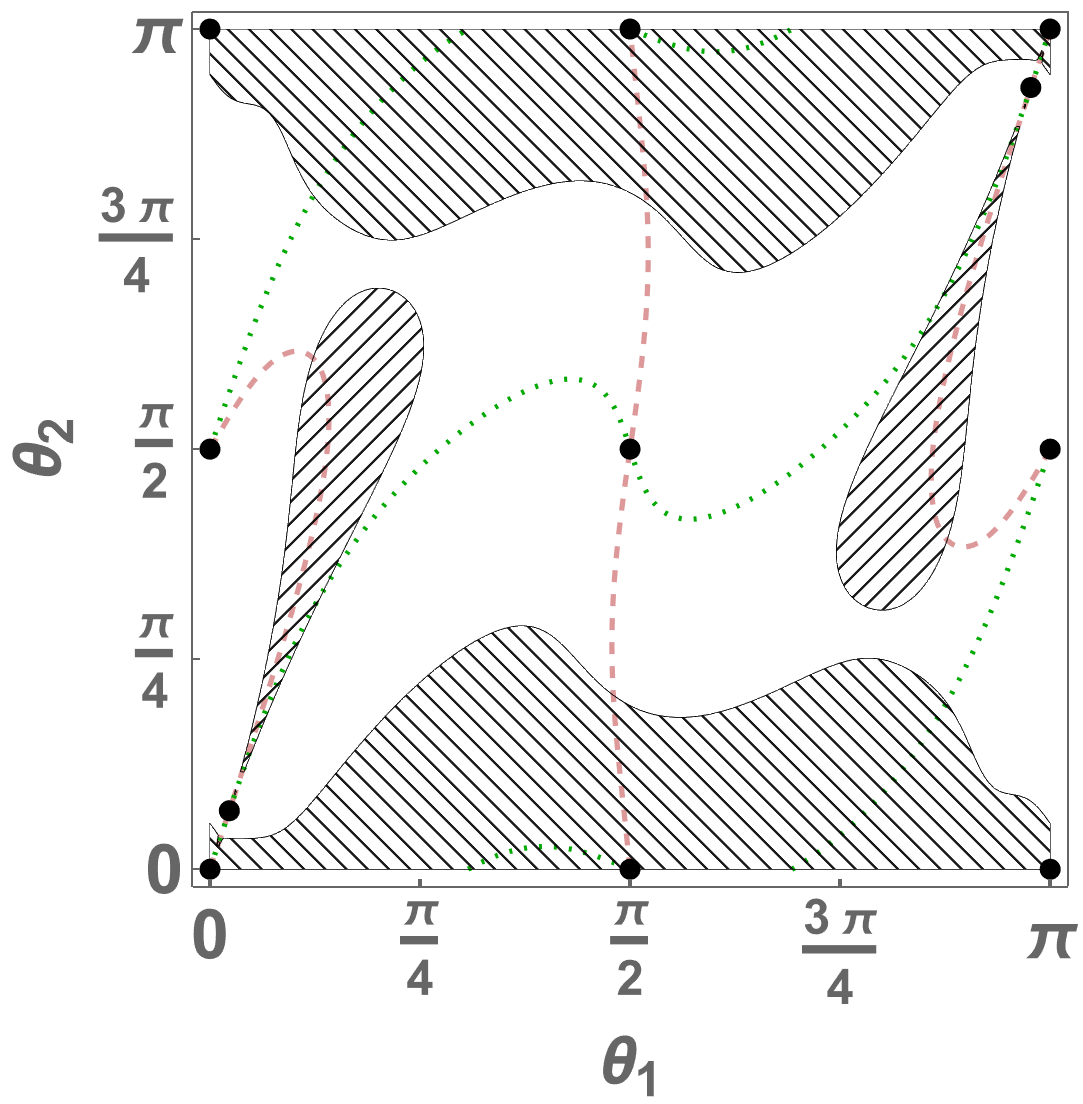}\label{fig:EqMassRSlice_l1_NotHalf_r1}}
	\hspace{0.7cm}
	\subfloat[$r=0.362$,\\subsequent to $\mathcal{B}_{TP}(r_{2})$.\\Shows RE bifurcating from $\theta_{1,2}=\frac{\pi}{2}$.]
	{\includegraphics[width=0.35\textwidth]{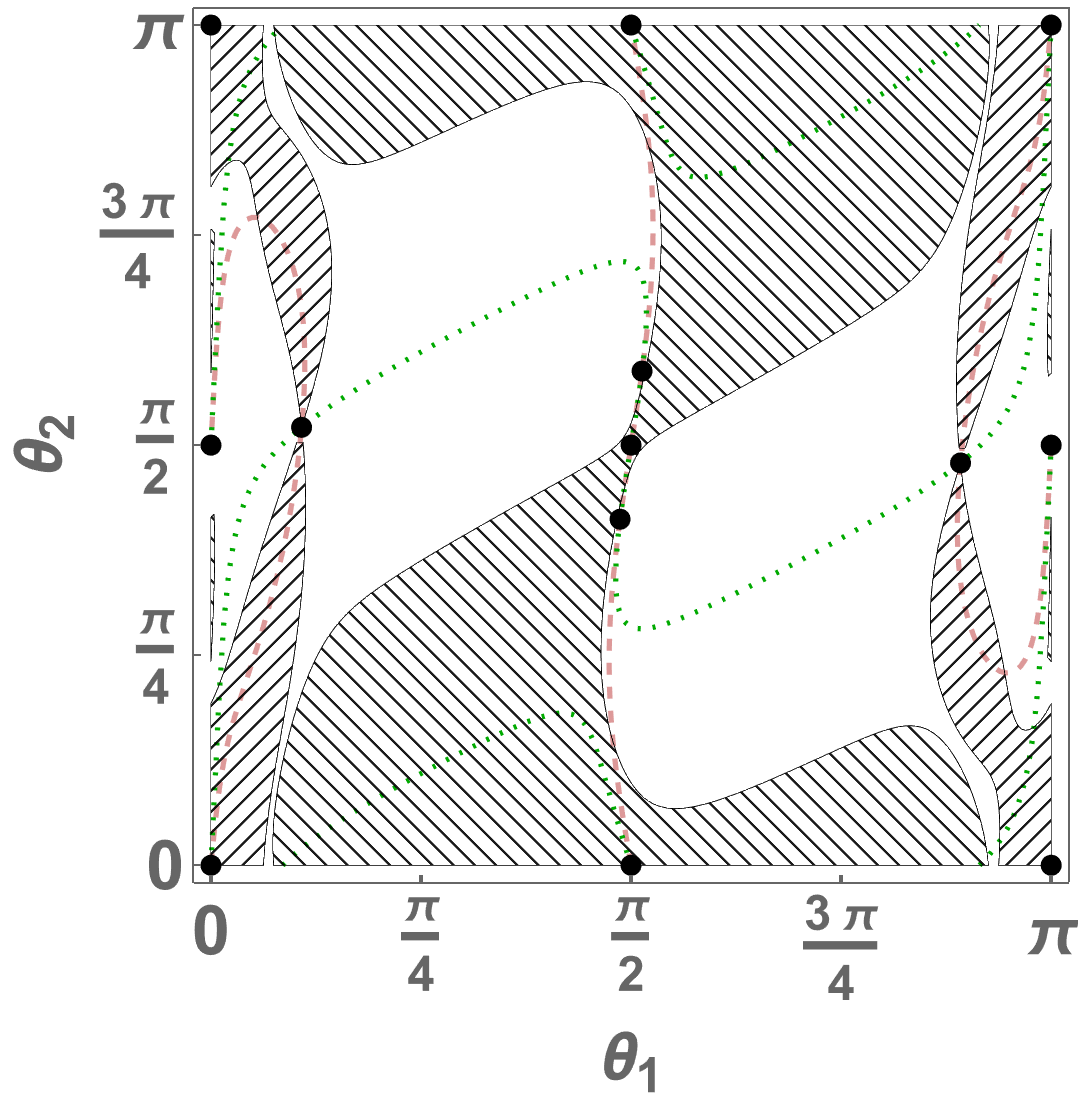}\label{fig:EqMassRSlice_l1_NotHalf_r2}}\raisebox{0.9\height}{\includegraphics[width=0.20\textwidth]{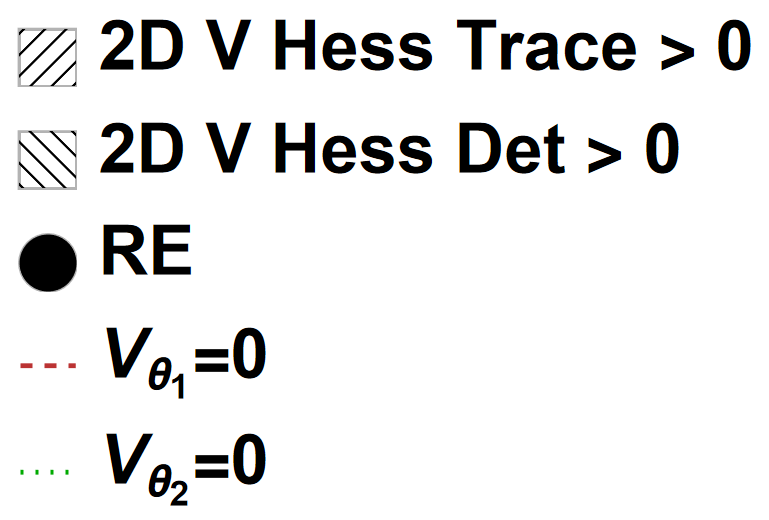}}
		\\[-1.5ex]
	\caption{\\\bfseries Equal Mass RE/Trace/Det.\\with $\ell_1=\frac{3}{4}$ for $r_1,r_2$}{\text{ }\\[-1.801em]}
	\label{fig:EqMassRSlice_l10p75_r0p252_r0p362}
\end{figure}
\begin{figure}[H]
	\captionsetup[subfigure]{justification=centering}
	\centering
	\subfloat[$r=0.364$. \\Subsequent to $\mathcal{B}_{PC}(r_{3})$.\\Shows RE bifurcating from $(\theta_1,\theta_2)=(0,\frac{\pi}{2})$.]
	{\includegraphics[width=0.35\textwidth]{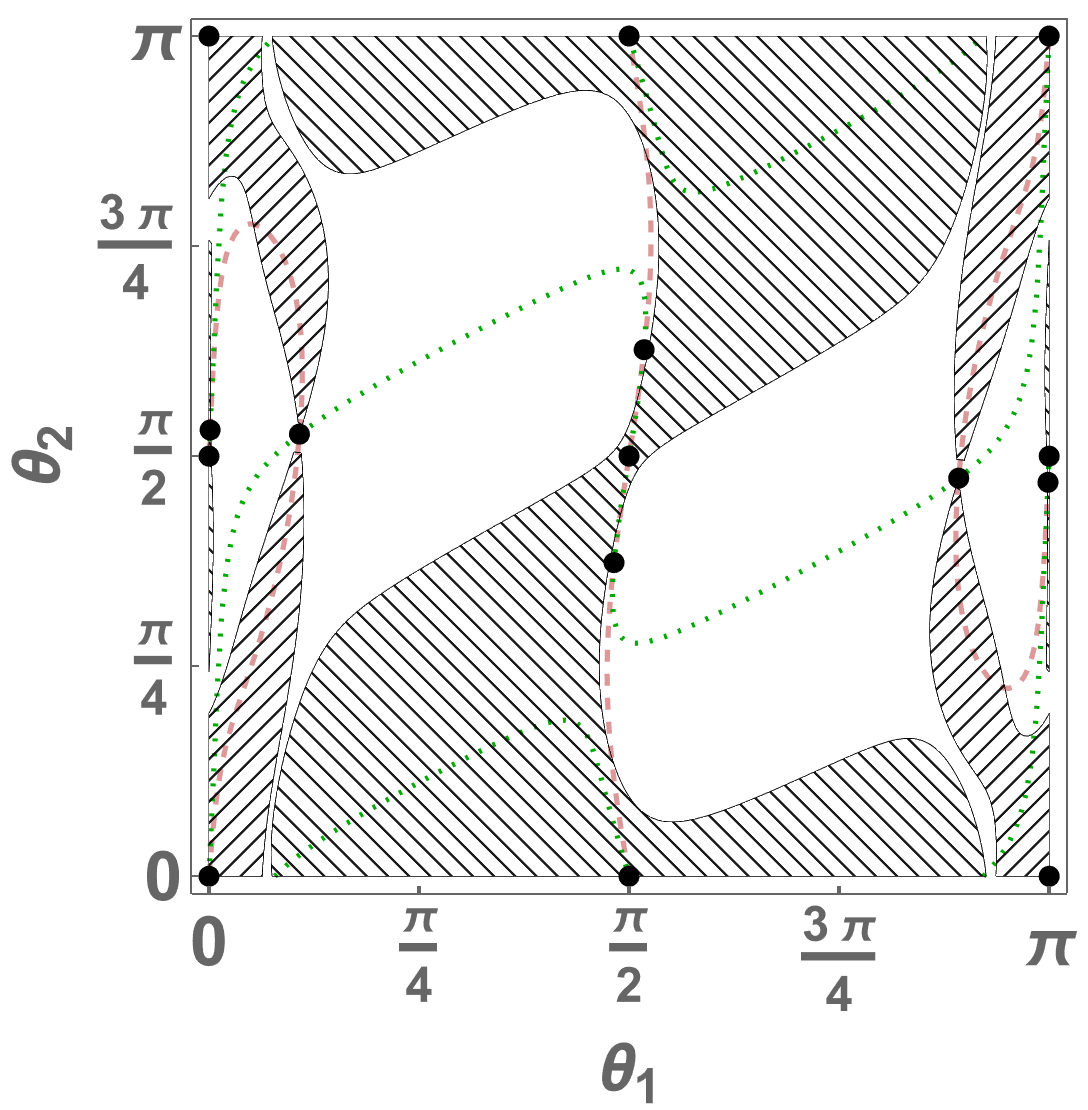}\label{fig:EqMassRSlice_l1_NotHalf_r3}}
	\hspace{0.7cm}
	\subfloat[$r=0.37$. \\Subsequent to $\mathcal{B}_{CP}(r_{4})$.\\Shows RE bifurcating from $\theta_{1,2}=0$.]
	{\includegraphics[width=0.35\textwidth]{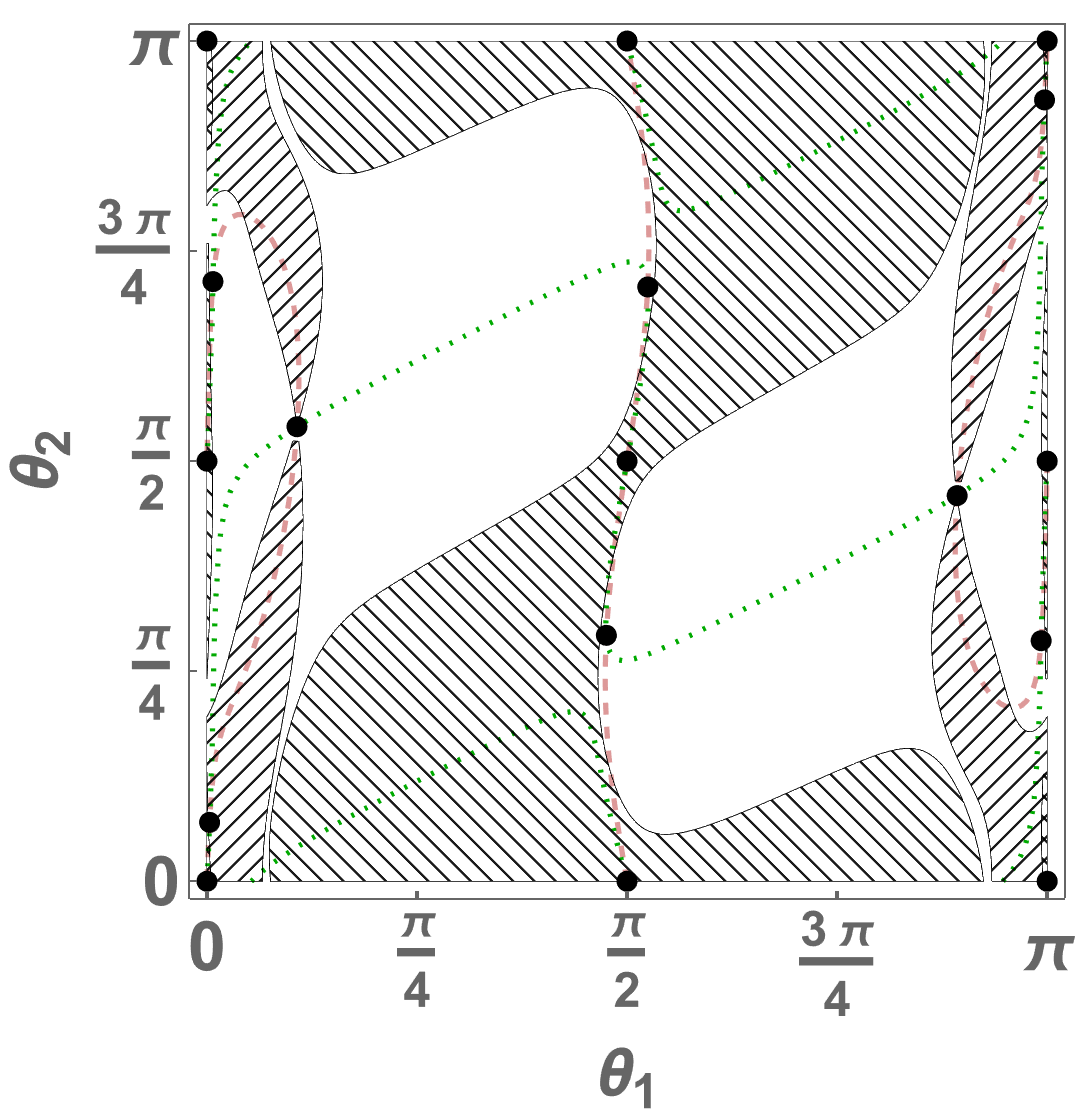}\label{fig:EqMassRSlice_l1_NotHalf_r4}}\raisebox{0.9\height}{\includegraphics[width=0.20\textwidth]{images/EqMass_rSlice/Legend.png}}
	\\[-1.5ex]
	\caption{\\\bfseries Equal Mass RE/Trace/Det.\\with $\ell_1=\frac{3}{4}$ for $r_3,r_4$}{\text{ }\\[-1.801em]}
	\label{fig:EqMassRSlice_l10p75_r0p364_r0p37}
\end{figure}
\begin{figure}[H]
	\captionsetup[subfigure]{justification=centering}
	\centering
	\subfloat[$r=0.38$. \\Prior to $\mathcal{B}_{PC}(r_{5})$.\\Shows RE merging with $\theta_{1,2}=0$]
	{\includegraphics[width=0.35\textwidth]{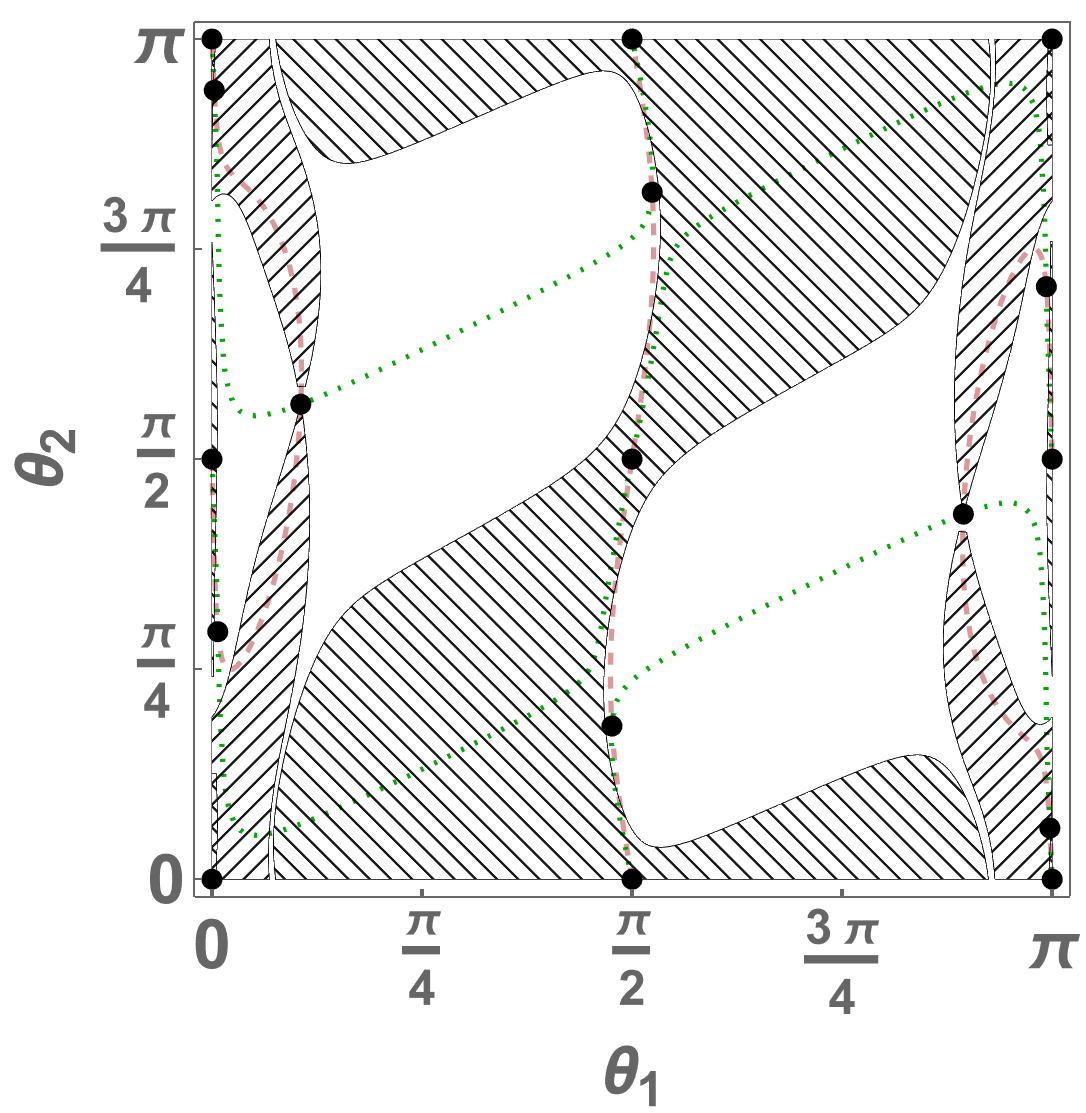}\label{fig:EqMassRSlice_l1_NotHalf_r5}}
	\hspace{0.7cm}
	\subfloat[$r=0.384$. \\Prior to $\mathcal{B}_{CP}(r_{6})$.\\Shows RE merging with $(\theta_1,\theta_2)=(0,\frac{\pi}{2})$]
	{\includegraphics[width=0.35\textwidth]{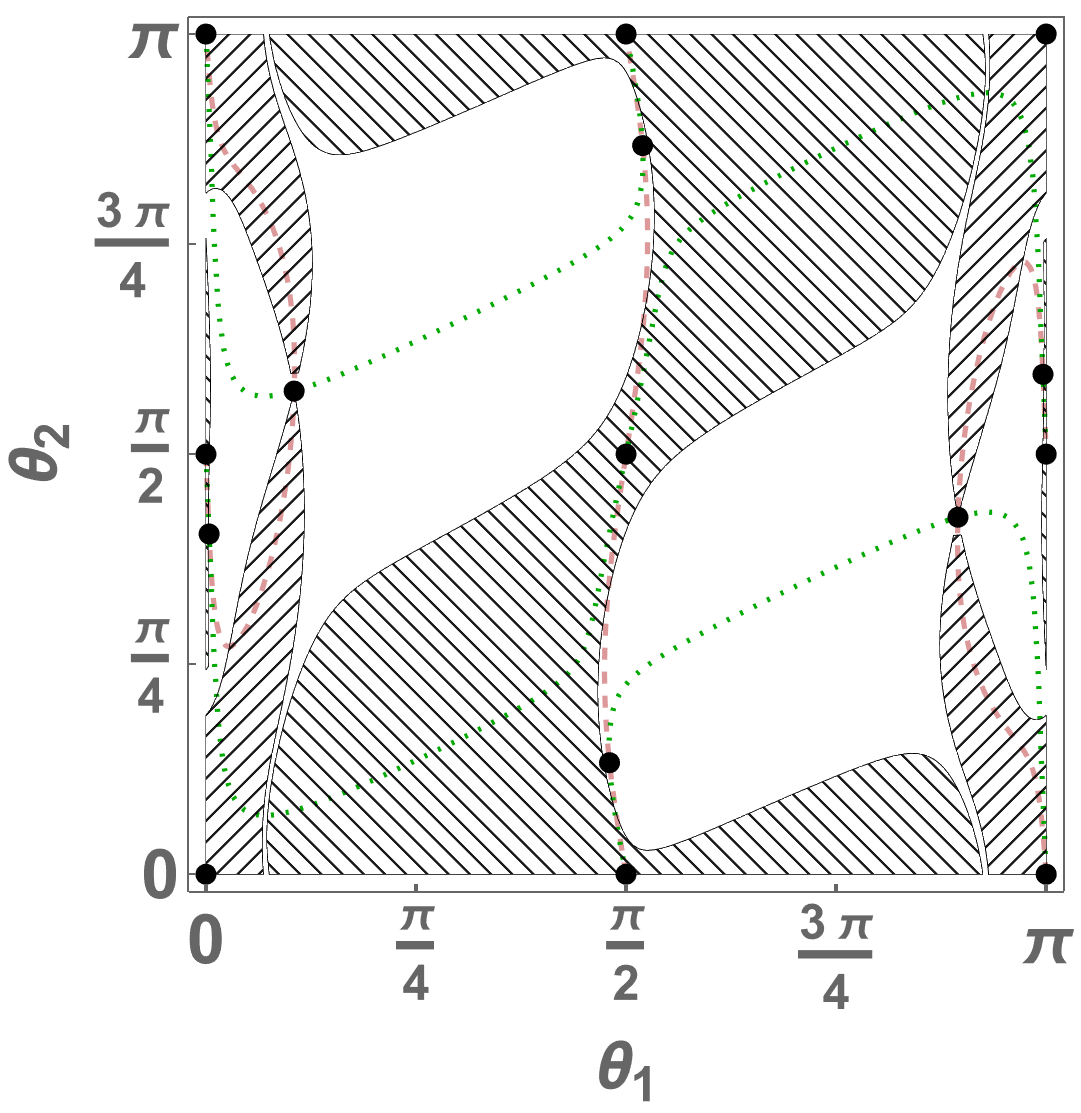}\label{fig:EqMassRSlice_l1_NotHalf_r6}}\raisebox{0.9\height}{\includegraphics[width=0.20\textwidth]{images/EqMass_rSlice/Legend.png}}
	\\[-1.5ex]
	\caption{\\\bfseries Equal Mass RE/Trace/Det.\\with $\ell_1=\frac{3}{4}$ for $r_5,r_6$}{\text{ }\\[-1.801em]}
	\label{fig:EqMassRSlice_l10p75_r0p38_r0p384}
\end{figure}
\begin{figure}[H]
	\captionsetup[subfigure]{justification=centering}
	\centering
	\subfloat[$r=0.388$. \\Prior to $\mathcal{B}_{TP}(r_{7})$.\\Shows RE merging with $(\theta_1,\theta_2)=(\frac{\pi}{2},0)$.\label{fig:EqMassRSlice_l1_NotHalf_r7}]
	{\includegraphics[width=0.35\textwidth]{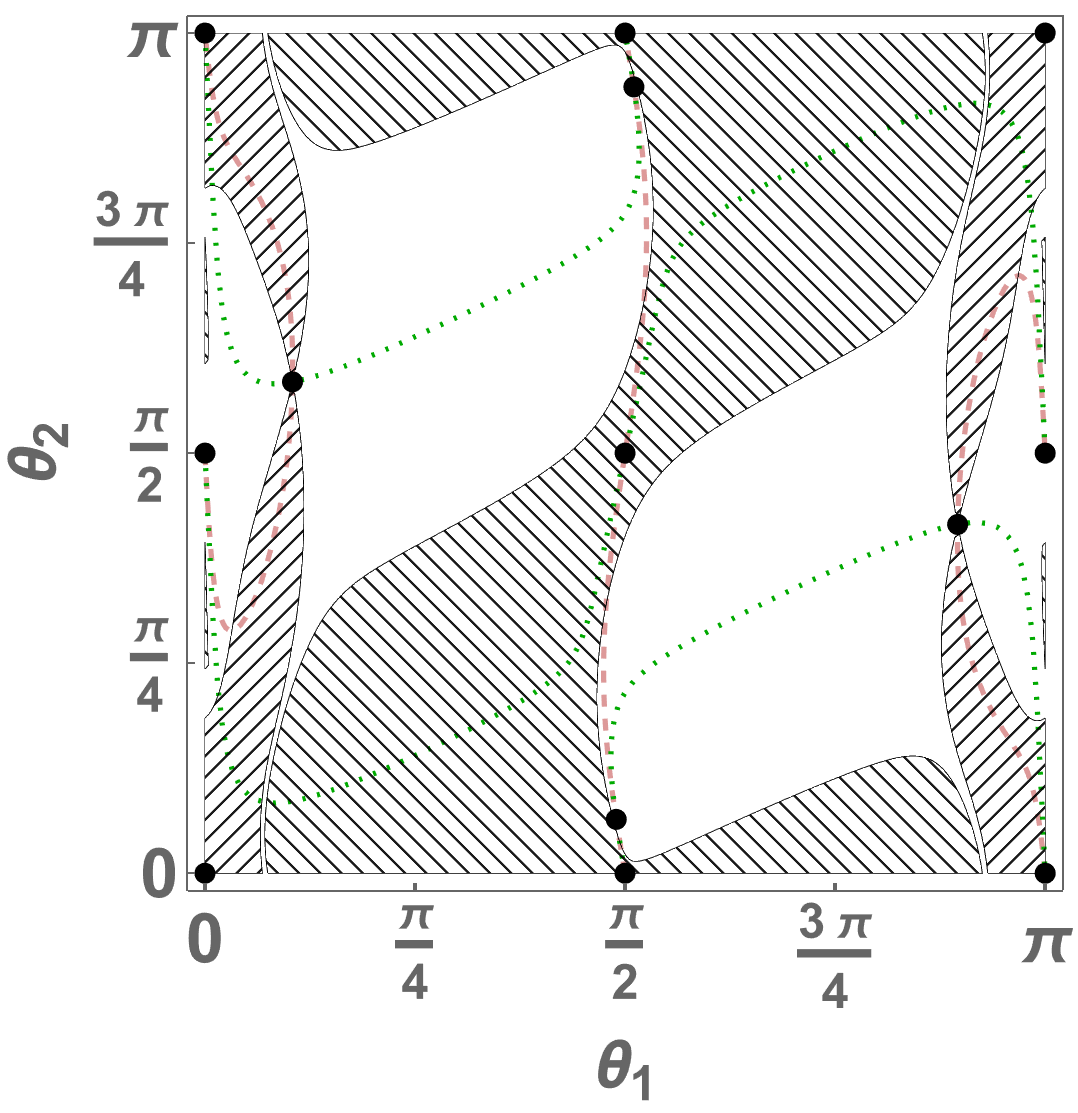}}
	\hspace{0.7cm}
	\subfloat[$r=0.498$. \\Prior to $\mathcal{B}_{CC}(r_{8})$.\\Shows RE merging with $\theta_{1,2}=0$.]
	{\includegraphics[width=0.35\textwidth]{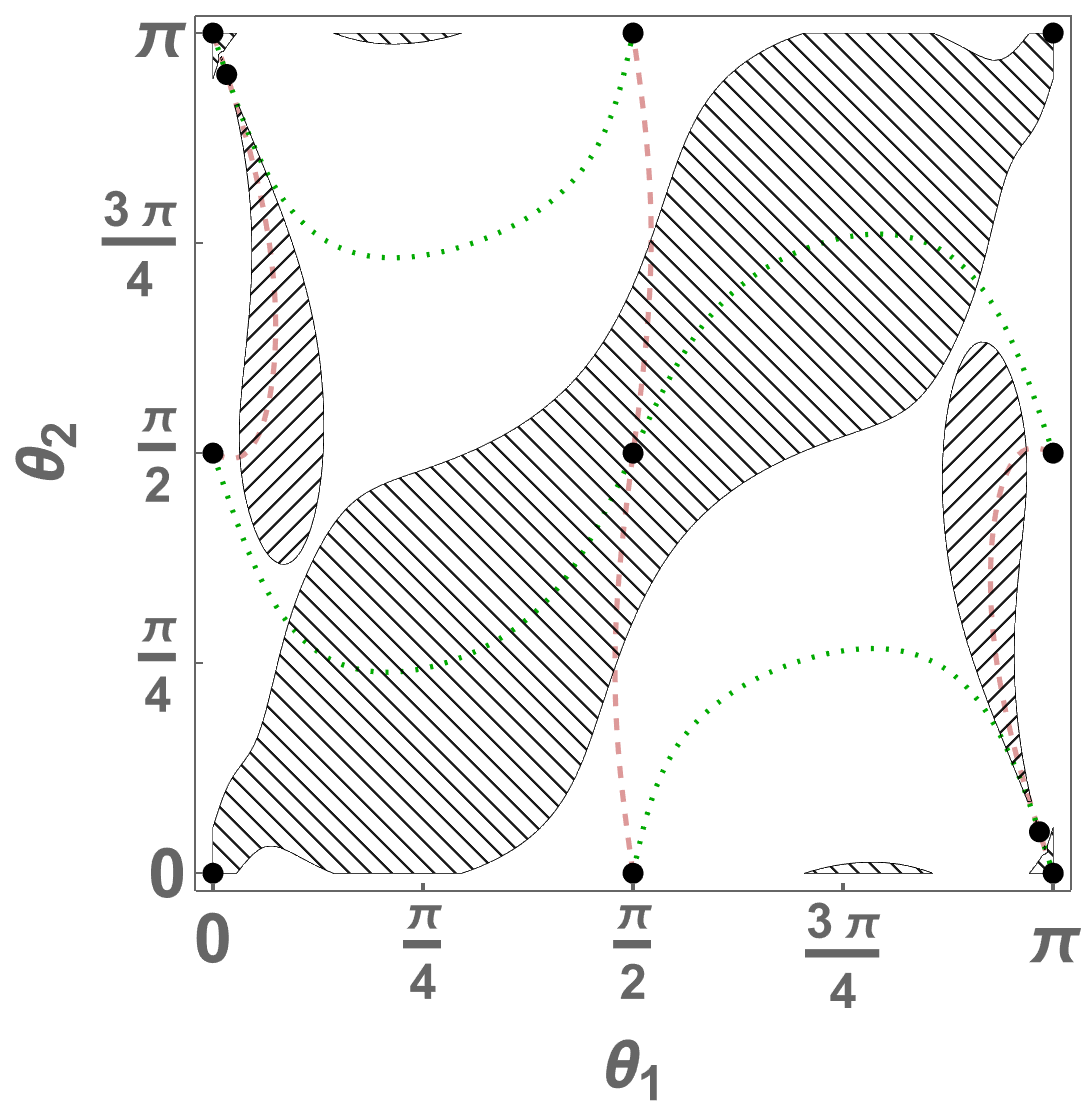}\label{fig:EqMassRSlice_l1_NotHalf_r8}}\raisebox{0.9\height}{\includegraphics[width=0.20\textwidth]{images/EqMass_rSlice/Legend.png}}
	\\[-1.5ex]
	\caption{\\\bfseries Equal Mass RE/Trace/Det.\\with $\ell_1=\frac{3}{4}$ for $r_7,r_8$}{\text{ }\\[-1.801em]}
	\label{fig:EqMassRSlice_l10p75_r0p388_r0p498}
\end{figure}
\begin{figure}[H]
	\centering
	\includegraphics[width=0.45\textwidth]{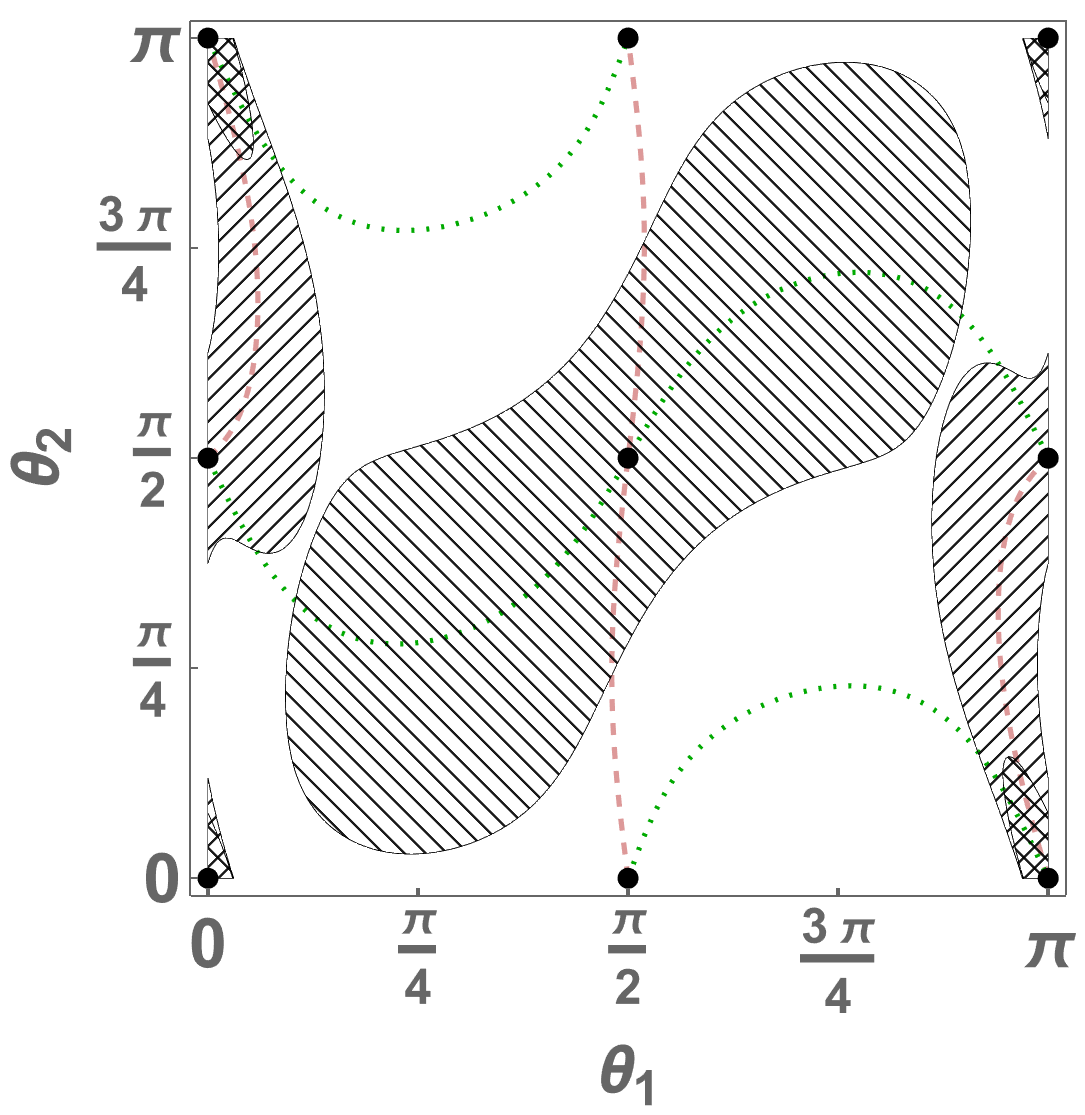}\raisebox{1.6\height}{\includegraphics[width=0.20\textwidth]{images/EqMass_rSlice/Legend.png}}
		\\[-1.5ex]
	\caption{\\\bfseries Equal Mass RE/Trace/Det.\\with $\ell_1=\frac{3}{4}$ for $r=0.552$}{\vspace{1em}Subsequent to $\mathcal{B}_{CC}(r_{8})$, only symmetric RE remain with stable colinear RE.\\[-0.801em]}
	\label{fig:EqMassRSlice_l10p75_r0p552}
\end{figure}
To get an idea of the long-term behavior, we graph $r=1000$: 
\begin{figure}[H]
	\centering
	\captionsetup{justification=centering}
	\includegraphics[width=0.45\textwidth]{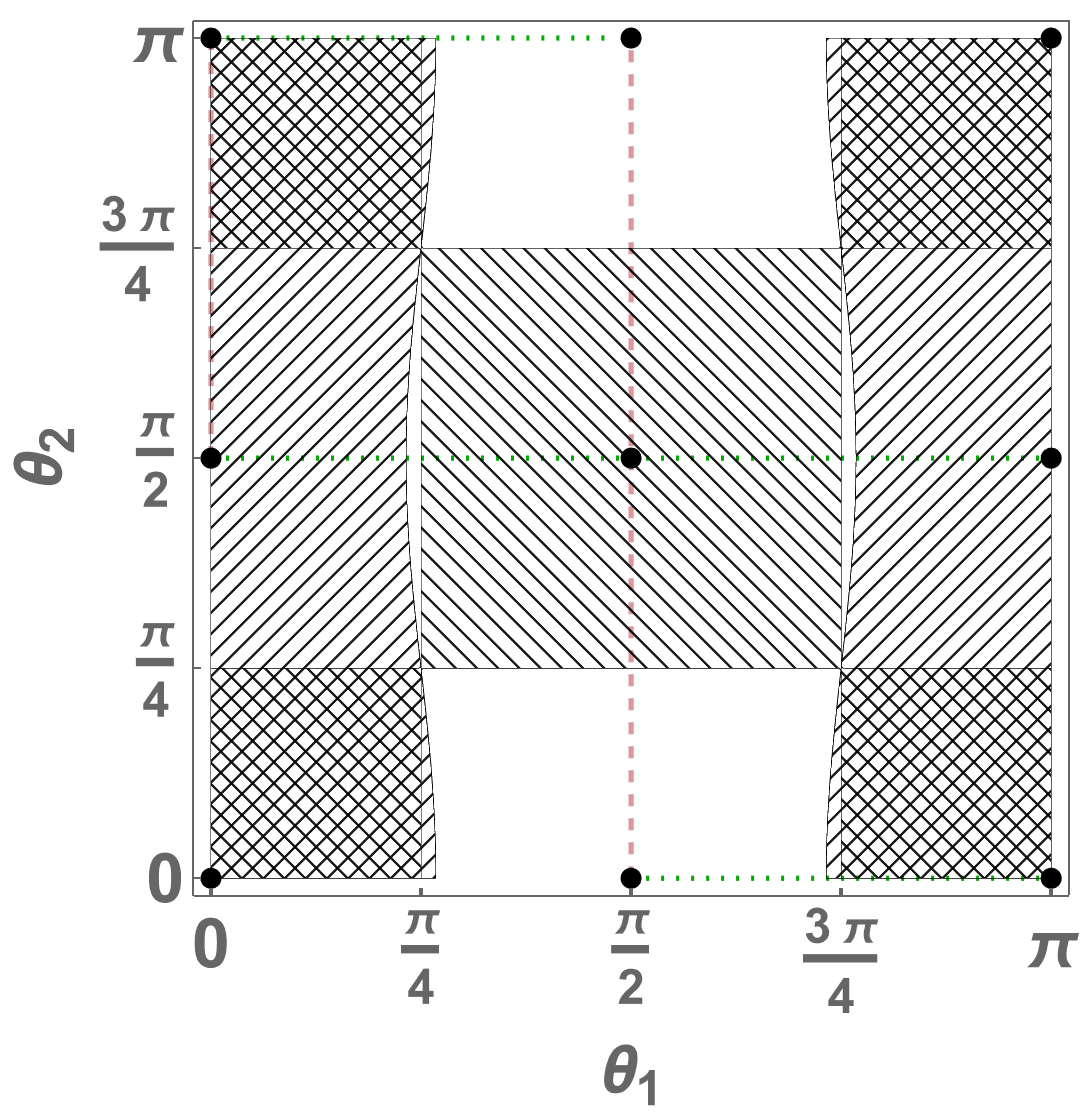}\raisebox{1.8\height}{\includegraphics[width=0.20\textwidth]{images/EqMass_rSlice/Legend.png}}
		\\[-1.5ex]
	\caption{\\\bfseries Equal Mass RE/Trace/Det.\\with $\ell_1=\frac{3}{4}$ for $r=1000$}{\vspace{0.8em}\footnotesize Long-term Behavior, stable colinear RE\\[-0.801em]}
	\label{fig:EqMassRSlice_l10p75_r1000p}
\end{figure}
%Go to the following URL to see an animation of graphs similar to the one above as $r$ varies from $0$ to $0.62$ for various values of $\ell_{1}$ : 
%\\ {\tiny http://www-users.math.umn.edu/\symbol{126}moreyjc/research/2DB/img/EqMass/rSlices/Several\_l1\_XRay.MP4}
\textbf{Bifurcation curves for $\ell_1=\frac{1}{2}$}. We will now describe the bifurcations shown in Figures \ref{fig:EQMassRSlice_l1_0p5_r_0},...,\ref{fig:EQMassRSlice_l10p5_r0p528} below. For this case when the lengths of the bodies are equal, observe that $r=0$ and $\theta_{1}=\theta_{2}\pm\frac{1}{2}n\pi$ with $n\in\mathbb{Z}$ trivially satisfy the angular requirements $V_{\theta_{i}}=0$ \eqref{eq:EqMass_Angular_Requirements}. Therefore, they represent a family $\mathcal{R}_0$ of RE at $r=0$. For small $r$, the following RE curves can be seen bifurcating from points within $\mathcal{R}_0$ (Figure \ref{fig:EqMassRSlice_l1_Half_r0p01}), and later merge with the symmetric families at some radii. 
\begin{itemize}[parsep=3pt]
	\item Two curves which we denote $\mathcal{B}_{C^{\pm}}$, bifurcate from $\left(\theta_1,\theta_2\right)\in\left\{\left(\frac{\pi}{4},\frac{3\pi }{4}\right),\left(\frac{3\pi }{4},\frac{\pi}{4}\right)\right\}$, and later ($r_1$) merge with $\mathcal{R}_{C}$ (Figure \ref{fig:EqMassRSlice_l1_Half_r1}).
\item Two curves which we denote $\mathcal{B}_{LP^{\pm}}$, bifurcate from $\theta_1,\theta_2=\cot ^{-1}\sqrt{2}$ (at collision), and later ($r_2$) merge with $\mathcal{R}_{P_{1/2}}$ (Figure \ref{fig:EqMassRSlice_l1_Half_r2}). And similarly two curves which we label $\mathcal{B}_{RP^{\pm}}$, bifurcate from $\theta_1,\theta_2=\cos^{-1}\left(-\sqrt{\frac{2}{3}}\right)$ (at collision), and later ($r_2$) merge with $\mathcal{R}_{P_{1/2}}$ (Figure \ref{fig:EqMassRSlice_l1_Half_r2}).
\item Two curves which we denote $\mathcal{B}_{T^{\pm}}$, bifurcate from the trapezoid RE $\theta_1,\theta_2=\frac{\pi}{2}$ (at collision), and later ($r_3$) merge with $\mathcal{R}_{C}$ (Figure \ref{fig:EqMassRSlice_l1_Half_r3a}). 
\end{itemize}
Below we graph a 2D slice of the configuration space at $r=0.01$ that reveals these RE families. You'll note the three RE $\left\{\mathcal{R}_{T},\mathcal{B}_{T^{+}}(r),\mathcal{B}_{T^{-}}(r)\text{ }\right\} $ tightly grouped near $\left(\frac{\pi}{2},\frac{\pi}{2}\right)$. However, the RE $\mathcal{B}_{LP^{\pm }}$ and $\mathcal{B}_{RP^{\pm}} $ are so close together (near $(\cot^{-1}\sqrt{2},\cot^{-1}\sqrt{2})$) it is difficult to distinguish them in Figure \ref{fig:EqMassRSlice_l1_Half_r0p01}, so we have included Figure \ref{fig:EQMassRSlice_l1_0p5_r_0p002_Zoom_In} zoomed in on $\theta_{1,2}=\cot^{-1}\sqrt{2}\approx 
%0.61548
0.6155,$ where one is able to distinguish between $\mathcal{B}_{LP^{+}}$ and $\mathcal{B}_{LP^{-}}.$ A similar zoomed in graph exists for $\mathcal{B}_{RP^{\pm}}.$ 
\begin{figure}[H]
	\captionsetup[subfigure]{justification=centering}
	\centering
	\subfloat[\footnotesize Shows low radii RE, including bifurcations from $\theta_{1,2}\in\{\cot ^{-1}\sqrt{2},\frac{\pi}{2},\cos^{-1}\left(-\sqrt{\frac{2}{3}}\right)\}$.]
	{\includegraphics[width=0.33\textwidth]{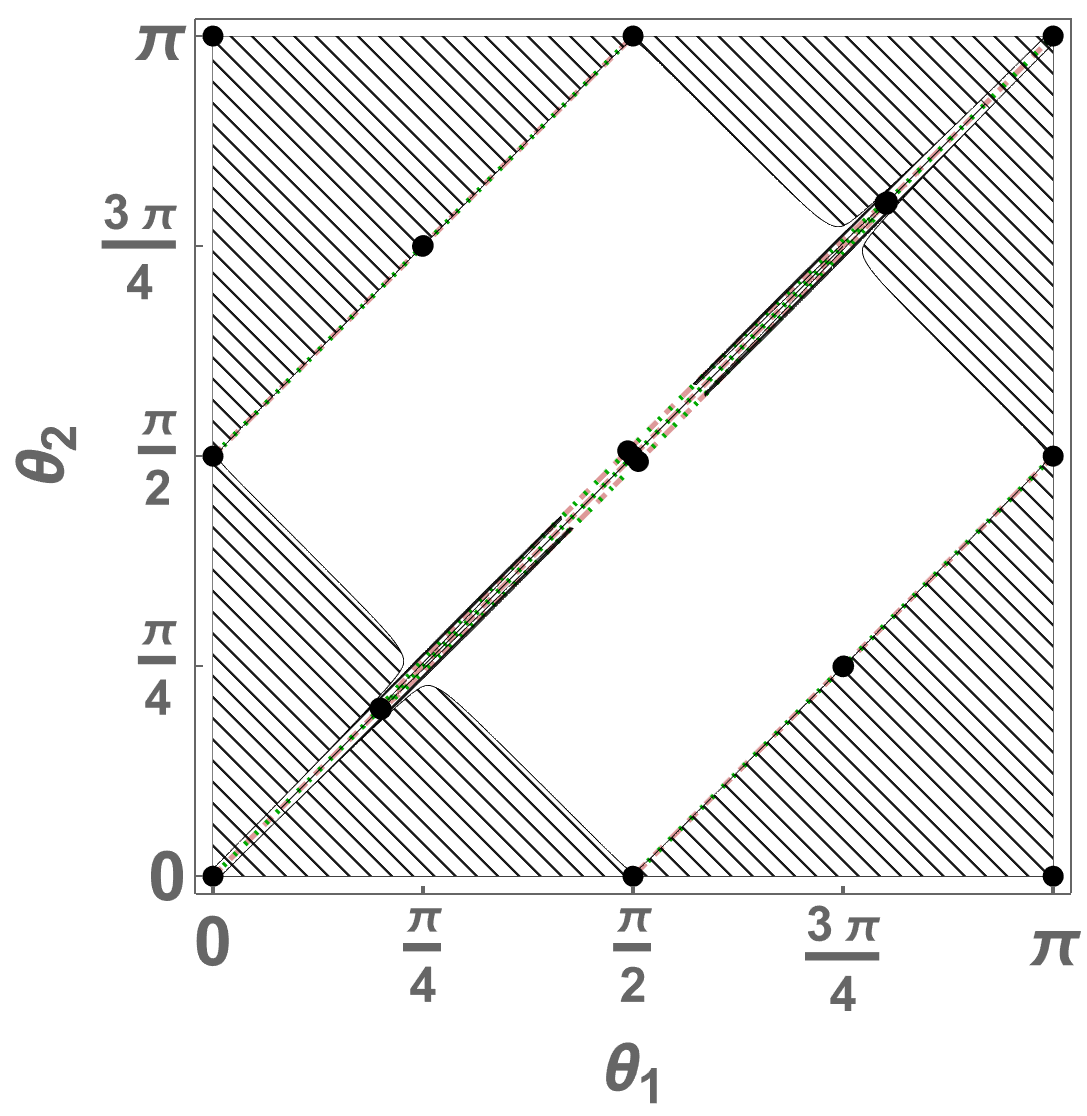}\label{fig:EqMassRSlice_l1_Half_r0p01}}
	\hspace{0.5cm}
	\subfloat[\footnotesize Shows close-up of $\mathcal{B}_{LP}$\\with RE bifurcating from $\cot ^{-1}\sqrt{2}$.]
	{\includegraphics[width=0.40\textwidth]{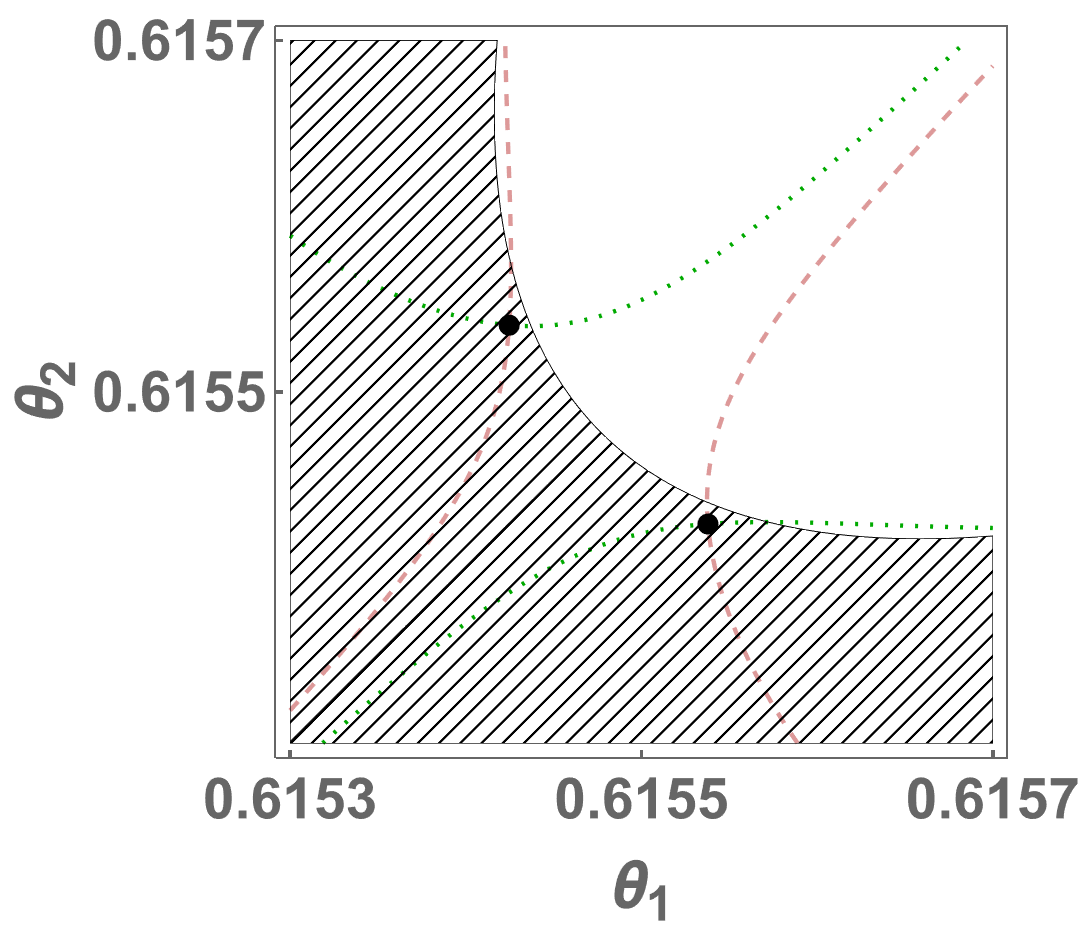}\label{fig:EQMassRSlice_l1_0p5_r_0p002_Zoom_In}}\raisebox{0.9\height}{\includegraphics[width=0.20\textwidth]{images/EqMass_rSlice/Legend.png}}
	\\[-1.5ex]
	\caption{\\\bfseries Equal Mass RE/Trace/Det.\\with $\ell_1=\frac{1}{2}$ for $r=0.01$}{\text{ }\\[-1.801em]}
	\label{fig:EQMassRSlice_l1_0p5_r_0}
\end{figure}
\begin{figure}[H]
	\captionsetup[subfigure]{justification=centering}
	\centering
	\subfloat[$r=0.316$, prior to $\mathcal{B}_C(r_1)$.\\Shows RE merging with $\theta_{1,2}=0$.]
	{\includegraphics[width=0.35\textwidth]{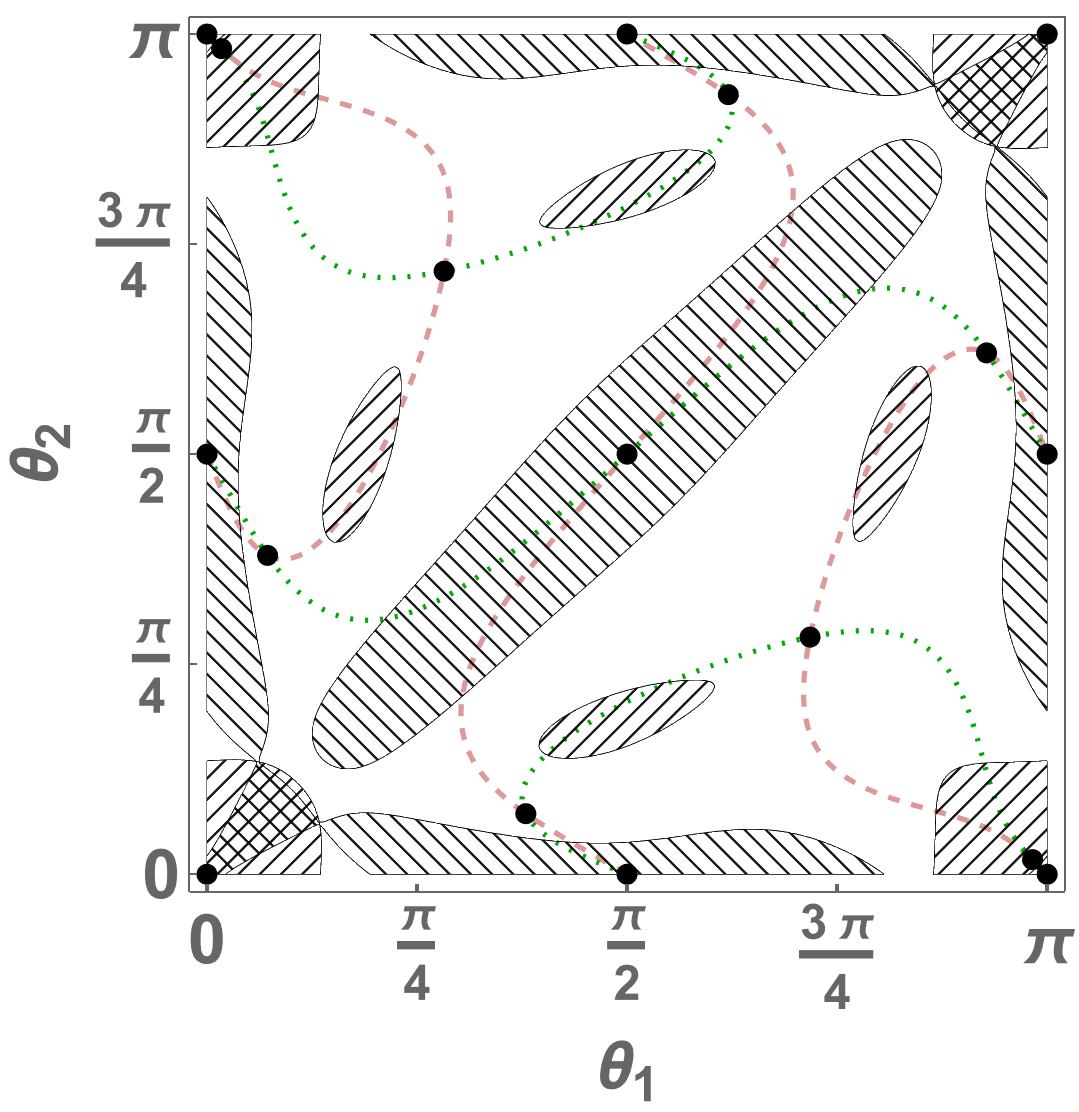}\label{fig:EqMassRSlice_l1_Half_r1}}
	\hspace{0.7cm}
	\subfloat[$r=0.336$, prior to $\mathcal{B}_{LP}(r_2)$/$\mathcal{B}_{RP}(r_2)$.\\Shows RE merging with $(\theta_1,\theta_2)\in\{(0,\frac{\pi}{2}),(\frac{\pi}{2},0)\}$.]
	{\includegraphics[width=0.35\textwidth]{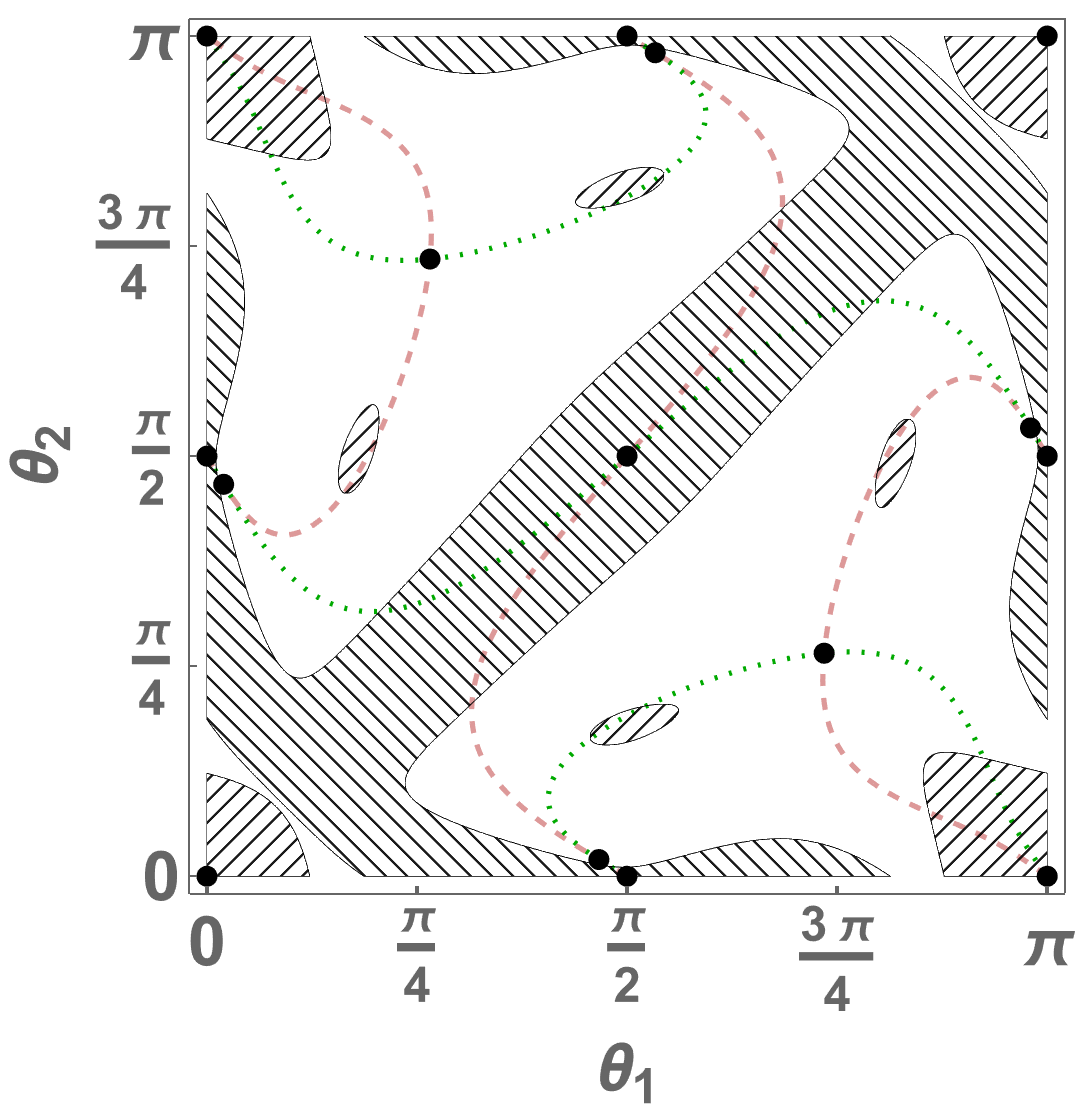}\label{fig:EqMassRSlice_l1_Half_r2}}\raisebox{0.9\height}{\includegraphics[width=0.20\textwidth]{images/EqMass_rSlice/Legend.png}\label{fig:EqMassRSlice_l1_Half_r8}}
	\\[-1.5ex]
	\caption{\\\bfseries Equal Mass RE/Trace/Det.\\with $\ell_1=\frac{1}{2}$ for $r_1,r_2$}{\text{ }\\[-1.801em]}
	\label{fig:EQMassRSlice_l10p5_r0p336}
\end{figure}
\begin{figure}[H]
	\captionsetup[subfigure]{justification=centering}
	\centering
	\subfloat[$r=0.498$, prior to $\mathcal{B}_T(r_3)$.\\Shows RE merging with $\theta_{1,2}=0$.]
	{\includegraphics[width=0.35\textwidth]{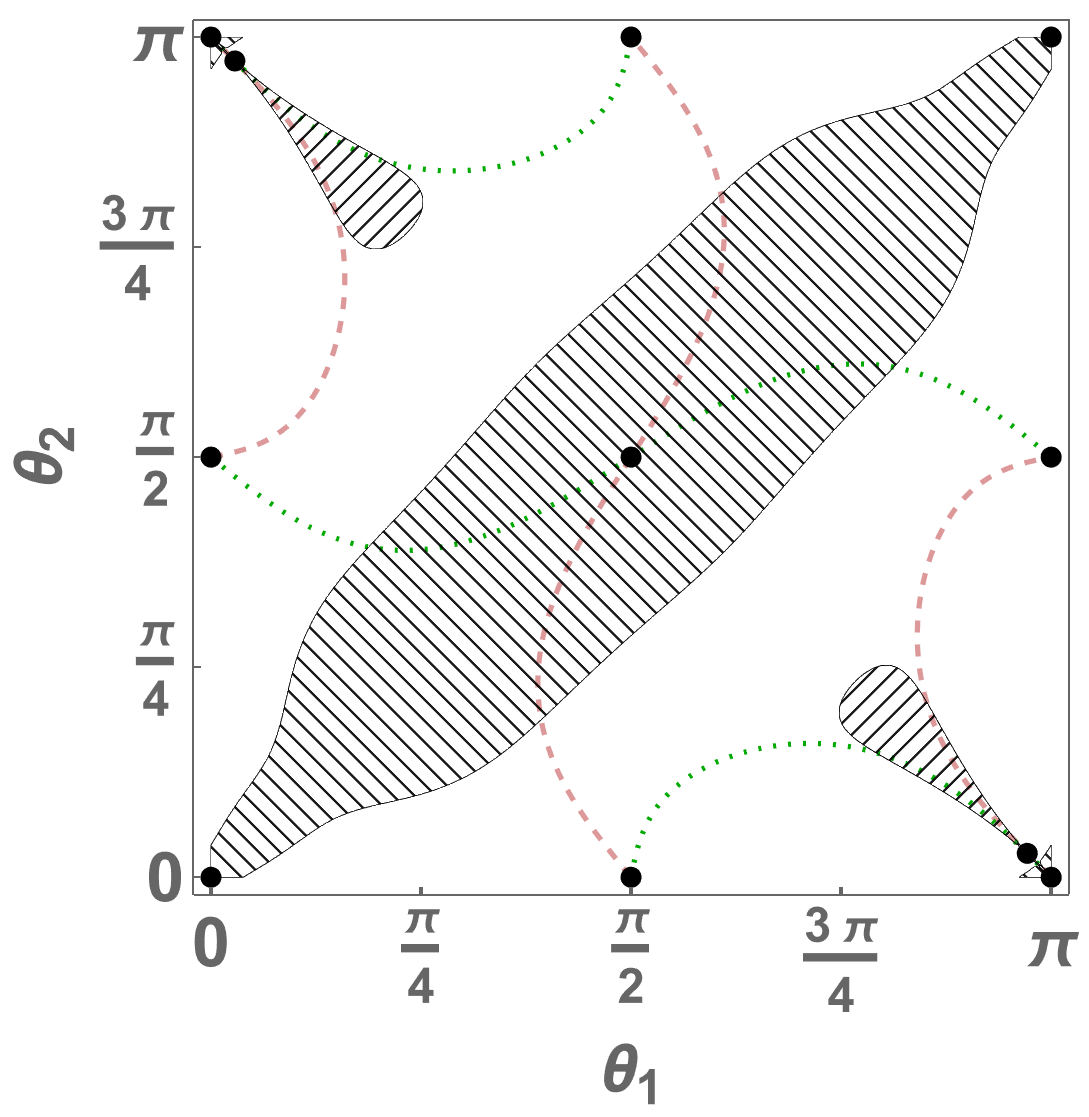}\label{fig:EqMassRSlice_l1_Half_r3a}}
	\hspace{0.7cm}
	\subfloat[$r=0.528$, subsequent to $\mathcal{B}_T(r_3)$.\\Shows $\theta_{1,2}=0$ stability after merging.]
	{\includegraphics[width=0.35\textwidth]{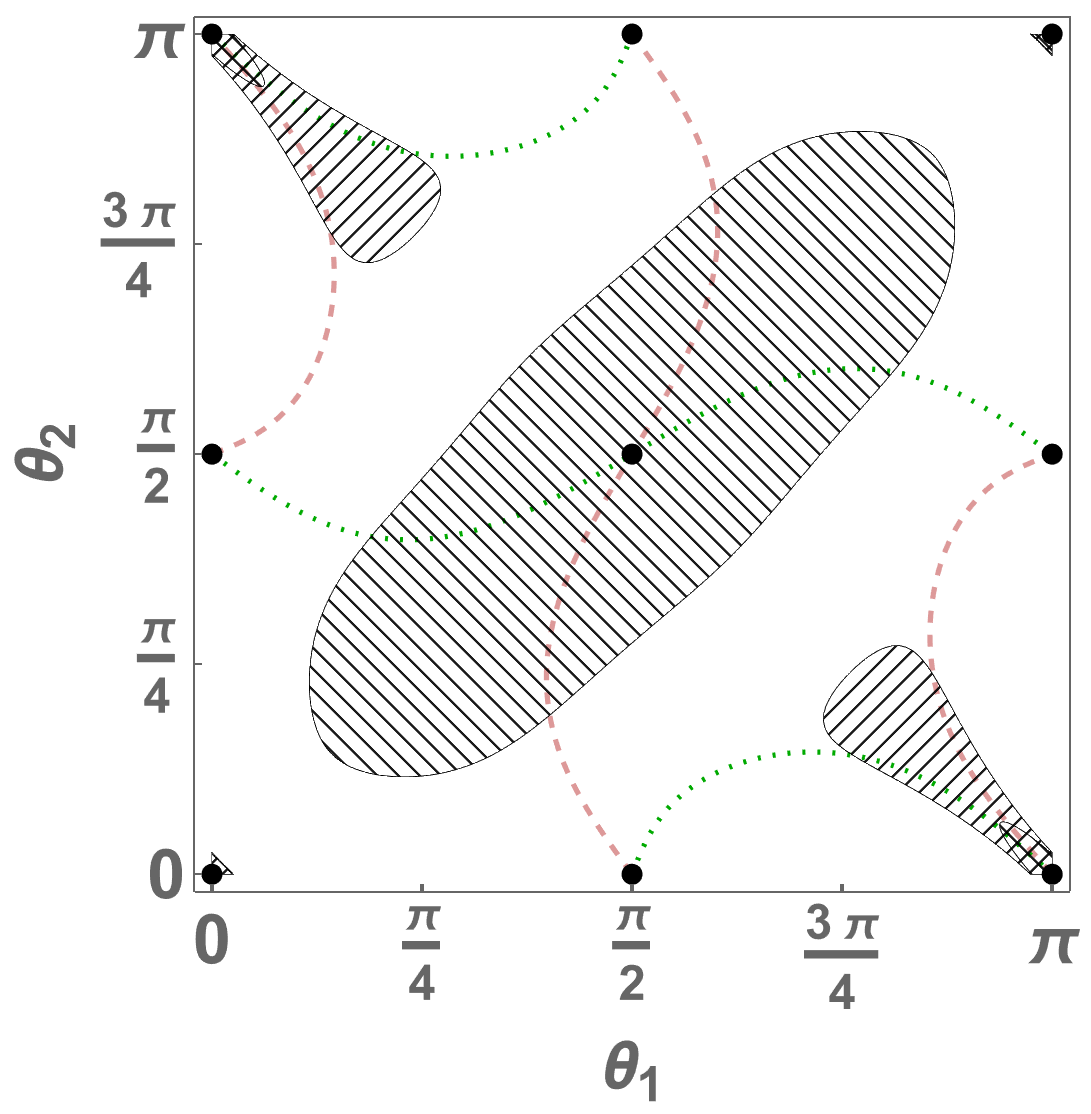}\label{fig:EqMassRSlice_l1_Half_r3b}}\raisebox{0.9\height}{\includegraphics[width=0.20\textwidth]{images/EqMass_rSlice/Legend.png}\label{fig:EqMassRSlice_l1_Half_r0p552}}
	\\[-1.5ex]
	\caption{\\\bfseries Equal Mass RE/Trace/Det.\\with $\ell_1=\frac{1}{2}$ for $r_3$}{\text{ }\\[-1.801em]}
	\label{fig:EQMassRSlice_l10p5_r0p528}
\end{figure}
To get an idea of the long-term behavior, we graph $r=1000$ below:
\begin{figure}[H]
	\centering
	\includegraphics[width=0.45\textwidth]{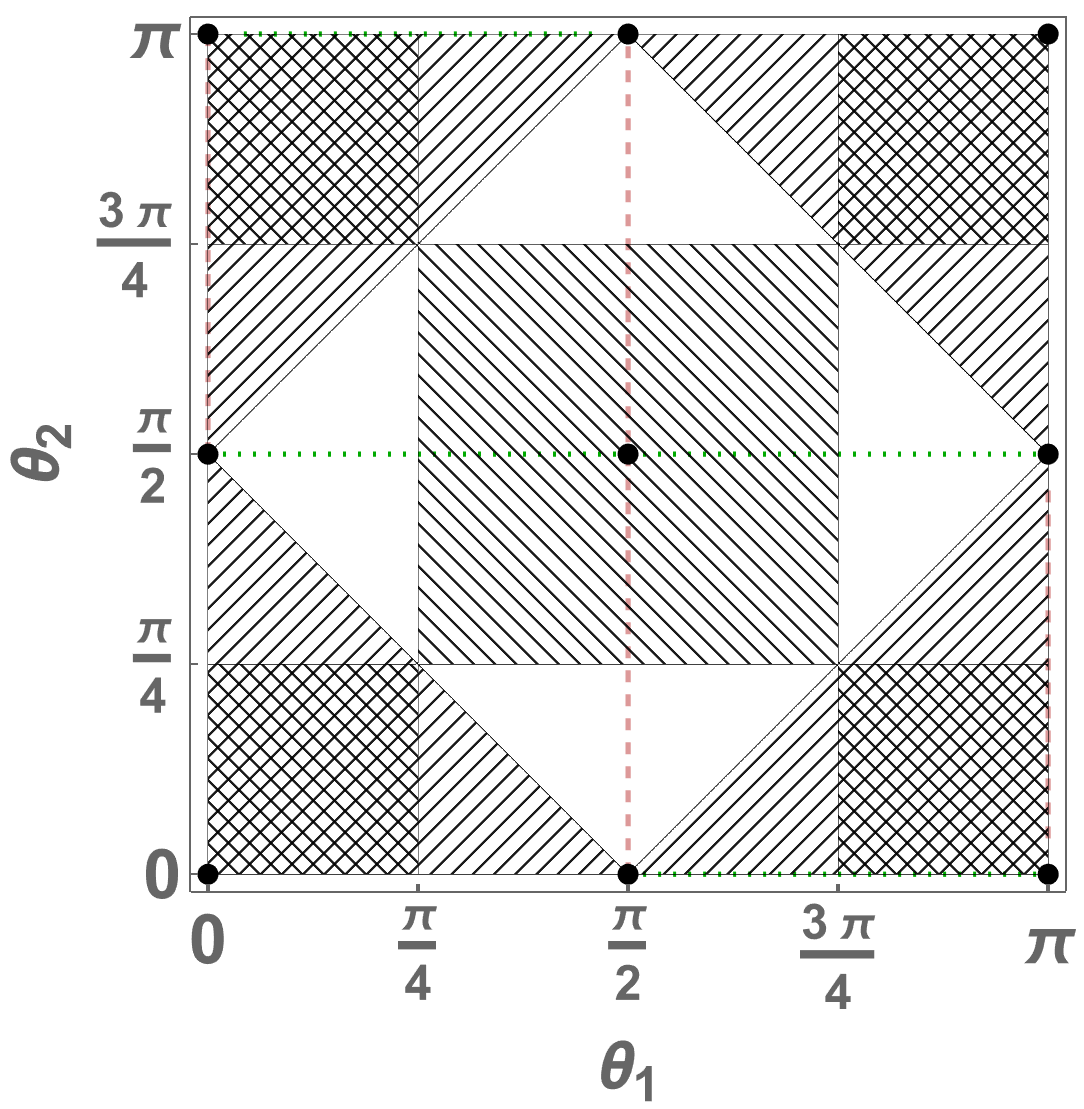}\raisebox{1.8\height}{\includegraphics[width=0.20\textwidth]{images/EqMass_rSlice/Legend.png}\label{fig:EqMassRSlice_l1_Half_r1000}}
	\\[-1.5ex]
	\caption{\\\bfseries Equal Mass RE/Trace/Det. \\with $\ell_1=\frac{1}{2}$ for $r=1000$}{\vspace{0.8em}\footnotesize Long-term Behavior, stable colinear RE\\[-0.931em]}
	\label{fig:EQMassRSlice_l10p5_r1000p}
\end{figure}%\vspace{1pt}Go to the following URL to see an animation of the above graph as $r$ varies from $0$ to $0.62$ : {\tiny http://www-users.math.umn.edu/\symbol{126}moreyjc/research/2DB/img/EqMass/rSlices/Several\_l1\_XRay.MP4}
When looking at the numerically found curves through our configuration space, we calculate that the angular momentum is nonphysical (either infinite or complex) for $\mathcal{B}_{CC}$, $\mathcal{B}_{PC}$ while $\ell \neq \frac{1}{2};$ and for $\mathcal{B}_{C}$, $\mathcal{B}_{T}$ while $\ell_{1}=\frac{1}{2}.$ Also, the angular momentum becomes unbounded for $\mathcal{B}_{LP/RP}$ as the curves approach their collision branching points at $r=0$. As a result, we will focus our attention on the bifurcation points located where the angular momentum is physically relevant, namely $\left\{\mathcal{B}_{TP}\left(r_{2}\right) ,\mathcal{B}_{TP}\left(r_{7}\right) ,\mathcal{B}_{CP}\left(r_{4}\right) ,\mathcal{B}_{CP}\left(r_{6}\right)\right\}$ for $\ell \neq\frac{1}{2}$ and $\mathcal{B}_{LP/RP}\left(r_{2}\right)$ for $\ell =\frac{1}{2}$, with $r>0$.

Let us perform the bifurcation analysis described above for these curves to confirm their existence. First, we will do a change of variables $\left(\theta_{1},\theta_{2}\right)\rightarrow\left(\theta_{1}+\theta_{1}^{\ast },\theta_{2}+\theta_{2}^{\ast }\right) $ so that our bifurcation points are at $\left(\theta_{1},\theta_{2}\right) =(0,0).$ For $\left(\theta_{1}^{\ast },\theta_{2}^{\ast }\right) \in\left\{\left(0,0\right) ,\left(0,\frac{\pi}{2}\right) ,\left(\frac{\pi}{2},0\right) ,\left(\frac{\pi}{2},\frac{\pi}{2}\right) \right\}$, it is easily confirmed that $\left(\widetilde{f},\widetilde{g}\right)\left({\theta}_{1}+{\theta}_{1}^{\ast }{ ,\theta}_{2}+{\theta}_{2}^{\ast },{r}\right) $ is odd in $\left(\theta_{1},\theta_{2}\right)$. Next, we will apply our bifurcation analysis to $\ell_{1}=\frac{3}{4}$ for $\mathcal{B}_{TP}$. This is the curve which bifurcates from the trapezoid configuration $\mathcal{R}_{T}$ at $r_2$, then the dumbbell bodies rotate until the curve merges with the perpendicular configuration $\mathcal{R}_{P}$ at $r_7$.

\paragraph{$\bm{\mathcal{B}_{TP}\left(r_7\right)}$ Pitchfork for $\bm{\ell_1=\frac{3}{4}}$}
When we are at $\left(\theta_{1}^{\ast },\theta_{2}^{\ast }\right) =\left(\frac{\pi}{2},0\right) $, our system \eqref{eq:2DB_Angular_Requirements} becomes: \\
$\text{\enspace\enspace} \widetilde{f}\left(\vec{{\theta}};{r}\right) :=\partial_{{\theta}_{1}}V\left(\vec{{\theta}},{r}\right) =\frac{1}{2}\left(\frac{1}{8}\sin\left({\theta}_{1}{-\theta}_{2}\right){-r}\sin{\theta}_{1}\right)\left(\frac{1}{d_{11}^{3}}-\frac{1}{d_{21}^{3}}\right)$\\
$\text{\enspace\enspace\enspace\enspace\enspace\enspace\enspace\enspace\enspace\enspace}+\frac{1}{2}\left(\frac{1}{8}\sin\left({\theta}_{1}{-\theta}_{2}\right){+r}\sin{\theta}_{1}\right)\left(\frac{1}{d_{22}^{3}}-\frac{1}{d_{12}^{3}}\right)$, and

$\text{\enspace\enspace} \widetilde{g}\left(\vec{{\theta}};{r}\right) \,:=\partial_{{\theta}_{2}}V\left(\vec{{\theta}},{r}\right) =\frac{1}{2}\left(\frac{3}{8}\sin\left({\theta}_{1}{-\theta}_{2}\right){ -r}\sin{\theta}_{2}\right)\left(\frac{1}{d_{12}^{3}}-\frac{1}{d_{11}^{3}}\right) $\\
$\text{\enspace\enspace\enspace\enspace\enspace\enspace\enspace\enspace\enspace\enspace}-\frac{1}{2}\left(\frac{3}{8}\sin\left({\theta}_{1}{-\theta}_{2}\right){+r}\sin{\theta}_{2}\right)\left(\frac{1}{d_{21}^{3}}-\frac{1}{d_{22}^{3}}\right)$,

where $\scriptstyle d_{wv}\left(\vec{{\theta}};r\right) :=\sqrt{{r}^{2}-\left({-1}\right)^{w}\frac{3}{4}{r}\cos{\theta}_{1}+\left({-1}\right)^{v}\frac{1}{4}{r}\cos{\theta}_{{2}}-\left({-1}\right)^{w+v}\frac{3}{32}\cos\left({\theta}_{1}{-\theta}_{2}\right)+\frac{5}{32}}.$

For $\vec{\theta}=\vec{0}$, we locate $\left\vert D\left(\widetilde{f},\widetilde{g}\right) _{\left(\vec{0};r\right)}\right\vert =0$ for $r^{\ast }\approx
%0.389342198769476.
0.3893$. We then calculate the eigenvalues $\left(\mu _{1},\;\mu _{2}\right) $ of $\left\vert D\left(f,g\right)_{\left(\vec{0};0\right)}\right\vert$, the eigenvectors, and the transition
matrix $P=\,\begin{bmatrix}
	P_{11} & P_{12} \\ 
	P_{12} & -P_{11}\end{bmatrix}$
to the Jordan normal form, where $\left(P_{11},P_{12}\right)\approx\left(
%0.166851
0.1669,\;
%-0.985982
-0.9860\right) $. This allows us to change variables to align our axes with the tangent plane to the RE curves. Observe that $\left\vert P\right\vert =1$, such that $P=P^{-1}.$ So, let $\vec{u}=P\vec{\theta},$ or $\theta_{1}=P_{11}u_{1}+P_{12}u_{2}$ and $\theta_{2}=P_{12}u_{1}-P_{11}u_{2}.$ Our new functions become: $\left(f(\vec{u}),g(\vec{u})\right) :=P^{-1}\left(\widetilde{f}\left(P^{-1}\vec{u};r\right) ,\widetilde{g}\left(P^{-1}\vec{u};r\right)\right)$.

If we then define: $\begin{bmatrix}
	f_{1} & f_{2} \\ 
	g_{1} & g_{2}\end{bmatrix}:=D\left(f,g\right) _{\vec{u}},$ we find that at the bifurcation point we have:

$\text{\enspace\enspace}\begin{bmatrix}
	f_{1} & f_{2} \\ 
	g_{1} & g_{2}\end{bmatrix}|_{\left(\vec{0};0\right)}=\begin{bmatrix}
	0 & 0 \\ 
	0 & 
	%-1.29509
	-1.295\end{bmatrix}.$

In other words: $\mu _{1}\left(0\right) =0,\;\mu _{2}\left(0\right) =:\mu\approx
%-1.29509
-1.295\neq 0$, $k:=\frac{d\mu _{1}}{dr}=f_{1r}|_{\left(\vec{0};0\right)}\approx
%0.921324
0.9213\neq 0,$ and $l:=\frac{1}{3}f_{{ u}_{1}{ u}_{1}{ u}_{1}}\left(0;0\right)\approx
%0.0589714
0.05897.$ And using our conclusions \eqref{eq:Pitchfork_Sols} above, we have:
\begin{subequations}\label{eq:EqMass_Pitchfork_Sols}
	\makeatletter\@fleqntrue\makeatother
	\begin{align}
		\begin{split}
			&\text{$u_{2}=0+\mathcal{O}\left(u_{1}^{3}\right) ,\;$}
		\end{split}\\
		\begin{split}
			&\text{$r=-\frac{l}{2k}u_{1}^{2}+\mathcal{O}\left(u_{1}^{3}\right)\approx-\frac{
					%0.0589714
				0.05897}{2\left(
				%0.921324
				0.9213\right)}u_{1}^{2}+\mathcal{O}\left(u_{1}^{3}\right)\approx
			%-0.0320036
			-0.03200\,u_{1}^{2}+\mathcal{O}\left(u_{1}^{3}\right).$}
		\end{split}
	\end{align}
\end{subequations}
To compare \eqref{eq:EqMass_Pitchfork_Sols} to our numerical results, we will reverse our previous change of coordinates. In general we have:
$r\rightarrow r-r^{\ast },$ $ u_{1}\rightarrow P_{11}\left(\theta_{1}-\theta_{1}^{\ast }\right)+P_{12}\left(\theta_{2}-\theta_{2}^{\ast }\right) \ $and $\;u_{2}\rightarrow P_{12}\left(\theta_{1}-\theta_{1}^{\ast }\right) -P_{11}\left(\theta_{2}-\theta_{2}^{\ast}\right)$. So \eqref{eq:EqMass_Pitchfork_Sols} becomes:
\begin{subequations}
	\begin{align}
		\begin{split}
			\text{$\theta_{2}$}&	\text{$=\theta_{2}^{\ast}+\frac{P_{12}}{P_{11}}\left(\theta_{1}-\theta_{1}^{\ast }\right)+\mathcal{O}\left(\left(\theta_{1}-\theta_{1}^{\ast }\right)^{3}\right)\nonumber$}
		\end{split}\\
		\begin{split}
			\text{$r$}&	\text{$=r^{\ast}-\frac{l}{2k}\left(P_{11}\left(\theta_{1}-\theta_{1}^{\ast}\right)+P_{12}\left(\theta_{2}-\theta_{2}^{\ast}\right)\right)^{2}+\mathcal{O}\left(\left(\theta_{1}-\theta_{1}^{\ast }\right)^{3}\right)\nonumber$}
		\end{split}\\
%		\begin{split}
%			&\text{$=r^{\ast}-\frac{l}{2k}\left(P_{11}\left(\theta_{1}-\theta_{1}^{\ast}\right)+P_{12}\left(\frac{P_{12}}{P_{11}}\left(\theta_{1}-\theta_{1}^{\ast}\right)+\mathcal{O}\left(\left(\theta_{1}-\theta_{1}^{\ast}\right)^{3}\right)\right)\right)^{2}+\mathcal{O}\left(\left(\theta_{1}-\theta_{1}^{\ast}\right)^{3}\right)\nonumber$}
%		\end{split}\\
	\begin{split}
		&\text{$=r^{\ast}-\frac{l}{2k}\left(P_{11}+\frac{P_{12}^{2}}{P_{11}}\right)^{2}\left(\theta_{1}-\theta_{1}^{\ast}\right)^{2}+\mathcal{O}\left(\left(\theta_{1}-\theta_{1}^{\ast}\right)^{3}\right)\nonumber$}
\end{split}
	\end{align}
\end{subequations}

And our parameterized graph in $\left(\vec{\theta},r\right)$ becomes:\\$\text{\enspace\enspace}G\left(\theta_{1}\right) :=\left(\theta_{1},\text{ }\theta_{2}^{\ast }+\frac{P_{12}}{P_{11}}\left(\theta_{1}-\theta_{1}^{\ast}\right) ,\text{ }r^{\ast }-\frac{l}{2k}\left(P_{11}+\frac{P_{12}^{2}}{P_{11}}\right)^{2}\left(\theta_{1}-\theta_{1}^{\ast }\right)^{2}\right) $.

In particular, for $\mathcal{B}_{TP}\left(r_{7}\right) $ we take: $r\rightarrow r-
%0.389342198769476
0.3893,\;$ $u_{1}\rightarrow
%0.166851
0.1669\theta_{1}-
%0.985982
0.9860\left(\theta_{2}-\frac{\pi}{2}\right)$, and\\$\;u_{2}\rightarrow
%-0.985982
-0.9860\theta_{1}
%-0.166851
-0.1669\left(\theta_{2}-\frac{\pi }{2}\right)$. Also, $\frac{P_{12}}{P_{11}}=\frac{
%-0.985982
-0.9860}{
%0.166851
0.1669}\approx
%-5.90937
-5.909,$ and $\ -\frac{l}{2k}\left(P_{11}+\frac{P_{12}^{2}}{P_{11}}\right)^{2}\\ \approx
%-0.0320036
-0.03200\left(
%35.9207
35.92\right) \approx
%-1.1496
-1.150.$ Therefore, we find the parameterized graph of $\mathcal{B}_{TP}\left(r_{7}\right)$ as\\
$G=\left(\theta_{1},\text{ }
%-5.90937
-5.909\,\theta_{1},\text{ }
%0.389342198769476-1.1496
0.3893-1.150\,\left(\theta_{1}-\frac{\pi}{2}\right)^{2}\right) $. Below, this graph is plotted in the 
%$u_1 r-$ plane, as well as the 
$r\theta_2$-plane. 
%In the second plot, w
We also include the numerically found RE for comparison.
\begin{figure}[H]
	\captionsetup[subfigure]{justification=centering}
	\centering
%	\subfloat[$\mathcal{B}_{TP}\left(r_{7}\right)$ in $u_1,r$-plane]
%	{\includegraphics[width=0.40\textwidth]{images//Bifurcation/TP1a.png}}
%	\hspace{1cm}
%	\subfloat[$\mathcal{B}_{TP}\left(r_{7}\right)$ in $\theta_2,r$-plane.\\Numerical (darker curve),\\G (lighter curve)]
	{\includegraphics[width=0.35\textwidth]{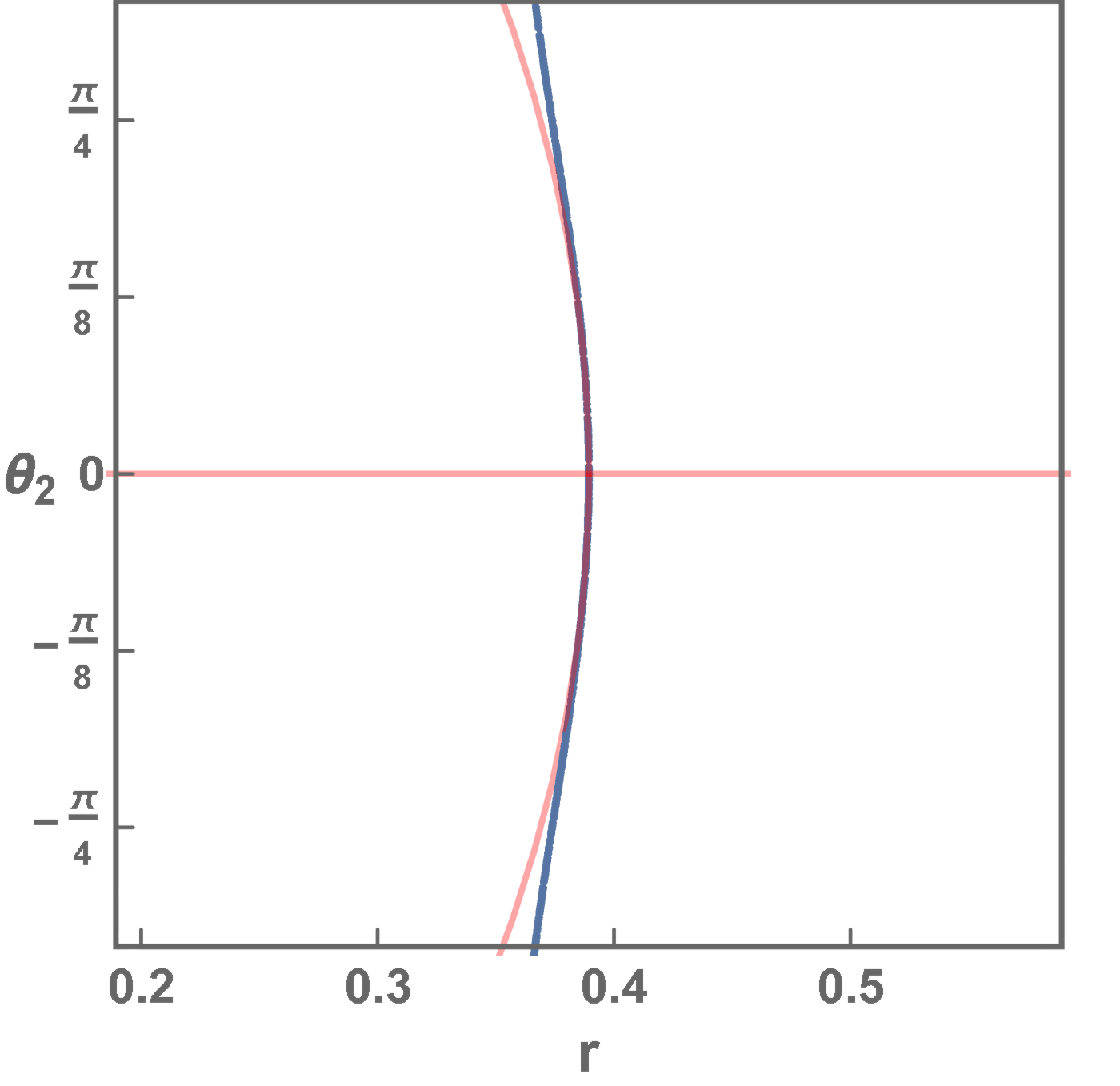}}
	\caption{\\\bfseries Equal Mass $\mathcal{B}_{TP}\left(r_{7}\right)$ Bifurcation with $\ell_1=\frac{3}{4}$}{\vspace{0.8em}\footnotesize Numerical (darker curve), Quadratic Approximation G (lighter curve)\\This shows the end of the curve which transitions from a trapezoid to a perpendicular configuration.\\[-0.051em]}
	\label{fig:Bifurcation_TP1}
\end{figure}
\paragraph{$\bm{\mathcal{B}_{TP}\left(r_2\right)}$ Pitchfork for $\bm{\ell_1=\frac{3}{4}}$}
Similarly, for the start of the trapezoid to perpendicular bifurcation $\mathcal{B}_{TP}\left(r_{2}\right)$ we find:
$G=\left(\theta_{1},\text{ }\frac{\pi}{2}-6.44012\,\left(\theta_{1}-\frac{\pi}{2}\right) ,\text{ }0.360032-0.0253042\left(42.4752\right)\,\left(\theta_{1}-\frac{\pi}{2}\right)^{2}\right) $. And below we include a plot comparing the numerically found results to this curve.
\begin{figure}[H]
	\centering
	\includegraphics[width=0.32\textwidth]{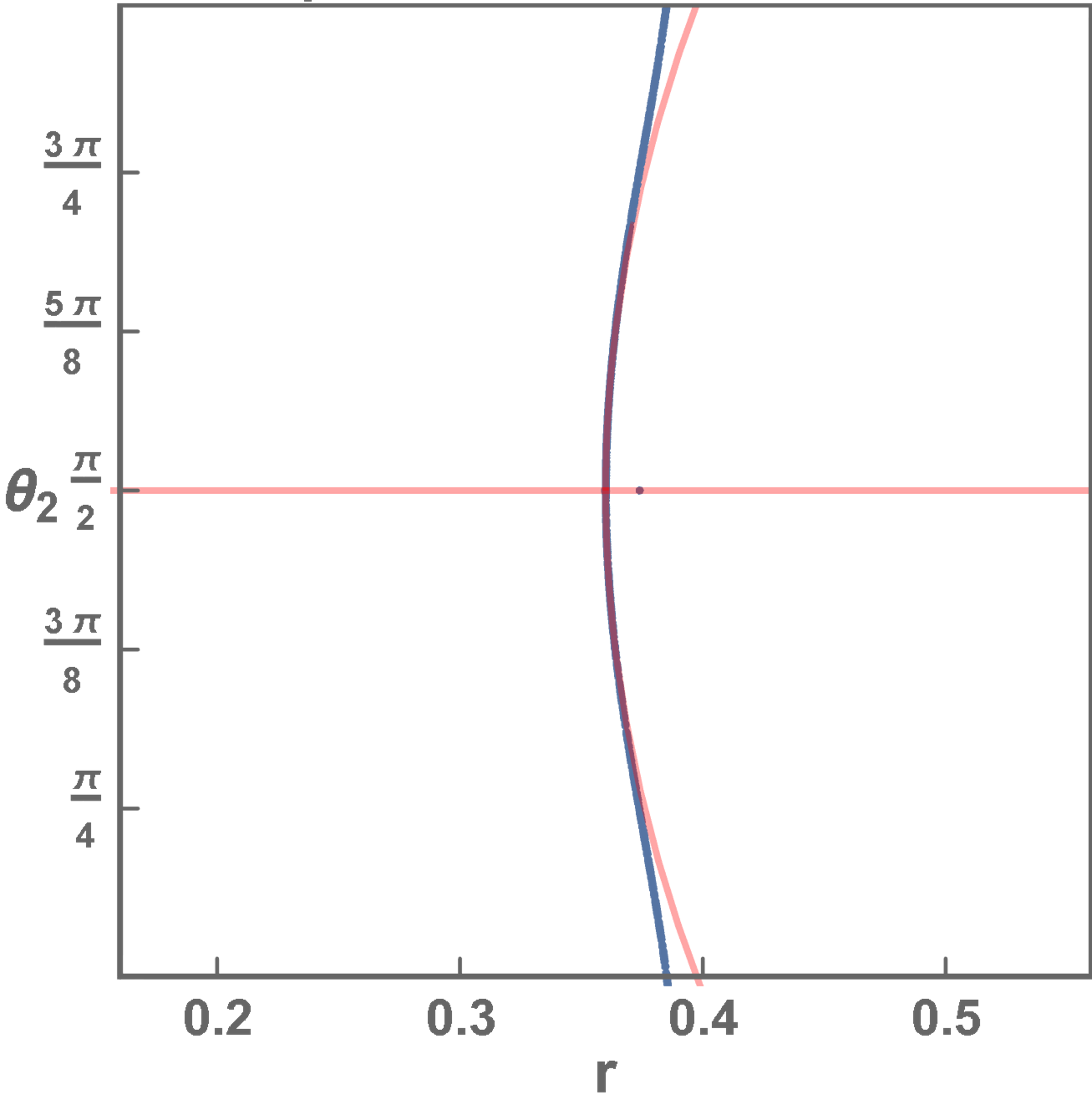}
	\caption{\\\bfseries Equal Mass $\mathcal{B}_{TP}\left(r_{2}\right)$ Bifurcations with $\ell_1=\frac{3}{4}$}{\vspace{0.8em}\footnotesize Numerical (darker curve), Quadratic Approximation G (lighter curve)\\This shows the start of the curve which transitions from a trapezoid to a perpendicular configuration.\\[-2.5951em]}
	\label{fig:Bifurcation_TP1a}
\end{figure}
\paragraph{$\bm{\mathcal{B}_{CP}}$ Pitchforks for $\bm{\ell_1=\frac{3}{4}}$}
$\mathcal{B}_{CP}$ is the curve which bifurcates from the colinear configuration, and later merges with the perpendicular configuration. For the two ends of the curve we find (respectively):\\
$\text{\hskip20pt}G=\left(\theta_{1},\text{ }
%19.9938
19.99\theta_{1},\text{ }
%0.368631+0.0187383\left(400.754\right)
0.3686+0.01874\left(400.8\right)\theta_{1}^{2}\right)$, and \\
$\text{\hskip20pt}G=\left(\theta_{1},\text{ }\frac{\pi}{2}-21.611\theta_{1},\text{ }
%0.386456-0.0168436\left(468.035\right)
0.3865-0.01684\left(468.0\right)\,\theta_{1}^{2}\right)$.\\Below are the plots comparing the numerically found results to these curves.
\begin{figure}[H]
	\captionsetup[subfigure]{justification=centering}
	\centering
	\subfloat[$\mathcal{B}_{CP}\left(r_{4}\right)$]
	{\includegraphics[width=0.32\textwidth]{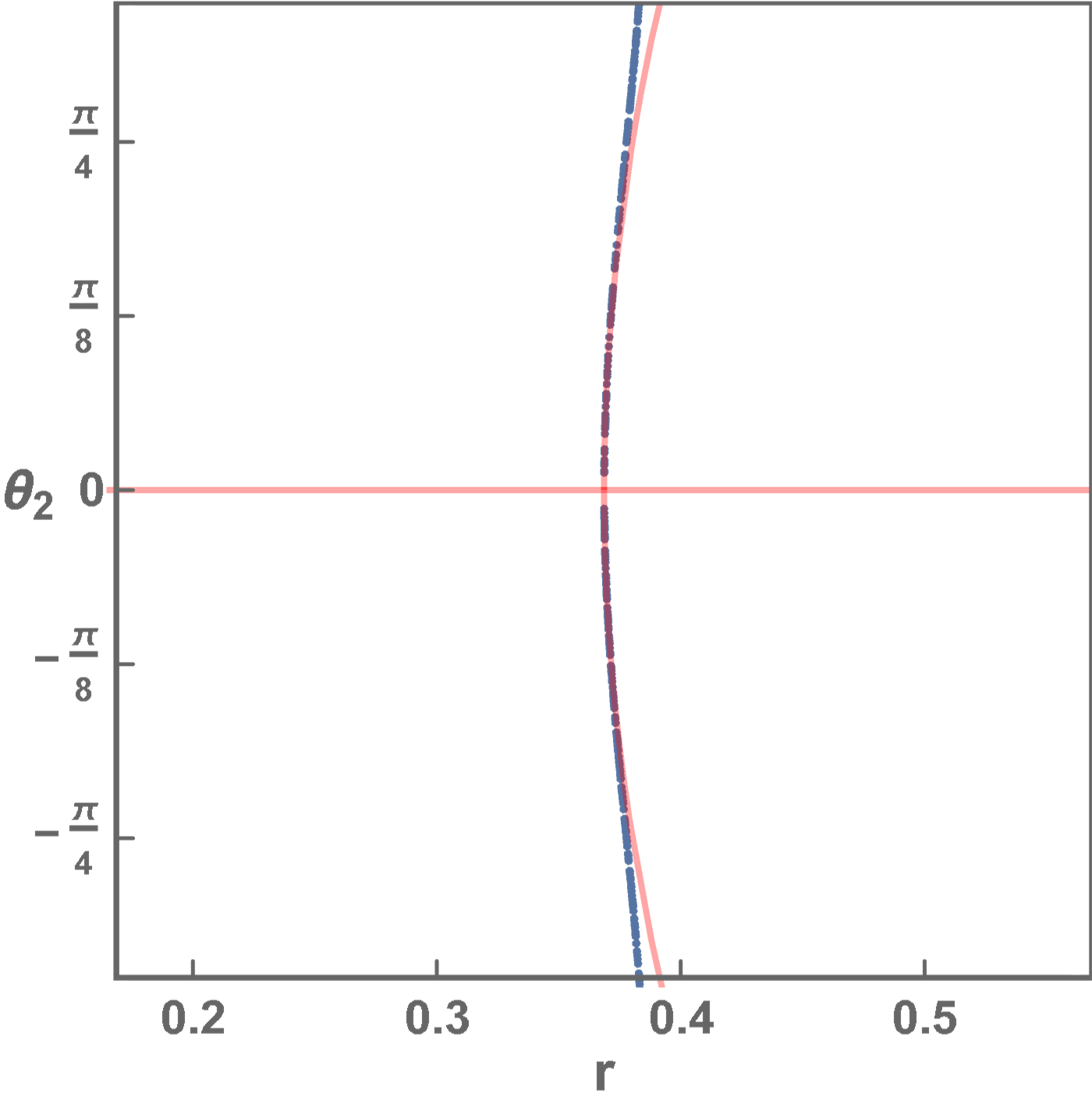}}
	\hspace{1cm}
	\subfloat[$\mathcal{B}_{CP}\left(r_{6}\right)$]
	{\includegraphics[width=0.32\textwidth]{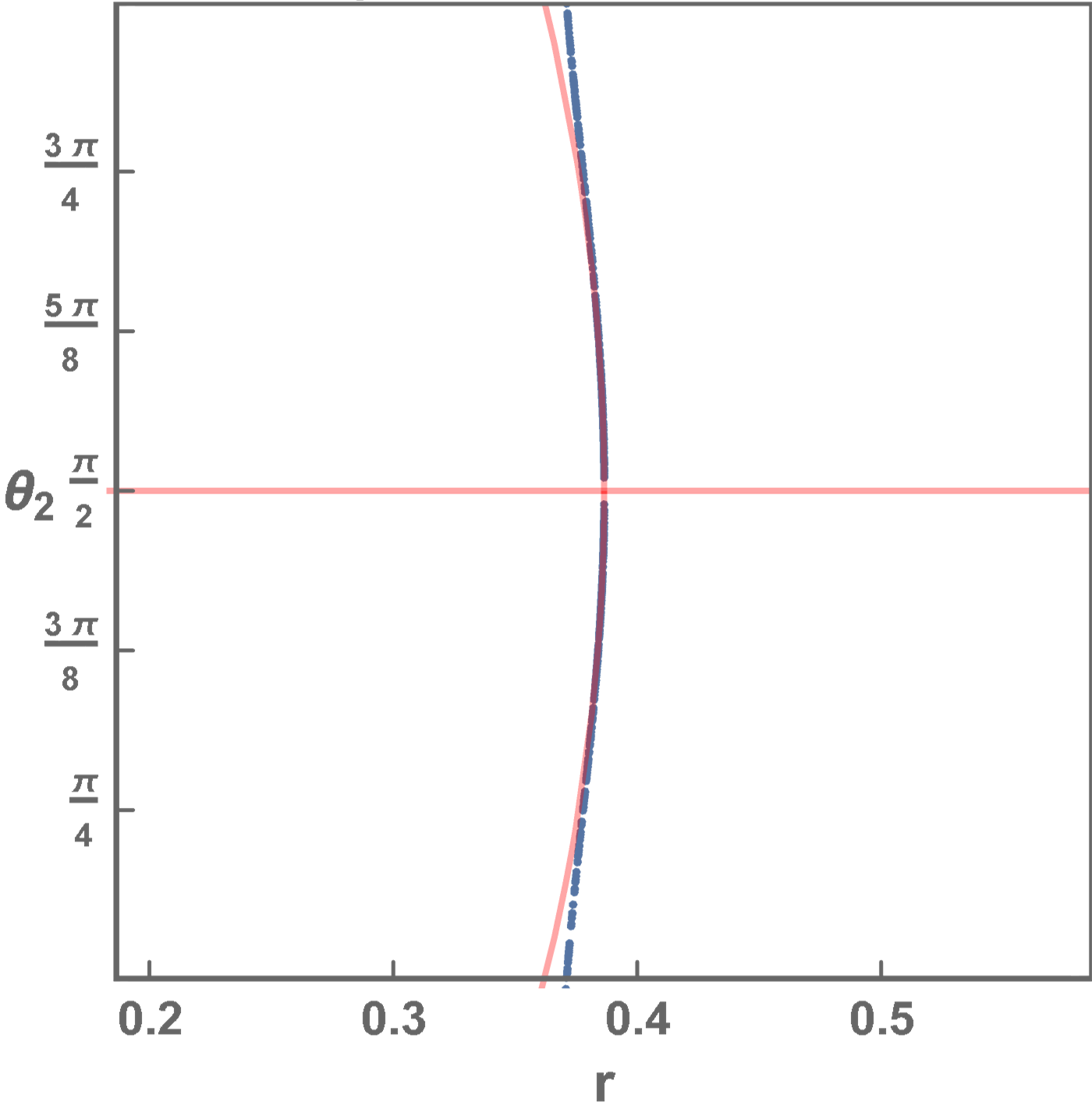}}
	\caption{\\\bfseries Equal Mass $\mathcal{B}_{CP}$ Bifurcation with $\ell_1=\frac{3}{4}$}{\vspace{0.8em}\footnotesize Numerical (darker curve), Quadratic Approximation G (lighter curve)\\This shows both ends of the curve which transitions from a colinear to a perpendicular configuration.\\[-1.991em]}
	\label{fig:Bifurcation_CPib}
\end{figure}
\paragraph{$\bm{\mathcal{B}_{LP}}$ Pitchforks for $\bm{\ell_1=\frac{1}{2}}$}
$\mathcal{B}_{LP}$ is the curve which bifurcates from $(\theta_{1},\theta_{2})=\left(\cot ^{-1}\sqrt{2},\cot ^{-1}\sqrt{2}\right)$ at $r=0$ (a collision), and later merges with the perpendicular configuration. For the end of the curve at $r_2$ we find: $G=\left(\theta_{1},\text{ }\frac{\pi}{2}
%-1.68595\theta_{1},\text{ }0.337696-0.106847\left(3.84244
-1.686\theta_{1},\text{ }0.3377-0.1068\left(3.842\right)\theta_{1}^{2}\right)$. Below we compare the numerically found results to this curve.
\begin{figure}[H]
	\centering

	\includegraphics[width=0.4\textwidth]{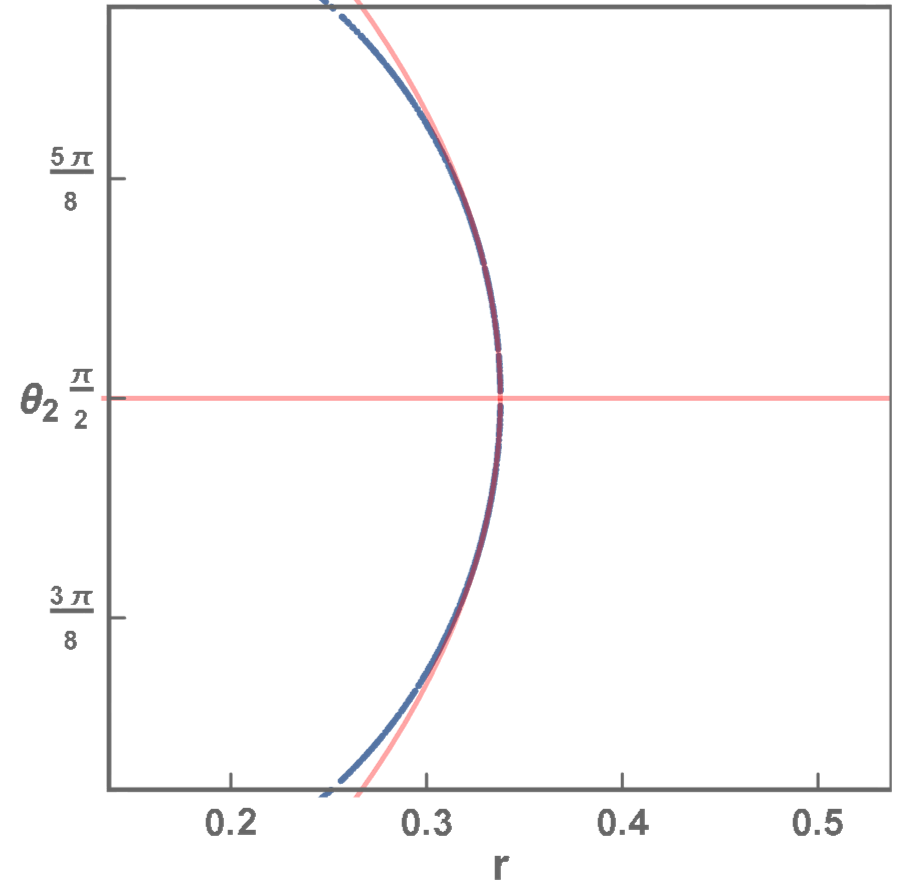}
	\caption{\\\bfseries Equal Mass $\mathcal{B}_{LP}(r_2)$ Bifurcation with $\ell_1=\frac{1}{2}$}{\vspace{0.8em}\footnotesize Numerical (darker curve), Quadratic Approximation G (lighter curve)\\This shows the end of the curve which starts at $(r,\theta_{i})=(0,\cot ^{-1}\sqrt{2})$\\and subsequently merges with the perpendicular configuration.\\[-1.851em]}
	\label{fig:Bifurcation_LP}
\end{figure}

\textbf{$L^2$ Bifurcation Analysis}\\
Now that we have confirmed the existence of the trapezoid and asymmetric RE curves, let us determine the regions of the parameter space that differ by the number of RE. For the asymmetric curves, we first calculate angular momentum for the numerically located RE families. Recall for $\ell_1\neq\frac{1}{2}$, we had a curve bifurcating from the trapezoid RE and subsequently merging with the perpendicular RE (see Figure \ref{fig:EQMassTrapToPerpBif_l10p75}), and a curve bifurcating from the colinear RE and subsequently merging with the perpendicular RE (see Figure \ref{fig:EQMassColToPerpBif_l10p75}).  And for $\ell_1=\frac{1}{2}$, we had one family of solutions in which a curve bifurcated from a collision of the two dumbbells, and subsequently merged with the perpendicular RE (see Figure \ref{fig:EQMassPerpBifurcation_l10p5}). For each of these families, the value of $M_1$ did not qualitatively change the shape of the angular momentum graphs. Below you see the graphs for these families over the relevant radii. 
\begin{figure}[H]
	\captionsetup[subfigure]{justification=centering}
	\centering
	\subfloat[$\mathcal{B}_{TP}$ for $\ell_1=\frac{3}{4}$]
	{\includegraphics[width=0.3\textwidth]{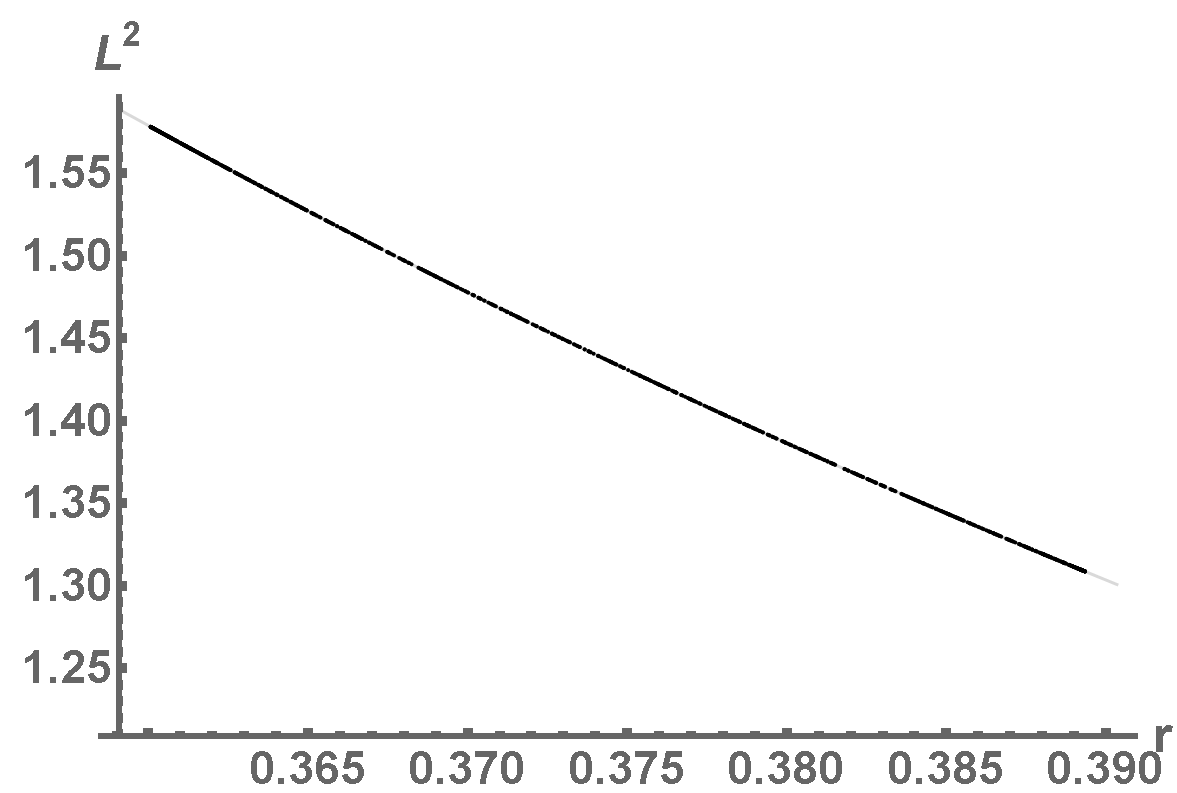}}
	\hspace{.4cm}
	\subfloat[$\mathcal{B}_{CP}$ for $\ell_1=\frac{3}{4}$]
	{\includegraphics[width=0.3\textwidth]{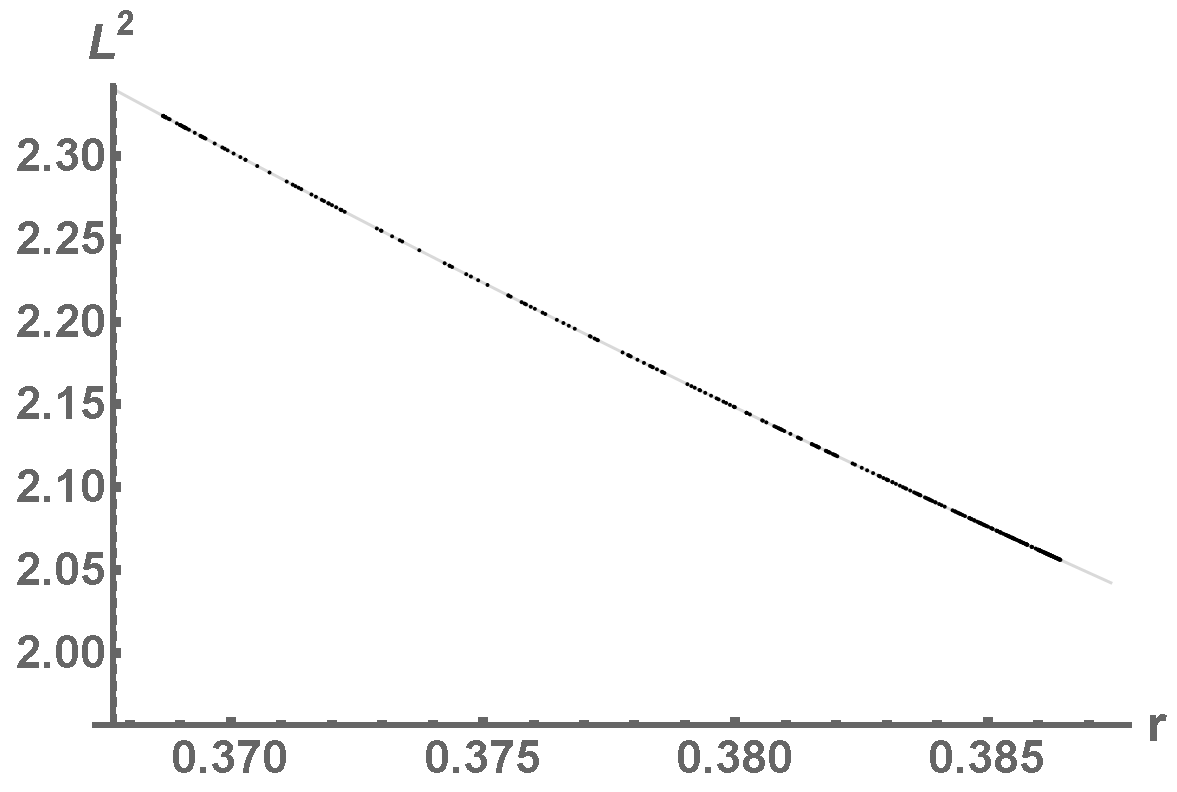}}
	\hspace{.4cm}
	\subfloat[$\mathcal{B}_{LP/RP}$ for $\ell_1=\frac{1}{2}$]
	{\includegraphics[width=0.3\textwidth]{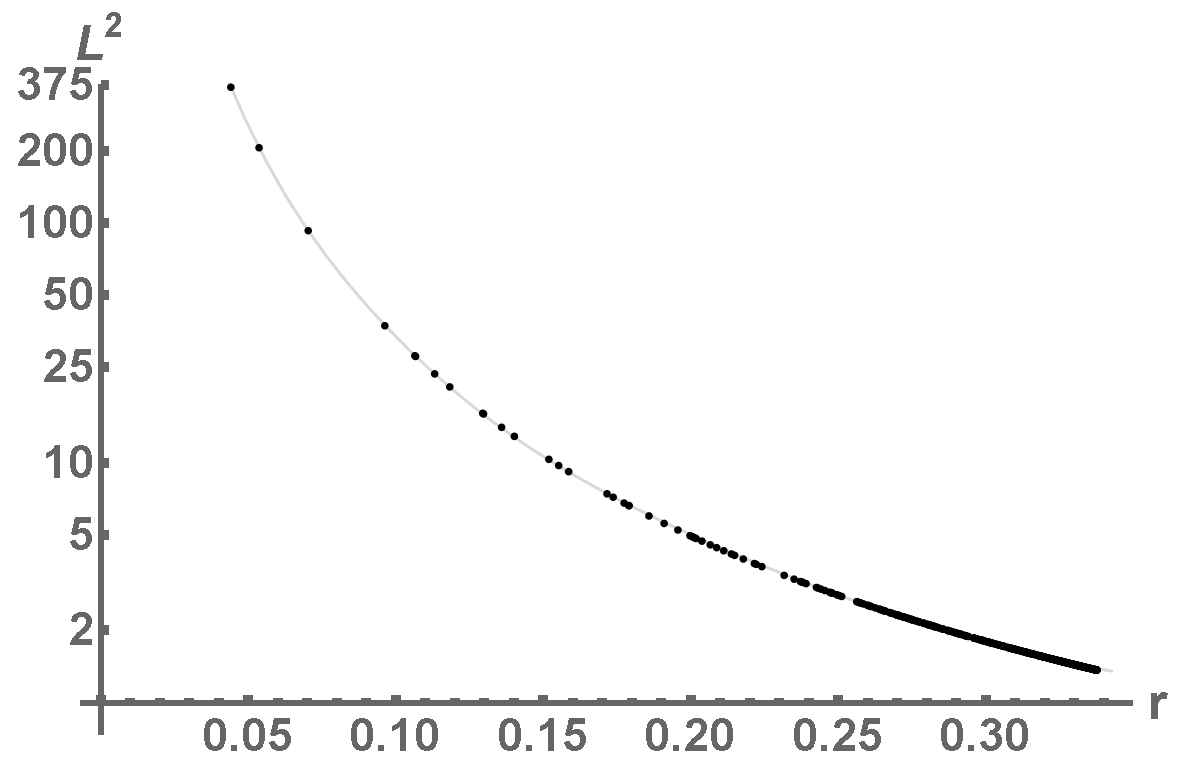}}
	\caption{\\\bfseries Equal Mass Angular Momenta for $\mathcal{B}_{TP}$, $\mathcal{B}_{CP}$, $\mathcal{B}_{LP/RP}$ with $M_1=\frac{1}{2}$}{\vspace{0.8em}\footnotesize The $L^2$ curves are strictly decreasing, and therefore have only one RE\\for each allowed angular momentum.\\[-0.851em]}
	\label{fig:Bifurcation_LP}
\end{figure}
Observe these curves are all strictly decreasing.  So for the angular momentum range allowed by these curves, there is only one RE per $L$.

For the trapezoid RE, the $L^2$ curves look qualitatively similar to the ones below.
\begin{figure}[H]
	\captionsetup[subfigure]{justification=centering}
	\centering
	{\includegraphics[width=0.55\textwidth]{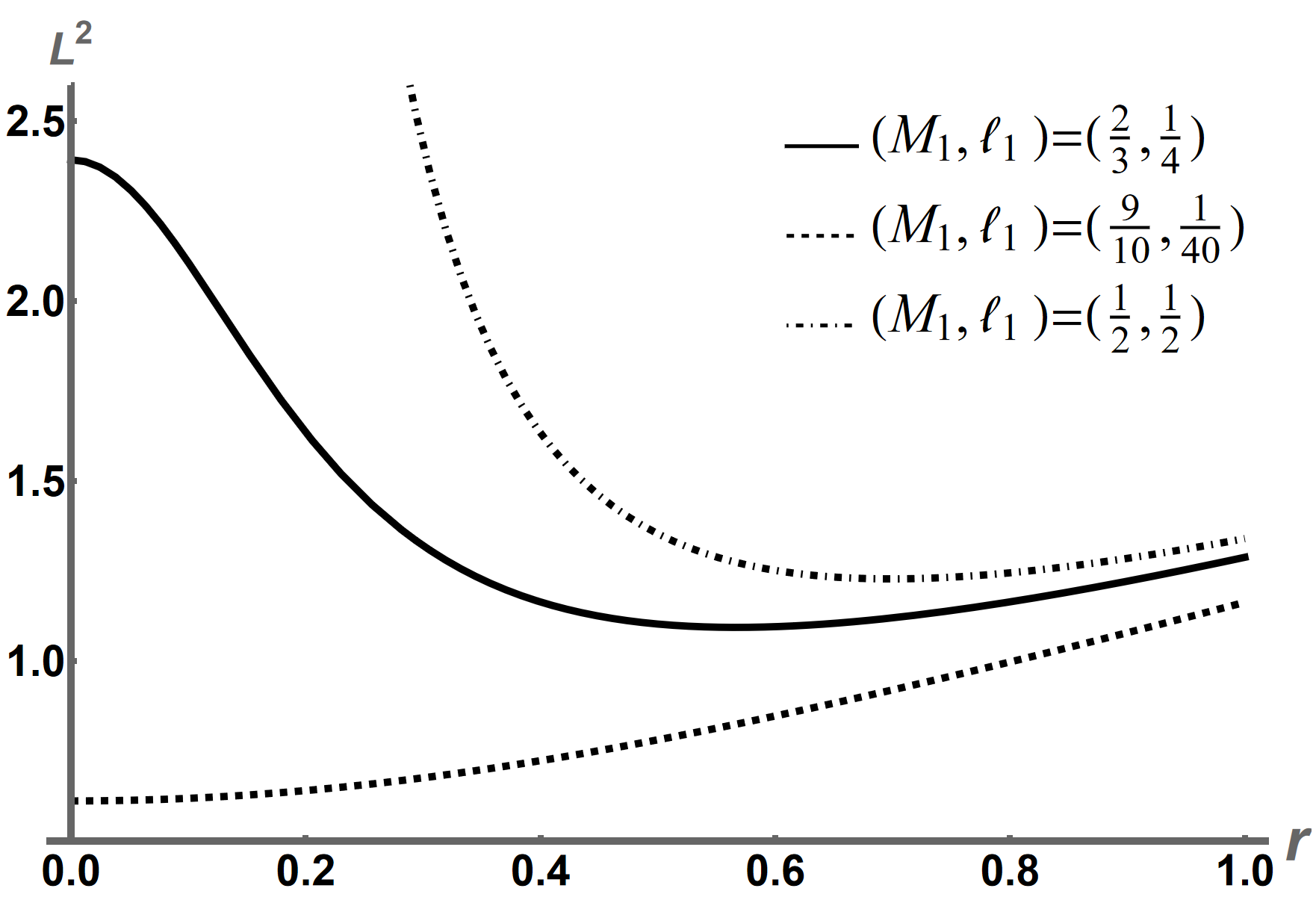}}
	\caption{\\\bfseries Trapezoid $L^2$ Curves}{\vspace{0.8em}\footnotesize Depending upon $(M_1,\ell_1)$, as $L^2$ increases there are either zero, then two, then one RE (solid line);\\zero, then one RE (dashed line); or zero, then two RE (dash-dotted line). \\[-0.851em]}
	\label{fig:LS_Trapezoid}
\end{figure}
The $L^2$ equation \ref{eq:EqMass_LS} for the trapezoid configuration becomes $\frac{(r^2+B_1+B_2)^2}{2}\left(\frac{1}{d^3_{11}}+\frac{1}{d^3_{12}}\right)$. Note geometrically that when the length of the dumbbells are unequal, and $r=0$, the distances $d_{11},d_{12}$ are strictly positive, and by inspection $L^2$ is finitely positive. When the lengths are equal, the distance $d_{11}$ goes to zero, and the expression becomes unbounded (as in Figure \ref{fig:LS_Trapezoid}). 

Taking a derivative, we have:\\ $\partial_rL^2=-\frac{3}{2}r(r^2+B_1+B_2)^2(\frac{1}{d^5_{11}}+\frac{1}{d^5_{12}})+2r(r^2+B_1+B_2)(\frac{1}{d^3_{11}}+\frac{1}{d^3_{12}}).$\\We can see by inspection that for dumbbells of unequal length, this slope goes to zero as $r\rightarrow0$. Additionally, since $d_{ij}\rightarrow r$ as $r\rightarrow\infty$, the expression goes to $1$, and angular momentum is unbounded as $r\rightarrow\infty$.
%-(3/2)r^5(1/d12^5+1/d11^5)+2r^3(1/d12^3+1/d11^3)    ->     -3+4=1.

Looking at the curvature at $r=0$, we calculate $\partial_r^2L|_{r=0}$. However, this gives a very complicated expression, ill-suited to determining the sign.  If instead, we make substitutions \ref{eq:2DB_Col_LS_Curvature_COVs} in $\partial_r^2L|_{r=0}$, then after some simplification we get the following when $\ell_1<\ell_2$: \\
%$\scriptstyle\text{\enspace\enspace\enspace\enspace\enspace}8\ell^7(1-3M)M + 28\ell^6(1-3M)M - 2\ell^5(3+8M+93M^2) - 5\ell^4(3+22M+51M^2) - 10\ell^3(3+16M+21M^2)\\ \text{\enspace\enspace\enspace\enspace\enspace\enspace\enspace\enspace} - 2\ell^2(15+58M+51M^2) - 3\ell(5+14M+9M^2) - 3(1+M)^2$, and\\
$\scriptstyle\text{\enspace\enspace\enspace\enspace\enspace} - 3\ell^7(1 + M)^2 - 3\ell^6(9 + 14 M + 5 M^2) - 
2\ell^5(51 + 58 M + 15 M^2) - 10\ell^4(21 + 16 M + 3 M^2) - 
5\ell^3(51 + 22 M + 3 M^2)\\ \text{\enspace\enspace\enspace\enspace\enspace\enspace\enspace\enspace} - 
2\ell^2(93 + 8 M + 3 M^2) +28\ell(M-3) + 8(M-3)$.\\
Observe by inspection that each term is negative when $M<3$, or equivalently when $M_1<\frac{3}{4}$ (same result occurs for $\ell_2<\ell_1$ when $M_2<\frac{3}{4}$). So we have negative curvature (solid line in Figure \ref{fig:LS_Trapezoid}).  Observe that if $M>3$, the last two terms are positive such that for small $\ell$ ($\ell_1$ small) we have positive curvature (dotted line in Figure \ref{fig:LS_Trapezoid}).

So the boundary behavior matches what we see in the graphs. 
However, proving the shape of the graph between the boundaries for all choices of parameters ($\ell_1,M_1$) is nontrivial. Numerically, we find that the above graphed results reflect the possibilities.  That is, when $M_1<\frac{3}{4}$ or with $\ell_1$ large, we have one minimum, and as $L^2$ increases, we have zero, then two, then one RE. When $M_1>\frac{3}{4}$ and $\ell_1$ small, we have no minimum, and as $L^2$ increases, we have zero, then one RE. And when $\ell_1=\frac{1}{2}$, we have zero, then two RE.

\vspace*{-5mm}\subsubsection{Energetic Stability for Equal Mass}
\label{2DB_EQMass_Energetic_Stability}Now that we have located RE for the equal mass configuration, let us consider its stability. Since we have already characterized energetic stability for the colinear and perpendicular configurations, we will examine only the trapezoid configuration and the asymmetric bifurcation curves with $L^2>0$. As we saw in Section \ref{AmendedPotential}, to determine energetic stability we check if the RE are strict minima of the amended potential $V$. However, for the equal mass configuration, the amended potential $V$'s Hessian is not (in general) block diagonal. So it is not simple to characterize conditions under which $H$ is positive definite.
%Also, depending on $\ell_{1},$ there are several 1D families of RE curving through the configuration space $(r,\theta_{1},\theta_{2})$. In particular, there are symmetric families which consist of the colinear configuration $\mathcal{R}_{C}$ where $\left(\theta_{1},\theta_{2}\right) =\left(0,0\right) $, the perpendicular configurations $\mathcal{R}_{P_{1}},\mathcal{R}_{P_{2}}$ where $\left(\theta_{1},\theta_{2}\right) \in\left\{\left(\frac{\pi}{2},0\right) ,\left(0,\frac{\pi}{2}\right) \right\}$, and the trapezoid configuration $\mathcal{R}_{T}$ where $\left(\theta_{1},\theta_{2}\right) =\left(\frac{\pi}{2},\frac{\pi}{2}\right)$. Their existence and location are independent of $\ell_{1}.$ There are also asymmetric solutions which bifurcate from these symmetric ones. 
%We will explore the stability of the RE families located in the previous section, and particularly how the stability changes as $r$ varies.

However, note that for each choice of $r$, we can consider the two-dimensional trace and determinant of the amended potential $V$'s Hessian \eqref{eq:2DB_Hession}. We can therefore identify the regions where the RE would be 2D maxima, minima, or saddles. In our effort to identify the 3D minima, being a 2D minimum becomes a necessary criteria, and will help us eliminate unstable RE.

After setting $r$, the configuration space consists of a $(\theta_{1},\theta_{2})$ torus. In our analysis below, you can reference the graphs supplied in the equal mass bifurcation section of \ref{2DB_EqualMass_Bifurcation_Analysis}. These will provide visual references consistent with the stability we find below. Note that in addition to the RE graphed in these figures, we also hatch marked the positive regions of the two-dimensional $V$ trace and determinant. This way, the RE are stable if they are in a crosshatched region.

In Figure \ref{fig:EqMassRSlice_elThreeQuarters} you see an example of such a graph ($r=0.018,\;\ell_{1}=\frac{3}{4}$). We see that the only RE in the positive 2D determinant region are $\mathcal{R}_{C}$ and $\mathcal{R}_{P_{1}}$. And of these, only $\mathcal{R}_{C}$ also has a positive 2D trace. So we may conclude that $\mathcal{R}_{C}$ is a 2D minimum, $\mathcal{R}_{P_{1}}$ is a 2D maximum, and $\mathcal{R}_{P_{2}},\mathcal{R}_{T}$ are 2D saddles.
\begin{figure}[H]
	\centering
	\includegraphics[width=0.45\linewidth]{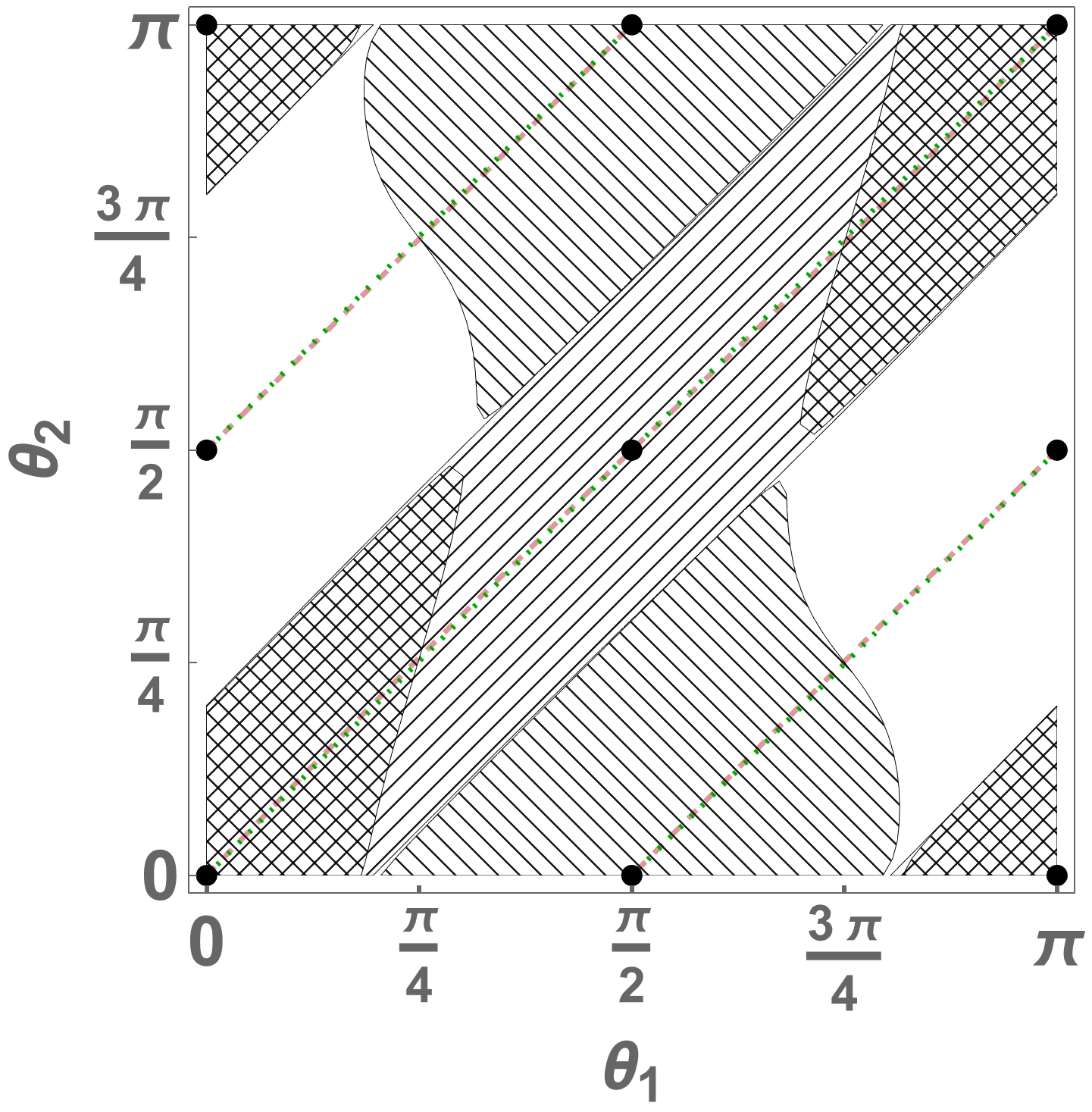}\raisebox{1.3\height}{\includegraphics[width=0.23\textwidth]{images/EqMass_rSlice/Legend.png}}
	\caption{\\\bfseries Equal Mass RE/Trace/Det.\\with $\ell_1=\frac{3}{4}$ for $r=0.018$}
	\label{fig:EqMassRSlice_elThreeQuarters}{\vspace{0.8em}\footnotesize This type of graph reveals the 2D stability of RE.\\We see the stable crosshatched minimum region around $(\theta_{1},\theta_{2})=(0,0)$,\\a maximum at $(\frac{\pi}{2},0)$, and saddles at $(\frac{\pi}{2},\frac{\pi}{2})$, $(0,\frac{\pi}{2})$. \\[-0.851em]}
\end{figure}
As noted earlier, we find $\ell_{1}=\frac{1}{2}$ leads to qualitatively different bifurcations than when $\ell_{1}\neq \frac{1}{2}.$ So let us explore the energetic stability of the RE for these two cases.
\paragraph{Case: $\ell_1\neq \frac{1}{2}$}
Recall the description for the $\ell_{1}\neq \frac{1}{2}$ curves given in the equal mass bifurcation section of \ref{2DB_EqualMass_Bifurcation_Analysis}, where we determined that, besides the colinear and perpendicular which we covered previously, the only curves with physically realizable angular momenta are the trapezoid and bifurcated curves $\mathcal{B}_{TP},\mathcal{B}_{CP}$. So we restrict our analysis to these below.
Looking at the specific case of $\ell_{1}=\frac{3}{4}$, we calculate bifurcation radii of interest as: $\left(r_{1},r_{2},r_{3},r_{4}\right)\approx\left(\frac{1}{4},0.3600,0.3630,0.3686\right)$, and $\left(r_{5},r_{6},r_{7},r_{8}\right) \approx\left(0.3812,0.3865,0.3893,\frac{1}{2}\right)$, (consistent with the figures in the equal mass bifurcation section).

In the tables below, we map the signs of the full 3D eigenvalues ($Sgn\left(e_{r},e_{\theta_{i}}\right) \in\left(\pm ,\pm ,\pm\right) $) for each curve as $r$ varies through the above mentioned radii. These signs are consistent with the graphs in the equal mass bifurcation section of \ref{2DB_EqualMass_Bifurcation_Analysis}. These signs are found by looking at the eigenvalues of the Hessian at the numerically located RE in the configuration space. Since the signs are from the 3D Hessian, the sign of the radial eigenvalue seen in the tables below will change at radii not indicated in our 2D images above. By comparing the 2D to the 3D eigenvalues, we can also discern which is the radial eigenvalue.

\textbf{Trapezoid}\begin{adjustwidth}{0.5cm}{0cm}In addition to the bifurcation seen above at $r_{2}\approx 0.360$, for some $r_{t}$ (which depends upon $M_{1}$), the trapezoid radial eigenvalue turns from negative to positive. 
%Also of note is that angular momentum through these radii is always real, and bounded away from infinity. 
Below is a table of the signs of $H$'s eigenvalues \eqref{eq:2DB_Hession} as $r$ ranges through the relevant radii. Observe we find no stability ($Sgn\left(e_{r},e_{\theta_{i}}\right)\ne\left(+,+,+\right)$).
%[Calculations documented at 2\_23\_21 in EqualMassConfig\_2\_23\_2021.nb]
\end{adjustwidth}
\captionof{table}{\\Equal Mass Trapezoid Eigensigns with $\ell_1\ne\frac{1} {2}$}\begin{table}[h]\centering\begin{tabular}{|l|c|c|c|c|c|c|}
	\hline
	$\left(r,\theta_{1},\theta_{2}\right) $ & $\left(0,\frac{\pi}{2},\frac{\pi }{2}\right) $ & $\left(0^{+},\frac{\pi}{2},\frac{\pi}{2}\right) $ & $\left(r_{2}^{-},\frac{\pi}{2},\frac{\pi}{2}\right) $ & $\left(r_{2}^{+},\frac{\pi}{2},\frac{\pi}{2}\right) $ & $\left(r_{t}^{-},\frac{\pi}{2},\frac{\pi}{2}\right) $ & $\left(r_{t}^{+},\frac{\pi}{2},\frac{\pi}{2}\right) $ \\ \hline
	\multicolumn{1}{|r|}{$\mathcal{R}_{T}$ signs} & $\left(0,0,+\right) $ & $\left(-,-,+\right) $ & $\left(-,-,0^{+}\right) $ & $\left(-,-,0^{-}\right) $ & $\left(0^{-},-,-\right) $ & $\left(0^{+},-,-\right) $ \\ \hline
\end{tabular}\end{table}
$\vspace{10pt}$\newpage
Below you find visualizations of the dumbbell configuration for a couple of radii in this range.
\begin{figure}[H]
	\captionsetup[subfigure]{justification=centering}
	\centering
	\subfloat[$r=0$:\\$L^2$ singular.]
	{\includegraphics[width=0.12\textwidth]{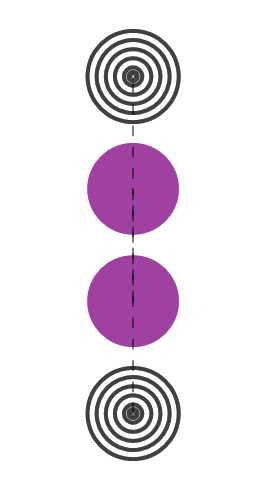}}
	\hspace{2cm}
	\subfloat[$r_2: $ Bifurcation to $B_{TP}$.\\$L^2>0$]
	{\includegraphics[width=0.29\textwidth]{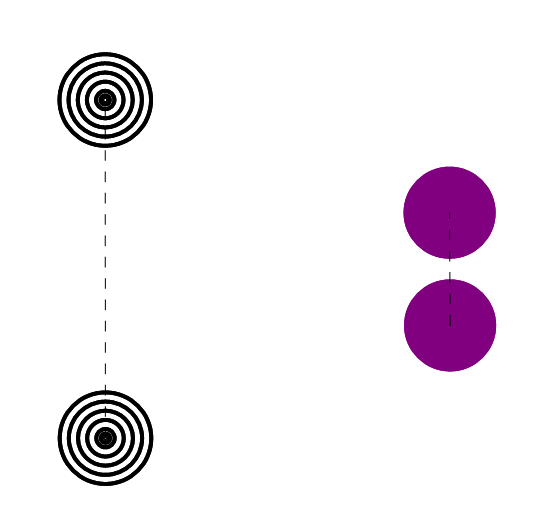}}
	\caption{\\\bfseries Equal Mass Trapezoid Visualization with $\ell_1\ne\frac{1}{2}$}{\vspace{0.8em}\footnotesize Finite positive angular momenta when $r>0$, but no stability.\\[0.191em]}
	\label{fig:EQMassTrap_l10p75_r0p0_r_2}
\end{figure}

\textbf{TP Bifurcation}
\begin{adjustwidth}{0.5cm}{0cm}
We saw bifurcations in Figures \ref{fig:EqMassRSlice_l10p75_r0p252_r0p362}b, \ref{fig:EqMassRSlice_l10p75_r0p388_r0p498}a at $\left(r_{2},r_{7}\right) \approx\left(0.3600,0.3893\right).$ 
%Angular momenta through these radii is always real. 
Below is a table of the signs of $H$'s eigenvalues \eqref{eq:2DB_Hession} as $r$ ranges through the relevant radii. Note that the notation $0^{\{-,+,-\}}$ is used below to imply that the eigenvalue is near zero, but takes on positive or negative values depending upon the value of $M_{1}.$ In particular, $0^{\{-,+,-\}}$ is meant to imply that for smaller values of $M_{1},$ the eigenvalue is negative, for middle range values of $M_{1}$ the eigenvalue is positive, and for larger values of $M_{1}$ the eigenvalue is again negative. Observe we find no stability ($Sgn\left(e_{r},e_{\theta_{i}}\right)\ne\left(+,+,+\right)$).
%[Calculations documented at 2\_24\_21 in BatonReduction\_manual\_cube\_roots\_c.nb]
\end{adjustwidth}
\captionof{table}{\\Equal Mass Trapezoid$\rightarrow$Perpendicular Eigensigns with $\ell_1\ne\frac{1} {2}$}\begin{table}[h]\centering\begin{tabular}{|l|c|c|c|}
	\hline
	$\left(r,\theta_{1},\theta_{2}\right) $ & $\left(r_{2}^{+},\frac{\pi}{2},\frac{\pi}{2}\right) $ & $\left(r_{7}^{-},\frac{\pi}{2},0\right) /\left(
	r_{7}^{-},0,\frac{\pi}{2}\right),\text{\small radial eigenroots at } M_{1}=\frac{1}{10},\frac{1}{2}$ & $r_{7}^{+}$ \\ \hline
	\multicolumn{1}{|r|}{$\mathcal{B}_{TP^{\pm }}$ signs} & $\left(
	-,-,0^{+}\right) $ & $\left(0^{\{-,+,-\}},-,0^{+}\right) $ & N/A \\ \hline
\end{tabular}\end{table}\newpage
Below you find visualizations of the dumbbell configuration for various radii in this range.
\begin{figure}[H]
	\captionsetup[subfigure]{justification=centering}
	\centering
	\subfloat[$r_2$: Bifurcating from $\mathcal{R}_{T}$. \\$L^2>0. $]
	{\includegraphics[width=0.27\textwidth]{images/EQMassTrapBif/l10p75_Num1.png}}
	\hspace{.4cm}
	\subfloat[$r\approx0.3788$: \\$L^2>0.$]
	{\includegraphics[width=0.3\textwidth]{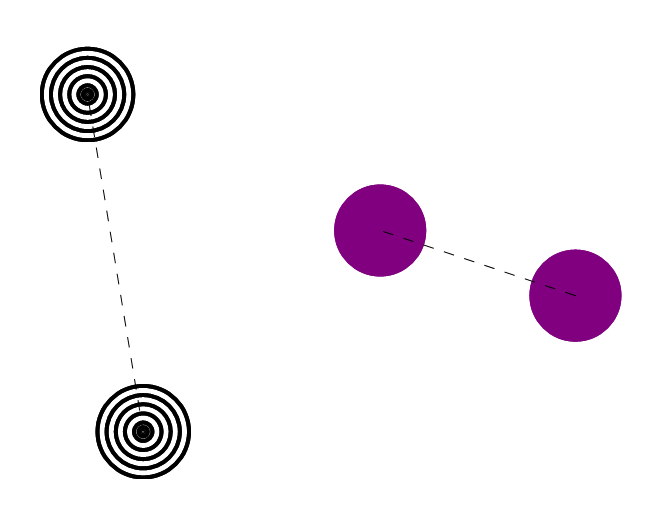}}
	\hspace{.4cm}
	\subfloat[$r_7$: Merging with $\mathcal{R}_{P_{1}}$. \\$L^2>0.$]
	{\includegraphics[width=0.3\textwidth]{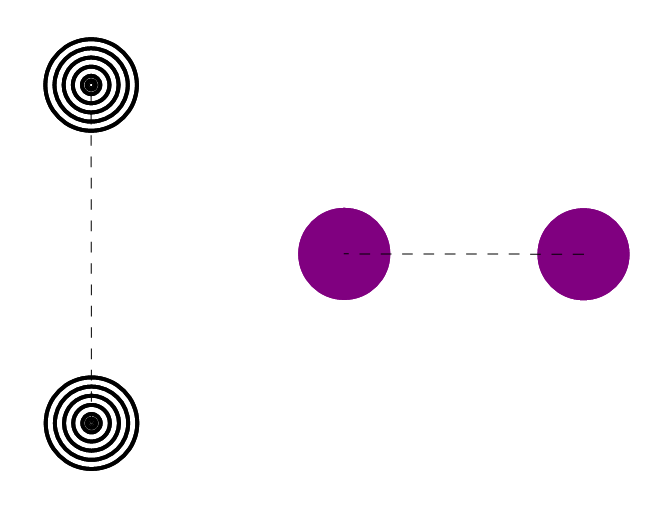}}
	\caption{\\\bfseries Equal Mass Trapezoid$\rightarrow$Perpendicular Visualization with $\ell_1\ne\frac{1}{2}$}{\vspace{0.8em}\footnotesize Finite positive angular momenta, but no stability.\\[0.191em]}
	\label{fig:EQMassTrapToPerpBif_l10p75}
\end{figure}

\textbf{CP Bifurcation}
\begin{adjustwidth}{0.5cm}{0cm}
We saw bifurcations in Figures \ref{fig:EqMassRSlice_l1_NotHalf_r4}, \ref{fig:EqMassRSlice_l1_NotHalf_r6} at $\left(r_{4},r_{6}\right) \approx\left(0.3686,0.3865\right).$ 
%Angular momentum through these radii is always real. 
Below is a table of the signs of $H$'s eigenvalues \eqref{eq:2DB_Hession} as $r$ ranges through these radii. Observe we find no stability ($Sgn\left(e_{r},e_{\theta_{i}}\right) \ne\left(+,+,+\right)$).
%[Calculations documented at 2\_24\_21 in BatonReduction\_manual\_cube\_roots\_c.nb]
\end{adjustwidth}
\captionof{table}{\\Equal Mass Colinear$\rightarrow$Perpendicular Eigensigns with $\ell_1\ne\frac{1} {2}$}\begin{table}[h]\centering\begin{tabular}{|l|c|c|c|}
	\hline
	$\left(r,\theta_{1},\theta_{2}\right) $ & $\left(r_{4}^{+},0^{+},0^{+}\right) $ & $\left(r_{6}^{-},0^{+},\frac{\pi}{2}^{-}\right) $ & $r_{6}^{+}$ \\ \hline
	\multicolumn{1}{|r|}{$\mathcal{B}_{CP^{\pm }}$ signs} & $\left(-,+,0^{-}\right) $ & $\left(0^{-},+,-\right) $ & N/A \\ \hline
\end{tabular}\end{table}
Below you find visualizations of the dumbbell configuration for various radii in this range.
\begin{figure}[H]
	\captionsetup[subfigure]{justification=centering}
	\centering
	\subfloat[$r_4$: Bifurcating from $\mathcal{R}_{C}$. \\$L^2>0$]
	{\includegraphics[width=0.3\textwidth]{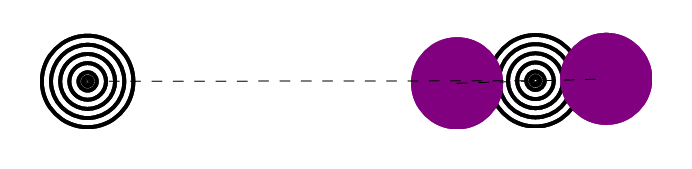}}
	\hspace{0.4cm}
	\subfloat[$r=0.382836$.  \\$L^2>0$]
	{\includegraphics[width=0.3\textwidth]{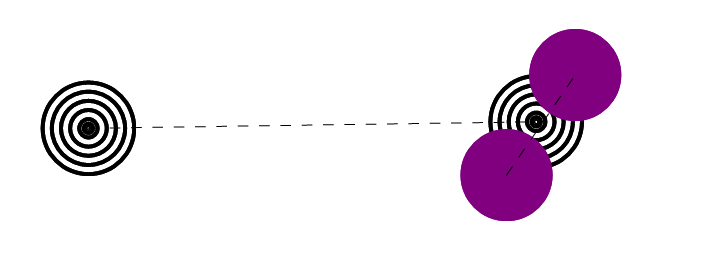}}
	\hspace{0.4cm}
	\subfloat[$r_6$: Merging with $\mathcal{R}_{P_2}$. \\$L^2>0$]
	{\includegraphics[width=0.3\textwidth]{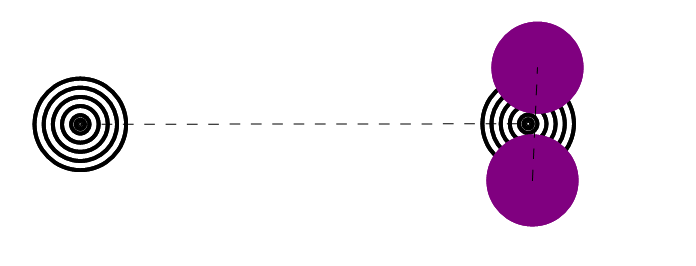}}
	\caption{\\\bfseries Equal Mass Colinear$\rightarrow$Perpendicular Visualization with $\ell_1\ne\frac{1}{2}$}{\vspace{0.8em}\footnotesize Finite positive angular momenta, but no stability.\\[0.191em]}
	\label{fig:EQMassColToPerpBif_l10p75}
\end{figure}

\textbf{Stability Conclusions}\\
In our effort to find energetic stability, we were looking for eigenvalues $\left(e_{r},e_{\theta_{i}}\right) $ such that $Sgn\left(e_{r},e_{\theta_{i}}\right) =\left(+,+,+\right)$, which are associated with strict minima. In Section \ref{2DB_Colinear_Energetic_Stability}, we saw that $\mathcal{R}_{C},$ with sufficiently large radius such that $\partial_rL^2>0,$ is energetically stable. We note from above that as $r$ varies, none of the other RE curves' signs become $\left(+,+,+\right) $. For the most part, we find saddles. But for the interval $\left(0,r_{7}\right) $ of $\mathcal{R}_{P_{1}}$, for certain $M_{1}$ values we have maxima. Similarly we find maxima for $\mathcal{R}_{T}$ on $\left(r_{2}^{+},\frac{\pi}{2},\frac{\pi}{2}\right)$. Although, as has been noted previously, the positive definite nature of the kinetic energy in the Hamiltonian makes these RE saddles in the energy manifold.

\paragraph{Case: $\ell_1=\frac{1}{2}$}
Now we will take a look at $\ell_{1}=\frac{1}{2}$, since it has qualitatively different bifurcation curves than $\ell_{1}\ne\frac{1}{2}$. Recall the description for the $\ell_{1}=\frac{1}{2}$ curves given in the equal mass bifurcation section of \ref{2DB_EqualMass_Bifurcation_Analysis} where we determined that, besides the colinear and perpendicular which we covered previously, the only curves with physically realizable angular momenta are the trapezoid and $\mathcal{B}_{LP/RP}$. So we restrict our analysis to these below.

In the following tables, we look at each of these curves, and map the signs of the full 3D eigenvalues ($Sgn\left(e_{r},e_{\theta_{i}}\right) \in\left(\pm ,\pm ,\pm\right) $) as $r$ varies through the above mentioned radii. These signs are consistent with the graphs in the equal mass bifurcation section of \ref{2DB_EqualMass_Bifurcation_Analysis}. These signs are found by looking at the eigenvalues of the Hessian at numerically located RE in the configuration space. Note that the signs of the 3D radial eigenvalue, seen in the tables below, will change at radii not suggested in our 2D images above.

\textbf{Trapezoid}
\begin{adjustwidth}{0.5cm}{0cm}
In addition to the bifurcation seen in Figure \ref{fig:EqMassRSlice_l1_Half_r0p01} at $r=0$, for some $r_{t}$ (which depends upon $M_{1}$), the trapezoid radial eigenvalue turns from negative to positive. 
%[Couldn't show, documented at 2\_16\_21 b, but $r_{t}$ appears to always be $>0.698412$] Also of note is that the angular momentum through these radii is always real.
Below is a table of the signs of $H$'s eigenvalues \eqref{eq:2DB_Hession} as $r$ ranges through the relevant radii. Observe we find no stability ($Sgn\left(e_{r},e_{\theta_{i}}\right) \neq\left(+,+,+\right)$).
%[Calculations documented at 2\_23\_21 in EqualMassConfig\_2\_23\_2021.nb]
\end{adjustwidth}
\captionof{table}{\\Equal Mass Trapezoid Eigensigns with $\ell_1=\frac{1} {2}$}\begin{table}[h]\centering\begin{tabular}{|l|c|c|c|c|}
	\hline
	$\left(r,\theta_{1},\theta_{2}\right) $ & $\left(0,\frac{\pi}{2},\frac{\pi }{2}\right) $ & $\left(0^{+},\frac{\pi}{2},\frac{\pi}{2}\right) $ & $\left(r_{t}^{-},0,0\right) $ & $\left(r_{t}^{+},0,0\right) $ \\ \hline
	\multicolumn{1}{|r|}{$\mathcal{R}_{T}$ signs} & collision & $\left(-,-,0^{-}\right) $ & $\left(0^{-},-,-\right) $ & $\left(
	0^{+},-,-\right) $ \\ \hline
\end{tabular}\end{table}\newpage

Below you find visualizations of the dumbbell configuration for a couple of radii in this range.
\begin{figure}[H]
	\captionsetup[subfigure]{justification=centering}
	\centering
	\subfloat[Near collision.\\$L^2\rightarrow+\infty$ as $r\rightarrow0$.]
	{\includegraphics[width=0.22\textwidth]{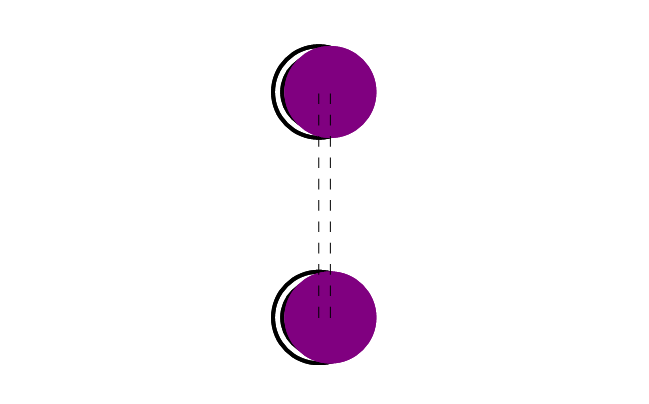}}
	\hspace{1cm}
	%\subfloat[$r\approx0.179152$: \\$L^2>0$.]
	\subfloat[$r\approx0.1792$: \\$L^2>0$.]
	{\includegraphics[width=0.22\textwidth]{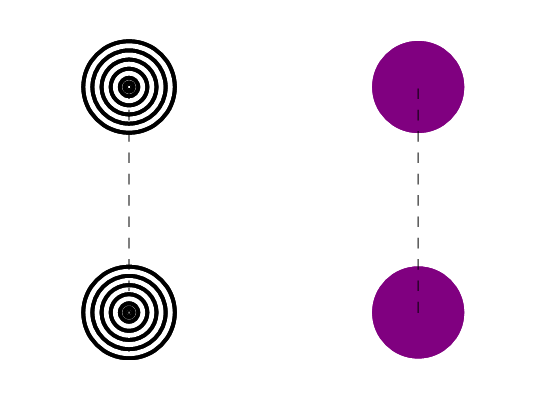}}
	\caption{\\\bfseries Equal Mass Trapezoid Visualization with $\ell_1=\frac{1}{2}$}{\vspace{0.8em}\footnotesize Finite positive angular momenta when $r>0$, but no stability.\\[-0.191em]}
	\label{fig:EQMassTrapBifurcation}
\end{figure}

\textbf{Perpendicular Bifurcated}
\begin{adjustwidth}{0.5cm}{0cm}
We saw a bifurcation in Figure \ref{fig:EqMassRSlice_l1_Half_r2} at $r_{2}\approx 0.3377.$ 
%Angular momenta are real throughout this curve. 
Below is a table of the signs of $H$'s eigenvalues \eqref{eq:2DB_Hession} as $r$ ranges through the relevant radii. Observe we find no stability ($Sgn\left(e_{r},e_{\theta_{i}}\right)\neq\left(+,+,+\right)$).
%[Calculations documented at 2\_24\_21 in BatonReduction\_manual\_cube\_roots\_c.nb]
\end{adjustwidth}
\captionof{table}{\\Equal Mass Perpendicular Bifurcation Eigensigns with $\ell_1=\frac{1} {2}$}\begin{table}[h]\centering\begin{tabular}{|l|c|}
	\hline $\left(r,\theta_{1},\theta_{2}\right) $ & $\left(0,\cot ^{-1}\sqrt{2},\cot ^{-1}\sqrt{2}\right) /\left(0,\cos ^{-1}\left(-\sqrt{\frac{2}{3}}\right) ,\cos ^{-1}\left(-\sqrt{\frac{2}{3}}\right)\right) $
	  \\ \hline
	\multicolumn{1}{|r|}{$\mathcal{B}_{LP^{\pm }}/\mathcal{B}_{RP^{\pm }}$ signs} & collision\\ \hline
\end{tabular}\end{table}
$\vspace{20pt}\text{\enspace\enspace\enspace\enspace}$
\begin{tabular}{|c|c|c|}
	\hline
	$0^{+}$ & $r_{2}^{-}$ & $r_{2}^{+}$  \\ \hline
	 $\left(-,+,-\right) $ & $\left(0^{-},+,-\right) $ & N/A \\ \hline
\end{tabular}\\
Below you find visualizations of the dumbbell configuration for various radii in this range.
\begin{figure}[H]
	\captionsetup[subfigure]{justification=centering}
	\centering
	\subfloat[Near Collision: \\$L^2\rightarrow+\infty$ as $r\rightarrow0$.]
	{\includegraphics[width=0.17\textwidth]{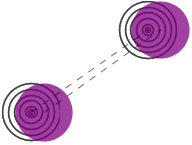}}
	\hspace{0.2cm}
	\subfloat[$r=0.2882$: \\$L^2>0$.]
	{\includegraphics[width=0.22\textwidth]{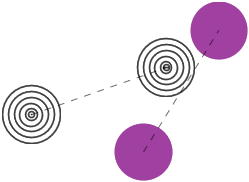}}
	\hspace{0.2cm}
	\subfloat[$r=0.3306$: \\$L^2>0$.]
	{\includegraphics[width=0.215\textwidth]{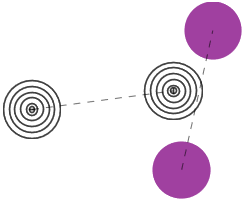}\hspace{10pt}}
	\hspace{0.2cm}
	\subfloat[$r_2$: Merging with $\mathcal{R}_{P_2}$. \\$L^2>0$.]
	{\includegraphics[width=0.205\textwidth]{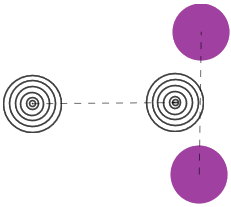}}
	\caption{\\\bfseries Equal Mass Perpendicular Bifurcation  Visualization with $\ell_1=\frac{1}{2}$}{\vspace{0.8em}\footnotesize Finite positive angular momenta when $r>0$, but no stability.\\[-0.191em]}
	\label{fig:EQMassPerpBifurcation_l10p5}
\end{figure}

\newpage
\textbf{Stability Conclusions for Equal Mass}\\
\label{Stability_Conclusions_for_Equal_Mass}In our effort to find energetic stability, we were looking for eigenvalues $\left(e_{r},e_{\theta_{i}}\right) $ such that $Sgn\left(e_{r},e_{\theta_{i}}\right) =\left(+,+,+\right)$, which are associated with strict minima and stability. In Section \ref{2DB_Colinear_Energetic_Stability}, we saw that $\mathcal{R}_{C},$ with sufficiently large radii such that $\partial_rL^2>0,$ is energetically stable. We note from above that as $r$ varies, none of the other RE curves' signs become $\left(+,+,+\right) $. For the most part, we find saddles in $V$. But we do find maxima in $V$ for low radius $\mathcal{R}_{T}$, which are saddles for the Hamiltonian when considering the positive definite kinetic energy.

\subsubsection{Linear Stability for Equal Mass}
%[Documented 8_11_21 BatonReduction_manual_cube_roots_d]
Despite finding energetic stability in only the colinear case, when we map the linear stability criteria \eqref{eq:2DB_Linear_Stability_Criteria} to the equal mass configuration, we see some linear stability for each of the symmetric, but none of the asymmetric cases. As we have already examined linear stability for the colinear and perpendicular configurations, we will restrict ourselves to the trapezoid. Particularly, for small $r$ in the $\theta_{1}\theta_{2}$-plane (see Figure \ref{fig:EqMassLinearization_1}) we have linear stability in the trapezoid case when $\ell_{1}\neq \frac{1}{2}$. Below are graphs in the $\theta_{1}\theta_{2}$ and $rM_1$-planes at $r,\ell_{1},M_{1}$ values where linear stability exists.

In Figure \ref{fig:EqMassLinearization_RM1_Trap} we note a lack of stability for the trapezoid configuration when $\ell_{1}=\frac{1}{2}.$ We also see that as $\ell _{1}\rightarrow 1,$ stability appears to converge to those found by Beletskii and Ponomareva \cite{beletskii1990} for the dumbbell/point mass problem in their fig 5. All linear stability occurs at low $r.$ Similar figures for $\ell_1<\frac{1}{2}$ exist requiring $M_2$ large when $\ell_2$ large.

Physically speaking, linear stability here requires that if one body is long, it is also the massive body.  This would require the shorter body to be less massive.  This seems to fit well with how mass and size tend to work in real life, and so does not impose a requirement which is difficult to satisfy.
\paragraph{Equal Mass Linear Stability Conclusions}
For the equal mass configuration, in addition to our previous colinear and perpendicular results,
we found linear stability for the trapezoid case when $r$ is small and $\ell_1\neq\frac{1}{2}$, but none for the asymmetric RE curves.
\\\\
As referenced earlier, we now provide a theorem which geometrically restricts the location of planar rigid bodies when in planar RE with a dumbbell.  We noted earlier how the RE located in this paper obeyed these restrictions.
\begin{figure}[H]
	\captionsetup[subfigure]{justification=centering}
	\centering
	{\includegraphics[width=0.4\textwidth]{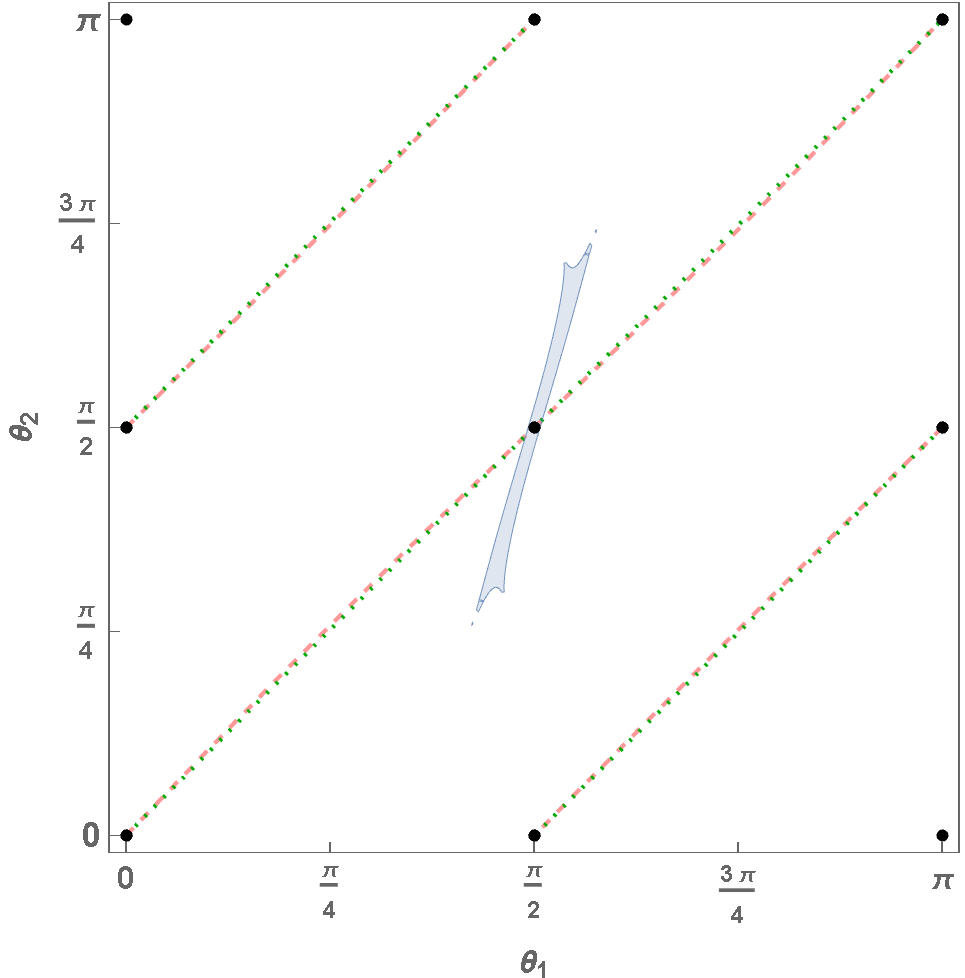}\label{fig:EqMassLinearization_1A}}
	\caption{\\\bfseries Equal Mass Trapezoid Linear Stability in $\theta_1\theta_2$-plane}
{\vspace{0.8em}\footnotesize $\ell_1=\frac{3}{4}, M_1=\frac{1}{2},r=0.02$\\$V_{\theta_i}=0$ (dashed lines), RE (black dots), Linear Stability (shaded regions).\\Graph shows linear stability for a small radius.\\[-1.591em]}
	\label{fig:EqMassLinearization_1}
\end{figure}

\begin{figure}[H]
	\captionsetup[subfigure]{justification=centering}
	\centering
	\subfloat[$\ell _{1}=\frac{1}{2}$]
	{\includegraphics[width=0.27\textwidth]{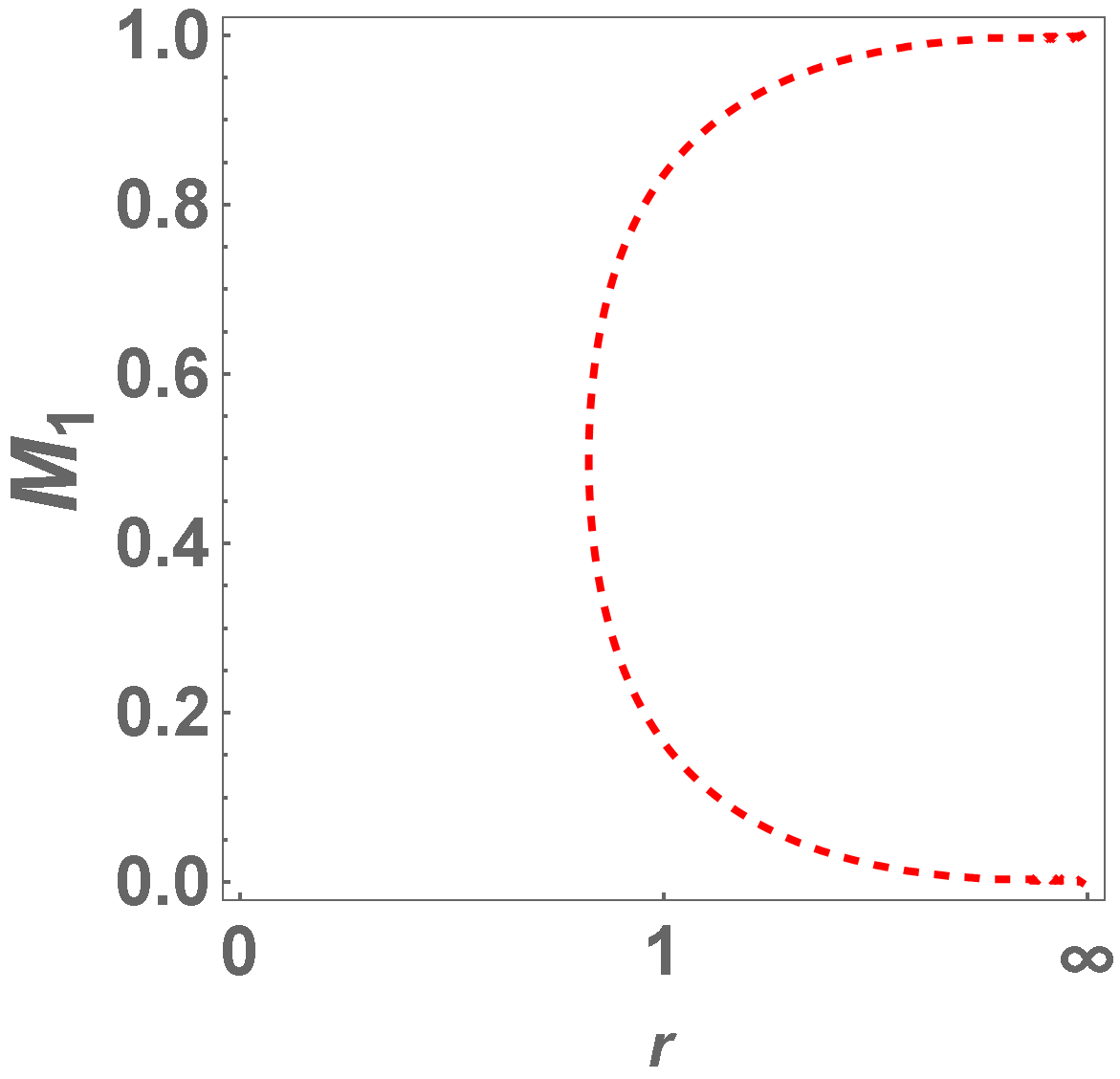}}
	\hspace{1cm}
	\subfloat[$\ell _{1}=0.8$]
	{\includegraphics[width=0.27\textwidth]{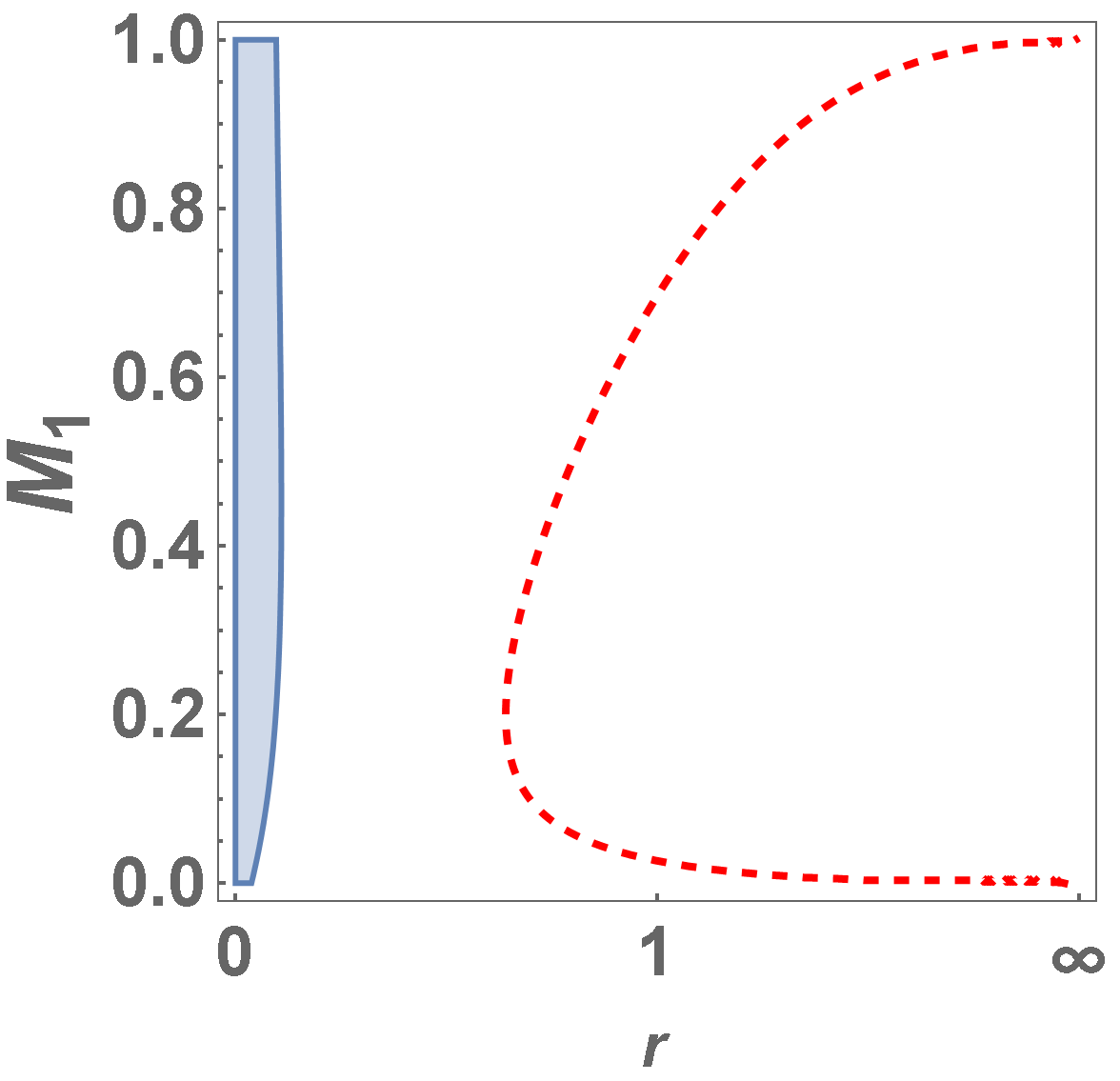}}\\
	\subfloat[$\ell _{1}=0.9$]
	{\includegraphics[width=0.27\textwidth]{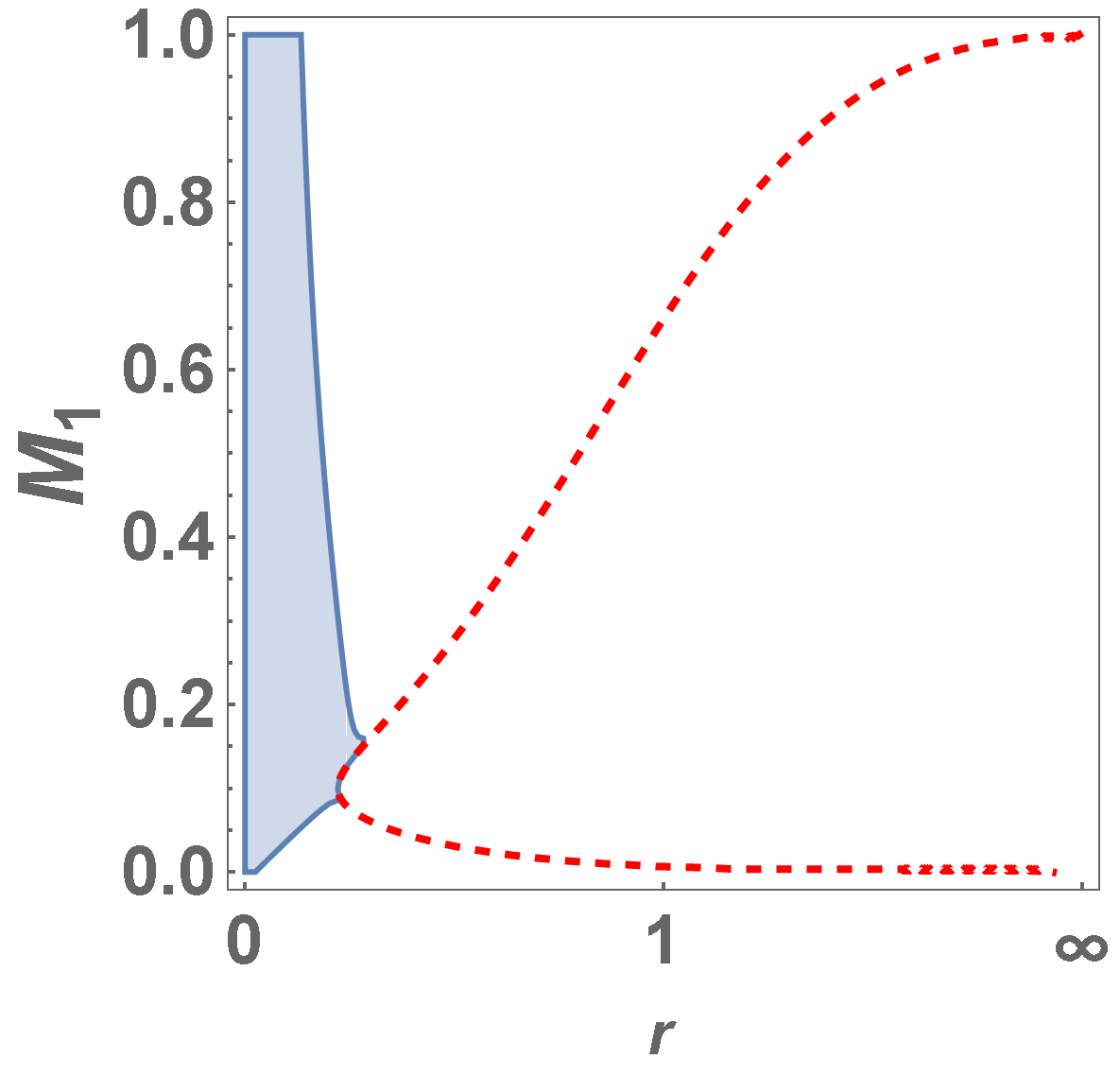}}
	\hspace{1cm}
	\subfloat[$\ell _{1}=0.95$]
	{\includegraphics[width=0.27\textwidth]{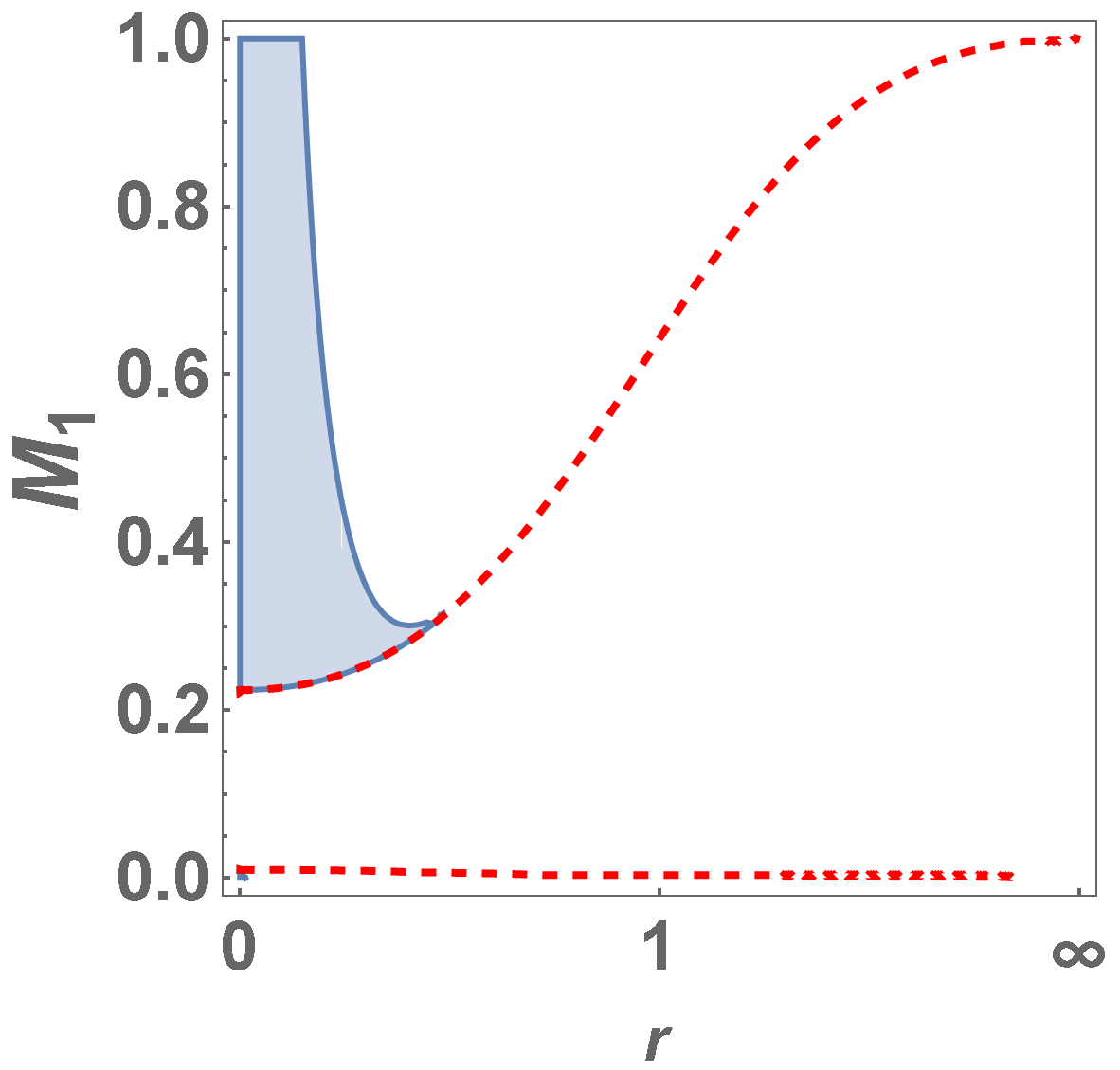}}
	\caption{\\\bfseries Equal Mass Trapezoid Linear Stability in $rM_1$-plane}{\vspace{0.8em}\footnotesize $\partial_r L^2=0$ (dashed lines), Linear Stability (shaded regions).\\Graphs show linear stability for low radii, bounded by $\partial_r L^2=0$.\\[-0.891em]}
	\label{fig:EqMassLinearization_RM1_Trap}
\end{figure}

%%%%%%%%%%%%%%%%%%%%%%%%%%%%%%%%%%%%%%%%%%%%%%%%%%%%%%%%%%%%%%%%%%%%%%%%%%%%%%%%
\newpage\section{Perpendicular Bisector Theorem for RE of a Dumbbell and Rigid Planar Bodies}\label{PerpBisThm}
In 1990, Conley and Moeckel developed the perpendicular bisector theorem which restricts the possible geometries of central configurations \cite{Moeckel1990}. For each pair of point masses $\vec{r}_i,\vec{r}_j$, the theorem asks one to consider the four quadrants formed by the line containing $\vec{r}_i,\vec{r}_j$, and its perpendicular bisector. The hourglass shape which is formed from the union of the 1st and 3rd quadrants is called a cone, similarly with the 2nd and 4th quadrants. The term ``open cone'' refers to a cone minus the axes.\\
\begin{theorem}[Perpendicular Bisector Theorem] Let $q = \left\{\vec{r}_1, \vec{r}_2,...\right\}$ be a planar central configuration and let $\vec{r}_i$ and $\vec{r}_j$ be any two of its points. Then, if one of the two open cones determined by the line through $\vec{r}_i$ and $\vec{r}_j$ and its perpendicular bisector contains points of the configuration, so does the other one.
\end{theorem}\vspace{1em}

We will prove an extension of this theorem as it relates to a dumbbell and several rigid planar bodies in planar RE.  In particular, the theorem will also apply to discretized bodies.  A discretized body is one which consists of point masses, all connected by massless rods, with a point mass body being trivially discretized. And of course, a dumbbell is a discretized body.  
\begin{wrapfigure}{r}{0.41\textwidth}
	\begin{minipage}{1\linewidth}
		\centering\captionsetup[subfigure]{justification=centering}
		%{\includegraphics[width=.9\textwidth]{images/Perp_Biff_Rigid_Bodies.png}}
		%{\footnotesize(a) Arbitrary System Rotation}
		{\includegraphics[width=.9\textwidth]{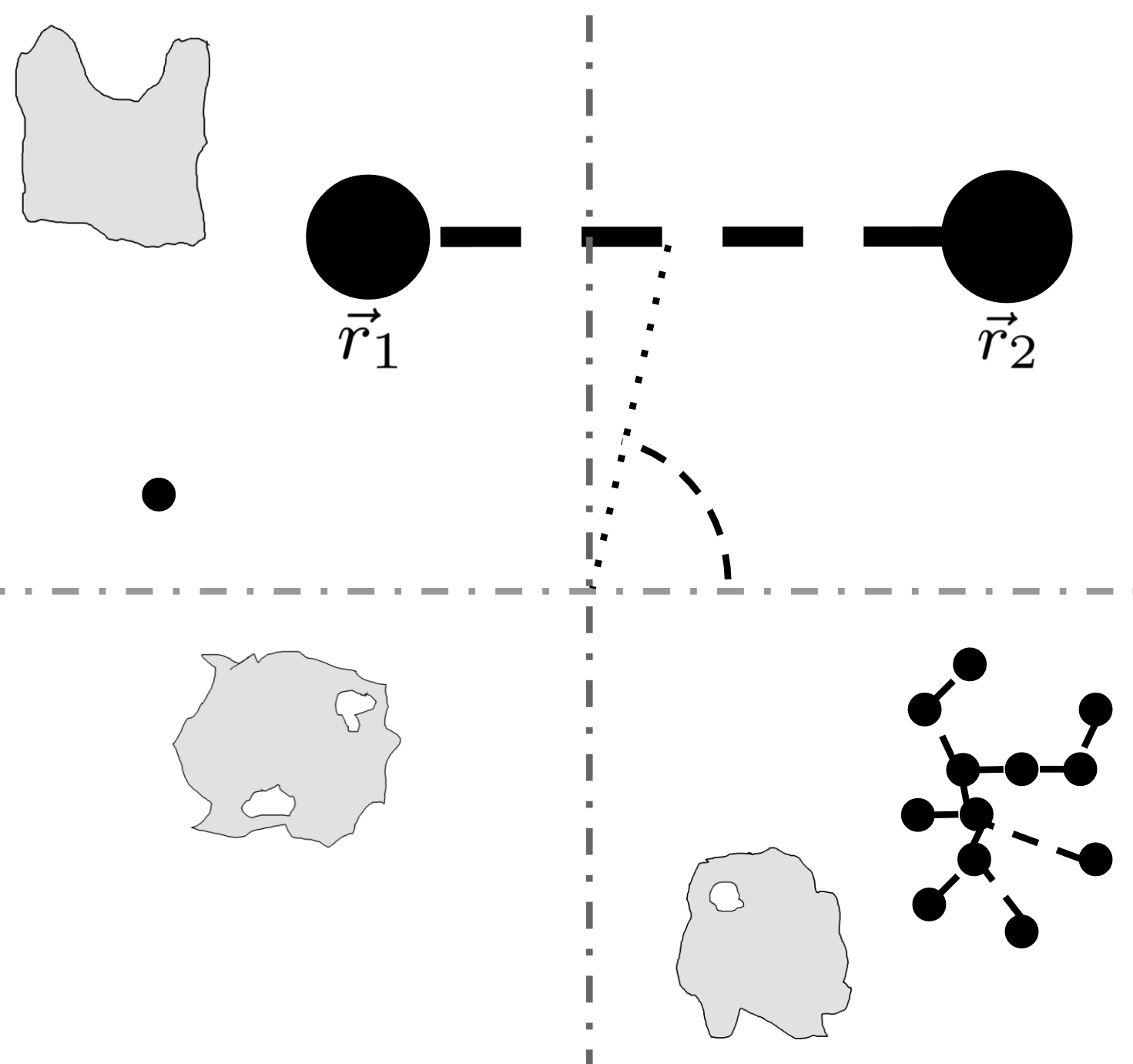}}
		%{\footnotesize(b) Dumbbell parallel to horizontal}
		\end{minipage}
	\caption{\\\bfseries Dumbbell and Rigid Bodies}{\vspace{0.8em}\footnotesize Initial system rotation chosen such that the dumbbell is parallel to the horizontal axis.\\[-4.8em]}
	\label{fig:Perp_Biff_Rigid_Rotated}
\end{wrapfigure}

For this analysis, reference Figure \ref{fig:Perp_Biff_Rigid_Rotated}. As before, we let $\vec{r}_1$ and $\vec{r}_2$ be the locations of the dumbbell's masses (with mass ratios $x_1,x_2$). We denote the body mass of the dumbbell as $M_1$, the dumbbell body as $\mathcal{B}_1$, and the other rigid bodies as $\mathcal{B}_2,...,\mathcal{B}_n$. We assume a reference frame rotating such that a RE configuration will be at equilibrium. Note that the dynamics do not depend upon our choice initial system rotation, so for convenience of calculation, and without loss of generality, we choose this rotation such that the dumbbell is parallel to the horizontal axis, see Figure \ref{fig:Perp_Biff_Rigid_Rotated}.\newpage
\begin{theorem}[Perpendicular Bisector Theorem for RE of a Dumbbell and Rigid Planar Bodies]
	\label{Perpendicular_Bisector_Theorem_for_Dumbbell_and_Rigid_Bodies}
	Let a dumbbell $\overline{r_1r_2}$ and one or more planar rigid bodies $\mathcal{B}_2,...,\mathcal{B}_n$ be in a planar RE. Then, if one of the two open cones determined by the lines through $\overline{r_1r_2}$ and its perpendicular bisector contains one or more rigid bodies, the other open cone cannot be empty.
\end{theorem}
\begin{proof}
The idea of the proof is to calculate the rotational acceleration $\ddot{\theta}$ of the dumbbell, and to show that it will be nonzero if the rigid bodies are contained in just one open cone. Note that each of the point masses on the dumbbell experiences acceleration due to gravitation and centrifugal forces. Our choice of system rotation has positioned our dumbbell such that the rotational acceleration is determined by the second of our vector components (perpendicular to the dumbbell). Therefore, to determine the rotational acceleration of the dumbbell, we subtract the accelerations of the two point masses, and look at the vertical component. 
\vspace*{-0mm}\subsection{Centrifugal Force}\label{Discrete Body Centrifugal Force}
Note that the centrifugal forces for our masses are:\\ $\text{\hskip20pt}\vec{F}_{c_{i}}=\frac{M_1x_i\protect\vec{v}_i^{\text{\hskip2pt}2}}{\left\vert\protect\vec{r}_i\right\vert^2}\protect\vec{r}_i$$=M_1x_i\overset{\cdot}{\varphi}^{2}\vec{r}_{i} =:M_{1}x_{i}\overset{\cdot}{\varphi}^{2}\left(r_{ix},r_{iy}\right)$. \\And the accompanying centrifugal accelerations can be written as $k\left(r_{ix},r_{iy}\right)$, with $k$ depending upon rotation rate. The centrifugal accelerations perpendicular to the massless rod are $kr_{iy}.$
Due to the dumbbell being parallel with the horizontal axis, note that the second (vertical) component of $\vec{r}_{i}$ does not depend on $i$. Therefore, when we subtract these components to determine the rotational acceleration due to centrifugal force ($\ddot{\theta}_{c}=kr_{2y}-kr_{1y}$), the terms cancel and we find centrifugal force does not contribute to rotational acceleration. 
%The calculations for centrifugal force done in Theorem \ref{Discrete_Dumbbell_Theorem} can be applied here unaltered.  In other words, the perpendicular components of the centrifugal forces applied to the two point masses cancel each other out and therefore don't contribute to the rotational acceleration.
\vspace*{-0mm}\subsection{Gravitational Force}
Observe that the gravitational effect of $\vec{r}_{1}$ on $\vec{r}_{2}$ is exactly canceled out by an equal and opposite force of $\vec{r}_{2}$ on $\vec{r}_{1}$ being transmitted through the massless rod connecting them. Therefore, for each point mass on the dumbbell, gravitationally we need only take into account the force exerted by the rigid bodies. 
Note that the gravitational force and acceleration on each $r_{i}$ is $-\partial_{\overrightarrow{r}_{i}}U$ and $-\frac{\partial_{\overrightarrow{r}_{i}}U}{M_1x_i}$, respectively. If we let $\delta(\vec{p})$ represent the density function, we have:\\
$\text{\enspace\enspace\enspace\enspace\enspace}
\textstyle \frac{\partial_{\overrightarrow{r}_2}U}{M_1x_2}-\frac{\partial_{\overrightarrow{r}_1}U}{M_1x_1}=-\frac{1}{M_1}\left(\frac{\partial_{\overrightarrow{r}_2}}{x_2}-\frac{\partial_{\overrightarrow{r}_1}}{x_1}\right)\left(\sum_{k=2}^{n}\int_{\overrightarrow{p}\in\mathcal{B}_k}\left(\frac{x_1}{\left\vert \overrightarrow{p}-\overrightarrow{r}_1\right\vert ^2}+\frac{x_2}{\left\vert\overrightarrow{p}-\overrightarrow{r}_2\right\vert^2}\right)\delta \left(\overrightarrow{p}\right) d\overrightarrow{p}\right)$\\
% =-\frac{1}{M_{1}}\left( \frac{\partial \overrightarrow{r}_{2}}{x_{2}}\left( \sum_{k=2}^{n}\int_{\overrightarrow{p}\in \mathcal{B}_{k}}\left( \frac{x_{1}}{\left\vert \overrightarrow{p}-\overrightarrow{r}_{1}\right\vert ^{2}}+\frac{x_{2}}{\left\vert \overrightarrow{p}-\overrightarrow{r}_{2}\right\vert ^{2}}\right) \delta \left( \overrightarrow{p}\right) d\overrightarrow{p}\right) -\frac{\partial _{\overrightarrow{r}_{1}}}{x_{1}}\left( \sum_{k=2}^{n}\int_{\overrightarrow{p}\in\mathcal{B}_{k}}\left( \frac{x_{1}}{\left\vert \overrightarrow{p}-\overrightarrow{r}_{1}\right\vert ^{2}}+\frac{x_{2}}{\left\vert\overrightarrow{p}-\overrightarrow{r}_{2}\right\vert ^{2}}\right) \delta\left( \overrightarrow{p}\right) d\overrightarrow{p}\right) \right)
$\text{\enspace\enspace\enspace\enspace\enspace}\textstyle  =-\frac{1}{M_{1}}\sum_{k=2}^{n}\left( \int_{\overrightarrow{p}\in\mathcal{B}_{k}}\partial _{\overrightarrow{r}_{2}}\left( \frac{\delta \left( \overrightarrow{p}\right)}{\left\vert \overrightarrow{p}-\overrightarrow{r}_{2}\right\vert ^{2}}\right) -\partial _{\overrightarrow{r}_{1}}\left(\frac{\delta \left( \overrightarrow{p}\right) }{\left\vert\overrightarrow{p}-\overrightarrow{r}_{1}\right\vert ^{2}}\right) d\overrightarrow{p}\right)$\\
$\text{\enspace\enspace\enspace\enspace\enspace}\textstyle  =\frac{2}{M_{1}}\sum_{k=2}^{n}\int_{\overrightarrow{p}\in\mathcal{B}_{k}}\left(\frac{\overrightarrow{p}-\overrightarrow{r}_{2}}{\left\vert\overrightarrow{p}-\overrightarrow{r}_{2}\right\vert^{4}}-\frac{\overrightarrow{p}-\overrightarrow{r}_{1}}{\left\vert\overrightarrow{p}-\overrightarrow{r}_{1}\right\vert^{4}}\right)\delta\left(\overrightarrow{p}\right) d\overrightarrow{p}.$\\
%$ =\frac{2}{M_{1}}\sum_{k=2}^{n}\left( \int_{\overrightarrow{p}\in\mathcal{B}_{k}}\left( \frac{\overrightarrow{p}_{x}-\overrightarrow{r}_{2x}}{\left\vert \overrightarrow{p}-\overrightarrow{r}_{2}\right\vert ^{4}}-\frac{\overrightarrow{p}_{x}-\overrightarrow{r}_{1x}}{\left\vert \overrightarrow{p}-\overrightarrow{r}_{1}\right\vert 4}\right)\delta \left( \overrightarrow{p}\right) d\overrightarrow{p},\;\int_{\overrightarrow{p}\in \mathcal{B}_{k}}\left( \frac{\overrightarrow{p}_{y}-\overrightarrow{r}_{2y}}{\left\vert \overrightarrow{p}-\overrightarrow{r}_{2}\right\vert ^{4}}-\frac{\overrightarrow{p}_{y}-\overrightarrow{r}_{1y}}{\left\vert \overrightarrow{p}-\overrightarrow{r}_{1}\right\vert ^{4}}\right) \delta \left( \overrightarrow{p}\right) d\overrightarrow{p}\right)$
Looking at the vertical components, we can calculate the total rotational acceleration for the dumbbell. And, since the dumbbell is horizontal, note that $p_{y}-r_{2y}=p_{y}-r_{1y}$, allowing for some simplification in the following calculation:\\
$\text{\enspace\enspace\enspace\enspace\enspace}\overset{\cdot \cdot}\theta=\frac{2}{M_{1}}\sum_{k=2}^{n}\int_{\overrightarrow{p}\in \mathcal{B}_{k}}\left(\frac{p_{y}-r_{2y}}{\left\vert\overrightarrow{p}-\overrightarrow{r}_{2}\right\vert^{4}}-\frac{p_{y}-r_{1y}}{\left\vert\overrightarrow{p}-\overrightarrow{r}_{1}\right\vert ^{4}}\right) \delta \left( \overrightarrow{p}\right) d\overrightarrow{p}$.
\begin{equation}\label{eq:RigidBody_dumbbell_acceleration}
	\text{\enspace\enspace}=\frac{2}{M_{1}}\sum_{k=2}^{n}\int_{\overrightarrow{p}\in \mathcal{B}_{k}}(p_{y}-r_{1y})\left(\frac{1}{\left\vert\overrightarrow{p}-\overrightarrow{r}_{2}\right\vert^{4}}-\frac{1}{\left\vert\overrightarrow{p}-\overrightarrow{r}_{1}\right\vert ^{4}}\right) \delta \left( \overrightarrow{p}\right) d\overrightarrow{p}.
\end{equation}
Now consider the quadrants determined by the line through the dumbbell's rod and that rod's perpendicular bisector (see Figure \ref{fig:Perp_Biff_Rigid_Integral}).
Since density is always positive, we see that if a particular $\mathcal{B}_{k}$ is in the 4th quadrant, then $|\vec{p}-\vec{r}_2|<|\vec{p}-\vec{r}_1|$,  $p_{y}-r_{iy}<0$ for all $\vec{p}\in\mathcal{B}_{k}$, and the integral will be negative. If a $\mathcal{B}_{k}$ is in the 3rd quadrant, then $|\vec{p}-\vec{r}_1|<|\vec{p}-\vec{r}_2|$,  $p_{y}-r_{iy}<0$ for all $\vec{p}\in\mathcal{B}_{k}$, and the integral will be positive. Similarly, the integral will be negative when a $\mathcal{B}_{k}$ is in the 2nd quadrant and positive in the 1st quadrant. The preceding analysis for discretized bodies is nearly identical, except for the use of summations over the discrete points in \ref{eq:RigidBody_dumbbell_acceleration}, instead of integrals. Therefore, if the $\mathcal{B}_{k}$ are all in the open cone of the 2nd and 4th quadrants, the dumbbell will accelerate clockwise. And if they are in the 1st and 3rd, the dumbbell will accelerate counterclockwise. So, for RE, either both cones are empty or both are occupied.
\end{proof}
\begin{figure}[H]
\captionsetup[subfigure]{justification=centering}
\centering
%\subfloat[$\ell _{1}=\frac{1}{2}$]
{\includegraphics[width=0.35\textwidth]{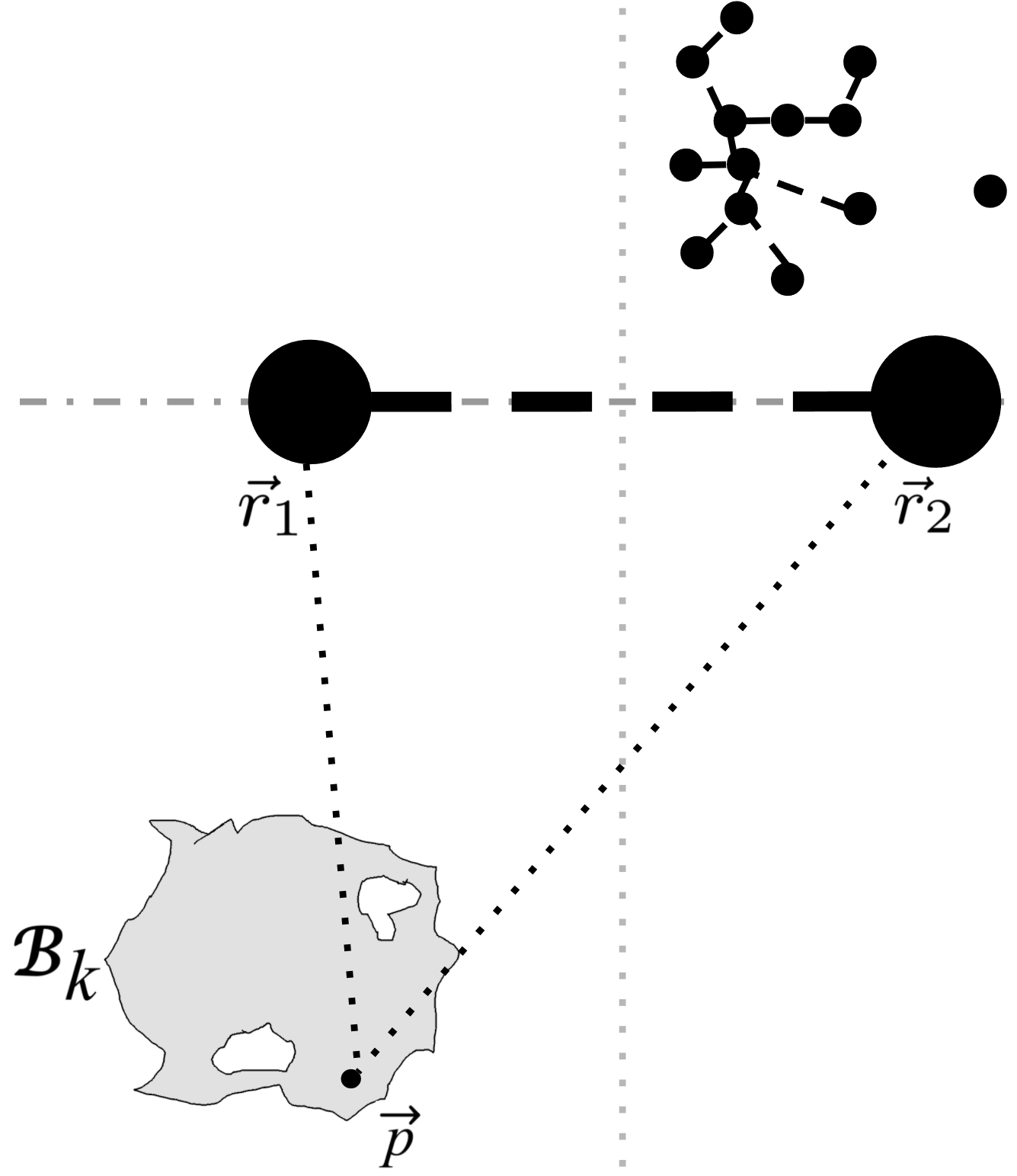}}
\caption{\\\bfseries Dumbbell and Rigid Bodies in Non-Re}{\vspace{0.8em}\footnotesize Comparing the distances between $\vec{p}$ and $\vec{r}_1,\vec{r}_2$.\\[-2.8em]}
\label{fig:Perp_Biff_Rigid_Integral}
\end{figure}
%%%%%%%%%%%%%%%%%%%%%%%%%%%%%%%%%%%%%%%%%%%%%%%%%%%%%%%
\chapter{Conclusion}
\section{Dumbbell/Point Mass Problem}
We verified the RE (colinear and isosceles) and stability for the dumbbell/point mass problem found by Beletskii and Ponomareva. We also found RE, including some stable RE in the overlapped colinear region. We performed bifurcation analyses for all RE in order to characterize qualitatively different regions of the $L^2$ and $x_1 M_1$ parameter space where the number of RE differ.
\paragraph{Colinear}
Our bifurcation analysis discovered that in the non-overlapped colinear configuration, irrespective of $x_1,M_1,$ for small angular momenta there are no RE. For sufficiently large angular momenta there are two RE (one at a closer radius, and another farther away). For the overlap region, when $x_1>x_2$ we found no RE for any angular momenta.  However, we discovered that when $x_1\ll x_2$ we have either one or three RE depending upon the angular momentum, and when $x_1<x_2$ with $x_1$ sufficiently large we have one RE for all angular momenta.

\textbf{Stability}\\
We verified stability found in Beletskii and Ponomareva for the non-overlap region, with sufficiently large radius.
%, as in Beletskii and Ponomareva we find energetic stability for sufficiently large radii.  In particular however, looking at the graph of the square of the angular momentum $L^2(r)$, we find stability when $\partial_{r}L^2>0$. 
In the overlapped region we also found radial intervals of energetic stability for sufficiently small $x_1$. Smaller $x_1$ is associated with larger intervals of stability. For physically realizable angular momenta, linear stability coincided with the energetic stability.
\paragraph{Isosceles}
%For the isosceles configuration, we discovered a minimal radius for RE not identified by Beletskii and Ponomareva: $\frac{\left\vert x_{2}-x_{1}\right\vert}{2}.$ 
In our bifurcation analysis, we found as angular momentum increases, that for a sufficiently massive point mass body there is either no RE, or one RE. For a less massive point mass body, as angular momentum increases we find no RE, then two RE, then one RE (see Figure \ref{fig:1DB_Isos_L_Plot}).\newpage
\textbf{Stability}\\
We found no energetic stability for the isosceles configuration, but for sufficiently large vertical body mass, we find radial intervals of linear stability. Depending on $x_1,M_1$, there are either one or two such radial intervals.
\section{Two Dumbbell Problem}
We found RE and examined stability for symmetric (colinear, perpendicular, equal mass) configurations, as well as asymmetric RE curves bifurcating from them.
\paragraph{Colinear}
We discovered that irrespective of the choice of parameters, with sufficiently low angular momenta $L$, there are no RE. However, for some $L_b>0$, and at some radius $r_b$, two RE bifurcate. And for all angular momenta greater than $L_b$, there are two RE (see Figure \ref{fig:6_3a_no_overlap}).

\textbf{Stability}\\
We showed that the RE of the colinear configuration are stable when the radius is sufficiently large, in particular when $\partial_{r}L^2>0$. We found linear stability coincided with the energetic stability.
\paragraph{Perpendicular}
We proved that for the perpendicular configuration, there is a family of isosceles RE (where $d_{11}=d_{21}$ and $d_{12}=d_{22}$).
%Also, there is a subfamily of rhombus RE for $r<\frac{\ell_2}{2}$ (dumbbells are overlapped) when the $d_{ij}$ are equal.
%\textbf{Rhombus}\\Through a bifurcation analysis we showed that as angular momentum increases, there is first zero, then one RE, then zero RE.
%\textbf{Stability}\\
%We found there is neither energetic nor linear stability for the rhombus case.\textbf{Isosceles}\\
We found RE for every radius of the isosceles family. Numerically, we found there were many qualitatively different $L^2$ bifurcation graphs for the family, depending upon the choice of parameters (see Figure \ref{fig:2D_Perp_Isos}).\\
%For the first configuration, as angular momentum increases we find no RE, then one RE, then three RE, then one RE. For the second configuration, as angular momentum increases we find no RE, and then one RE radially increasing from $r_0=0$ without bound as $L^2\rightarrow\infty$. For the third configuration, as angular momentum increases we find  no RE, then two bifurcating RE, then one RE.
\textbf{Stability}\\
We showed there is no energetic stability for the isosceles family. However, we did find radial intervals of linear stability for for large (resp. small) $M_{1}$ when $(\theta_{1},\theta_2)=(\frac{\pi}{2},0)$ (resp. $(\theta_{1},\theta_2)=(0,\frac{\pi}{2})$) (see Figure \ref{fig:Linearization_Isos_By_M1}).
\paragraph{Equal Mass}
We identified a symmetric equal mass trapezoid RE ($(\theta_{1},\theta_2)=(\frac{\pi}{2},\frac{\pi}{2})$). Additionally, a bifurcation analysis (this time using $r$ as our bifurcation parameter) revealed several families of RE bifurcating from the symmetric ones and curving through $r\theta_{1}\theta_2$-space. While some of these RE are nonphysical due to their complex or unbounded angular momenta, some are not. In particular, for $\ell_{1}\neq\frac{\pi}{2}$ we found a curve bifurcating from the trapezoid RE and subsequently merging with the perpendicular RE (see Figure \ref{fig:EQMassTrapToPerpBif_l10p75}), and a curve bifurcating from the colinear RE and subsequently merging with the perpendicular RE (see Figure \ref{fig:EQMassColToPerpBif_l10p75}). For $\ell_{1}=\frac{1}{2}$ we found a curve bifurcating from a collision of the two dumbbells, and subsequently merging with the perpendicular RE (see Figure \ref{fig:EQMassPerpBifurcation_l10p5}).
%Regarding the number of RE in regions of the parameter space, the results for the symmetric solutions are found above.  For the 

Regarding the $L^2$ bifurcation, we found 3 possibilities for the trapezoid configuration.  When $M_1<\frac{3}{4}$ or with $\ell_1$ large, we have one minimum, and as $L^2$ increases, we have zero, then two, then one RE. When $M_1>\frac{3}{4}$ and $\ell_1$ small, we have no minimum, and as $L^2$ increases, we have zero, then one RE. And when $\ell_1=\frac{1}{2}$, we have zero, then two RE as $L^2$ increases. For the bifurcating families of asymmetric (and physically realizable) RE, we have exactly one RE for each angular momentum within the range of angular momenta occurring in these curves.

\textbf{Stability}\\
The trapezoid configuration and asymmetric RE showed no energetic stability. The asymmetric RE also showed no linear stability. However, we did find linear stability for the trapezoid case for small $r$ when $\ell_{1}\neq\frac{1}{2}$.
\paragraph{Perpendicular Bisector Theorem for a Dumbbell and Planar Rigid Bodies}
We also proved an extension of the Conley Perpendicular Bisector Theorem.\\ Let a dumbbell $\overline{r_1r_2}$ and one or more planar rigid bodies $\mathcal{B}_2,...,\mathcal{B}_n$ be in a planar RE. Then, if one of the two open cones determined by the lines through $\overline{r_1r_2}$ and its perpendicular bisector contains one or more rigid bodies, the other open cone cannot be empty.
\appendix\chapter*{Appendices}
\addcontentsline{toc}{chapter}{Appendices}
\renewcommand{\thesection}{\Alph{section}}
\section{Notation Used in This Paper}
\label{appendix:Notation Used in This Paper}
\begin{tabular}{|l|l|}
\hline
\textbf{Notation}   & \textbf{Meaning}\\ \hline
$O$& Origin, also system's center of mass                                              \\ \hline
$\mathcal{B}_i$& Gravitational bodies in the system \\ \hline
$\vec{C},\vec{C}_i$& Location of center of mass for bodies \\ \hline
%$\vec{C}_{Tk}$& Location of center of mass for all bodies except $\mathcal{B}_k$ \\ \hline
%$R_{k}$& $\vert\vec{C}_k\vec{C}_{Tk}\vert$, ``radius'' from kth body to center of (other) system mass \\ \hline
$\vec{r}_p$& Location of point mass body\\ \hline
$\vec{r}_i,\vec{r}_{ij},\vec{r}_{kj}$& Locations of points on dumbbell or discretized body\\ \hline
$M_i$& Mass of body $i$ \\ \hline
$\vec{r}, r$& Vector and distance between bodies' centers of mass\\ \hline
$R$& distance between bodies' centers of mass after change of variable\\ \hline
$\ell_i$& Length of massless rod connecting point masses $\vec{r}_{ij}$ on dumbbell \\ \hline
$\phi$& Acute angle between positive horizontal axis and $\vec{r}$\\ \hline
$\theta,\theta_i$& Angle between $\vec{r}$ and dumbbell's rod\\ \hline
$\vec{v}_i, \vec{v}_{ij}$& Velocity with which mass $\vec{r}_{i},\vec{r}_{ij}$ is moving  \\ \hline
$\mathcal{L},T,U,V$& Lagrangian, kinetic, potential energy, and amended potential \\ \hline
$d_{i},d_{ij}$ & Distance from $\vec{r}_{p}$ to $\vec{r}_{i}$, and from $\vec{r}_{1i}$ to $\vec{r}_{2j}$, respectively \\ \hline
$x_{i},x_{ij}$ & Dumbbell mass ratios  \\ \hline
$B,B_{i}$& Moment of inertia for dumbbells\\ \hline
%$\vec{F}_{gikj}$& Gravitational force on $\vec{r}_i$ by $\vec{r}_{kj}$\\ \hline
$\vec{F}_{c_i}$ & Centrifugal force experienced by $\vec{r}_i$\\ \hline
%$\vec{a}_{gikj}$ & Gravitational acceleration on $\vec{r}_i$ due to $\vec{r}_{kj}$\\ \hline
%$\vec{a}_{ci}$ & Centrifugal acceleration on ith dumbbell mass\\ \hline
L & Scalar angular momentum\\ \hline
$\mathcal{R}_{C},\mathcal{R}_{P_{1,2}},\mathcal{R}_{T}$ & Equal mass colinear, perpendicular, and trapezoid RE\\ \hline
\setlength\tabcolsep{0pt}\begin{tabular}{ll} $\mathcal{B}_{CC^{\pm}},\mathcal{B}_{TP^{\pm}},$ \\ $\mathcal{B}_{PC^{\pm}},\mathcal{B}_{CP^{\pm}}$ \end{tabular}& Equal mass $\ell_1\neq\frac{1}{2}$ RE curves bifurcating from symmetric RE\\ \hline
\setlength\tabcolsep{0pt}\begin{tabular}{ll} $\mathcal{B}_{C^{\pm}},\mathcal{B}_{T^{\pm}},$ \\ $\mathcal{B}_{LP^{\pm}},\mathcal{B}_{RP^{\pm}}$ \end{tabular}& Equal mass $\ell_1=\frac{1}{2}$ RE curves bifurcating from symmetric RE\\ \hline
$\delta(\vec{p})$ & Density function\\ \hline
\end{tabular}

\end{document}